\documentclass[11pt]{article}
\usepackage{amsmath,amsthm,amsfonts,amssymb,mathrsfs,bm,wasysym}
\usepackage{epsfig}
\usepackage[usenames]{color}
\usepackage[usenames,dvipsnames]{xcolor}
\usepackage{verbatim}
\usepackage{multicol}
\usepackage{comment}
\usepackage{graphicx}
\usepackage{centernot}
\usepackage{tikz}
\usepackage{color}
\usepackage{enumerate}
\usepackage{comment}
\usepackage{thmtools}
\usepackage{float}
\usepackage{caption}
\usepackage{subcaption}
\usepackage{textcomp}
\usepackage[normalem]{ulem}
\usepackage[utf8]{inputenc} 
\usepackage[T1]{fontenc}
\usepackage[color=green, textsize=tiny]{todonotes}
\usepackage{hyperref}
\usepackage{titlesec}

\parskip \medskipamount
\titlespacing*{\subsubsection}
{0pt}{0.4\baselineskip}{0.2\baselineskip}

\declaretheoremstyle[
qed={//}
]{defstyle}

\newtheorem{theorem}{Theorem}[section]
\declaretheorem[style=defstyle,sibling=theorem]{definition}
\newtheorem{assumption}[theorem]{Assumption}
\numberwithin{equation}{section}
\newtheorem{lemma}[theorem]{Lemma}
\newtheorem{proposition}[theorem]{Proposition}
\newtheorem{corollary}[theorem]{Corollary}

\newtheorem{remark}[theorem]{Remark}

\newtheorem*{claim*}{Claim}
\newtheorem*{lemma*}{Lemma}

\newtheorem{conjecture}[theorem]{Conjecture}
\newtheorem{open_problem}[theorem]{Open Problem}
\numberwithin{equation}{section}

\newcommand{\Var}{{\mathrm{Var}}}
\newcommand{\1}{{\text{\Large $\mathfrak 1$}}}

\newcommand{\Z}{\mathbb{Z}}
\newcommand{\Zpos}{\Z_{\ge1}}
\newcommand{\Znonneg}{\Z_{\ge0}}
\newcommand{\R}{\mathbb{R}}
\newcommand{\Rpos}{\R_{>0}}

\newcommand{\tree}{T}
\newcommand{\quasi}{\text{quasi-tree}}
\newcommand{\rball}[1]{#1-{$R$}-\text{ball}}
\newcommand{\til}{\widetilde}

\newcommand{\trel}{t_{\mathrm{rel}}}
\newcommand{\trelabs}{\trel^{\mathrm{abs}}}
\newcommand{\tmix}[1]{t_{\mathrm{mix}}\left(#1\right)}
\newcommand{\hit}[2]{\mathrm{hit}_{#1}\left(#2\right)}
\newcommand{\tunif}[1]{t_{\mathrm{unif}}\left(#1\right)}
\newcommand{\tauLR}{\tau_{\mathrm{LR}}}
\newcommand{\tmixtext}{t_{\mathrm{mix}}}
\newcommand{\hittext}{\mathrm{hit}}

\newcommand{\tavetext}{t_{\mathrm{ave}}}
\newcommand{\diam}[1]{\mathrm{diam}\left(#1\right)}
\DeclareMathOperator{\argmin}{argmin}

\newcommand{\B}{\mathcal{B}}
\newcommand{\h}{\mathfrak{h}}
\newcommand{\V}{\mathfrak{V}}
\newcommand{\F}{\mathcal{F}}
\newcommand{\I}{\mathcal{I}}
\newcommand{\U}{\mathcal{U}}
\newcommand{\Dirform}{\mathcal{E}}

\newcommand{\trunc}[2]{\mathrm{Tr}(#1,#2)}
\newcommand{\trunctil}[2]{\widetilde{\mathrm{Tr}}(#1,#2)}
\newcommand{\truncprime}[1]{\mathrm{Tr}'(#1)}

\newcommand{\Geom}[1]{\mathrm{Geom}\!\left(#1\right)}
\newcommand{\Geompos}[1]{\mathrm{Geom}_{\ge1}\!\left(#1\right)}
\newcommand{\Geomnonneg}[1]{\mathrm{Geom}_{\ge0}\!\left(#1\right)}
\newcommand{\Bin}[2]{\mathrm{Bin}\!\left(#1,#2\right)}
\newcommand{\NegBin}[2]{\mathrm{NegBin}\!\left(#1,#2\right)}

\newcommand{\pr}[1]{\mathbb{P}\!\left(#1\right)}
\newcommand{\E}[1]{\mathbb{E}\!\left[#1\right]}
\newcommand{\estart}[2]{\mathbb{E}_{#2}\!\left[#1\right]}
\newcommand{\prstart}[2]{\mathbb{P}_{#2}\!\left(#1\right)}
\newcommand{\prcond}[3]{\mathbb{P}_{#3}\!\left(#1\;\middle\vert\;#2\right)}
\newcommand{\econd}[2]{\mathbb{E}\!\left[#1\;\middle\vert\;#2\right]}

\newcommand{\vr}[1]{\mathrm{Var}\left(#1\right)}
\newcommand{\vrc}[2]{\mathrm{Var}\left(#1\;\middle\vert\;#2\right)}
\newcommand{\vrstart}[2]{\mathrm{Var}_{#2}\left(#1\right)}

\newcommand{\cvc}[3]{\mathrm{Cov}\left(#1,#2\;\middle\vert\;#3\right)}
\newcommand{\cvstart}[3]{\mathrm{Cov}_{#3}\left(#1,#2\right)}

\newcommand{\Div}[2]{\mathrm{D}\left(#1||#2\right)}

\newcommand{\dtv}[2]{d_{\mathrm{TV}}\left({#1},{#2}\right)}

\newcommand{\dinfty}[2]{\left|\left|{#1}-{#2}\right|\right|_{\infty}}

\newcommand{\eqdist}{\stackrel{\mathrm{d}}{=}}

\newcommand{\eps}{\varepsilon}

\newcommand{\lazytext}{\mathrm{lazy}}

\begin{document}
\title{\bf Phase transition for random walks on graphs with added weighted random matching}
\author{Zsuzsanna Baran
\thanks{University of Cambridge, Cambridge, UK. {zb251@cam.ac.uk}}
\and
Jonathan Hermon
\thanks{University of British Columbia, Vancouver, CA. {jhermon@math.ubc.ca}. }
\and An{\dj}ela \v{S}arkovi\'c
\thanks{University of Cambridge, Cambridge, UK. {as2572@cam.ac.uk}}
\and Perla Sousi
\thanks{
University of Cambridge, Cambridge, UK. {p.sousi@statslab.cam.ac.uk}}
}
\date{}
\maketitle
\vspace{-1cm}

\begin{abstract}
For a finite graph $G=(V,E)$ let $G^*$ be obtained by considering a random perfect matching of $V$ and adding the corresponding edges to $G$ with weight $\eps$, while assigning weight 1 to the original edges of $G$. We consider whether for a sequence $(G_n)$ of graphs with bounded degrees and corresponding weights $(\eps_n)$, the (weighted) random walk on $(G_n^*)$ has cutoff. For graphs with polynomial growth we show that  $\log\left(\frac{1}{\eps_n}\right)\ll\log|V_n|$ is a sufficient condition for cutoff. Under the additional assumption of vertex-transitivity we establish that this condition is also necessary. For graphs where the entropy of the simple random walk grows linearly up to some time of order $\log|V_n|$ we show that $\frac{1}{\eps_n}\ll\log|V_n|$ is sufficient for cutoff. In case of expander graphs we also provide a complete picture for the complementary regime $\frac{1}{\eps_n}\gtrsim\log|V_n|$.
\end{abstract}

\section{Introduction}
This paper is motivated by the question of what little random perturbation one can apply to a given sequence $(G_n)$ of graphs so that the random walk on the resulting graphs $(G_n^*)$ will exhibit cutoff whp.

\begin{definition}\label{def:Gstar}
Let $(G_n=(V_n,E_n))$ be a sequence of connected graphs, with $|V_n|\to\infty$ as $n\to\infty$ and $|V_n|$ even for each $n$, and let $(\eps_n)$ be a sequence taking values in $(0,1)$. We define the sequence $(G_n^*)$ of weighted random graphs as follows. We let $G_n^*$ have vertex set $V_n$ and edge set $E_n\sqcup E^*_n$ where $E^*_n$ is a uniformly random perfect matching of $V_n$, and let each edge in $E_n$ have weight 1 and each edge in $E^*_n$ have weight $\eps_n$. (Note that $G_n^*$ might have multiple edges.)\\
If $|V_n|$ is odd then we can define $G_n^*$ similarly, considering matchings of $V_n$ that leave one vertex unmatched. For simplicity we will assume in the proofs that $|V_n|$ is even.
\end{definition}

An unweighted version of the above model (i.e.\ having $\eps_n\equiv1$) was considered in 2013 by Diaconis~\cite{Diaconis_question} who asked about the order of the mixing time of the random walks on $G_n^*$. In 2022~\cite{random_matching} proved that for a sequence of connected graphs $(G_n)$ with uniformly bounded degrees, the walks on $(G_n^*)$ exhibit cutoff whp at a time of order $\log|V_n|$.

If $\eps_n$ is of constant order, it follows analogously to~\cite{random_matching} that the walks on $(G_n^*)$ exhibit cutoff whp. On the other hand, it is easy to see intuitively that if $\eps_n$ is sufficiently small so that by the mixing time of $G_n$ the walk is unlikely to cross any of the added edges $E^*_n$ of $G^*_n$, then the added edges do not affect the occurrence of cutoff, i.e.\ $(G_n^*)$ exhibits cutoff if and only if $(G_n)$ does.
In this paper we seek to understand what happens between these two regimes.

Intuitively, one would expect that if the random walk on $(G_n)$ did not exhibit cutoff, then for the added edges to introduce cutoff, the walk on $G_n^*$ would need to cross a diverging number of the added edges up to its mixing time. In fact this can be verified very similarly to the proof of Proposition~\ref{pro:smaller_eps}. So in order to introduce cutoff it is a necessary condition to have $\tmixtext^{G^*_n}\gtrsim\frac1{\eps_n}$. A very closely related condition, that is often easily verifiable, would be that for any starting vertex the entropy of the first added edge crossed by the random walk is of strictly smaller order than $\log|V_n|$. (See Section~\ref{sec:general_graphs} for more discussion on this condition.) In case of two very large families of graphs (graphs with polynomial growth and graphs with linear growth of entropy) we prove that this latter assumption is sufficient for $(G_n^*)$ to have cutoff. In case of two families (vertex-transitive graphs with polynomial growth and expander graphs) we also prove that if the entropy is of strictly larger order than $\log|V_n|$ then the added edges do not affect whether cutoff occurs, which can be viewed as a phase transition.

Before stating the main results, we recall the definition of total variation mixing time and cutoff.

Given a finite connected graph $G$ we define the total variation mixing time of the discrete-time random walk $X$ on $G$ as follows. Let $P$ be the transition matrix of $X$ and let $\pi$ be the corresponding invariant distribution. Then for any $\theta\in(0,1)$ we define
\[\tmixtext^G(\theta):=\quad\min\left\{t:\dtv{P^t(x,\cdot)}{\pi(\cdot)}\le\theta\text{ for all }x\right\},\]
where the total variation distance $\dtv{\mu}{\nu}$ of distributions $\mu$ and $\nu$ on state space $S$ is defined as
\[\dtv{\mu}{\nu}:=\quad\frac12\sum_{x\in S}\left|\mu(x)-\nu(x)\right|.\]
We say that a sequence $(G_n)$ of graphs exhibits cutoff at time $t_n$ with window of order $s_n$ if $\frac{s_n}{t_n}\to0$ as $n\to\infty$, and for every $\theta\in(0,1)$ there exists a constant $c(\theta)$ such that for all $n$ we have
\begin{align}\label{eq:cutoff_def}
t_n-c(\theta)s_n \quad\le\quad\tmixtext^{G_n}(\theta)\quad\le\quad t_n+c(\theta)s_n.
\end{align}

For a random sequence $(G_n)$ of graphs we say that $(G_n)$ exhibits cutoff with high probability (whp) if \eqref{eq:cutoff_def} holds with probability tending to 1 as $n\to\infty$.
The lazy version of a random walk with transition matrix $P$ is defined to have transition matrix $P_{\lazytext}:=\frac12(I+P)$. For lazy random walks we define the mixing time and cutoff analogously and write \emph{lazy} in the superscript.

We also recall that the entropy of a random variable $W$ talking values in a countable set $\mathcal{W}$ is defined as 
\[H(W):=\quad\sum_{w\in\mathcal{W}}\pr{W=w}\left(-\log\pr{W=w}\right),\]
where $p(-\log p)$ is considered to be $0$ when $p=0$.

Given two functions $f,g:\Znonneg\to\Rpos$ or $f,g:\Rpos\to\Rpos$ we write $f(t)\lesssim g(t)$ if there exists a constant $c>0$ such that we have $f(t)\le cg(t)$ for all sufficiently large values of $t$. We write $f(t)\ll g(t)$ if we have $\frac{f(t)}{g(t)}\to0$ as $t\to\infty$. We define $\gtrsim$ and $\gg$ analogously. We write $f(t)\asymp g(t)$ if we have $f(t)\lesssim g(t)$ and $f(t)\gtrsim g(t)$.
In case $f,g$ take negative values, we write $f(t)\lesssim g(t)$ if $(-f(t))\gtrsim(-g(t))$, and we define $\gtrsim$, $\ll$, $\gg$ and $\asymp$ analogously.

\subsection{Results}

Now we state the main results of the paper.

\begin{theorem}\label{thm:results_lin_entropy}
Let $\Delta$, $a$, $b$ and $c$ be given positive constants. Let $(\eps_n)$ be a sequence of constants in $(0,1)$ and let $(G_n)$ be a sequence of connected graphs with the following properties: $|V_n|\to\infty$ as $n\to\infty$, all vertices have degree $\le\Delta$, and for any $n$ and any $x_n\in V_n$ the entropy of the simple random walk on $G_n$ from $x_n$ at any time $t\le a\log|V_n|$ satisfies $H(X_t)\in[bt,ct]$. Let $G_n^*$ be as in Definition~\ref{def:Gstar}.
\begin{enumerate}[(a)]
\item If $\eps_n\gg\frac{1}{\log|V_n|}$ then the random walk on $(G_n^*)$ exhibits cutoff whp, at a time of order $\log|V_n|$.
\item If $\eps_n\ll\frac{1}{\tmixtext^{G_n}(\theta)}$ for some $\theta\in(0,1)$, then the random walk on $(G_n^*)$ exhibits cutoff if and only if the random walk on $(G_n)$ does.
\item If $\eps_n\asymp\frac{1}{\tmixtext^{G_n,\lazytext}\left(\frac14\right)}$ and the lazy random walk on $(G_n)$ does not exhibit cutoff, then neither does the random walk on $(G_n^*)$.
\end{enumerate}
\end{theorem}

\begin{theorem}\label{thm:results_poly_balls}
Let $\Delta$ be a given positive constant and let $p$ be a given polynomial. Let $(\eps_n)$ be a sequence of constants in $(0,1)$ and let $(G_n)$ be a sequence of connected graphs with the following properties: $|V_n|\to\infty$ as $n\to\infty$, all vertices have degree $\le\Delta$, and for any $r$ the volume of any ball of radius $r$ in any $G_n$ is upper bounded by $p(r)$. Let $G_n^*$ be as in Definition~\ref{def:Gstar}.
\begin{enumerate}[(a)]
\item If $\eps_n\gtrsim|V_n|^{-o(1)}$ then the random walk on $(G_n^*)$ exhibits cutoff with high probability, at a time of order $\frac{1}{\eps_n}\frac{\log|V_n|}{\log\left(\frac{1}{\eps_n}\right)}$.
\item If $\eps_n\ll\frac{1}{\tmixtext^{G_n}(\theta)}$ for some $\theta\in(0,1)$, then the random walk on $(G_n^*)$ exhibits cutoff if and only if the random walk on $(G_n)$ does.
\item If $\eps_n\asymp\frac{1}{\tmixtext^{G_n,\lazytext}\left(\frac14\right)}$ and the lazy random walk on $(G_n)$ does not exhibit cutoff, then neither does the random walk on $(G_n^*)$.
\end{enumerate}
\end{theorem}

\begin{remark}
We will refer to graphs satisfying the assumptions of Theorem~\ref{thm:results_lin_entropy} as graphs having linear growth of entropy, while we refer to graphs satisfying the assumptions of Theorem~\ref{thm:results_poly_balls} as graphs having polynomial growth of balls.
\end{remark}

We say that a sequence of graphs $(G_n)$ is an expander family if $|V_n|\to\infty$ as $n\to\infty$, all vertices have degree $\le\Delta$ for some constant $\Delta$, and there exists a constant $a>0$ such that for any set $A_n\subseteq V_n$ with $|A_n|\le\frac12|V_n|$ we have $|\partial A_n|\ge a|A_n|$, where $\partial A_n$ is the set of edges between $A_n$ and $V_n\setminus A_n$ in $G_n$.

In Section~\ref{sec:specific_graphs} we present a proof of the standard fact that expander graphs have a linear growth of entropy. In this case we get a phase transition for the occurrence of cutoff around the value $\eps_n=\frac{1}{\log|V_n|}$ as the following result shows.

\begin{theorem}\label{thm:results_expanders}
Let $(\eps_n)$ be a sequence of constants in $(0,1)$ and let $(G_n)$ be a family of connected expander graphs. Let $G_n^*$ be as in Definition~\ref{def:Gstar}.
\begin{enumerate}[(a)]
\item If $\eps_n\gg\frac{1}{\log|V_n|}$ then the random walk on $(G_n^*)$ exhibits cutoff with high probability, at a time of order $\log|V_n|$.
\item If $\eps_n\ll\frac{1}{\log|V_n|}$ then the random walk on $(G_n^*)$ exhibits cutoff if and only if the random walk on $(G_n)$ does.
\item If $\eps_n\asymp\frac{1}{\log|V_n|}$ and the lazy random walk on $(G_n)$ does not exhibit cutoff, then neither does the random walk on $(G_n^*)$.
\item There exists a sequence of expanders $(G_n)$ such that the lazy random walk on $(G_n)$ exhibits cutoff, and for any sequence $(\eps_n)$ with $\eps_n\asymp\frac{1}{\log|V_n|}$, the random walk on $(G_n^*)$ also exhibits cutoff whp.
\item There exists a sequence of expanders $(G_n)$ such that the lazy random walk on $(G_n)$ exhibits cutoff, but for any sequence $(\eps_n)$ with $\eps_n\asymp\frac{1}{\log|V_n|}$, the random walk on $(G_n^*)$ does not exhibit cutoff whp.
\end{enumerate}
\end{theorem}

\begin{theorem}\label{thm:results_poly_growth_vertex_transitive}
Let $\Delta$ be a given positive integer and let $p$ be a given polynomial. Let $(\eps_n)$ be a sequence of constants in $(0,1)$ and let $(G_n)$ be a sequence of connected graphs with the following properties: each $G_n$ is vertex-transitive, $|V_n|\to\infty$ as $n\to\infty$, all vertices have degree $\Delta$, and for any $r$ the volume of any ball of radius $r$ in any $G_n$ is upper bounded by $p(r)$. Let $G_n^*$ be as in Definition~\ref{def:Gstar}.
\begin{enumerate}[(a)]
\item If $\eps_n\gtrsim|V_n|^{-o(1)}$ then the random walk on $(G_n^*)$ exhibits cutoff with high probability, at a time of order $\frac{1}{\eps_n}\frac{\log|V_n|}{\log\left(\frac{1}{\eps_n}\right)}$.
\item\label{part:poly_growth_no_cutoff} If $\eps_n\lesssim|V_n|^{-a_n}$ where $a_n\asymp1$, then whp the random walk on $(G_n^*)$ does not exhibit cutoff.
If in addition we have $\eps_n\gg \frac{1}{\diam{G_n}^2}$ where $\diam{G_n}$ is the diameter of $G_n$, then whp the mixing time of $G_n^*$ satisfies $\tmixtext^{G_n^*}\left(\frac{1}{4}\right)\asymp \frac{1}{\eps_n}$.
\end{enumerate}
\end{theorem}

We also obtain a result for general families of graphs with bounded degrees, as follows.

\begin{theorem}\label{thm:results_general_graphs}
Let $\Delta$ be a given positive constant. Let $(\eps_n)$ be a sequence of constants in $(0,1)$ and let $(G_n)$ be a sequence of connected graphs such that $|V_n|\to\infty$ as $n\to\infty$ and all vertices have degree $\le\Delta$. Let $G_n^*$ be as in Definition~\ref{def:Gstar}.
\begin{enumerate}[(a)]
\item\label{cond:eps_gg_loglogn_over_logn} If $\eps_n\gg\frac{\log\log|V_n|}{\log|V_n|}$ then the random walk on $(G_n^*)$ exhibits cutoff with high probability.
\item\label{cond:general_graph_small_eps} If $\eps_n\ll\frac{1}{\tmixtext^{G_n}(\theta)}$ for some $\theta\in(0,1)$, then the random walk on $(G_n^*)$ exhibits cutoff if and only if the random walk on $(G_n)$ does.
\item\label{cond:general_graph_eps_one_over_tmix} If $\eps_n\asymp\frac{1}{\tmixtext^{G_n,\lazytext}\left(\frac14\right)}$ and the lazy random walk on $(G_n)$ does not exhibit cutoff, then neither does the random walk on $(G_n^*)$.
\end{enumerate}
\end{theorem}

\begin{remark}
We conjecture that the above condition in~\eqref{cond:eps_gg_loglogn_over_logn} can be improved and in fact the random walk on $(G_n^*)$ exhibits cutoff for any sequence $\eps_n\gg\frac{1}{\log|V_n|}$. See Section~\ref{subsec:conjectures} for further discussion.
\end{remark}

\subsection{Relation to other works}

In this work we establish cutoff for random walks on randomly generated graphs at an entropic time. There have been multiple recent works proving cutoff at an entropic time, including~\cite{RWs_on_random_graph}, \cite{comparing_mixing_times_sparse_random_graphs}, \cite{cutoff_for_almost_all_RWs_on_Abelian_groups}, \cite{speeding_up_MCs}, \cite{cutoff_random_lifts}, \cite{cutoff_at_entorpic_time_RWs_on_covered_expanders}, \cite{cutoff_at_entropic_time_sparse_MCs}, \cite{mixing_time_of_Chung-Diaconis-Graham_process} and~\cite{random_matching}. For a more detailed overview please refer to~\cite{random_matching}. For some exciting recent progress which extends the connection between cutoff and entropic concentration to non-random graphs, see~\cite{non-negatively_curved_MCs}, \cite{varentropy} and~\cite{entropic_pf_of_cutoff_Ramanujan_graphs}.

Our model is a generalisation of the model in~\cite{random_matching}, where $\eps_n\equiv1$. In both cases the graph $G^*$ locally looks like a tree-like structure, as we explain in the overview below. In the case when $\eps_n$ is of constant order or goes to 0 sufficiently slowly, the proofs in~\cite{random_matching} become more technical, but can be adapted to prove cutoff. To establish cutoff for the full range of $\eps_n$ as in Theorems~\ref{thm:results_lin_entropy} and~\ref{thm:results_poly_balls} we need to use a different method. (We explain in more detail why this is necessary in Section~\ref{subsec:overview}.)

Our approach is inspired by~\cite{cutoff_random_regular_graphs} and~\cite{cutoff_NBRW_on_sparse_random_graphs} that establish cutoff for non-backtracking random walks. To approximate the time $2t$ transition probability between two vertices $x$ and $y$ of the random graph they consider reversing the second half of a length $2t$ path and study the first $t$ steps of two independent walks from $x$ and $y$. These in turn can be approximated by the first $t$ steps of two independent walks on two independent copies of the limiting tree. As far as we are aware this method has only been used for non-backtracking walks at fixed times, which makes the reversal of the second half of the path straightforward and allows to get a good control over the position of the walk on the tree.

In our work we use this idea for a simple random walk, and instead of looking at a fixed time we use it for a random time $\tau$. This means that we face additional challenges regarding both the reversal of paths and the study of a walk on the limiting tree-like structure. We prove that at the random time $\tau$ the position of the walk is close to the uniform distribution on the vertices, which has bounded $\ell^{\infty}$ distance from the stationary distribution, and we also prove that $\tau$ is concentrated around some given time. To conclude cutoff our approach relies crucially on the connection between cutoff and concentration of hitting times of large sets established in~\cite{characterisation_of_cutoff}. To the best of our knowledge, this is the first time the results of~\cite{characterisation_of_cutoff} have been utilised in this fashion. We believe this method and variants of the arguments in this paper can also be used to analyse the random walk on more involved random graph models.

In case $\eps_n\ll\frac{1}{\tmixtext^{G_n}}$ it is intuitively clear whether or not the walk on $G^*_n$ has cutoff and we formalise this intuition. In Theorem~\ref{thm:results_poly_growth_vertex_transitive}~\eqref{part:poly_growth_no_cutoff} there is a regime with $\eps_n\gtrsim\frac{1}{\tmixtext^{G_n}}$ but no cutoff and the proof of this is quite demanding. Here we need to use a different argument relying on reversing the second half of a path and considering walks on independent tree-like structures.

\subsection{Overview}\label{subsec:overview}

Below we give an overview of the methods we use to establish an upper bound on the mixing time in the cutoff regime, which is the most difficult part of the proof.

In many commonly used random graph models (e.g.\ in Erd\H os-R\'enyi graphs $G_n\sim\mathcal{G}\left(n,\frac{d}{n}\right)$) the graph can be approximated locally with a Galton-Watson tree. In our case the situation is not this simple; the graph $G^*$ also retains some of the original structure of $G$, so it does not quite look like a tree, but it can still be approximated locally with a tree-like structure.

Similarly to~\cite{random_matching} we define the random \emph{quasi-tree} $T$ corresponding to a graph $G$, radius $R$ and weight $\eps$ as follows. We consider a ball $B$ of radius $R$ around a uniformly chosen vertex $\rho$ of $G$. For each vertex $v$ of the ball, except for $\rho$, we draw a new edge from $v$ and attach an independently sampled copy of $B$ to the other end. We repeat this for all vertices of the newly added balls, except their centres. Then proceed similarly, resulting in an infinite graph. We call the edges joining different balls of $T$ \emph{long-range edges} and we assign weight $\eps$ to them. We call $\rho$ the root of $T$. We sometimes refer to the balls in $T$ as $R$-balls. We will usually work with the "long-range distance" on $T$. (For vertices $x$ and $y$ this is defined as the minimum number of long-range edges a path from $x$ to $y$ has to cross.)

\begin{figure}[h]
\centering
\includegraphics[width=70mm]{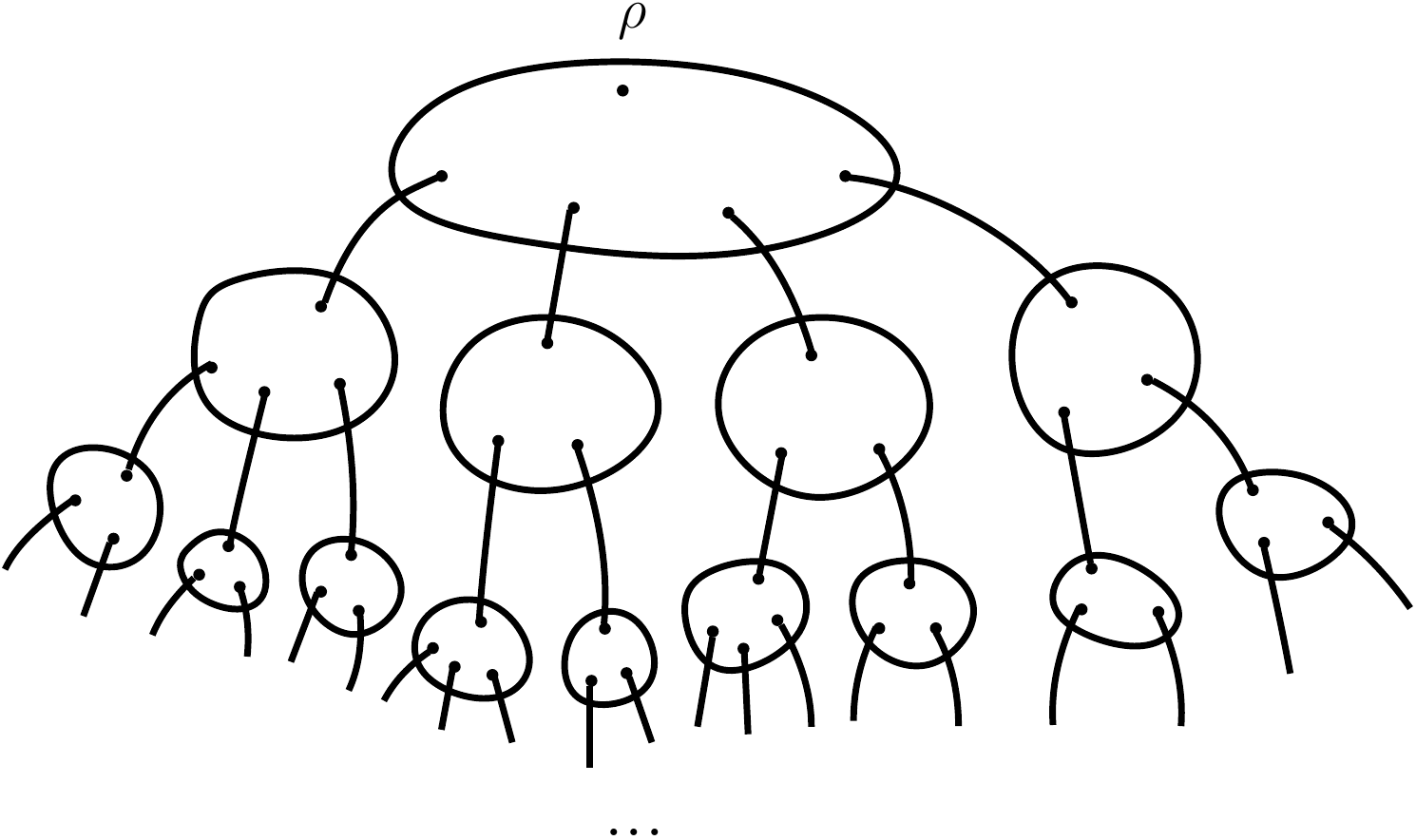}
\caption{An illustration of a quasi-tree.}
\label{fig:eg_quasi_tree}
\end{figure}

Note that the offsprings of each ball in $T$ are independent and identically distributed (iid), which will make it easier to study the behaviour of a random walk $\til{X}$ on $T$ than studying a random walk $X$ on $G^*$ directly. We will show that locally $G^*$ and a random walk $X$ on it can be well approximated by a quasi-tree $T$ and a random walk $\til{X}$ on it from its root.

In~\cite{random_matching} the proof proceeds as follows. The first step is to obtain a concentration result on the speed and entropy of a random walk $\til{X}$ on the quasi-tree $T$. After this a coupling is defined between the random graph $G^*$ and the walk $X$ on it, and the random quasi-tree $T$ and the walk $\til{X}$. Using these it is then proved that at the entropic time, which is of order $\log n$, the $L_2$-distance between the distribution of the walk $X$ conditioned on a certain typical event and the invariant distribution $\pi_{G^*}$ is bounded by $\exp\left(c\sqrt{\log n}\right)$ for a positive constant $c$. After that the proof of the upper bound on the mixing time can be concluded by showing that the absolute relaxation time is of constant order and using the Poincar\'e inequality.

In our model we are also able to establish concentration results for the speed and entropy of the walk on the quasi-tree (see Lemma~\ref{lem:speed} and Proposition~\ref{pro:aux_entropy}). The bounds on the fluctuations of speed and entropy and the upper bound on the absolute relaxation time are all functions of $\eps$, so the above approach would only work for weights $\eps$ approaching 0 sufficiently slowly. To prove cutoff for $\eps$ decaying faster, we need a different approach.

More precisely in order for the method of~\cite{random_matching} to work one requires an entropic concentration estimate, which is analogous to the varentropy condition of~\cite[Theorem 5]{non-negatively_curved_MCs}. Namely, the required estimate is that the standard deviation of the random variable whose mean is the entropy of the law of the random walk on the quasi-tree at the entropic time is $o\left(\frac{\tmixtext}{\trel}\right)$. In case of some graphs covered by Theorems~\ref{thm:results_lin_entropy} and~\ref{thm:results_poly_balls} the deviations of the entropy are too large.\footnote{For Theorem~\ref{thm:results_lin_entropy} the examples are graphs with Cheeger constant $O\left(\eps\right)$, for example a lamplighter graph on a $d\ge3$ dimensional torus or a locally expanding graph in the sense of Proposition~\ref{pro:local_expanders_lin_entropy} with a bottleneck. In this case one can show that $\trel(G^*)\asymp\frac1\eps$. For Theorem~\ref{thm:results_poly_balls} one can consider two graphs of comparable sizes with polynomial growth of balls whose ratio of entropy is bounded away from 1 (e.g.\ two tori of different dimensions), connected by relatively few edges.}
For the graphs from Theorem~\ref{thm:results_poly_growth_vertex_transitive} one can prove stronger entropic concentration estimates which satisfy the above varentropy condition, however because in Theorem~\ref{thm:results_poly_growth_vertex_transitive} we consider very small values of $\eps$, an additional technical difficulty arises. The argument in~\cite{random_matching} involves a certain exploration process in which a portion of the random graph is coupled with a quasi-tree. When $\eps$ is sufficiently small, a substantial challenge that arises is that it becomes difficult to control the size of the revealed graph and ensure it is much smaller than $n$, which is required for the coupling to work. With the approach we describe below we only have to reveal slightly more than $\sqrt{n}$ vertices, regardless of the value of $\eps$.


In order to upper bound the mixing time, we define a random time $\tau$ and show that wherever the walk $X$ starts from, (i) $\tau$ is likely to be $\le t$ for a given $t$ that agrees with our lower bound on the $1-o(1)$ mixing time up to lower order terms, and (ii) the distribution of $X_{\tau}$ is close to the uniform distribution $\U$ on the set $V$ of vertices. This implies that for any fixed small $\theta\in\left(0,\frac14\right)$, if $n$ is sufficiently large, then any set $A\in V$ with $\pi(A)\ge1-\theta$ is hit by time $t$ with probability at least $1-\theta$. Using a result from~\cite{characterisation_of_cutoff} relating mixing and hitting times, together with an upper bound of order $\frac1\eps$ on the absolute relaxation time of the walk on $G^*$, we get that $\tmixtext^{G^*}(2\theta)\le t+\frac{C(\theta)}{\eps}$.


Assuming the walk $X$ starts from a vertex $x$ with a sufficiently nice neighbourhood in $G^*$, the random time we consider is $\tau_{2L}$, which is roughly speaking the first time when $X$ has travelled long-range distance $2L$ from $x$. \footnote{The actual definition of $\tau_{2L}$ is more involved and uses the fact that the long-range distance is well-defined on the quasi-tree.} To upper bound
\[\dtv{\prcond{X_{\tau_{2L}}=\cdot}{G^*}{x}}{\U(\cdot)}\quad=\quad\sum_{y}\left(\frac1n-\prcond{X_{\tau_{2L}}=y}{G^*}{x}\right)^+,\]
it is sufficient to lower bound $\prcond{X_{\tau_{2L}}=y}{G^*}{x}$ for most values of $y$.

We would like to be able to express $\prcond{X_{\tau_{2L}}=y}{G^*}{x}$ as
\begin{align}\label{eq:Xtau_transitions_approx_sum}
\prcond{X_{\tau_{2L}}=y}{G^*}{x}\quad\approx\quad\sum_{w,z}\prcond{X_{{\tau}^{(X)}_{L}}=z}{G^*}{x}\prcond{Y_{{\tau}^{(Y)}_{L}}=w}{G^*}{\eta(y)}\1_{w=\eta(z)},
\end{align}
where $X$ and $Y$ are independent random walks on $G^*$ from $x$ and $y$ respectively, ${\tau}^{(X)}_{L}$ and ${\tau}^{(Y)}_{L}$ are the first times when $X$ and $Y$ respectively, have travelled long-range distance $L$, and $\eta(v)$ denotes the long-range neighbour of vertex $v$. A decomposition similar to~\eqref{eq:Xtau_transitions_approx_sum} was considered in~\cite{cutoff_NBRW_on_sparse_random_graphs} and~\cite{cutoff_two_community}, but there it was applied for a non-backtracking random walk and deterministic times $2t$ and $t$ instead of $\tau_{2L}$ and $\tau_L$, hence it could be written as an exact equality.

Now we explain how to obtain the approximate equality in~\eqref{eq:Xtau_transitions_approx_sum} and how this decomposition can be used to lower bound $\prcond{X_{\tau_{2L}}=y}{G^*}{x}$.

\subsubsection*{Approximate reversibility at a random time}

The reason we do not have an exact equality in~\eqref{eq:Xtau_transitions_approx_sum} is that for a path $(z_0,...,z_k)$ the event that $k$ is the first time when the path reached long-range distance $L$ is not equivalent to the event that $k$ is the first time that the path $(z_{k-1},...,z_0,\eta(z_0))$ reached long-range distance $L$. See a counterexample in Figure~\ref{fig:eg_not_reversible}. For a path $z=(z_0,...,z_k)$ we define an event $\Omega^{z}_1(L)$ which is equivalent for the path and its reversal and on this event $k$ being the first time when $z$ reaches long-range distance $L$ is equivalent to $k$ being the first time that $(z_{k-1},...,z_0,\eta(z_0))$ reaches long-range distance $L$. See an illustration in Figure~\ref{fig:eg_reversibility_condition}. Using that the walks only rarely cross the same long-range edge multiple times, we will show that for any $G^*$ the event $\Omega^X_1(L)$ holds whp.

\begin{figure}[h]
\centering
\includegraphics[width=100mm]{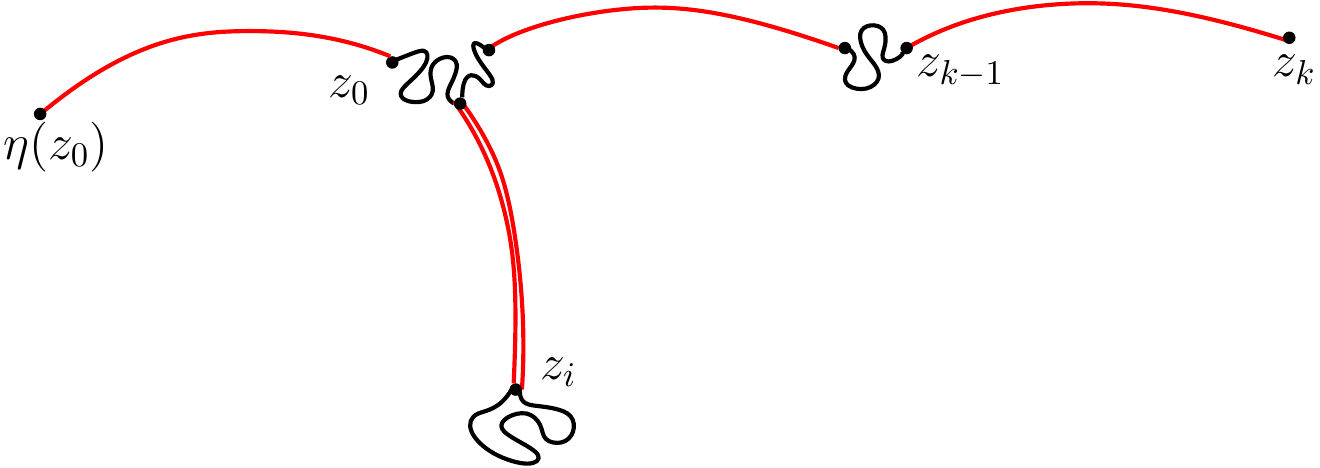}
\caption{Long-range edges are denoted by red. In this example the path $(z_0,...,z_k)$ first reaches long-range distance 2 from $z_0$ at $z_k$, but the reverse path $(z_{k-1},...,z_0,\eta(z_0))$ already reaches long-range distance 2 from $z_{k-1}$ at $z_i$.}
\label{fig:eg_not_reversible}
\end{figure}

\begin{figure}[h]
\centering
\includegraphics[width=130mm]{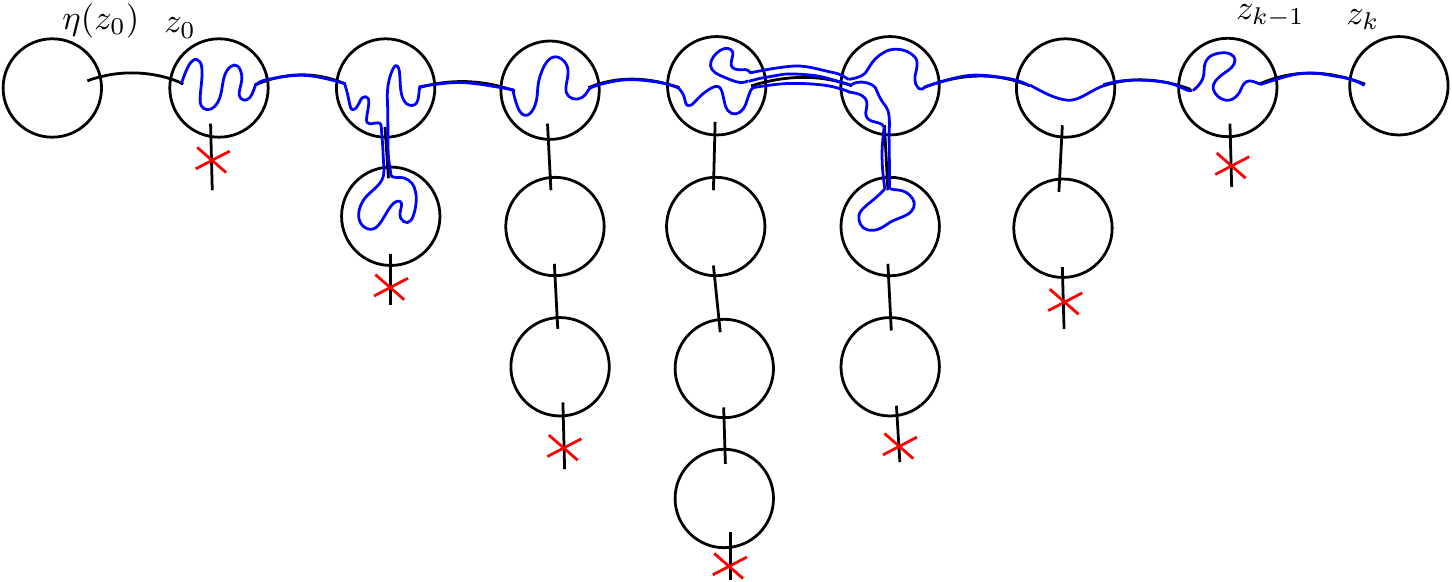}
\caption{An illustration of the condition $\Omega_1^z(7)$ for a path $z=(z_0,...,z_k)$.}
\label{fig:eg_reversibility_condition}
\end{figure}

\subsubsection*{Lower bounding $\prcond{X_{\tau_{2L}}=y}{G^*}{x}$ via~\eqref{eq:Xtau_transitions_approx_sum}}

We couple certain random neighbourhoods of $x$ and $y$ in $G^*$ and the walks $X$ and $Y$ on them by independent random quasi-trees $T_x$ and $T_y$ and independent random walks $\til{X}$ and $\til{Y}$ on them. We then define subsets $\partial T_x$ and $\partial T_y$ of the vertices of $T_x$ and $T_y$, respectively so that the long-range edges from these vertices are not yet revealed, for each pair $z\in\partial T_x$, $w\in\partial T_y$ the probability $\prcond{X_{{\tau}^{(X)}_{L}}=\eta(z)}{G^*}{x}\prcond{Y_{{\tau}^{(Y)}_{L}}=\eta(w)}{G^*}{\eta(y)}$ is sufficiently small, but the overall probability $\sum_{z\in\partial T_x,w\in\partial T_y}\prcond{X_{{\tau}^{(X)}_{L}}=\eta(z)}{G^*}{x}\prcond{Y_{{\tau}^{(Y)}_{L}}=\eta(w)}{G^*}{\eta(y)}$ is close to 1.

Conditioning on the explorations of the neighbourhoods of $x$ and $y$, we can complete the exploration of $G^*$ by choosing a uniform random matching of the yet unmatched long-range half-edges. 
We then use the following concentration result from~\cite{cutoff_NBRW_on_sparse_random_graphs} (based on~\cite{Steins_method}) to lower bound~\eqref{eq:Xtau_transitions_approx_sum}.

\begin{lemma}[Lemma 5.1 in \cite{cutoff_NBRW_on_sparse_random_graphs}]\label{lem:w_pair_sum}
Let $\I$ be a set of even size. Let $(w_{i,j})_{(i,j)\in\I\times\I}$ be some non-negative weights and let $\eta$ be a uniform random paring of $\I$. Then for all $a>0$ we have
\[\pr{\sum_{i\in\I}w_{i,\eta(i)}<m-a}\quad\le\quad\exp\left(-\frac{a^2}{4bm}\right)\]
where $m=\frac{1}{|\I|-1}\sum_{i\in\I}\sum_{j\in\I\setminus\{i\}}w_{i,j}$ and $b=\max_{i\ne j}(w_{i,j}+w_{j,i})$.
\end{lemma}

We let $w_{ij}=\prcond{X_{{\tau}^{(X)}_{L}}=\eta(i)}{G^*}{x}\prcond{Y_{{\tau}^{(Y)}_{L}}=\eta(j)}{G^*}{\eta(y)}\1_{i\in\partial T_x,j\in\partial T_y}$ and $a=\frac12m$. Having defined $T_x$, $\partial T_x$, ... appropriately, this will give us that with probability $1-o\left(\frac{1}{n^2}\right)$ the graph $G^*$ is such that for most pairs of vertices $x$, $y$ (for all pairs satisfying a certain condition) the sum~\eqref{eq:Xtau_transitions_approx_sum} is lower bounded by $\frac1n(1-\theta)$, where $\theta$ is an arbitrarily small constant.

Note that the probability $1-o\left(\frac{1}{n^2}\right)$ in the above result is high enough so that we can take a union bound over all pairs $x$, $y$ satisfying the required condition and still get a high probability statement.

\subsection{Organisation}

In Section~\ref{sec:quasi-tree} we define the notion of a random quasi-tree $T$ and prove concentration results on the speed and entropy of a random walk on $T$. In this section we work under minimal assumptions on $G$ and $\eps$.

In Section~\ref{sec:relating_Gstar_and_T} we introduce additional assumptions. We define an exploration process of $G^*$ and a coupling of $X$ on $G^*$ to $\til{X}$ on $T$. In this section we prove that the coupling succeeds with a large probability.

In Section~\ref{sec:proof_of_cutoff} we still work under the assumptions of Section~\ref{sec:relating_Gstar_and_T}. Here we prove the upper and lower bounds on the mixing time.

In Section~\ref{sec:specific_graphs} we prove that sequences of graphs with polynomial growth of balls or linear growth of entropy satisfy the assumptions of Section~\ref{sec:relating_Gstar_and_T}, hence they have cutoff in the regimes $\eps_n=|V_n|^{-o(1)}=o(1)$ and $\frac{1}{\log|V_n|}\ll\eps_n\ll1$ respectively. We also show that expanders and locally expanding graphs (to be defined later) have linear growth of entropy.

In Section~\ref{sec:other_eps} we discuss the other regimes for $\eps_n$ and in case of expanders and vertex-transitive graphs with polynomial growth of balls we complete the picture for every possible sequence $\eps_n$.

In Section~\ref{sec:general_graphs} we discuss some conjectures regarding more general families of graphs and present a sketch of the proof of Theorem~\ref{thm:results_general_graphs}.

Appendix~\ref{app:aux_results} lists a few standard results that are used in various proofs. Some of the proofs about speed and entropy that are analogous to the corresponding proofs in~\cite{random_matching} are deferred to Appendix~\ref{app:proofs_speed_entropy}, and some other proofs are also deferred to Appendices~\ref{app:proof_of4lemmas},~\ref{app:other_remainings_pfs} and~\ref{app:estimates_for_transitive_poly_growth_graphs}.

\section{Limiting tree}\label{sec:quasi-tree}

In what follows we will work with $(G_n)$, $(\eps_n)$ and $(G^*_n)$ as in Definition~\ref{def:Gstar} and in addition we also assume that all vertices of all graphs have degree $\le\Delta$ (where $\Delta$ is a given constant) and $\eps_n\to0$ as $n\to\infty$. For ease of notation we will also assume that $|V_n|=n$.

In this section we start by defining a tree-like structure that will be used to approximate the graph~$G_n^*$. It is the same as the definition of a {\emph{\quasi}} in~\cite{random_matching} except that the weights of the long-range edges in our setting are equal to $\eps_n$.
The goal is then to study the entropy of a random walk on these weighted {\emph{quasi-trees}}.

\begin{definition}\label{def:tree}

Given a graph $G$ and constants $\eps>0$ and $R\in\Zpos$ we define the associated \emph{quasi-tree} $\tree=\tree_{G,R,\eps}$, which is a random weighted infinite graph, as follows. 

Let $B$ be a random ball (in the graph distance of $G$) obtained by first sampling a uniform vertex and then considering its $R$-neighbourhood. We call such a ball a $\tree\rball$.

Let $\rho$ be the centre of $B$ that we will call the root of $\tree$. Next join by an edge each other vertex of $B$ (except for the root) to the centre of an i.i.d.\ copy of $B$. Repeat the same procedure for every vertex of the new balls except for their centres. We call edges joining different balls \emph{long-range edges} and we assign weight $\eps$ to them, while we assign weight $1$ to all other edges.

Let $\mathcal{T}$ be the topological space consisting of all rooted locally finite unlabelled connected graphs with a collection of distinguished edges, called \emph{long-range edges}, with the property that every simple path between a pair of vertices must cross the same collection of long-range edges. (In other words, the long-range edges give rise to a tree structure.) We call a member $T$ of $\mathcal{T}$ a \emph{quasi-tree} and we call the connected components of $T$ without the long-range edges the $T$-balls of $T$. The graph $T_{G,R,\eps}$ described above is a random variable taking values in this topological space $\mathcal{T}$ and its $T$-$R$-balls are exactly the $T$-balls in the above sense.

For a quasi-tree $T$ and $x,y \in \tree$ we write $d_\tree(x,y)$ (or simply $d(x,y)$ when $\tree$ is clear from the context), for the number of long-range edges on the shortest path from $x$ to $y$. Note that this is not the usual graph distance on $\tree$, but for us this will be a useful notion of distance. One can think of this distance as ``the long-range distance", but since we rarely consider the graph distance on $T$, we do not use this terminology.  A level of $T$ consists of all vertices at the same distance from $\rho$, i.e.\ when $d(\rho,x)=r$, then~$x$ belongs to the $r$-th level. We denote the level of a vertex $x$ by $\ell(x)$. We write $\B_r(x)=\B(x,r)=\{y:d_\tree(x,y) \le r \}$ for the ball of radius $r$ centred at $x$.
\end{definition}

\begin{definition}\label{def:long_range_dist_on_G}
In the graph $G^*$ we define the (long-range) distance between $x$ and $y$ to be the minimal number of long-range edges needed to be crossed to go from $x$ to $y$, when we only allow at most~$R$ consecutive edges of $G$ in the path from $x$ to $y$ and we do not allow any long-range edge to be crossed more than once. (Like for the {\quasi} $\tree$, we rarely use the regular graph distance on~$G^*$, so unless specified otherwise the term ''distance'' will refer to the aforementioned distance.)

We write $\B_K^*(x)$ for the ball of radius~$K$ and centre~$x$ in this metric.   
We write $\B_{G}(x,r)$ for a ball centred at $x$ of radius $r$ in the graph metric of $G$. When $r=R$, we call it the $G^*\rball$ centred at~$x$.    
\end{definition}

Below we establish a few more notations and conventions.

Let $X$ be a random walk on the weighted {\quasi} $\tree$. 
For an edge $(x,y)$ we write 
\[
\tau^+_{(x,y)} = \inf\{t\geq 1: (X_{t-1},X_t)= (x,y)\},
\]
i.e.\ $\tau^+_{(x,y)}$ is the first time that $X$ crosses the edge $(x,y)$.

Similarly we let $\tau_x$ be the first time when $X$ hits vertex $x$ and $\tau_k$ be the first time when $X$ reaches (long-range) distance $k$.

For a long-range edge $e$ we will often write $e=(e^-,e^+)$ where $e^-$ is the endpoint of $e$ closer to the root (or other reference vertex made clear from the context) and $e^+$ is the other endpoint. We define $\ell(e):=\ell(e^+)$. For a vertex $x$ we will denote its long-range neighbour by $x^+$ if it is further from the root (or reference vertex) than $x$, and denote by $x^-$ if it is closer.

In what follows most statements will concern sequences of graphs or sequences of quasi-trees. To simplify notation, we will often drop the super-/subscripts $n$. Some statements are about random quasi-trees associated to a sequence of graphs. In this case we assume that the graphs and the sequence of weights satisfy the assumptions in Definition~\ref{def:Gstar}, the graphs $G_n$ have bounded degrees, $\eps_n\to0$, and that $R_n\gtrsim\frac{1}{\eps_n}$.
Some statements are about sequences of non-random quasi-trees from the topological space $\mathcal{T}$. In this case (unless specified otherwise) we assume that the sequence is a possible realisation of random quasi-trees associated to a sequence of graphs as above.

Later we will talk about the boundary of an $R$-ball. By that (with some abuse of notation) we mean the vertices in the $R$-ball that are at graph distance $R$ from its centre.

\begin{definition}\label{def:loop-erasure}
Given a random walk $X$ on a transient quasi-tree $T$ let us define its loop-erasure~$\xi$ as follows. Let $\xi_k$ be the long-range edge between levels $k-1$ and $k$ that was last visited by $X$, which exists almost surely by transience. Note that this is equivalent to erasing the loops of $X$ in the chronological order in which they were created and then only keeping the long-range edges. We say that a random sequence $\xi$ of long-range edges is a loop-erased random walk (LERW) on $T$ if it has the same distribution as the loop-erasure of a random walk on $T$.
\end{definition}

\subsection{Some preliminary bounds}

\begin{lemma}\label{lem:neverbacktrack}
Let $C>0$ and $\Delta\in \Zpos$. Let $(\tree^{(n)})$ be a sequence of weighted {\quasi}s and suppose that the weights $\eps_n$ converge to 0 as $n\to\infty$, $R_n\geq\frac{C}{\eps_n}$ for all $n$ and each vertex in each tree has degree $\le\Delta$ (not counting the long-range edges). Let $X^{(n)}$ be a random walk on $\tree^{(n)}$ starting from its root.
Then
\[
\inf_{(x_n,y_n)}\prstart{\tau_{(y_n,x_n)}^+=\infty}{y_n}=1-\delta_n,\quad\text{where}\quad\delta_n=o(1)\quad\text{as}\quad n\to\infty
\]
and the infimum is taken over all long-range edges $(x_n,y_n)$ of $T^{(n)}$ (with $x_n$ being closer to the root than~$y_n$).

Furthermore there exist positive constants $c_1$, $c_2$ depending only on $\Delta$ such that
\[
c_1\eps_n\le\delta_n\le c_2\eps_n^{1/3}.
\]
\end{lemma}

The lemma above immediately implies the following.

\begin{corollary}\label{cor:neverbacktrack}
In the setup of Lemma~\ref{lem:neverbacktrack} for any $z_n$ in the ball of $y_n$ we have
\[
\prstart{\tau_{(y_n,x_n)}^+<\infty}{z_n}\le\delta_n.
\]
\end{corollary}

\begin{proof}[Proof of Lemma~\ref{lem:neverbacktrack}]

We fix a long-range edge $(x,y)$. We immediately get that $\prstart{\tau_{(y,x)}^+<\infty}{y}\ge\prstart{X_1=x}{y}\ge\frac{\eps}{\Delta+\eps}$, showing the lower bound on $\delta$.

Let $N$ be the first time when $X$ leaves the $\tree\rball$ of $y$. We will show that
\begin{align}\label{eq:delta_upper_bd_goal}
\prstart{X_N=x}{y}\le c\eps^{1/3}
\end{align}
as $n\to\infty$ for some constant $c$ depending only on $\Delta$. Once we have that we can establish the upper bound on $\delta$ as follows. Let $Y$ be the random walk on the directed long-range edges induced by $X$. Then $d(\rho,Y)$ stochastically dominates a biased random walk that steps to the right with probability $1-c\eps^{1/3}$ and to the left with probability $c\eps^{1/3}$. The probability of this biased walk returning to the starting point is $\asymp\eps^{1/3}$ as $n\to\infty$.

We will now prove \eqref{eq:delta_upper_bd_goal}. Note that

\begin{align*}
\prstart{X_N=x}{y}\le\prstart{N\le\frac{1}{\eps^{2/3}}}{y}+\sum_{i\ge\frac{1}{\eps^{2/3}}}\prstart{N=i,X_i=x}{y}\qquad\text{and}\\
\prstart{N=i,X_i=x}{y}=\sum_{(x_0,...,x_{i-1})}\prstart{X_1=x_1,...,X_{i-1}=x_{i-1}}{y}\prstart{X_1=x}{y}
\end{align*}

where the sum is taken over all paths $(x_0,...,x_{i-1})$ inside the $\tree\rball$ of $y$ with $x_0=x_{i-1}=y$. This sum is equal to
\[
\sum_{(x_0,...,x_{i-1})}\left(\prod_{j=0}^{i-2}\frac{1}{\deg(x_j)+\eps}\right)\frac{\eps}{\deg(y)+\eps},
\]
where the long-range edges are not counted towards $\deg$. Let $Z$ be a simple random walk on the~$\tree\rball$ of~$y$. The last expression above then equals
\begin{align*}
&\sum_{(x_0,...,x_{i-1})}\left(\prod_{j=0}^{i-2}\frac{1}{\deg(x_j)}\right)\left(\prod_{j=0}^{i-2}\frac{\deg(x_j)}{\deg(x_j)+\eps}\right)\frac{\eps}{\deg(y)+\eps}\\
&=\quad\sum_{(x_0,...,x_{i-1})}\prstart{Z_1=x_1,...,Z_{i-1}=x_{i-1}}{y}\left(\prod_{j=0}^{i-2}\frac{\deg(x_j)}{\deg(x_j)+\eps}\right)\frac{\eps}{\deg(y)+\eps}\\
&\le\quad\prstart{Z_{i-1}=y}{y}\left(\frac{\Delta}{\Delta+\eps}\right)^{i-1}\frac{\eps}{1+\eps}.
\end{align*}
Since $Z$ is a simple random walk on a graph with $\ge R$ vertices and degrees bounded by $\Delta$, from Lemma~\ref{lem:Pt_upper_bound} we know that for every $j\le R^2$ we have $\prstart{Z_j=y}{y}\lesssim\frac{1}{\sqrt{j}}$.
Since $\prstart{Z_{2j}=y}{y}\le\prstart{Z_{2j-2}=y}{y}$ (see for instance~\cite[Proposition~10.18]{MTMC_book}) and by Cauchy-Schwarz $\prstart{Z_{2j+1}=y}{y}\le\sqrt{\prstart{Z_{2j+2}=y}{y}\prstart{Z_{2j}=y}{y}}$, we also know that for all $j>R^2$ we have $\prstart{Z_j=y}{y}\lesssim\frac{1}{R}$.

Since at every step the probability of $X$ leaving the $\tree\rball$ of $y$ is $\ge\frac{\eps}{\Delta+\eps}$, we also have by a union bound
\[
\prstart{N\le\frac{1}{\eps^{2/3}}}{y}\le \frac{1}{\eps^{2/3}}\frac{\eps}{\Delta+\eps}\le\eps^{1/3}.
\]

Together these give
\begin{align*}
\prstart{X_N=x}{y}\quad&\lesssim\quad \eps^{1/3} + \sum_{i=\frac{1}{\eps^{2/3}}}^{R^2}\frac{1}{\sqrt{i-1}}\left(\frac{\Delta}{\Delta+\eps}\right)^{i-1}\frac{\eps}{1+\eps} + \sum_{i\ge R^2}\frac{1}{R}\left(\frac{\Delta}{\Delta+\eps}\right)^{i-1}\frac{\eps}{1+\eps}\\
&\lesssim\quad\eps^{1/3}+\eps^{1/3}\frac{\Delta+\eps}{1+\eps}+\frac{1}{R}\frac{\Delta+\eps}{1+\eps}\quad\lesssim\quad\eps^{1/3}+\frac1R\quad\asymp\quad\eps^{1/3}.
\end{align*}

(Note that $\frac1R\lesssim\eps\ll\eps^{1/3}$, so the constant in the last $\asymp$ does not depend on $\inf_n\{R_n\eps_n\}$.)
\end{proof}

\begin{definition}\label{def:sigma_varphi}
For a walk $X$ on a quasi-tree $T$ let us say that $t$ is a regeneration time of $X$ if $(X_{t-1},X_{t})$ is a long-range edge and is crossed by the walk exactly once. For a walk $X$ on $T$ and a given non-negative integer $K$ let us define the sequence $(\sigma_i)_{i\ge0}$ as follows. Let $\sigma_0$ be the first time when $X$ reaches the $K$th level, and for each $i\ge1$ let $\sigma_i$ be the $i$th regeneration time that happens at a level $\ge K+1$. If $K$ is not specified, we take $K=0$.
Let $\varphi_i=d_T(\rho,X_{\sigma_i})$. (Remember that we are working with long-range distances.)
\end{definition}

The following two statements will allow us to relate the walk at the regeneration times to the loop-erased random walk.

\begin{lemma}\label{lem:firstgen}
Let $(\tree^{(n)})$ be a sequence of weighted {\quasi}s as in Lemma~\ref{lem:neverbacktrack}, let $\rho_n$ be the root of~$\tree^{(n)}$ and let $X^{(n)}$ be a random walk on $\tree^{(n)}$. Let $\sigma_1^{(n)}$ be the first regeneration time of $X^{(n)}$. Let~$\xi^{(n)}$ be a loop-erased random walk on $\tree^{(n)}$ and let $(x_n,y_n)$ be a long-range edge of $\tree^{(n)}$ in the first level. Then we have
\[
\prstart{\left(X^{(n)}_{\sigma_1^{(n)}-1},X^{(n)}_{\sigma_1^{(n)}}\right)=(x_n,y_n)}{\rho_n} = (1-o(1))\cdot  \prstart{\xi^{(n)}_1=(x_n,y_n)}{\rho_n},
\]
uniformly in $(x_n,y_n)$.
\end{lemma}

\begin{proof}
Let $(x,y)$ be a long-range edge. We have
\begin{align*}
&\prstart{\xi_1=(x,y)}{\rho}=\prstart{\tau_y<\infty}{\rho}\sum_{k=0}^{\infty}\prstart{\tau_x<\infty}{y}^k\prstart{\tau_y<\infty}{x}^k\prstart{\tau_x=\infty}{y},\\
&\prstart{\left(X_{\sigma_1-1},X_{\sigma_1}\right)=(x,y)}{\rho}=\prstart{\tau_y<\infty}{\rho}\prstart{\tau_x=\infty}{y}.
\end{align*}
Lemma~\ref{lem:neverbacktrack} gives $\prstart{\tau_x<\infty}{y}=o(1)$, hence $\sum_{k=0}^{\infty}\prstart{\tau_x<\infty}{y}^k\prstart{\tau_y<\infty}{x}^k=1+o(1)$.
\end{proof}

\begin{lemma}\label{lem:bound_prob_Xsigma1_equals_x}
Let $T$ be a quasi-tree as in Definition~\ref{def:tree} with root $\rho$, let $X$ be a random walk on $T$ from $\rho$ and let $\xi$ be a loop-erased random walk on $T$ from $\rho$. Let $K$ be any non-negative integer and let $(\sigma_i)_{i\ge0}$ be as in Definition~\ref{def:sigma_varphi}. Then for any $r\ge1$ and any vertex $y$ that is the centre of an $R$-ball at (long-range) distance $K+r$ from $\rho$ we have
\begin{align*}
\prstart{X_{\sigma_1}=y}{\rho}\quad\le\quad (2\delta^2)^{r-1}\prstart{\tau_y<\infty}{\rho}\quad\le\quad (2\delta^2)^{r-1}\frac{1}{1-\delta}\prstart{\xi_{K+r}=(x,y)}{\rho},\end{align*}
where $x$ is the long-range neighbour of $y$ and $\delta$ is defined as in Lemma~\ref{lem:neverbacktrack} via
\[
\delta =1- \inf_{(w,z)}\prstart{\tau_{(z,w)}^+=\infty}{w}.\]
\end{lemma}

\begin{proof} Let $z$ be the level $K$ centre of the $R$-ball that is the ancestor of $x$ and $y$. (Note that in case $K=0$ we have $z=\rho$.) Then
\begin{align*}
\prstart{X_{\sigma_1}=y}{\rho}&=\prstart{\tau_z<\infty}{\rho}\prcond{X_{\sigma_1}=y}{\tau_z<\infty}{\rho}\qquad\text{and} \\
\prstart{\tau_y<\infty}{\rho}&=\prstart{\tau_z<\infty}{\rho}\prstart{\tau_y<\infty}{z}.
\end{align*}

Let $y_0=z$ and let $(x_i,y_i)_{i=1}^r$ be the directed long-range edges between $z$ and $y$. Let $Y$ be the walk on the long-range edges induced by $X$. Then
\begin{align*}
&\prcond{X_{\sigma_1}=y}{\tau_z<\infty}{\rho}\\
&=\prstart{Y\text{ crosses each }(x_i,y_i)_{i=1}^{r-1}\text{ at least 3 times, then crosses }(x_r,y_r)\text{ for the first and last time}}{z}\\
&\le\prstart{Y\text{ crosses each }(x_i,y_i)_{i=1}^{r-1}\text{ at least 3 times, then crosses }(x_r,y_r)\text{ for the first time}}{z}.
\end{align*}
For a path $p$ of $Y$ starting from $z$, ending at $(x_r,y_r)$ and crossing each of $(x_i,y_i)_{i=1}^{r-1}$ at least three times, let us associate a path $\til{p}$ as follows. Let $(x_{j_1},y_{j_1})$ be the furthest edge $p$ reaches before crossing $(y_1,x_1)$ again. Then let $\til{p}$ start as $(x_1,y_1),...,(x_{j_1},y_{j_1}),(y_{j_1},x_{j_1}),...,(y_1,x_1),$ $(x_1,y_1),...,(x_{j_1},y_{j_1})$. Note that $p$ must also cross these edges in this order before first crossing $(x_{j_1+1},y_{j_1+1})$. Then in each step let $(x_{j_{k+1}},y_{j_{k+1}})$ be the furthest edge $p$ reaches before crossing $(y_{j_k+1},x_{j_k+1})$ again. Then append $(x_{j_k+1},y_{j_k+1}),...,(x_{j_{k+1}},y_{j_{k+1}}),(y_{j_{k+1}},x_{j_{k+1}}),$ $...,(y_{j_k+1},x_{j_k+1}),(x_{j_k+1},y_{j_k+1}),...,(x_{j_{k+1}},y_{j_{k+1}})$ to the end of $\til{p}$. Continue this until $\tilde{p}$ reaches $(x_{r-1},y_{r-1})$. Then finally append $(x_r,y_r)$ to it. Note that $\til{p}$ is a subsequence of $p$. Also note that $\til{p}$ is always of length $(3r-2)$, can take $2^{r-1}$ different possible values (we can choose $(j_k)$ to be any subsequence of $(1,2,...,r-1)$) and that for each $\til{p}$ the sum of the probabilities of the paths $p$ associated to it is upper bounded by
\[\prod_{i=1}^{3r-2}\prstart{\tau_{\til{p}_{i}}<\infty}{\til{p}_{i-1}}\le
\left(\prod_{i=1}^r\prstart{\tau_{(x_i,y_i)}<\infty}{(x_{i-1},y_{i-1})}\right)\delta^{2r-2}
=\delta^{2r-2}\prstart{\tau_{(x_r,y_r)}<\infty}{z},\]
where $\til{p}_i$ is the $i$th point of $\til{p}$.
Summing over all $\tilde{p}$ gives
\[\prcond{X_{\sigma_1}=y}{\tau_z<\infty}{\rho}\quad\le\quad2^{r-1}\delta^{2r-2}\prstart{\tau_y<\infty}{z}.\]
We also know that 
\[
\prstart{\xi_{K+r}=(x,y)}{\rho}\quad\ge\quad\prstart{\tau_y<\infty}{\rho}\prstart{\tau_x=\infty}{y}\quad\ge\quad(1-\delta)\prstart{\tau_y<\infty}{\rho}.\]
This finishes the proof.
\end{proof}

Summing over all $y$ at level $K+r$ in Lemma~\ref{lem:bound_prob_Xsigma1_equals_x} we immediately get the following bounds.

\begin{corollary}\label{cor:varphi_tail_bound}
In the setup of Lemma~\ref{lem:bound_prob_Xsigma1_equals_x} for any $r\ge1$ we have
\begin{align*}\prstart{\varphi_1=K+r}{\rho}\quad&\le\quad \frac{1}{1-\delta}(2\delta^2)^{r-1},\quad\text{and}\\
\prstart{\varphi_1\ge K+r}{\rho}\quad&\le\quad \frac{1}{(1-\delta)(1-2\delta^2)}(2\delta^2)^{r-1}.
\end{align*}
\end{corollary}

\begin{definition}\label{def:Tprime}
Given a quasi-tree $T$ with root $\rho$, let $T^a$ be obtained from $T$ by adding an $\eps$-weighted long-range edge from $\rho$ to a new vertex $\rho^a$. Given a quasi-tree $T$ and a vertex $x$ that is the centre of an $R$-ball in $T$, let $T(x)$ denote the quasi-tree formed by the $R$-ball of $x$ and the $R$-balls that are its descendants.

Given a random quasi-tree $T$ corresponding to a graph $G$, let $(T',X')$ be distributed as $(T^a,X^a)$ conditioned on $X^a$ never hitting $\rho^a$ (where $T^a$ is as above and $X^a$ is a random walk on it started from $\rho$). Note that in general $T'$ does not have the same distribution as $T$, and given $T'$, the process $X'$ is not a random walk on it. Let $\xi'$ be a loop-erased random walk on $T'$ from its root and let the times $(\sigma'_i)_{i\ge0}$ and their levels $(\varphi'_i)_{i\ge0}$ be defined as in Definition~\ref{def:sigma_varphi} for the tree $T'$ and the process $X'$ (instead of $T$ and $X$).
\end{definition}

Then $T'$, $X'$ and $\xi'$ also satisfy a version of Lemma~\ref{lem:firstgen} and Lemma~\ref{lem:bound_prob_Xsigma1_equals_x} as follows.

\begin{lemma}\label{lem:bound_prob_Xprime_sigma1_equals_x}
Let $K=0$ and let $T'$, $X'$, $\sigma'_1$ and $\xi'$ be as in Definition~\ref{def:Tprime}. Then for any realisation of $T'$ and for any $r\ge1$ and any vertex $y$ that is the centre of an $R$-ball at distance $r$ from $\rho$ we almost surely have 
\[\prcond{X'_{\sigma'_1}=y}{T'}{\rho}\quad\lesssim\quad (2\delta^2)^{r-1}\prcond{\xi'_{r}=(x,y)}{T'}{\rho}\]
where $x$ is the long-range neighbour of $y$.
If $r=1$ then we almost surely also have
\[\prcond{X'_{\sigma'_1}=y}{T'}{\rho}\quad\asymp\quad\prcond{\xi'_{1}=(x,y)}{T'}{\rho}.\]
\end{lemma}

\begin{proof}Using the definition of $X'$, and that by Lemma~\ref{lem:neverbacktrack} we have $\prcond{\tau_{\rho^a}=\infty}{T^a}{\rho}=1-o(1)$ and writing $(\sigma_i^a)$ for the regeneration times of $X^a$ we get that 
\begin{align*} \prcond{X'_{\sigma'_1}=y}{T'}{\rho}\quad&=\prcond{X^a_{\sigma^a_1}=y}{T^a,\tau_{\rho^a}=\infty}{\rho}\\
=\frac{\prcond{X^a_{\sigma^a_1}=y,\tau_{\rho^a}=\infty}{T^a}{\rho}}{\prcond{\tau_{\rho^a}=\infty}{T^a}{\rho}}\quad&\asymp\prcond{X^a_{\sigma^a_1}=y,\tau_{\rho^a}=\infty}{T^a}{\rho}\quad\le\prcond{X^a_{\sigma^a_1}=y}{T^a}{\rho}.\end{align*}

Repeating the proof of Lemma~\ref{lem:bound_prob_Xsigma1_equals_x} for $T^a$ instead of $T$ we get that this is
\[\lesssim\quad(2\delta^2)^{r-1}\prcond{\xi^a_{r}=(x,y)}{T^a}{\rho}\quad =\quad (2\delta^2)^{r-1}\prcond{\xi'_{r}=(x,y)}{T'}{\rho}   ,\]
where the last equality follows by noting that $\xi^a$ and $\xi'$ are loop-erased random walks on $T^a$ and $T'$ respectively.
This finishes the proof of the first part of the result. For the second part of the result note that
\begin{align*}\prcond{X'_{\sigma'_1}=y}{T'}{\rho}\quad&\asymp\quad\prcond{X^a_{\sigma^a_1}=y,\tau_{\rho^a}=\infty}{T^a}{\rho}\quad\gtrsim\quad\prcond{\tau_{y}<\infty,\tau_{y}<\tau_{\rho^a}}{T^a}{\rho}\\
&\gtrsim\quad\prcond{\xi^a_1=(x,y),\tau_{y}<\tau_{\rho^a}}{T^a}{\rho}\quad\asymp\quad\prcond{\xi'_{1}=(x,y)}{T'}{\rho}\end{align*}
and this finishes the proof.
\end{proof}

The proof of the following lemma is identical to the proof of the first and third points of~\cite[Lemma 3.6]{random_matching} and we omit the proof.

\begin{lemma}\label{lem:indep_between_regenerations}
Let $T$ be a random quasi-tree associated to a graph $G$, with root $\rho$. Let $K\ge0$ and let $T_0$ be a realisation of the first $K$ levels of $T$. Let $X$ be a simple random walk on $T$ started from the root.
Conditional on $\B(\rho,K)=T_0$ we have that
\begin{enumerate}[(i)]
\item $\left(T(X_{\sigma_i})\setminus T(X_{\sigma_{i+1}}),(X_t)_{\sigma_i\le t\le\sigma_{i+1}}\right)$ are i.i.d.\ for $i\ge1$ and are jointly independent from\\ $\left(T\setminus T(X_{\sigma_{1}}),(X_t)_{0\le t\le\sigma_{1}}\right)$, and
\item for all $i\ge1$, the pair $\left(T(X_{\sigma_i}),(X_t)_{t\ge\sigma_i}\right)$ has the same distribution as $(T',X')$.
\end{enumerate}
\end{lemma}

The tail probability bounds for $(\sigma_i)$ and $(\varphi_i)$ from~\cite[Lemma~3.6]{random_matching} are no longer valid here and instead we get the following.

\begin{lemma}\label{lem:tail_bounds}
There exist positive constants $c_1, c_2$ and $C$ so that the following holds. Let $T$ be a quasi-tree as in Definition~\ref{def:tree}, let $X$ be a random walk on it from its root $\rho$ and let $(\sigma_i)$ and $(\varphi_i)$ be defined as in Definition~\ref{def:sigma_varphi} with $K=0$. Assume that $\eps$ is small enough so that $\delta<\frac13$. Then for any $r\ge1$ we have
\begin{align*}
\prstart{\varphi_1\ge r}{\rho}\quad&\leq \quad  C e^{-c_1(r-1)\log\left(\frac1\eps\right)},\\
\prstart{\sigma_1\ge r}{\rho}\quad&\leq \quad Ce^{-c_2r\eps}.
\end{align*}
\end{lemma}

\begin{lemma}\label{lem:tail_bounds_Tprime}
There exist positive constants $c_1, c_2$ and $C$ so that the following holds. Let $T$ be a quasi-tree with $\rho$ as in Definition~\ref{def:tree}. Let $(\sigma'_i)$ and $(\varphi'_i)$ be defined for $X'$ as before, with $K=0$. Assume that $\eps<\frac{1}{27}$. Then
\begin{align*}\prstart{\varphi'_1\ge r}{\rho}\quad&\leq\quad Ce^{-c_1(r-1)\log\left(\frac1\eps\right)},\\
\prstart{\sigma'_1\ge r}{\rho}\quad&\leq \quad  Ce^{-c_2r\eps}.
\end{align*}
We also have
$\estart{\sigma'_1}{\rho}\geq \frac1{C\eps}.$
\end{lemma}

Before proving these tail bounds, we will state and prove two immediate corollaries.

\begin{corollary}\label{cor:tail_bounds}
There exist positive constants $c_1, c_2$ and $C$ so that the following holds.
Let $T$ be a quasi-tree as in Definition~\ref{def:tree}, let $X$ be a random walk from its root $\rho$ and let $(\sigma_i)$ and $(\varphi_i)$ be defined as in Definition~\ref{def:sigma_varphi} for some $K\ge0$. Then for all $i\ge1$ we have
\begin{align*}
\prstart{\varphi_{i+1}-\varphi_{i}\ge r}{\rho}\leq C e^{-c_1(r-1)\log\left(\frac1\eps\right)}, \quad& \prstart{\sigma_{i+1}-\sigma_{i}\ge r}{\rho}\leq C e^{-c_2r\eps} \quad \text{ and }\\
&\estart{\sigma_{i+1}-\sigma_{i}}{\rho}\quad\geq \quad\frac1{C\eps}.
\end{align*}
In case $K\ge1$ we also have
$$\prstart{\varphi_{1}-\varphi_{0}\ge r}{\rho}\quad\leq\quad Ce^{-c_1(r-1)\log\left(\frac1\eps\right)}.$$
(Note that in case $K=0$ we already have this bound.)
\end{corollary}

\begin{proof}
Applying Lemmas~\ref{lem:indep_between_regenerations} and~\ref{lem:tail_bounds_Tprime} gives the bounds for $(\sigma_{i+1}-\sigma_{i})$ and $(\varphi_{i+1}-\varphi_{i})$. In case $K\ge1$, applying Lemma~\ref{lem:tail_bounds_Tprime} to $T(X_L)$, where $L$ is the last time when $X$ enters level $K$ gives the bound for $(\varphi_1-\varphi_0)$.
\end{proof}

Corollary~\ref{cor:varphi_tail_bound} and Lemma~\ref{lem:tail_bounds} immediately imply the following.

\begin{corollary}\label{cor:varphi_sigma_E_Var_bounds} In the setup of Corollary~\ref{cor:tail_bounds} for any $i\ge1$ we have
\[\E{\varphi_i-\varphi_{i-1}-1}\quad\lesssim\quad\delta^2,\qquad \E{(\varphi_i-\varphi_{i-1}-1)^2}\quad\lesssim\quad\delta^2,\]
and for any $i\ge2$ we have
\[\E{\sigma_i-\sigma_{i-1}}\quad\asymp\quad\frac{1}{\eps},\qquad \E{\left(\sigma_i-\sigma_{i-1}\right)^2}\quad\lesssim\quad\frac{1}{\eps^2}.\]
\end{corollary}

\begin{proof}[Proof of Lemma~\ref{lem:tail_bounds}]
By Corollary~\ref{cor:varphi_tail_bound} and the bound on $\delta$ from Lemma~\ref{lem:neverbacktrack} we have
\[\prstart{\varphi_1\ge r}{\rho}\quad\le\quad \frac{1}{(1-\delta)(1-2\delta^2)}(2\delta^2)^{r-1}\quad\lesssim\quad(2\delta^2)^{r-1}\quad\lesssim\quad e^{-c_1(r-1)\log\left(\frac1\eps\right)}.\]
Let $Y$ be the random walk on the directed long-range edges induced by $X$ and let $\sigma^{Y}_1$ be the number of long-range edges crossed by $X$ up to time $\sigma_1$. Then for any positive constants $a$ and $b$ we have
\[\prstart{\sigma_1\ge r}{\rho}\quad\le\quad
\prstart{\varphi_1\ge a\eps r}{\rho}\:+\:
\prstart{\varphi_1\le a\eps r,\sigma^{Y}_1\ge b\eps r}{\rho}\:+\:
\prstart{\sigma^{Y}_1\le b\eps r,\sigma_1\ge r}{\rho}.\]
By the first part of the proof we have
\[\prstart{\varphi_1\ge a\eps r}{\rho}\quad\le\quad e^{-c_1\left(\lceil a\eps r\rceil-1\right)\log\left(\frac1\eps\right)}\quad\lesssim\footnotemark\quad e^{-c_1'\eps r}\]
for some constant $c_1'>0$. \footnotetext{If $a\eps r\le1$, then LHS=1 and $\eps r\lesssim1$, so $e^{-c_1'\eps r}\asymp1$. If $a\eps r>1$, then $\left(\lceil a\eps r\rceil-1\right)\log\left(\frac1\eps\right)\ge\frac12a\eps r\log\left(\frac1\eps\right)\gtrsim \eps r$.}

For $i=1,2,...$ let $\tau^{(i)}$ be the number of steps between the $i$th and $(i+1)$th time that the walk crosses a long-range edge. Note that $\left(\tau^{(i)}\right)_{i\ge1}$ stochastically dominates a sequence of independent $\Geom{\eps}$ random variables. Also note that if $\sigma^{Y}_1\le b\eps r$ and $\sigma_1\ge r$ then $\tau^{(1)}+...+\tau^{({\lfloor b\eps r\rfloor})}\ge r$. Hence (given that $b<1$) we have
\begin{align*}
\prstart{\sigma^{Y}_1\le b\eps r,\sigma_1\ge r}{\rho}\quad&\le\quad \pr{\NegBin{\lfloor b\eps r\rfloor}{\eps}\ge r}\quad=\quad\pr{\Bin{r}{\eps}\le \lfloor b\eps r\rfloor}\\
\le\footnotemark\quad\exp\left(-r\Div{b\eps}{\eps}\right)\quad&=\quad
\exp\left(-rb\eps\log\left(\frac{b\eps}{\eps}\right)-r(1-b\eps)\log\left(\frac{1-b\eps}{1-\eps}\right)\right).
\end{align*}

\footnotetext{Here we use the following tail bound for binomial random variables. If $0<q<p<1$ then $\pr{\Bin{n}{p}\le qn}\le\exp\left(-n\Div{q}{p}\right)$ where $\Div{q}{p}=q\log\left(\frac{q}{p}\right)+(1-q)\log\left(\frac{1-q}{1-p}\right)$. This can be proved using the Chernoff bound.}
Note that $\eps<\frac12$ and $b<1$, hence $\frac{(1-b)\eps}{1-\eps}<1$ and so
\[\log\left(\frac{1-b\eps}{1-\eps}\right)\quad=\quad\log\left(1+\frac{(1-b)\eps}{1-\eps}\right)\quad\ge\quad\frac12\frac{(1-b)\eps}{1-\eps}.\]
Also for any sufficiently small $b\le\frac14$ we have
\[\frac12(1-b\eps)\frac{(1-b)\eps}{1-\eps}\quad>\quad\frac12(1-b)\eps\quad\ge\quad b\log\left(\frac1b\right)\eps.\]
Together these show that for $b\le\frac14$ and some positive constant $c'$ (depending on $b$) we have
\[\prstart{\sigma^{Y}_1\le b\eps r,\sigma_1\ge r}{\rho}\quad\le\quad\exp\left(-c'r\eps\right).\]

Note that if $\varphi_1\le a\eps r$ and $\sigma^{Y}_1\ge b\eps r$, then $\inf_{j\ge b\eps r}\ell(Y_{j})\le a\eps r$. We know that $\ell(Y)$ stochastically dominates a biased random walk $S$ on $\Z$ that steps right with probability $1-\delta$ and left with probability $\delta$. Hence (given that $\frac{2a+b}{2}\le (1-\delta)b$) we have
\begin{align*}
&\prstart{\varphi_1\le a\eps r,\sigma^{Y}_1\ge b\eps r}{\rho}\quad\le\quad
\prstart{\inf_{j\ge b\eps r}\ell(Y_{j})\le a\eps r}{\rho}\quad\le\quad
\prstart{\inf_{j\ge b\eps r}S_j\le a\eps r}{0}\\
&\le\quad\prstart{S_{\lceil b\eps r\rceil}\le 2a\eps r}{0}\quad+\quad\prstart{\inf_{j}S_j\le a\eps r}{\lceil 2a\eps r\rceil}\\
&\le\quad\pr{\Bin{\lceil b\eps r\rceil}{1-\delta}\le\frac{2a\eps r+\lceil b\eps r\rceil}{2}}\quad+\quad
\left(\frac{\delta}{1-\delta}\right)^{a\eps r}\\
&\le\footnotemark\:\exp\left(-2\lceil b\eps r\rceil\left(1-\delta-\frac{2a\eps r+\lceil b\eps r\rceil}{2\lceil b\eps r\rceil}\right)^2\right)+\exp\left(-a\eps\left(\log\left(\frac1\delta\right)-\log\left(\frac{1}{1-\delta}\right)\right)r\right)\\
&\le\quad\exp\left(-2b\eps r\left(1-\delta-\frac{a\eps r}{\lceil b\eps r\rceil}-\frac12\right)^2\right)\quad+\quad\exp\left(-\frac12a\eps\log\left(\frac1\delta\right)r\right).
\end{align*}
\footnotetext{Here we used the following tail bound. If $0<q<p<1$ then $\pr{\Bin{n}{p}\le qn}\le\exp\left(-2n\left(p-q\right)^2\right)$. This can be proved using Hoeffding's inequality.}
Choosing say $b=\frac14$ and $a=\frac{1}{12}b$ we have
$1-\delta-\frac{a\eps r}{\lceil b\eps r\rceil}-\frac12\ge 1-\frac13-\frac{a}{b}-\frac12\ge\frac{1}{12}$.\\
This gives an overall bound of form
$$\prstart{\sigma_1\ge r}{\rho}\quad\lesssim\quad e^{-c_2r\eps}$$
and this finishes the proof.
\end{proof}

\begin{proof}[Proof of Lemma~\ref{lem:tail_bounds_Tprime}]
Repeating the above proof for $T^a$ and $X^a$ we get the desired tail bounds for $\varphi^a_1$ and $\sigma^a_1$.

Using that $\pr{X^a\text{ does not hit }\rho^a}\ge1-\delta\gtrsim1$ we get that
\begin{align*}
\prstart{\varphi'_1\ge r}{\rho}\quad&=\quad\prcond{\varphi^a_1\ge r}{\tau_{\rho^a}=\infty}{\rho}\quad=\quad\frac{\prstart{\varphi^a_1\ge r,\tau_{\rho^a}=\infty}{\rho}}{\prstart{\tau_{\rho^a}=\infty}{\rho}}\\
&\lesssim\quad\prstart{\varphi^a_1\ge r}{\rho}\quad\lesssim\quad e^{-c_1(r-1)\log\left(\frac1\eps\right)}.
\end{align*}
Similarly, we get
\[\prstart{\sigma'_1\ge r}{\rho}\quad\lesssim\quad e^{-c_2r\eps}.\]

Let $\tau^a_1$ be the first time when $X^a$ crosses a long-range edge. Then $\sigma^a_1\ge\tau^a_1\ge_{\textrm{st}}\Geompos{\eps}$, hence

\begin{align*}
\estart{\sigma'_1}{\rho}\quad&=\quad\frac{\estart{\sigma^a_1\1_{\tau_{\rho^a}=\infty}}{\rho}}{\prstart{\tau_{\rho^a}=\infty}{\rho}}\quad\ge\quad\estart{\tau^a_1\1_{\tau_{\rho^a}=\infty}}{\rho}\quad=\quad\sum_{k\ge1}\prstart{\tau^a_1\1_{\tau_{\rho^a}=\infty}\ge k}{\rho}\\
&=\quad\sum_{k\ge1}\prstart{\tau^a_1\ge k}{\rho}\prcond{\tau_{\rho^a}=\infty}{\tau^a_1\ge k}{\rho}\quad\ge\quad\sum_{k\ge1}(1-\eps)^k(1-\delta)\quad\asymp\quad\frac1\eps.
\end{align*}

The last inequality holds because $\prstart{\tau^a_1\ge k}{\rho}\ge(1-\eps)^k$ and 
\begin{align*}\prcond{\tau_{\rho^a}=\infty}{\tau^a_1\ge k}{\rho}\quad&=\quad \sum_{u}\prcond{X^a_{k-1}=u}{\tau^a_1>k-1}{\rho}\prstart{\tau_{\rho^a}=\infty}{u}\\
&\ge\quad\sum_{u}\prcond{X^a_{k-1}=u}{\tau^a_1>k-1}{\rho}(1-\delta)\quad=\quad1-\delta.
\end{align*}
This finishes the proof.
\end{proof}

We conclude this section by proving two useful statements that we will use later.

\begin{lemma} \label{lem:compare_inxi_fisthit}
There exist positive constants $c_1, c_2$ and $C$ so that the following holds.
Let $T$ be a quasi-tree with root $\rho$ and let $e$ be a long-range edge from the $R$-ball of $\rho$. Let $X$ be a random walk and $\xi$ a loop-erased random walk on $T$. Let $\tau_e$ be the first time when $X$ crosses $e$ and let $\tau_1$ be the first time when $X$ crosses a long-range edge (i.e.\ it hits level 1). Then for any vertex $v$ in the $R$-ball of $\rho$ and for any $\ell_0$ we have
\[
\prstart{e\in\xi}{v}\quad=\quad(1-o(1))\prstart{\tau_e<\infty}{v}\quad\leq \quad C (e^{c_1\ell_0\eps\delta}\prstart{X_{\tau_1}=e^+}{v}+e^{-c_2\ell_0\eps}).
\]
\end{lemma}

\begin{proof} Let $e=(e^-,e^+)$ where $e^-$ is the vertex in the $R$-ball of $\rho$, and for each vertex $x\neq \rho$ in this $R$-ball let $x^+$ denote its long-range neighbour. 
\\
For the equality simply note that
\[\prstart{\tau_e<\infty}{v}\quad\ge\quad\prstart{e\in\xi}{v}\quad\ge\quad(1-\delta)\prstart{\tau_e<\infty}{v}.\]
For the inequality, grouping the possible paths to $e$ by their restriction $v=x_0,x_1,...,x_{\ell-1}=e^-$ to the $R$-ball of $\rho$, we get
\begin{align*}
\prstart{\tau_e<\infty}{v}&= \sum_{\ell}\sum_{x_0,...,x_{\ell-1}}\left(\prod_{i=0}^{\ell-2}\left(P(x_i,x_{i+1})\sum_{j=0}^{\infty}\left(P(x_i,x_i^+)\prstart{\tau_{x_i}<\infty}{x_i^+}\right)^j\right)\right)P(e^-,e^+)\\
&\le\quad\sum_{\ell}\sum_{x_0,...,x_{\ell-1}}\left(\frac{1}{1-\eps\delta}\right)^{\ell}\prstart{X_1=x_1,...,X_{\ell-1}=x_{\ell-1},X_\ell=e^+}{v}
\end{align*}
where the sum is taken over all sequences such that $x_0=v$, $x_{\ell-1}=e^-$ and $x_0,...,x_{\ell-1}$ are in the $R$-ball of $\rho$, and $P$ is the transition matrix of $X$. For the inequality we used that $P(x_i,x_i^+)\prstart{\tau_{x_i}<\infty}{x_i^+}\le\eps\delta$. The sum on the last line above is then equal to 
\begin{align*}
&=\quad\sum_{\ell}\left(\frac{1}{1-\eps\delta}\right)^\ell\prstart{\tau_1=\ell,X_{\tau_1}=e^+}{v}\\
&\le\quad\left(\frac{1}{1-\eps\delta}\right)^{\ell_0}\prstart{X_{\tau_1}=e^+}{v}\quad+ \quad\sum_{\ell\ge\ell_0}\left(\frac{1}{1-\eps\delta}\right)^\ell\left(1-\frac{\eps}{\Delta+1}\right)^{\frac{\ell}{2}-1}\eps\\
&\lesssim\quad\left(\frac{1}{1-\eps\delta}\right)^{\ell_0}\prstart{X_{\tau_1}=e^+}{v}\quad+ \quad\left(\frac{\sqrt{1-\frac{\eps}{\Delta+1}}}{1-\eps\delta}\right)^{\ell_0}\frac{\eps}{1-\left(\frac{\sqrt{1-\frac{\eps}{\Delta+1}}}{1-\eps\delta}\right)}\\
&\lesssim\quad e^{c_1\ell_0\eps\delta}\prstart{X_{\tau_1}=e^+}{v}\quad+\quad e^{-c_2\ell_0\eps}.
\end{align*}
For the second sum in the second line above we used that
for any $x\neq \rho$ there is a long-range edge coming out of $x$ and that the walk cannot be at $\rho$ at two consecutive steps.
In the last inequality we used that $\log(1-\theta)\asymp-\theta$ around 0 and that $1-\left(\frac{\sqrt{1-\frac{\eps}{\Delta+1}}}{1-\eps\delta}\right)\asymp\eps$.
\end{proof}

\begin{lemma}\label{lem:visit_far_then_hit_e}
There exist positive constants $C$ and $c'$ so that the following holds.
Let $T$ be a quasi-tree as in Definition~\ref{def:tree}. Let $v$ be a vertex of $T$ and let $e$ be a long-range edge from the $R$-ball of $v$. Let $X$ be a random walk on $T$ from $v$. Then for any $\ell_0$ we have
\[\prstart{X\text{ hits level }\ell(v)+m\text{ and then hits }e^+}{v}\hspace{3cm}\]
\[\hspace{3cm}\le\quad  C\left(\ell_0\delta^m\prstart{X\text{ hits }e^+}{v}+\delta^{m}e^{-c'\eps\ell_0}\left(\ell_0+\frac1\eps\right)\right).\]
\end{lemma}

\begin{proof}
By grouping the possible paths from $v$ to $e^+$ that pass through level $\ell(v)+m$ by their restriction to the $R$-ball of $v$ we get that
\begin{align*}&\prstart{X\text{ hits level }\ell(v)+m\text{ and then hits }e^+}{v}\\
&\le\quad\sum_{\ell}\sum_{x_0,...,x_\ell}\sum_{r=0}^{\ell-1}\left(\prod_{s\ne r}\prstart{\text{hit }x_{s+1}\text{ without visiting another vertex in the }R\text{-ball}}{x_s}\right)\\
&\quad\cdot\prstart{\text{hit }x_{r+1}\text{ after visiting level }\ell(v)+m,\text{ without visiting another vertex in the }R\text{-ball}}{x_r}
\end{align*}
where the second sum is taken over all paths $x_0,...,,x_\ell$ such that $x_0=v$, $x_{\ell}=e^{+}$ and $x_0,...x_{\ell-1}$ are in the $R$-ball of $v$. This is then
\[\le\quad\sum_{\ell}\sum_{x_0,...,x_\ell}\ell\delta^{m}\prod_{s}\prstart{\text{hit }x_{s+1}\text{ without visiting another vertex in the }R\text{-ball}}{x_s}.\]
First let us bound the terms with $\ell\le\ell_0$.
\begin{align*}
&\sum_{\ell\le\ell_0}\sum_{x_0,...,x_\ell}\ell\delta^{m}\prod_{s}\prstart{\text{hit }x_{s+1}\text{ without visiting another vertex in the }R\text{-ball}}{x_s}\\
&\le\quad\sum_{\ell\le\ell_0}\ell\delta^{m}\prstart{\text{hit }e^+\text{ after visiting exactly }(\ell-1)\text{ vertices in the }R\text{-ball}\footnotemark}{v}\\
&\le\quad\ell_0\delta^m\prstart{\text{hit }e^+}{v}.
\end{align*}

\footnotetext{counting with multiplicity}

Now let us bound the terms with $\ell>\ell_0$. Note that 
\begin{align*}&\prstart{\text{hit }x_{s+1}\text{ without visiting another vertex in the }R\text{-ball}}{x_s}\\
&=\quad\sum_{j=0}^\infty\left(P(x_s,x_s^+)\prstart{\text{hit }x_s}{x_s^+}\right)^jP(x_s,x_{s+1})\\
&\le\quad \sum_{j=0}^\infty\left(\eps\delta\right)^jP(x_s,x_{s+1})\quad=\quad\frac{1}{1-\eps\delta}P(x_s,x_{s+1}).
\end{align*}
Hence
\begin{align*}&\sum_{\ell>\ell_0}\sum_{x_0,...,x_\ell}\ell\delta^{m}\prod_{s}\prstart{\text{hit }x_{s+1}\text{ without visiting another vertex in the }R\text{-ball}}{x_s}\\
&\le\quad\sum_{\ell>\ell_0}\sum_{x_0,...,x_\ell}\ell\delta^{m}\left(\frac{1}{1-\eps\delta}\right)^\ell\prod_{s}P(x_s,x_{s+1})\\
&=\quad\sum_{\ell>\ell_0}\ell\delta^{m}\left(\frac{1}{1-\eps\delta}\right)^\ell\prstart{\text{hit }e^+\text{ after exactly }\ell\text{ steps, without leaving the }R\text{-ball}}{v}.
\end{align*}
(More precisely, in the last probability above we mean that only the last step is allowed to be outside of the $R$-ball.) At each step when the walk is not in $\rho$, it has probability $\le1-\frac{\eps}{\Delta+1}$ of not leaving the $R$-ball. When the walk is in $e^-$ it has probability $\le\eps$ of crossing to $e^+$. The walk cannot be at~$\rho$ in two consecutive steps, hence
\[\prstart{\text{hit }e^+\text{ after exactly }\ell\text{ steps, without leaving the }R\text{-ball}}{v}\quad\le\quad \left(1-\frac{\eps}{\Delta+1}\right)^{\frac{\ell}{2}-1}\eps.\]
This then gives
\begin{align*}&\sum_{\ell>\ell_0}\ell\delta^{m}\left(\frac{1}{1-\eps\delta}\right)^\ell\prstart{\text{hit }e^+\text{ after exactly }\ell\text{ steps, without leaving the }R\text{-ball}}{v}\\
&\le\quad\delta^{m}\sum_{\ell\ge\ell_0}\ell\left(\frac{1}{1-\eps\delta}\right)^\ell\left(1-\frac{\eps}{\Delta+1}\right)^{\frac{\ell}{2}-1}\eps\quad=\quad\delta^{m}\frac{\eps}{1-\frac{\eps}{\Delta+1}} \sum_{\ell\ge\ell_0}\ell\left(\frac{\sqrt{1-\frac{\eps}{\Delta+1}}}{1-\eps\delta}\right)^\ell\\
&\asymp\quad\delta^{m}\left(\frac{\sqrt{1-\frac{\eps}{\Delta+1}}}{1-\eps\delta}\right)^{\ell_0}\left(\ell_0+\frac1\eps\right)\quad\lesssim\quad \delta^{m}e^{-c'\eps\ell_0}\left(\ell_0+\frac1\eps\right).
\end{align*}

In the second last step we used that for any $\ell_0$ and any $p\in(0,1)$ we have
\[\sum_{\ell\ge\ell_0}\ell p^\ell\quad = \quad  p^{\ell_0}\left(\frac{\ell_0}{1-p}+\frac{p}{(1-p)^2}\right).\]
Putting everything together gives the result.
\end{proof}

\subsection{Speed and entropy}

\begin{lemma}\label{lem:speed}
Let $X$ be a simple random walk on a random quasi-tree $\tree$ induced by a graph $G$ and let $\nu:=\frac{\E {\varphi_2-\varphi_1}}{\E {\sigma_2-\sigma_1}} $. Then $\nu\asymp \varepsilon$ and almost surely we have 
\[
\frac{  d_\tree(\rho, X_t)}{t}\to \nu \text{ as } t\to\infty.
\]
Moreover, for all $\theta>0$, $J\ge1$ there exists a positive constant $C$ so that for all $t\ge\frac{J}{\eps}$ we have
\[
\pr{|d_\tree(\rho,X_t)-\nu t|>C\sqrt{\eps t}}\leq \theta \quad \text{ and } \quad \pr{\sup_{s:\,s\leq t}d_\tree(\rho,X_s)>\nu t+ 2C\sqrt{\eps t}}\leq \theta.
\] 
\end{lemma}

The proof of this is deferred to Appendix~\ref{app:proofs_speed_entropy}.

\begin{proposition}\label{pro:aux_entropy}
Let $\tree$ be a random quasi-tree corresponding to a graph $G$ as in Definition~\ref{def:tree}, let $\xi$ and $\til{\xi}$ be two independent loop-erased random walks on $\tree$, both started from the root $\rho$ and let $T'$, $X'$, $\sigma'_1$ and $\xi'$ be as in Definition~\ref{def:Tprime}, with $\xi'$ independent of $X'$ (conditional on $T'$).
Let
\begin{align*}
\h:=&\quad\frac{1}{\E{\varphi_2-\varphi_1}}\cdot\E{-\log\prcond{X'_{\sigma'_1}\in\xi'}{X',T'}{\rho}},\\
\V:=&\quad\h^2\vee\vr{-\log\prcond{X'_{\sigma'_1}\in\xi'}{T',X'}{\rho}}.
\end{align*}
Then almost surely
\begin{align*}
\frac{-\log \prcond{\xi_k\in \til{\xi}}{T,\xi}{} }{k}\to \h \text{ as } k\to\infty.
\end{align*}

Let $Y'=-\log\prcond{X'_{\sigma'_1}\in\xi'}{X',T'}{\rho}$. Assume that $\E{(Y'-\E{Y'})^2}\lesssim\E{Y'}^2$ and\\ $\E{(Y'-\E{Y'})^4}\lesssim\E{Y'}^4$. (This will in particular imply that $\V\asymp\h^2$.)

Fix $K\geq 0$ and let $T_0$ be a realisation of the first $K$ levels of $\tree$.
Let $b(R)$ be such that $b(R)\ge\frac{c}{\eps}$ and $\log b(R)\ge c\h$ for some constant $c$ and the number of vertices in any $R$-ball of $T$ is at most $b(R)$.
Then for all $\theta>0$, there exists a positive constant $C$ (depending only on $\theta$ and $\Delta$) so that for all $k\geq\frac{(K\log b(R))^2}{\V}\vee1$ we have
\begin{align}\label{eq:entropy_deviations}
\prcond{\left|-\log \prcond{\xi_k\in \til{\xi}}{T,\xi}{} -\h k\right|>C\sqrt{k\V}}{\B_K(\rho)=T_0}{}\leq \theta.
\end{align}
\end{proposition}

Below we will state and prove a lemma used in the proof of the above proposition. Given the lemma the proof of Proposition~\ref{pro:aux_entropy} is analogous to the proof of \cite[Proposition 3.15]{random_matching} and is deferred to Appendix~\ref{app:proofs_speed_entropy}.

\begin{lemma}\label{lem:aux_entropy_var_bound}
Let us consider the setup of Proposition~\ref{pro:aux_entropy}. Then we have 
\[\vrc{-\log\prcond{\xi_{\varphi_k}\in\til\xi}{\xi,T}{\rho}}{T_0}\quad\lesssim\quad k\V+\left(K\log b(R)\right)^2.\]
\end{lemma}

\begin{proof}[Proof of Lemma~\ref{lem:aux_entropy_var_bound}.]
The proof idea is similar to the proof of \cite[Lemma 3.14]{random_matching}. The full details are given in Appendix~\ref{app:proofs_speed_entropy}.
\end{proof}

The following lemma will help us to find the order of $\h$ and $\V$ (as functions of $n$ and $\eps$) in Proposition~\ref{pro:aux_entropy} for specific sequences of graphs and to check that the assumptions of \eqref{eq:entropy_deviations} hold. The proof is given in Appendix~\ref{app:proof_of4lemmas}.

\begin{lemma}\label{lem:main_aux_entropy}
Let $G$ be as before. Let $\rho$ be any vertex of $G$ and let $T$ be any realisation of the random quasi-tree corresponding to $G$, rooted at $\rho$. Let $v$ be any vertex of $G$ such that the ball of radius $\frac{R}{2}$ around $v$ is contained in the ball of radius $R$ around $\rho$ (the balls are in graph distance). Let $X$ and $\til{X}$ be independent random walks on $T$ from $v$ and let $\tau_1$ be the first time when $X$ hits level 1 and $\tilde{\tau}_1$ be the first time when $\til{X}$ hits level 1. Let $Y$ be a simple random walk on $G$ from $v$ and let $E$ be a random variable taking values on $\Znonneg$ such that $\prcond{E=k}{E\ge k,Y_k=u}{v}=\frac{\eps}{\deg(u)+\eps}$ for all $k$ and all $u$. Let $\left(\til{Y},\til{E}\right)$ be an independent copy of $\left(Y,E\right)$. Let $b\in\Zpos$.  Let $\xi$ and $\til{\xi}$ be independent loop-erased random walks on $T$. Also let $T'$, $X'$ and $\xi'$ be as in Definition~\ref{def:Tprime}, with $T'$ rooted at $\rho$. Let $\left(\til{X}',\til{\xi}'\right)$ be an independent copy of $\left(X',\xi'\right)$ given $T'$.

Assume that there exist $h_b$ and $b(R)$ (depending on $n$) such that the following properties hold.
\begin{enumerate}[(i)]
\item\label{assump:h_b} For any realisation of $T$ and any choice of $v$ we have 
\[\E{\left(-\log\prcond{Y_E=\til{Y}_{\til{E}}}{Y}{v}\right)^b}\quad\asymp\quad h_b.\]
\item\label{assump:ball_growth} The size of all $R$-balls in $G$ is upper bounded by $b(R)$, and $b(R)$ satisfies
\[\left(\log b(R)\right)^b\left(1-\frac{\eps}{\Delta+\eps}\right)^{\frac{R}{2}}\quad\ll\quad h_b.\]
\item\label{assump:R} $R\gg\frac1\eps$.
\item\label{assump:R_eq_n} The above assumptions also hold, with the same value of $h_b$, if we set the value of $R$ to be $n$.
\end{enumerate}
Then we have
\begin{align}\label{eq:aux_entropy_main}
\econd{\left(-\log\prcond{X'_{\sigma'_1}\in\til{\xi}'}{X',T'}{\rho}\right)^b}{T'}\quad\asymp\quad h_b.
\end{align}
\end{lemma}

In what follows we will use the following notation for entropy and the analogous expectation with a higher power of the $\log$. We will also make use of some simple results listed in Appendix~\ref{app:aux_results}.

\begin{definition}\label{def:Hb}
For $b\in\Zpos$ and for a sequence $(p_i)$ of real numbers taking values in $[0,1]$ let
\[H_b(p_1,p_2,...):=\quad\sum_{i}p_i(-\log p_i)^b,\]
and for a random variable $W$ taking values in a (countable) set $\mathcal{W}$ let
\[H_b(W):=\quad\sum_{w\in\mathcal{W}}\pr{W=w}\left(-\log\pr{W=w}\right)^b.\]
In both cases $p(-\log p)^b$ is considered to be $0$ when $p=0$.
\end{definition}

\begin{definition}\label{def:Hb_conditional}
For random variables $W$, $Z$ taking values in (countable) sets $\mathcal{W}$ and $\mathcal{Z}$ respectively let
\begin{align*}H_b(W|Z):&=\quad\sum_{z\in\mathcal{Z}}\pr{Z=z}H_b(W|Z=z)\\
&=\quad\sum_{z\in\mathcal{Z}}\pr{Z=z}\sum_{w\in\mathcal{W}}\prcond{W=w}{Z=z}{}\left(-\log\prcond{W=w}{Z=z}{}\right)^b.
\end{align*}
Here $\pr{Z=z}\prcond{W=w}{Z=z}{}\left(-\log\prcond{W=w}{Z=z}{}\right)^b$ is considered to be $0$ when\\ $\pr{Z=z}=0$ or $\prcond{W=w}{Z=z}{}=0$.
\end{definition}

\section{Relating $G^*$ and $T$}\label{sec:relating_Gstar_and_T}

As mentioned earlier, we will be mostly interested in two classes of graphs $G$, graphs with linear growth of entropy, and graphs with polynomial growth of balls. In the former case we will prove cutoff when $\eps_n\gg\frac{1}{\log n}$, i.e.\ when $\eps_n$ can be written as $\eps_n=\frac{g(n)}{\log n}$ where $g(n)\le\log n$ is a function growing to infinity. In the latter case we will prove cutoff when $\eps_n=n^{-o(1)}$, i.e.\ when $\eps_n$ can be written as $\eps_n=\exp\left(-\frac{\log n}{g(n)}\right)$ where $g(n)$ is a function growing to infinity. In the two cases we will be able to write $\h$ and $\V$ in the same form using $g(n)$ and $\log n$ and will also choose the parameters in the proofs to be of the same form in terms of $\eps_n$, $g(n)$ and $\log n$.

The reason that it is possible is that in both cases the entropy of the random walk on $G$ and the growth of balls of $G$ can be expressed using the same function $f$ that relates $\eps_n$ and $\frac{\log n}{g(n)}$. This motivates the following assumption.

\begin{assumption}\label{assump:f}
$ $
\begin{itemize}
\item $f(t)$ is a continuous increasing function $\Rpos\to\Rpos$ that satisfies $f(0^+)=0^+$, $f(\infty)=\infty$ and there exist positive constants $c$ and $C$ such that $c\log t\leq f(t)\leq C t$.
\item For any positive constant $\widehat{c},$ there exist positive constants $c$ and $C$ depending on $\widehat{c}$ such that for any $t\le \frac{\widehat{c}}{\eps_n}$, $b\in\{1,2,4\}$ and for all $x_0$, a random walk $X$ on $G$ with $X_0=x_0$ satisfies that $c f(t)^b\leq H_b(X_t)\leq C f(t)^b$.
\item There exist a function $b$ and positive constants $c$ and $C$ such that for any $t$ the size of any ball of radius $t$ in $G$ is upper bounded by $b(t)$, and we have $cf(t)\le \log b(t)\le C f(t)$ for all $t$.
\item There exist positive constants $C$ and $\beta$ such that $\eps_n\leq C(\log n)^{-\beta}$. \footnote{We will consider the regime $\eps_n\gg(\log n)^{-\frac13}$ in Section~\ref{sec:other_eps}.}
\item There exists a function $g(n)$ growing to infinity and there exist positive constants $c$ and $C$ such that $c\frac{\log n}{g(n)} \leq f\left(\frac1{\eps_n}\right)\leq C\frac{\log n}{g(n)}$.
\item For any positive constant $\widehat{c}$ there exist positive constants $c$ and $C$ such that $c f\left(\frac1{\eps_n}\right) \leq f\left(\frac{\widehat{c}}{\eps_n}\right)\leq C f\left(\frac1{\eps_n}\right)$ for all sufficiently large values of $n$.
\item There exists a positive constant $c$ such that for all $x,y\geq c$ we have $f(xy)\leq f(x)f(y)$.
\end{itemize}
\end{assumption}

Throughout this section and the next section, we will assume that Assumption~\ref{assump:f} holds (in addition to the assumptions in Definition~\ref{def:Gstar} and the assumption that the graphs $G_n$ have bounded degrees and $\eps_n\to0$).

In Section~\ref{sec:specific_graphs} we will show that graphs with linear growth of entropy satisfy it with $f(t)=t$, while graphs with polynomial growth of balls satisfy it with $f(t)=\log (1+t)$.

\begin{lemma}\label{lem:entropy_from_assump_f}
Assume that Assumption~\ref{assump:f} holds and assume that $R\gtrsim\frac1\eps\log g(n)$. Let $K\asymp\frac{g(n)\log g(n)}{\log n}\asymp\frac{\log g(n)}{f\left(\frac1\eps\right)}$ and let $T_0$ be a realisation of the first $K$ levels of the random quasi-tree $T$ corresponding to $G$.
Then for some $\h\asymp f\left(\frac1\eps\right)$ and $\V\asymp f\left(\frac1\eps\right)^2$ for all $\theta>0$ there exists a positive constant $C$ (depending only on $\theta$ and $\Delta$) so that for all $k\ge\frac{(K\log b(R))^2}{\V}$ we have
\begin{align}
\prcond{\left|-\log \prcond{\xi_k\in \til{\xi}}{T,\xi}{} -\h k\right|>C\sqrt{k\V}}{\B_K(\rho)=T_0}{}\leq \theta.
\end{align}
\end{lemma}

\begin{proof}
We will show that for $b\in\{1,2,4\}$ the assumptions of Lemma~\ref{lem:main_aux_entropy} hold with $h_b\asymp f\left(\frac1\eps\right)^b$ and that Proposition~\ref{pro:aux_entropy} holds.

Let $Y$ and $E$ be defined as in Lemma~\ref{lem:main_aux_entropy}. First we wish to show that $H_b(Y_E)\asymp f\left(\frac1\eps\right)^b$. By Lemma~\ref{lem:H_b_cond_properties} we know that
\[
H_b\left(Y_E\middle|E\right)\quad\le\quad H_b(Y_E)\quad\lesssim\quad H_b\left(Y_E\middle|E\right)+H_b(E),
\]
so it is sufficient to show that $H_b\left(Y_E\middle|E\right)\asymp f\left(\frac1\eps\right)^b$ and $H_b(E)\lesssim f\left(\frac1\eps\right)^b$.

First, note that for any $k\in\Znonneg$ we have
\begin{align*}
\pr{E=k}\quad&=\quad\left(\prod_{i=1}^k\prcond{E\ge i}{E\ge i-1}{}\right)\prcond{E=k}{E\ge k}{}\\
&\le\quad\left(\frac{\Delta}{\Delta+\eps}\right)^k\frac{\eps}{1+\eps}\quad\asymp\quad\left(\frac{\Delta}{\Delta+\eps}\right)^k\frac{\eps}{\Delta+\eps},
\end{align*}
therefore $H_b(E)\lesssim H_b\left(\Geomnonneg{\frac{\eps}{\Delta+\eps}}\right)\asymp\log\left(\frac1\eps\right)^b\lesssim f\left(\frac1\eps\right)^b$.

Note that for any $t$ we have
\[H_b(Y_t|E=t)\quad\lesssim\quad\left(\log b(t)\right)^b\quad\lesssim\quad f(t)^b.\]

Using this and that $f$ is increasing we get that for any fixed constant $C$ we have
\begin{align*}
H_b\left(Y_E\middle|E\right)d&=\sum_{t}\pr{E=t}H_b(Y_t|E=t)\lesssim\sum_{t}\pr{E=t}f(t)^b\\
&\le \pr{E\le \frac{C}{\varepsilon}}f\left(\frac{C}{\varepsilon}\right)^b +\sum_{t\geq \frac{C}{\eps}}\pr{E=t}f(t)^b\lesssim f\left(\frac{1}{\varepsilon}\right)^b+ \sum_{t\geq \frac{C}{\eps}}\pr{E=t}f(t)^b.
\end{align*}
Let $C$ be a sufficiently large constant so that $f(xy)\le f(x)f(y)$ for all $x,y\ge C$. Using this submultiplicativity and that $f(s)\lesssim s$, we get that
\begin{align*}
\sum_{t\ge \frac{C}{\eps}}\pr{E=t}f(t)^b\quad\lesssim\sum_{t\ge \frac{C}{\eps}}\pr{E=t}f\left(\frac1\eps\right)^b\left(\eps t\right)^b\quad\le f\left(\frac1\eps\right)^b\E{(\eps E)^b} \quad\asymp f\left(\frac1\eps\right)^b.
\end{align*}

This proves $H_b\left(Y_E\middle|E\right)\lesssim f\left(\frac{1}{\eps}\right)^b$.

For the lower bound on $H_b\left(Y_E\middle|E\right)$ we will show that for any positive constant $C$ for all $t\le\frac{C}{\eps}$ and all $u$ we have
\begin{align}\label{eq:Y_cond_E}
\prcond{Y_t=u}{E=t}{v}\quad\asymp\quad\prstart{Y_t=u}{v},
\end{align}
hence
\begin{align}\label{eq:Hb_Y_cond_E}
H_b(Y_t|E=t)\quad\asymp\quad H_b(Y_t)\quad\asymp\quad f(t)^b.
\end{align}
For any $t\le\frac{C}{\eps}$ and any $u$ we have $\prcond{Y_t=u}{E=t}{}=\frac{\prcond{E=t}{Y_t=u}{}}{\pr{E=t}}\prstart{Y_t=u}{}$. For any given path $u_1,...,u_t$ in $G$ we have
\[\left(\frac{1}{1+\eps}\right)^{t}\frac{\eps}{\Delta+\eps}\quad\le\quad\prcond{E=t}{Y_1=u_1,...,Y_t=u_t}{}\quad\le\quad\left(\frac{\Delta}{\Delta+\eps}\right)^{t}\frac{\eps}{1+\eps}.\]
Here
\[0\quad\ge\quad\log\left(\left(\frac{1}{1+\eps}\right)^t\right)\quad\asymp\quad t\log\left(1-\frac{\eps}{1+\eps}\right)\quad\asymp\quad-t\eps\quad\gtrsim\quad-1,\]
hence $\left(\frac{1}{1+\eps}\right)^{t}\frac{\eps}{\Delta+\eps}\asymp\eps$. Similarly $\left(\frac{\Delta}{\Delta+\eps}\right)^{t}\frac{\eps}{1+\eps}\asymp\eps$. This shows that for any $u$ we have $\prcond{E=t}{Y_t=u}{}\asymp\eps\asymp\pr{E=t}$. This finishes the proof of~\eqref{eq:Y_cond_E}.

Now using \eqref{eq:Hb_Y_cond_E} we get that
\begin{align*}
H_b(Y_E|E)\quad&=\quad\sum_{t}\pr{E=t}H_b(Y_t|E=t)\quad\gtrsim\quad\sum_{t=1}^{C/\eps}\pr{E=t}f(t)^b\\
&\ge\quad\pr{\frac{c}{\eps}\le E\le\frac{C}{\eps}}f\left(\frac{c}{\eps}\right)^b\quad\gtrsim\quad\pr{\frac{c}{\eps}\le E\le\frac{C}{\eps}} f\left(\frac{1}{\eps}\right)^b\quad\asymp\quad f\left(\frac{1}{\eps}\right)^b
\end{align*}
for sufficiently small values of the constant $c$ and sufficiently large values of the constant $C$. This proves $H_b\left(Y_E\middle|E\right)\gtrsim f\left(\frac{1}{\eps}\right)^b$, finishing the proof of $H_b(Y_E)\asymp f\left(\frac1\eps\right)^b$.

Now we show that all assumptions of Lemma~\ref{lem:main_aux_entropy} hold. We proved above that assumption~\eqref{assump:h_b} holds with $h_b\asymp f\left(\frac1\eps\right)^b$. It is immediate to see that assumption~\eqref{assump:R} holds. Note that
\[\left(\log b(R)\right)^b\quad\asymp\quad f(R)^b\quad\lesssim\quad f\left(\frac1\eps\right)^bf(\log g(n))^b\quad\lesssim\quad f\left(\frac1\eps\right)^b\left(\log g(n)\right)^b\]
and
\[\left(1-\frac{\eps}{\Delta+\eps}\right)^{\frac{R}{2}}\quad\lesssim\quad\exp\left(-c_1\eps R\right)\quad\lesssim\quad\exp\left(-c_2\log g(n)\right)\quad\ll\quad\left(\log g(n)\right)^{-b}\]
for some positive constants $c_1$, $c_2$, hence
\[\left(\log b(R)\right)^b\left(1-\frac{\eps}{\Delta+\eps}\right)^{\frac{R}{2}}\quad\ll\quad f\left(\frac1\eps\right)^b,\]
so assumption~\eqref{assump:ball_growth} also holds.

Note that they would also hold with $R=n$. Then using Lemma~\ref{lem:main_aux_entropy}, we get that in Proposition~\ref{pro:aux_entropy} we have $\h\asymp\E{Y'}\asymp f\left(\frac1\eps\right)$, $\V\asymp\E{(Y')^2}\asymp f\left(\frac1\eps\right)^2$, $\E{(Y'-\E{Y'})^2}\lesssim\E{(Y')^2}\asymp f\left(\frac1\eps\right)^2$ and $\E{(Y'-\E{Y'})^4}\lesssim\E{(Y')^4}\asymp f\left(\frac1\eps\right)^4$. Also $b(R)\ge R\gtrsim\frac1\eps$ and $\log b(R)\asymp f(R)\gtrsim f\left(\frac1\eps\right)\asymp\h$. Applying Proposition~\ref{pro:aux_entropy} gives the result.
\end{proof}

\subsection{Values of some parameters}

\begin{definition}\label{def:variable_values}
Let
\[t_0:=\frac{\log n}{\nu\h}\asymp\frac1\eps g(n),\qquad t_w:=\frac{1}{\nu\h}\frac{\log n}{\sqrt{g(n)}}\asymp\frac1\eps\sqrt{g(n)}\]
where $\nu$ is the speed from Lemma~\ref{lem:speed} and~$\h$ is from Lemma~\ref{lem:entropy_from_assump_f}.

In what follows also let
\[R:=C_1\frac1\eps\log g(n),\qquad K:=C_2\log g(n),\qquad M:=C_3\log g(n),\qquad L:=\frac12\nu(t_0+Bt_w)\]
for some constants $C_1$, $C_2$, $C_3$ and $B$ to be chosen later. Let us assume that $C_2\ge2C_3$.
\end{definition}

\subsection{$K$-roots and reversing paths}\label{subsec:K_roots_reversing_paths}

In this section we define $K$-roots the same way as in~\cite{random_matching} and we define events $\Omega_0$ and $\Omega_1$ on which we will be able to consider reversal of paths. We prove some results that will mean that in the later proofs it is sufficient to only consider $K$-roots and only consider paths that can be reversed.

\begin{definition}\label{def:Kroot}
We call a vertex $x$ of $G^*$ a $K$-root if $\B_K^*(x)$ (as in Definition~\ref{def:long_range_dist_on_G}) is a possible realisation of the first $K$ levels of the random quasi-tree $T$ corresponding to $G$. If $x$ is a $K$-root and $i \le K$, we denote by $\partial \B^*_{i}(x)$ the collection of centres of $\tree\rball$s at (long-range) distance $i$ from~$x$. 
\end{definition}

\begin{lemma}\label{lem:hit_K_root_quickly}
There exists a positive constant $b$ such that with high probability $G^*$ is such that starting a random walk from any vertex $x$ with high probability we hit a $K$-root within $\frac{bK}{\eps}$ steps.
\end{lemma}

\begin{proof}

Let $x$ be any vertex of $G$ and let us explore the ball $\B_{3K}^*(x)$ around it as follows. First let us consider the $R$-ball of $x$ (i.e.\ the set of vertices that are within graph distance $R$ from $x$). For each $v\ne x$ in this $R$-ball let us reveal the long-range edge starting from it and consider the $R$-ball around the newly revealed endpoint. If it contains any previously revealed vertex, let us say that an overlap occurred. Proceed similarly from each vertex that is in one of the new $R$-balls, but not a centre and was not already considered before. Continue the exploration for $3K$ levels.

Note that $|\B^*_{3K}(x)|\le b(R)^{3K+1}$, hence every time we reveal the long-range edge and the corresponding $R$-ball from a vertex $v$ of $\B_{3K}^*(x)$, the probability that an overlap occurs is $\le\frac{b(R)^{3K+2}}{n}$. Therefore the total number $I$ of overlaps is $\le_{\textrm{st}}\Bin{b(R)^{3K+1}}{\frac{b(R)^{3K+2}}{n}}$. Then
\[\pr{I\ge2}\quad\lesssim\quad\frac{1}{n^2}b(R)^{12K+6}\quad\ll\quad\frac{1}{n}\]
as $\log\left(b(R)^{12K+6}\right)\asymp K\log b(R)\asymp Kf(R)\lesssim Kf\left(\frac1\eps\right)f(\log g(n))\lesssim\log g(n)\frac{\log n}{g(n)}\log g(n)$ $\ll\log n$.

Taking union bound over all vertices of $G$ we get that whp for each $x$ there is at most one overlap in the ball $\B^*_{3K}(x)$. If there is no overlap, then $x$ is a $K$-root and we are done. Assume there is one overlap, say $y$ and $z$ are distinct vertices in $\B^*_{3K}(x)$ and $(y,y')$ and $(z,z')$ are long-range edges such that $\B_{G}(y',R)\cap\B_{G}(z',R)\ne\emptyset$. Let us consider the path from $x$ to $y$ that contains the fewest possible long-range edges, contains at most $R$ edges of $G$ in a row and does not cross the same long-range edge twice and does not pass through $z$. Let $y_0$ be the last vertex of this path that is in the $R$-ball of $x$. Let us define $z_0$ similarly. Repeating the proof of Lemma~\ref{lem:neverbacktrack} we get that with high probability the walk from $x$ will first leave the $R$-ball of $x$ via a vertex other than $y_0$ and $z_0$ and will reach (long-range) distance $2K$ without returning to this $R$-ball. In this case the first vertex reached by the walk that is at distance $2K$ from $x$ will be a $K$-root. Whp this vertex is reached in time $\le \frac{4K}{\nu}\asymp\frac{K}{\eps}$.
\end{proof}

\begin{lemma}\label{lem:few_non_K_roots}
With high probability the number of vertices in $G^*$ that are not $K$-roots is $o(n)$.
\end{lemma}

\begin{proof}
To determine whether a given vertex $x$ of $G$ is a $K$-root let us explore the ball $\B_K^*(x)$ around it as in the proof of Lemma~\ref{lem:hit_K_root_quickly}.

Every time we reveal the long-range edge and the corresponding $R$-ball from a vertex $v$ of $\B_K^*(x)$, the probability that an overlap occurs is $\le\frac{b(R)^{K+2}}{n}$. Therefore the total number $I$ of overlaps is $\le_{\textrm{st}}\Bin{b(R)^{K+1}}{\frac{b(R)^{K+2}}{n}}$. Then
\[\pr{x\text{ is not a }K\text{-root}}\quad=\quad\pr{I\ge1}\quad\lesssim\quad \frac1nb(R)^{2K+3}\quad=:p.\]
Similarly to the calculations in the proof of Lemma~\ref{lem:hit_K_root_quickly} we get that $p\ll1$.

Let $J$ be the total number of non-$K$-roots in $G^*$. Then for any constant $u>0$ we have
\[\pr{J\ge n\sqrt{p}}\quad\le\quad\frac{\E{J}}{n\sqrt{p}}\quad\le\quad\frac{np}{n\sqrt{p}}\quad=\quad\sqrt{p}\quad\ll\quad1.\]
This shows that whp the number of $K$-roots is $<n\sqrt{p}=o(n)$.
\end{proof}

\begin{lemma}\label{lem:few_close_vertices}
For any realisation of $G^*$ and any vertex $x$, the number of vertices $y$ with\\ $\B_K^*(x)\cap \B_K^*(y)\ne\emptyset$ is $o(n)$.
\end{lemma}

\begin{proof}
The number of such vertices $y$ is $\le|\B_{2K}^*(x)|\le b(R)^{2K+1}$, which is $o(n)$ since\\ $\log\left(b(R)^{2K+1}\right)\lesssim Kf(R)\ll\log n$ as seen before.
\end{proof}

Given a realisation of $G^*$ we now define the corresponding quasi-tree $T$ and given a walk $X$ on $G^*$ we define the corresponding walk $\til{X}$ on $T$. In what follows we will measure how much the walk $X$ travelled by considering the long-range distance travelled by $\til{X}$ (which does not coincide with the previous definition of long-range distance on $G^*$).

\begin{definition}\label{def:T_of_Gstar}
For a vertex $v$ of $G^*$ let $\eta(v)$ denote its long-range neighbour in $G^*$.	

Given the graph $G^*$ and a vertex $v$ in it, the associated full quasi-tree $T$ and the map $\iota:T\to G^*$ are constructed as follows. Let $T$ be rooted at a copy of $v$ and let us add a copy of the $n$-ball of $v$ in $G$ around it. Let $\iota$ map each vertex in this ball to the corresponding vertex of $G^*$. For each vertex $u$ in this ball (including the root) let us add a long-range edge of weight $\eps$ from $u$ and attach a copy of the $n$-ball of $\eta(\iota(u))$ in $G$ to it. Let $\iota$ map each vertex of the new $n$-ball to the corresponding vertex in $G^*$. Repeat the same procedure from every vertex of the newly added $n$-balls except for their centres. Then proceed similarly, resulting in an infinite quasi-tree.

Given the graph $G^*$ and a vertex $v$ the associated $R$-quasi-tree $T$ and map $\iota_R$ are constructed analogously, but instead of $n$-balls we only consider $R$-balls and we do not add a long-range edge from the root. Note that this $R$-quasi-tree is a subgraph of the associated full quasi-tree and $\iota_R$ is a restriction of the map $\iota$ on the full quasi-tree.

Given a walk $X$ on $G^*$ started from $v$ we define the associated walk $\til{X}$ on the associated full quasi-tree $T$ as follows. Let $\til{X}$ start from the root of $T$. Then $\iota(\til{X}_0)=X_0$. Assume that $\iota(\til{X}_k)=X_k$. If $X_{k+1}$ is a neighbour of $X_k$ in $G$, then let $\til{X}_{k+1}$ be the unique vertex in the $n$-ball of $\til{X}_k$ in $T$ such that $\iota(\til{X}_{k+1})=X_{k+1}$. Then $\til{X}_{k+1}$ is also a neighbour of $\til{X}_k$ in $T$. If $X_{k+1}$ is the long-range neighbour of $X_k$ in $G^*$ then let $\til{X}_{k+1}$ be the long-range neighbour of $\til{X}_{k}$ in $T$. Then we also have $\iota(\til{X}_{k+1})=X_{k+1}$. (If $X_{k+1}$ is a neighbour of $X_k$ in $G$ and is also the long-range neighbour of $X_k$, then we apply the first case with probability $\frac{1}{1+\eps}$ and the second case otherwise.)

The corresponding walk on the corresponding $R$-quasi-tree is the restriction of $\til{X}$ to the vertices of the $R$-quasi-tree, only defined up to the first time that $\til{X}$ leaves the set of these vertices.

Given a walk $X$ on $G^*$ from a vertex $v$ and given a positive integer $\ell$ we define $\tau_{\ell}$ to be the first time when the corresponding walk $\til{X}$ in the corresponding full quasi-tree $T$ reaches level $\ell$.
\end{definition}

Note that if $X$ is a random walk on $G^*$ started from $v$, then $\til{X}$ is a random walk on $T$ started from the root, with $\iota(\til{X}_k)=X_k$ for all $k$. Also if $\til{X}$ is a random walk on $T$ from the root, then $\iota(\til{X}_k)$ is a random walk on $G^*$ from $v$.

\begin{definition}\label{def:Omega_zero}
Let $X$ be a walk on $G^*$ and let $\ell$ be a positive integer. Let $T$ be the corresponding full quasi-tree and $\til{X}$ be the corresponding walk on $T$. Let $\Omega^{X}_0(\ell)$ be the event that up to time $\tau_{\ell}$ the walk $\til{X}$ visits at most $\frac{R}{2}$ distinct vertices in each ball of $T$ and that up to time $\tau_{\ell}$ the walk $\til{X}$ does not cross the long-range edge from the root of $T$.
\end{definition}

Note that if the event $\Omega^{X}_0(\ell)$ holds then the walk corresponding to $X$ on the corresponding $R$-quasi-tree is defined up to time at least $\tau_{\ell}$.

\begin{lemma}\label{lem:bound_Omega_zero_prob}
Let $X$ be a random walk on $G^*$ from vertex $x$ and let $\ell\lesssim g(n)$ be a positive integer. Then for any sufficiently large values of $C_1$ in the definition of $R$ we have
\[
\prcond{\Omega^X_0(\ell)}{G^*}{x}\quad=\quad1-o(1)
\]
uniformly over all possible realisations of $G^*$ and all choices of $x$.
\end{lemma}

\begin{proof}
Let $T$ be the corresponding full quasi-tree and $\til{X}$ the corresponding walk on $T$. Let $Z$ be the induced walk on the directed long-range edges of $T$. We know that $\ell(Z)$ dominates a biased random walk on $\Z$ from 0 that steps right with probability $1-o(1)$, therefore $\pr{\tau^Z_\ell\ge2\ell}=o(1)$. Also the probability that $\til{X}$ ever crosses the long-range edge $(\rho,\rho^+)$ from the root is $o(1)$.

At each step the probability that the walk $\til{X}$ crosses a long-range edge and later never crosses that long-range edge back is $\ge\frac{\eps}{\Delta+\eps}(1-\delta)=:\epsilon\ge\frac{\eps}{2\Delta}$, therefore
\[
1-\prcond{\Omega^X_0(\ell)}{G^*}{x}\le\prcond{\tau^{\til{X}}_{\rho^+}<\infty}{G^*}{}+\prcond{\tau^Z_\ell\ge2\ell}{G^*}{}+2\ell\cdot\pr{\Geomnonneg{\epsilon}\ge\frac12R}.
\]

We already know that $\prcond{\tau^{\til{X}}_{\rho^+}<\infty}{G^*}{}=o(1)$ and $\prcond{\tau^Z_\ell\ge2\ell}{G^*}{}=o(1)$. We have
\[
\log\pr{\Geomnonneg{\epsilon}\ge\frac12R}\quad=\quad\frac12R\log\left(1-\epsilon\right)\quad\le\quad -\frac{1}{4\Delta}R\eps\quad\le\quad-\frac{C_1}{4\Delta}\log g(n).
\]
This gives 
\[
2\ell\cdot\pr{\Geomnonneg{\epsilon}\ge\frac12R}\quad\lesssim\quad g(n)^{1-C_1/(4\Delta)}
\]
which is $o(1)$ for sufficiently large values of $C_1$.
\end{proof}

\begin{definition}\label{def:Omega_one}
Let $X$ be a walk on $G^*$ and let $\ell$ be a positive integer. Let $T$ be the corresponding full quasi-tree and let $\til{X}$ be the corresponding walk on $T$. Let $\Omega^{X}_1(\ell)$ be the event that the following is satisfied. For $k=1,2,...,\lfloor\frac{\ell}{2}\rfloor$ up to time $\tau_\ell$ the walk $\til{X}$ only visits level $2k-1$ of $T$ at descendants of $\til{\xi}^{(\ell)}_k$, where $\til{\xi}^{(\ell)}$ denotes the loop-erasure of $(\til{X}_i)_{i=0}^{\tau_\ell}$.
\end{definition}

\begin{lemma}\label{lem:bound_Omega_one_prob}
Let $X$ be a random walk on $G^*$ from vertex $x$ and let $\ell\lesssim g(n)$ be a positive integer. Then we have
\[
\prcond{\Omega^X_1(\ell)}{G^*}{x}\quad=\quad1-o(1)
\]
uniformly over all possible realisations of $G^*$ and all choices of $x$.
\end{lemma}

\begin{proof} Let $T$ be the corresponding full quasi-tree and let $\til{X}$ be the corresponding walk on it. Let $\Omega'_1$ be the event that for $k=1,2,...$ after hitting level $2k-1$ the walk $\til{X}$ never returns to level $k-1$. Note that if $\Omega'_1$ holds, then $\Omega_1(\ell)$ holds for all $\ell$. Also note that by Lemma~\ref{lem:neverbacktrack} the probability that $\til{X}$ ever returns to a level $k$ after hitting level $k+m$ is $\delta^m$. Hence
\[
1-\prcond{\Omega^X_1(\ell)}{G^*}{x}\quad\le\quad1-\prcond{\Omega'_1}{G^*}{x}\quad\le\quad\sum_{k=1}^{\infty}\delta^k\quad\lesssim\quad\delta\quad=\quad o(1).\qedhere
\]	
\end{proof}

For a finite path $\gamma$ in $G^*$ we define $\Omega_0^\gamma(\ell)$ and $\Omega_1^\gamma(\ell)$ analogously to Definitions~\ref{def:Omega_zero} and~\ref{def:Omega_one}. (If the path $\til{\gamma}$ on the corresponding full quasi-tree $T$ does not reach distance $\ell$, then these events are not defined.)

\begin{lemma}\label{lem:reversal_under_Omega_zero_one}
Let $x=(x_0,\ldots, x_{k})$ be a path in $G^*$ such that $k=\tau^x_{\ell}$ and $\Omega^x_0(\ell)$ and $\Omega^x_1(\ell)$ hold. Define a path $z$ by setting $z_0=\eta(x_k)=x_{k-1}$, $z_1=x_{k-2}$, $z_2=x_{k-3}$, ..., $z_{k-1}=x_0$, $z_k=\eta(x_0)$. Then we have $k=\tau^z_{\ell}$, and the events $\Omega^z_0(\ell)$ and $\Omega^z_1(\ell)$ hold.
\end{lemma}

\begin{proof}
Let $T$ be the full quasi-tree corresponding to $G^*$ and $x_0$, let $\iota$ be the corresponding map and let $y$ be the path on $T$ corresponding to $x$. Let $y_{-1}$ be the long-range neighbour of $y_0$. Note that rerooting $T$ at $y_{k-1}$ it becomes a full quasi-tree corresponding to $G^*$ and $z_{0}$, and $(y_{k-1-i})$ is the path on it corresponding to $z$.

The condition that $k=\tau^z_{\ell}$, $\Omega^z_0(\ell)$ and $\Omega^z_1(\ell)$ hold, and the condition that $k=\tau^x_{\ell}$, $\Omega^x_0(\ell)$ and $\Omega^x_1(\ell)$ hold are both equivalent to the following points being satisfied.
\begin{itemize}
\item The vertex $y_k$ is at long-range distance $(\ell+1)$ from $\eta(y_0)$.
\item The path $\eta(y_0),y_0,y_1,...,y_k$ has at most $\frac{R}{2}$ distinct vertices in each ball of $T$.
\item Let $B_{-1}$, $B_0$, $B_1$, ..., $B_\ell$ (in this order) be the balls in $T$ visited by the loop-erasure of the path $\eta(y_0),y_0,y_1,...,y_k$. For a ball $B$ in $T$ let $i_B=\argmin_{i\in\{-1,0,...,\ell\}}d_T(B,B_i)$, i.e.\ the index of the ball $B_i$ at minimal long-range distance from $B$. (This is well-defined because of the tree structure of $T$.)\\
The path $y_0,y_1,...,y_{k-1}$ never visits a ball $B$ with $i_B+d_T(B,B_{i_B})>\ell$ or $(\ell-i_B)+d_T(B,B_{i_B})>\ell$. (Note that this is equivalent to saying that $y_0,y_1,...,y_{k-1}$ never visits a ball $B$ that is at long-range distance $>\ell$ from $B_{-1}$ or at long-range distance $>\ell$ from $B_{\ell}$.) See Figure~\ref{fig:eg_reversibility_condition} for an illustration.
\end{itemize}
Also see Figures~\ref{fig:noneg_Omega_one_1} and~\ref{fig:noneg_Omega_one_2} for an example where $\Omega_1^x(\ell)$ does not hold and $\tau^z_{\ell}\ne k$.
\end{proof}

\begin{minipage}{0.6\linewidth}
\includegraphics[width=0.9\textwidth]{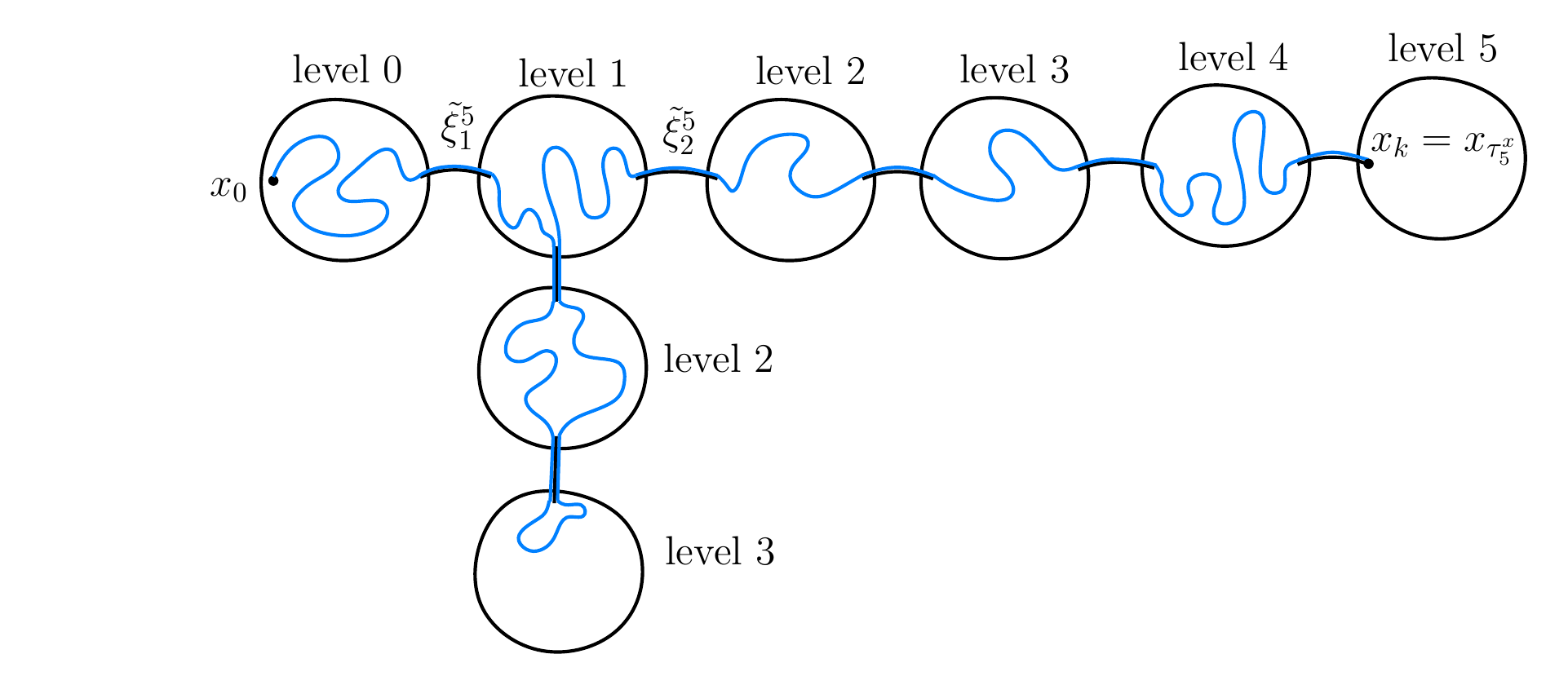}
\end{minipage}
\begin{minipage}{0.38\linewidth}
\captionof{figure}{In this example $\Omega^x_1(5)$ does not hold, since $x$ reaches (long-range) distance $3$ from $x_0$ at an $R$-ball that is not a descendant of $\til{\xi}^{5}_2$.}
\label{fig:noneg_Omega_one_1}
\end{minipage}

\begin{minipage}{0.6\linewidth}
\includegraphics[width=0.9\textwidth]{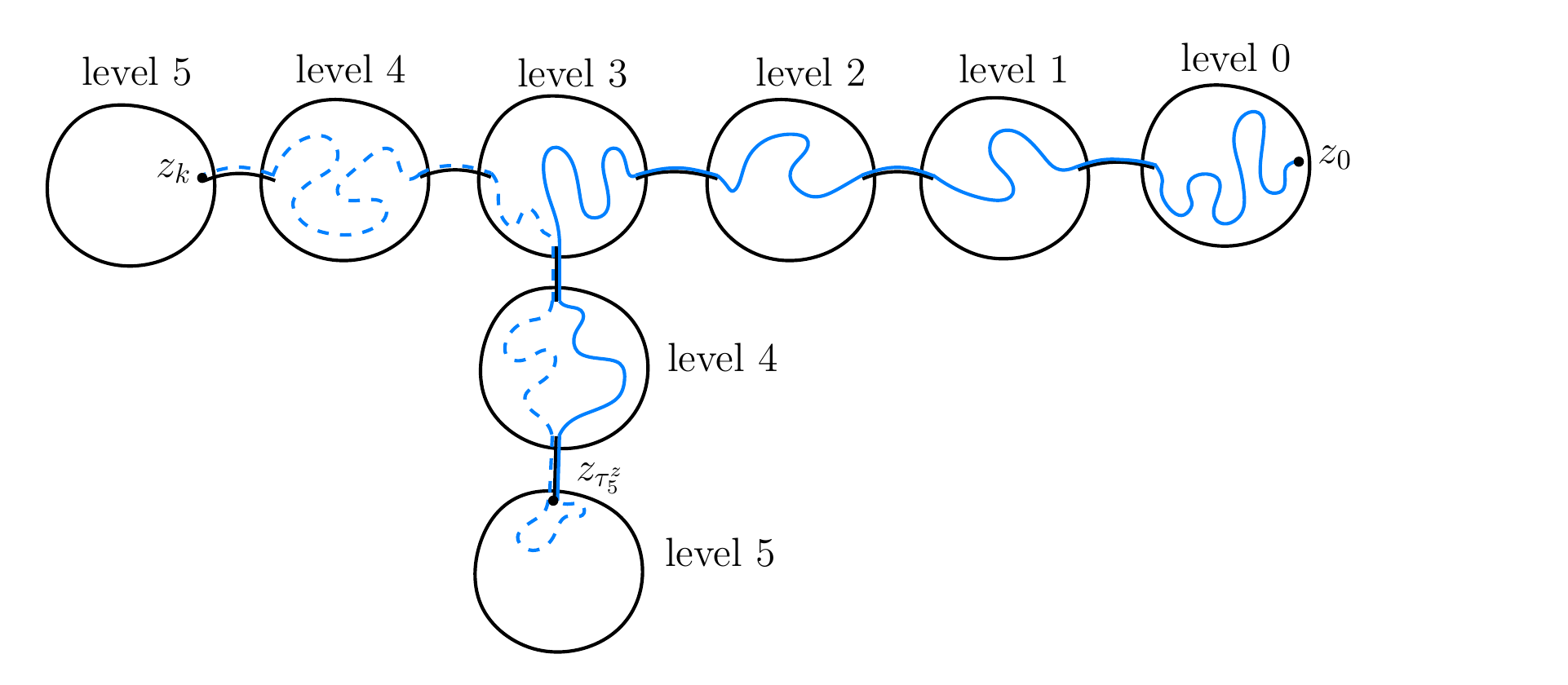}
\end{minipage}
\begin{minipage}{0.38\linewidth}
\captionof{figure}{Here $\tau^z_{\ell}\ne k$, since the reversed path $z$ reaches distance $5$ from $z_0$ before time $k$.}
\label{fig:noneg_Omega_one_2}
\end{minipage}

\begin{lemma}\label{lem:bound_tv_dist_at_tau_two_ell}
Let $\mu$ be a distribution on the vertices with $\mu(z)\asymp\frac1n$ $\forall z$ and let $\ell\lesssim g(n)$ be a positive integer. Then for any realisation of $G^*$ and any vertices $x$ and $y$ we have
\begin{align*}
&\dtv{\prcond{X_{\tau_{2\ell}}\in\cdot}{G^*}{x}}{\mu(\cdot)}\\
&\le\sum_y\left(\mu(y)-\sum_{z}\prcond{X_{\tau_{\ell}}=z,\Omega_0(\ell),\Omega_1(\ell)}{G^*}{x}\prcond{X_{\tau_{\ell}}=\eta(z),\Omega_0(\ell),\Omega_1(\ell)}{G^*}{\eta(y)}\right)^++o(1).
\end{align*}
\end{lemma}

\begin{proof}
In what follows we consider $G^*$ fixed and omit it from the notation.

For any $x$ and $y$ we have
\begin{align*}
\prstart{X_{\tau_{2\ell}}=y}{x}\quad \ge &\quad\sum_{z}\prstart{X_{\tau_{\ell}}=z,\Omega_0(\ell),\Omega_1(\ell)}{x}\prstart{X_{\tau_{1}}\ne\eta(z),X_{\tau_{\ell}}=y,\Omega_0(\ell),\Omega_1(\ell)}{z} \\ 
=& \quad \sum_{z}\prstart{X_{\tau_{\ell}}=z,\Omega_0(\ell),\Omega_1(\ell)}{x}\prstart{X_{\tau_{\ell}}=y,\Omega_0(\ell),\Omega_1(\ell)}{z} \\ &\quad - \sum_z\prstart{X_{\tau_{\ell}}=z,\Omega_0(\ell),\Omega_1(\ell)}{x}\prstart{X_{\tau_{1}}=\eta(z),X_{\tau_{\ell}}=y,\Omega_0(\ell),\Omega_1(\ell)}{z}
\end{align*}

Note that by Lemma~\ref{lem:reversal_under_Omega_zero_one} we have that 
\[
\prstart{X_{\tau_{\ell}}=y,\Omega_0(\ell),\Omega_1(\ell)}{z}\quad=\quad\prstart{X_{\tau_{\ell}}=\eta(z),\Omega_0(\ell),\Omega_1(\ell)}{\eta(y)}.
\]
Using the definition of the events $\Omega_0$ and $\Omega_1$ we also see that 
\[
\prstart{X_{\tau_{1}}=\eta(z),X_{\tau_{\ell}}=y,\Omega_0(\ell),\Omega_1(\ell)}{z}\leq 
\prstart{X_{\tau_{1}}=\eta(z)}{z}\prstart{X_{\tau_{\ell-1}}=y,\Omega_0(\ell-1),\Omega_1(\ell-1)}{\eta(z)}.
\]
Using this we then obtain
\begin{align*}
&\sum_{y,z}\prstart{X_{\tau_{\ell}}=z,\Omega_0(\ell),\Omega_1(\ell)}{x}\prstart{X_{\tau_{1}}=\eta(z),X_{\tau_{\ell}}=y,\Omega_0(\ell),\Omega_1(\ell)}{z}\\
&\le\quad\sum_{y,z}\prstart{X_{\tau_{\ell}}=z,\Omega_0(\ell),\Omega_1(\ell)}{x}\prstart{X_{\tau_{1}}=\eta(z)}{z}\prstart{X_{\tau_{\ell-1}}=y,\Omega_0(\ell-1),\Omega_1(\ell-1)}{\eta(z)}\\
&\le\quad\sum_{y,z}\prstart{X_{\tau_{\ell}}=z,\Omega_0(\ell),\Omega_1(\ell)}{x}\cdot\delta\cdot\prstart{X_{\tau_{\ell-1}}=y,\Omega_0(\ell-1),\Omega_1(\ell-1)}{\eta(z)}\\
&\le\quad\delta\sum_{z}\prstart{X_{\tau_{\ell}}=z,\Omega_0(\ell),\Omega_1(\ell)}{x}\quad\le\quad\delta\quad =\quad  o(1).
\end{align*}
Altogether these give
\begin{align*}
&\dtv{\prstart{X_{\tau_{2\ell}}\in\cdot}{x}}{\mu(\cdot)}\quad=\quad\sum_y\left(\mu(y)-\prstart{X_{\tau_{2\ell}}=y}{x}\right)^+\\
&\le\quad\sum_y\left(\mu(y)-\sum_{z}\prstart{X_{\tau_{\ell}}=z,\Omega_0(\ell),\Omega_1(\ell)}{x}\prstart{X_{\tau_{\ell}}=\eta(z),\Omega_0(\ell),\Omega_1(\ell)}{\eta(y)}\right)^++o(1).
\end{align*}
This now finishes the proof.
\end{proof}

In what follows we will work towards finding a lower bound on $$\sum_{z}\prstart{X_{\tau_{\ell}}=z,\Omega_0(\ell),\Omega_1(\ell)}{x}\prstart{X_{\tau_{\ell}}=\eta(z),\Omega_0(\ell),\Omega_1(\ell)}{\eta(y)}$$ for most values of $y$ and showing that whp the contribution from the remaining vertices $y$ is $o(1)$. This will in turn give an upper bound on the total variation distance from Lemma~\ref{lem:bound_tv_dist_at_tau_two_ell} and then in Section~\ref{sec:proof_of_cutoff} we will use a relationship between mixing times and hitting times to get an upper bound on the mixing time. 

\subsection{Truncation}\label{subsec:truncation}

\begin{definition}\label{def:W_tilW}
For a quasi-tree $T$, a long-range edge $e$ let
\[
\theta_T(e):=\quad\prcond{e\in\xi}{p(e)\in\xi}{\rho}
\]
where $\xi$ is a loop-erased random walk and $p(e)$ is the "parent" of $e$, i.e.\ the first long-range on the shortest path from $e$ to $\rho$. (If $e$ is between levels 0 and 1, then $p(e)$ is not defined and $\theta_T(e)=\prstart{e\in\xi}{\rho}$.) For a long-range edge $e$ at level $1$ we let
\[
\til{\theta}_{T,M}(e):=\quad\prstart{X_{\tau_M}\in V_e}{\rho}
\]
where $X$ is a simple random walk on $T$, $M$ is as in Definition~\ref{def:variable_values}, $\tau_M$ is the first time $X$ hits level~$M$ and $V_e$ is the set of descendants of $e$.

Note that
\[
W_T(e):=\quad-\log\prcond{e\in\xi}{T}{}\quad=\quad-\sum_{i=1}^{\ell(e)}\log\theta_{T(y_{i-1})}(e_i)
\]
where we recall that $\ell(e)$ is the level of $e$ and $e_1,...,e_{\ell(e)}$ are the long-range edges from $\rho$ to $e$, $e_i=(x_i,y_i)$ and $y_0=\rho$.

Also, define
\[
\til{W}_T(e):=\quad-\sum_{i=1}^{\ell(e)}\log\til\theta_{T(y_{i-1}),M}(e_i).\qedhere
\]
\end{definition}

\begin{definition}\label{def:trunc}
For a long-range edge $e$, constant $A>0$ and given positive integer $K$ define the truncation event as
\[
\trunc{e}{A}: =\quad \left\{\til{W}_\tree(e) > \frac12\log n  + A\sqrt{\frac{\V\log n}{\h}}\right\}\cap \{\ell(e)\geq K\},
\]
where $\V$ is from Lemma~\ref{lem:entropy_from_assump_f}.
\end{definition}

This is similar to the truncation criterion used in \cite{RWs_on_random_graph}.

Note that $\sqrt{\frac{\V\log n}{\h}}\asymp\frac{\log n}{\sqrt{g(n)}}$ and we know that $\frac{\log n}{g(n)}\asymp f\left(\frac1\eps\right)\gtrsim\log\left(\frac1\eps\right)\gtrsim\log\left((\log n)^{\beta}\right)\asymp\log\log n$, hence $g(n)\lesssim\frac{\log n}{\log\log n}$, therefore $\sqrt{\frac{\V\log n}{\h}}\gtrsim\sqrt{(\log\log n)\log n}$. We will use these 
multiple times in the following proofs.

\begin{lemma} \label{lem:if_W_large_then_Wtil_large}
Let $G$ be as before, assume that Assumption~\ref{assump:f} holds and let $R$, $K$, $M$ and $L$ be as in Definition~\ref{def:variable_values}. Let $T$ be the random quasi-tree associated to $G$.

Then for any positive constants $A$ and $A'<A$ the following holds for any sufficiently large $n$.
For any realisation of $T$ and for any long-range edge $e$ of $T$ at level $L$
\begin{align*}
\text{if}\qquad&\til{W}_T(e)\ge\frac12\log n+A\sqrt{\frac{\V\log n}{\h}},\\
\text{then}\qquad&{W}_T(e)\ge\frac12\log n+A'\sqrt{\frac{\V\log n}{\h}}.
\end{align*}
\end{lemma}

\begin{proof} The subscript $T$ will be omitted. We wish to show that for any $e$ at level $L$ we have
\[
W(e)\ge\frac12\log n+A'\sqrt{\frac{\V\log n}{\h}}\qquad\text{or} \qquad\til{W}(e)\le\frac12\log n+A\sqrt{\frac{\V\log n}{\h}}.
\]
Equivalently we wish to show that
\[\prod_{i=1}^L\theta_{T(y_{i-1})}(e_i)\quad\le\quad \exp\left(-\frac12\log n-A'\sqrt{\frac{\V\log n}{\h}}\right)\hspace{3cm}\]
\[\hspace{3cm}\text{or}\quad\prod_{i=1}^L\til\theta_{T(y_{i-1}),M}(e_i)\quad\ge\quad\exp\left(-\frac12\log n-A\sqrt{\frac{\V\log n}{\h}}\right).\]

Since $\theta_{T(y_{i-1})}(e_i)\le1$ for each $i$, it is sufficient to show that for any $A'<A$ we have
\begin{align*}
&\theta_{T(y_{i-1})}(e_i)\le \exp\left(-\frac12\log n-A'\sqrt{\frac{\V\log n}{\h}}\right)\quad\text{for  some  }i\quad\text{or}\\
&\prod_{i=1}^L\til\theta_{T(y_{i-1}),M}(e_i)\ge\prod_{i=1}^L\theta_{T(y_{i-1})}(e_i)\exp\left(-(A-A')\sqrt{\frac{\V\log n}{\h}}\right).
\end{align*}
It is sufficient to show that for any edge $e$ at level 1 and any $A'<A$ we have
\[
\theta_{T}(e)\le \exp\left(-\frac12\log n-A'\sqrt{\frac{\V\log n}{\h}}\right)\quad\text{or} \quad\til\theta_{T,M}(e)\ge\theta_{T}(e)\exp\left(-\frac{A-A'}{L}\sqrt{\frac{\V\log n}{\h}}\right).
\]
Equivalently, we wish to show that for any long-range edge $e$ at level 1 and any $A'<A$ we have the following.
\begin{align} \label{eq:inxi_prob_large}
\text{If}\qquad\prstart{e\in\xi}{\rho}\ge\exp\left(-\frac12\log n-A'\sqrt{\frac{\V\log n}{\h}}\right),
\end{align}
\begin{align} \label{eq:inVe_prob_large}
\text{then}\qquad \prstart{X_{\tau_M}\in V_e}{\rho}\ge\prstart{e\in\xi}{\rho}\exp\left(-\frac{A-A'}{L}\sqrt{\frac{\V\log n}{\h}}\right).
\end{align}

In the rest of the proof let us assume that \eqref{eq:inxi_prob_large} holds. We start by noting that
\begin{align*}
\prstart{e\in\xi}{\rho}\quad&=\quad\prstart{e\in\xi,X_{\tau_M}\in V_e}{\rho}+\prstart{e\in\xi,X_{\tau_M}\not\in V_e}{\rho}\\
&\le\quad\prstart{X_{\tau_M}\in V_e}{\rho}+\prstart{e\in\xi,X_{\tau_M}\not\in V_e}{\rho}.
\end{align*}
Now our goal is to show that
\[
\prstart{e\in\xi,X_{\tau_M}\not\in V_e}{\rho}\quad\le\quad\left(1-\exp\left(-\frac{A-A'}{L}\sqrt{\frac{\V\log n}{\h}}\right)\right)\prstart{e\in\xi}{\rho}.
\]

Note that for any constants $A$ and $A'<A$, using that $\h\asymp f\left(\frac{1}{\eps}\right)$ and $\V\asymp  f\left(\frac{1}{\eps}\right)^2$ from Lemma~\ref{lem:entropy_from_assump_f} and the value of $L$ from Definition~\ref{def:variable_values} we get $\frac{A-A'}{L}\sqrt{\frac{\V\log n}{\h}}\asymp\frac{\log n}{g(n)^{3/2}}$, and hence $1-\exp\left(-\frac{A-A'}{L}\sqrt{\frac{\V\log n}{\h}}\right)\ge \frac{1}{2}\wedge\frac{c_3\log n}{g(n)^{3/2}}$ for some constant $c_3$.

Note that
\[
\prstart{e\in\xi,X_{\tau_M}\not\in V_e}{\rho}\quad\le\quad\prstart{\text{hit level }M\text{ and then hit }e^+}{\rho}.
\]

By Lemma~\ref{lem:visit_far_then_hit_e} we know that for any $\ell_0$ this is
\begin{align}\label{eq:bound_in_xi_not_in_Ve}
\lesssim\quad \ell_0\delta^M\prstart{\tau_{e^+}<\infty}{\rho}+\delta^{M}e^{-c'\eps\ell_0}\left(\ell_0+\frac1\eps\right).
\end{align}

We know that $\prstart{\tau_{e^+}<\infty}{\rho}\asymp\prstart{e\in\xi}{\rho}\ge\exp\left(-c_4\log n\right)$ for some positive constant $c_4$. Let us choose $\ell_0=\frac{c_4\log n}{c'\eps}$. Then we get that~\eqref{eq:bound_in_xi_not_in_Ve} is
\[
\lesssim\quad \frac{\log n}{\eps}\delta^M\prstart{e\in\xi}{\rho}.
\]

We have
\[
\delta^{\frac{M}{2}}\quad=\quad\exp\left(-\frac{C_3}{2}\left(\log g(n)\right)\log\left(\frac{1}{\delta}\right)\right)\quad\lesssim\quad\exp\left(-\frac32\log g(n)\right)\quad=\quad\frac{1}{g(n)^{3/2}},
\]
and hence using that $\delta\lesssim \eps^{1/3}$ from Lemma~\ref{lem:neverbacktrack} we get that 
\[
\frac{\log n}{\eps}\delta^M\quad\ll\quad\frac{\log n}{g(n)^{3/2}},
\]
and also
\[
\frac{\log n}{\eps}\delta^M\quad\lesssim\quad\exp\left(\log\log n-\frac{C_3}{2}\left(\log g(n)\right)\log\left(\frac1\delta\right)\right)\quad\ll\quad1.
\]
For the last step above we used that $\log (1/\delta)\asymp \log (1/\eps)$ and the assumption that $\eps\lesssim(\log n)^{-\beta}$.

Together these show that for all sufficiently large $n$ we have
\[
\prstart{e\in\xi,X_{\tau_M}\not\in V_e}{\rho}\quad\le\quad\left(1-\exp\left(-\frac{A'-A}{L}\sqrt{\frac{\V\log n}{\h}}\right)\right)\prstart{e\in\xi}{\rho}
\]
which finishes the proof.
\end{proof}

\begin{proposition}\label{pro:bound_trunc_prob}
Let us consider the setup of Lemma~\ref{lem:if_W_large_then_Wtil_large}. Let $T_0$ be any realisation of the first $K$ levels of $T$. Let $X$ be a simple random walk on $T$ started from its root. Then for all $\theta\in(0,1)$ there exist $B$ (depending on $\theta$) and $A$ (depending on $\theta$ and $B$) sufficiently large such that
\[
\prcond{\bigcup_{k\le\tau_{L}}\trunc{(X_{k-1},X_k)}{A}}{B_K(\rho)=T_0}{}<\theta.
\]
\end{proposition}

\begin{proof} The proof is analogous to the proof of \cite[Lemma 4.5]{random_matching} and we omit the details.
\end{proof}

\begin{definition}\label{def:truncprime}
Given a quasi-tree $T$ and $L$ as in Definition~\ref{def:variable_values}, we define the level $L$ truncation event for a long-range edge $e$ as
\[\truncprime{e}:=\left\{\ell(e)=L,\quad-\log\prstart{X_{\tau_L}=e^+}{\rho}<\frac12\log n+\log\log n\right\}.\qedhere\]
\end{definition}

\begin{proposition} \label{pro:bound_truncprime_prob}
Let us consider the setup of Lemma~\ref{lem:if_W_large_then_Wtil_large}. Let $T_0$ be a realisation of the first $K$ levels of $T$. Let $X$ be a simple random walk on $T$ started from its root. Then for all $\theta\in(0,1)$ there exist $B$ (as in Definition~\ref{def:variable_values}, depending on $\theta$) sufficiently large such that
\[
\prcond{\truncprime{\left(X_{\tau_{L}-1},X_{\tau_{L}}\right)}}{B_K(\rho)=T_0}{}<\theta.
\]
\end{proposition}

\begin{proof}
For any long-range edge $e$ we have
\[
\prcond{e\in\xi}{T}{}\quad\ge\quad(1-o(1))\prcond{\left(X_{\tau_{\ell(e)}-1},X_{\tau_{\ell(e)}}\right)=e}{T}{},
\]
hence
\[
\prcond{\truncprime{\xi_L}}{B_K(\rho)=T_0}{}\quad\ge\quad(1-o(1))\prcond{\truncprime{\left(X_{\tau_{L}-1},X_{\tau_{L}}\right)}}{B_K(\rho)=T_0}{}
\]
and
$\truncprime{\xi_L}$ implies ${W}_T(\xi_{L})<\frac12\log n+\log\log n+o(1)$.

From Lemma~\ref{lem:entropy_from_assump_f} we know that for a sufficiently large value of $C$ we have
\[
\prcond{W_T(\xi_L)\le L\h-C\sqrt{L\V}}{T_0}{}<\theta.
\]

Note that
\[
L\h-C\sqrt{L\V}\quad=\quad\frac12\log n+\frac12B\frac{\log n}{\sqrt{g(n)}}-Ca\sqrt{\frac{(\log n)^2}{g(n)}+B\frac{(\log n)^2}{\sqrt{g(n)}}}
\]
where $a\asymp1$. For sufficiently large values of $B$ (in terms of $C$) this is $\ge\frac12\log n+\frac14B\frac{\log n}{\sqrt{g(n)}}>\frac12\log n+\log\log n+1$.

Putting these together gives the result.
\end{proof}

\subsection{Coupling}

\begin{definition}\label{def:coupling}
Let $x$ and $y$ be two distinct vertices of the graph $G$ and condition on both of them being $K$-roots in $G^*$ with disjoint $K$-neighbourhoods $T_{x,0}=\B_K^*(x)$ and $T_{y,0}=\B_K^*(y)$.

Let $z_1,...,z_{L_x}$ be the vertices of $\partial\B_{K/2}^*(x)$ and let $z_{L_x+1},...,z_{L_x+L_y}$ be the vertices of $\partial\B_{K/2}^*(y)$. For a $z_i$ in $\partial\B_{K/2}^*(x)$ let $V_{z_i}$ be the set of its descendants in $\partial\B_{K}^*(x)$.

For each $z_i$ we will define an exploration process of $G^*$ from $V_{z_i}$ by constructing a coupling between a subset of $G^*$ and two independent quasi-trees $T_x$, $T_y$ that are distributed like the random quasi-tree corresponding to graph $G$, conditioned to be $T_{x,0}$ and $T_{y,0}$ respectively at the first $K$ levels around their roots.

Let us assume that $i\le L_x$. (The exploration from $V_{z_i}$ with $L_x<i\le L_x+L_y$ is defined analogously.)

First let us reveal one by one all the long-range edges of $T_x$ which have one endpoint in the $R$-ball of some $w\in V_{z_i}$, but not in $V_{z_i}$ itself.
Let us couple each of these long-range edges with the corresponding long-range edge of $G^*$ by using the optimal coupling between the two uniform distributions at every step. (At each step the endpoint in $T_x$ is chosen uniformly from all vertices, while the endpoint in $G^*$ is chosen uniformly from all vertices whose long-range edge has not been revealed yet.)
If one of these couplings fails, let us truncate that edge and stop the exploration from this edge in $G^*$, but continue in $T_x$. If the coupling was successful for a given edge, let us also consider the $R$-ball around its newly revealed endpoint. Let us also truncate a long-range edge if this $R$-ball intersects the $R$-ball around any previously considered vertex. Once all long-range edges from the $R$-balls of $V_{z_i}$ are revealed, examine for each of them whether its level $M$ ancestor satisfies the truncation criterion $\trunc{e}{A}$ (which is defined w.r.t.\ $T_x$, not $G^*$; note that since $M\le\frac{K}{2}$, it only depends on edges we have already revealed). If the level $M$ ancestor satisfies the truncation criterion, then let us truncate that edge and stop the exploration from it in $G^*$.

Suppose we have already explored the long-range edges up to level $k$ of $T_x$ and for each edge that is a descendant of ${z_i}$ and neither the edge nor any of its ancestors  have been truncated, we have also explored the corresponding edge in $G^*$. Then let us reveal the edges between levels $k$ and $k+1$ in $T_x$ and for each edge that is a descendant of $z_i$ and whose ancestors and itself are not truncated, let us couple the corresponding edge of $G^*$ to it using the optimal coupling. If the optimal coupling fails or the $R$-ball around the newly added endpoint intersects any previously revealed $R$-ball or if the level $M$ ancestor of the edge satisfies the truncation criterion $\trunc{e}{A}$, then let us truncate it and stop the exploration from it in $G^*$. Let us continue this up to level $L-1$. Let us also consider the half-edges leading to level $L$, but do not reveal their level $L$ endpoints yet. For each such half-edge let us truncate it if its level $M$ ancestor satisfies the truncation criterion $\trunc{e}{A}$ or the half-edge itself satisfies $\truncprime{e}$ (note that we do not need the other endpoint to determine whether this holds).

Let us run the exploration processes from $V_{z_i}$ for $i=1,2,...,L_x+L_y$ in this order. Let $\F_i$ be the $\sigma$-algebra generated by $T_{x,0}$, $T_{y,0}$ and the exploration processes started from $V_{z_1},...V_{z_i}$. Say that the vertex $z_i$ is \emph{good} if no vertex from the $R$-balls of $V_{z_i}$ has been explored during the exploration processes corresponding to the sets $V_{z_1},...V_{z_{i-1}}$. Otherwise call $z_i$ \emph{bad}. Note that the event $\{z_i\text{ good}\}$ is $\F_{i-1}$-measurable.

Let $\partial T_x$ denote the set of vertices at level $L-1$ of $T_x$ that are the endpoints of non-truncated long-range half-edges leading to level $L$. Define $\partial T_y$ analogously.

Note that by construction $G^*$ agrees with $T_x$ and $T_y$ on the (random) regions around $x$ and $y$ enclosed by the truncated edges and the half-edges leading from levels $L-1$ to levels $L$.

After running the explorations from each $V_{z_i}$ let $\I$ be the set of yet unpaired long-range half-edges in $G^*$. Let us complete the exploration of $G^*$ by considering a uniform random matching of $\I$ and adding the corresponding long-range edges.

For $u\in V_{z_i}$ where $1\le i\le L_x$, given the explorations from all $V_{z_k}$ let us define the coupling of a random walk $X$ on $G^*$ from $u$ with a random walk $\til{X}$ on $T_x$ from $u$ as follows.
Let us run $\til{X}$ until it reaches level $L$ in $T_x$ and move $X$ together with it as long as none of the following happen.
\begin{enumerate}[(i)]
\item $\til{X}$ crosses a truncated edge;
\item $\til{X}$ visits level $K-1$ of $T_x$;
\item $\Omega^X_0(L-K)$ or $\Omega^X_1(L-K)$ fails to hold.
\end{enumerate}
If none of the above events happen, we say that the coupling is successful. Otherwise we say that the coupling fails.

We can also define the coupling of a random walk $X$ on $G^*$ from $x$ and a random walk $\til{X}$ on $T_x$ from $x$ as follows. Let us run $\til{X}$ until it hits level $K$ and move $X$ together with it. After that use the above coupling from the level $K$ vertices. We say that the coupling is successful if $\Omega_0^X(K)$ and $\Omega_1^X(K)$ both hold (note that these only depend on $\til{X}$ up to the first time it hits level $K$) and the coupling from the level $K$ vertices is also successful. Otherwise let us say that the coupling fails. Let $\Omega_2^X$ denote the event that the coupling is successful. Note that on this event $\Omega_0^X(L)$ and $\Omega_1^X(L)$ also hold.

For $v\in V_{z_j}$ where $L_x<j\le L_x+L_y$ let us define the coupling of random walks on $G^*$ and $T_y$ from $v$ analogously. Also define the coupling of random walks on $G^*$ and $T_y$ from $y$ and the event $\Omega_2^Y$ analogously.
\end{definition}

\begin{lemma} \label{lem:bound_num_bad}
Let us consider the setup of Definition~\ref{def:coupling}.

For each $i$ let $D_i$ be the set of vertices explored during the exploration process from the set $V_{z_i}$. Then for all sufficiently large values of $n$ we have
\[
\left|\bigcup D_i\right|\quad\le\quad N:=\quad\sqrt{n}\exp\left(2A\sqrt{\frac{\V\log n}{\h}}\right).
\]
Also there exists a positive constant~$C$ (not depending on $T_{x,0}$ and $T_{y,0}$) so that the number $\rm{Bad}$ of bad vertices $z$ satisfies
\[
\prcond{\rm{Bad}\ge C}{\B_K^*(x)=T_{x,0},\B_K^*(y)=T_{y,0},\B_K^*(x)\cap \B_K^*(y)=\emptyset}{}\quad=\quad o\left(\frac{1}{n^2}\right).
\]
\end{lemma}

\begin{proof} The proof is similar to the derivation of \cite[equation (3.11)]{RWs_on_random_graph}. \footnote{The improvement to constant $C$ compared to~\cite{RWs_on_random_graph} is due to the fact that we reveal much fewer vertices for each $z_i$.}

First we show that for each $k$ the number of vertices $\widehat{S}_k$ at level $k$ of $T_x$ whose level $M$ ancestor (if exists) does not satisfy the truncation criterion is $\le\sqrt{n}\exp\left(\frac32A\sqrt{\frac{\V\log n}{\h}}\right)$.

Let $\til{S}_k$ be the set of level $k$ vertices of $T_x$ not satisfying the truncation criterion. Note that by induction on $k$ we have
\[
\sum_{e\in \til S_k}\exp\left(-\til W_{T}(e)\right)\quad\le\quad1
\]
and by the definition of truncation criterion this gives that $|\til{S}_k|\le \sqrt{n}\exp\left(A\sqrt{\frac{\V\log n}{\h}}\right)$.

Then using that
\[
M\log b(R)\:\asymp\:\left(\log g(n)\right)f(R):\lesssim\:\left(\log g(n)\right)\frac{\log n}{g(n)}(\log g(n))\:\ll\:\frac{\log n}{\sqrt{g(n)}}\:\asymp\:\sqrt{\frac{\V\log n}{\h}}
\]
we get that
\[
|\widehat{S}_k|\quad\le\quad b(R)^M|\til{S}_{k-M}|\quad\le\quad\sqrt{n}\exp\left(\frac32A\sqrt{\frac{\V\log n}{\h}}\right).
\]

Analogously we get that the number of vertices at level $k$ of $T_y$ whose level $M$ ancestor does not satisfy the truncation criterion is also $\sqrt{n}\exp\left(\frac32A\sqrt{\frac{\V\log n}{\h}}\right)$.

For each non-truncated long-range edge in $T_x$ and $T_y$ we reveal its $R$-ball, which has size $\le b(R)$. We continue the exploration for $L$ levels in each tree, so overall the number of vertices we explore is
\[
\le\quad 2Lb(R)\sqrt{n}\exp\left(\frac32A\sqrt{\frac{\V\log n}{\h}}\right)\quad\le\quad \sqrt{n}\exp\left(2A\sqrt{\frac{\V\log n}{\h}}\right).
\]

The last inequality holds for sufficiently large values of $n$ and is because
\begin{align*}
&\log b(R)\quad\asymp\quad f(R)\quad\lesssim\quad\frac{\log n}{g(n)}\log g(n)\quad\ll\quad\frac{\log n}{\sqrt{g(n)}}\quad\asymp\quad\sqrt{\frac{\V\log n}{\h}}\qquad\text{and}\\
&\log L\quad\asymp\quad \log g(n)\quad\lesssim\quad\log\left(\frac{\log n}{\log\log n}\right)\quad\ll\quad\sqrt{(\log\log n)\log n}\quad\lesssim\quad\sqrt{\frac{\V\log n}{\h}}.
\end{align*}

At each step of the exploration the probability of intersecting $\partial T_{x,0}$ or $\partial T_{y,0}$ is
\[
\le\quad\frac{|\partial T_{x,0}|+|\partial T_{y,0}|}{n}\quad\le\quad \frac{2b(R)^{K+1}}{n},
\]
independently for different steps, so the probability of having $\ge C$ bad vertices is
\begin{align*}
&\prcond{{\rm{Bad}}\geq C }{T_0}{}\quad\le\quad\binom{N}{C} \left(\frac{2b(R)^{K+1}}{n}\right)^{C}\\
&\le\quad\exp\left(C\left(\frac12\log n+2A\sqrt{\frac{\V\log n}{\h}}+2K\log b(R)-\log n\right)\right)\quad\ll\quad\frac{1}{n^2}
\end{align*}
if $C$ is sufficiently large.

In the last step, we used that
\begin{align*}
K\log b(R)\quad&\lesssim\quad\log g(n)\frac{\log n}{g(n)}\log g(n)\quad\ll\quad \log n\qquad\text{and}\\
\sqrt{\frac{\V\log n}{\h}}\quad&\asymp\quad\frac{\log n}{\sqrt{g(n)}}\quad\ll\quad\log n.
\end{align*}
\end{proof}

\begin{lemma} \label{lem:coupling_succ_prob}
Let us consider the setup of Definition~\ref{def:coupling} and consider the coupling of $X$ with $\til{X}$ from $u\in V_{z_i}$ where $1\le i\le L_x+L_y$. Then for all $\theta>0$ there exist $B$ and $A$ sufficiently large such that for all large enough $n$ we have
\[
\prcond{\text{the coupling of $X$ and $\til{X}$ succeeds}}{\F_{i-1}}{u}  \quad\ge\quad (1-\theta)\1_{\{z_i\text{ good}\}}.
\]
\end{lemma}

\begin{proof}
We follow the proof of \cite[Lemma 5.6]{random_matching}. Using Propositions~\ref{pro:bound_trunc_prob} and~\ref{pro:bound_truncprime_prob} with sufficiently small constants in place of $\theta$ it only remains to check the following.
\begin{itemize}
\item The probability that up to the time that $\til{X}$ first reaches level $L$ it ever returns to a vertex after visiting its depth $K$ descendant should be $o(1)$.

The probability of ever returning to level $k$ after visiting level $k+K$ is $\le \delta^{K}$. We have
\[
\log\left(L\delta^K\right)=\log L-K\log\left(\frac1\delta\right)\asymp\log g(n)-\left(\log g(n)\right)\log\left(\frac1\eps\right)\to-\infty\quad\text{as}\quad n\to\infty
\]
since $\log\left(\frac1\eps\right)\gg1$. This proves the required bound.

\item The probability that the first time when the walk visits a given level (which is at most $L$) there is an overlap at one of the vertices within distance $2K$ at the same level should be $o(1)$.

We can upper bound this probability by $L b(R)^{2K+1}\frac{b(R)N}{n-N}$. Note that
\begin{align*}
\log L\quad&\asymp\quad\log g(n)\quad\ll\quad\log n,\\
K\log b(R)\quad&\lesssim\quad \frac{\log n}{g(n)}(\log g(n))^2\quad\ll\quad\log n,\\
N\quad&\lesssim\quad n^{2/3},
\end{align*}
therefore $L b(R)^{(2K+1)}\frac{b(R)N}{n-N}=o(1)$ as required.

\item The probability that the first time when the walk visits a given level (which is at most $L$) the optimal coupling fails in one of the vertices within distance $2K$ should be $o(1)$.

As in the previous point we get $L b(R)^{2K+1}\frac{N}{n}=o(1)$ as required.

\item The probability of hitting level $K-1$ should be $o(1)$. This is true by Lemma~\ref{lem:neverbacktrack}.

\item The probability of $\Omega_0^X(L)$ or $\Omega_1^X(L)$ failing to hold should be $o(1)$. This is also true by Lemmas~\ref{lem:bound_Omega_zero_prob} and \ref{lem:bound_Omega_one_prob}. \qedhere
\end{itemize}
\end{proof}

\begin{lemma}\label{lem:coupling_succ_prob_from_x}
Let us consider the setup of Definition~\ref{def:coupling} and consider the coupling of $X$ with $\til{X}$ from $x$. Then for all $\theta>0$ there exists $B$ and $A$ sufficiently large such that for all large enough $n$ we have
\[\mathbb{P} \bigg(\prcond{\text{the coupling of $X$ and $\til{X}$ succeeds}}{\F_{L_x+L_y}}{x}>1-\theta\bigg|\hspace{5cm}
\]
\[\hspace{3cm}\bigg|\B_K^*(x)=T_{x,0},\B_K^*(y)=T_{y,0},\B_K^*(x)\cap\B_K^*(y)=\emptyset\bigg)\quad\ge\quad1-o\left(\frac{1}{n^2}\right).
\]
\end{lemma}

An analogous result holds with analogous proof for $Y$ and $\til{Y}$.

\begin{proof}We use a martingale argument similar to the one used in the proof of \cite[Theorem 1.1]{random_matching}. To simplify notation we will write $\{T_{x,0},T_{y,0}\}$ for the event\\ $\left\{B_K^*(x)=T_{x,0},\B_K^*(y)=T_{y,0},\B_K^*(x)\cap\B_K^*(y)=\emptyset\right\}$. Throughout the proof we will work conditional on this event.

Set 
\[
V:=\left\{z_i\in\partial\B^*_{K/2}(x):\:\prcond{X_{\tau_K}\in V_{z_i},\text{ coupling of }X,\til{X}\text{ from }X_{\tau_K}\text{ fails}}{\F_{L_x+L_y}}{x}\ge\theta\right\},
\]
\[
\widehat{V}:=\bigcup_{z\in V}V_{z},
\]
and for $w\in\partial\B^*_K(x)$ set $h(w):=\prcond{X_{\tau_K}=w}{\F_{L_x+L_y}}{x}$.

We will show that
\begin{align}\label{eq:bound_h_Vtil}
\prcond{h(\widehat{V})>3\theta}{T_{x,0},T_{y,0}}{}\quad\ll\quad\frac{1}{n^2}.
\end{align}

Once we have this, we can finish the proof as follows. Note that by Lemmas~\ref{lem:bound_Omega_zero_prob} and~\ref{lem:bound_Omega_one_prob}
\begin{align*}&\prcond{\text{the coupling of $X$ and $\til{X}$ fails}}{\F_{L_x+L_y}}{x}\\
&\le\quad\prcond{\Omega_0(K)^c}{\F_{L_x+L_y}}{x}\quad+\quad\prcond{\Omega_1(K)^c}{\F_{L_x+L_y}}{x}\\
&\quad\quad+\quad\sum_{i}\prcond{X_{\tau_K}\in V_{z_i},\text{ coupling of }X,\til{X}\text{ from }X_{\tau_K}\text{ fails}}{\F_{L_x+L_y}}{x}\\
&\le\quad o(1)\quad+\quad o(1)\quad+\quad\theta\quad+\quad\prcond{X_{\tau_K}\in\widehat{V}}{\F_{L_x+L_y}}{x}.
\end{align*}
By~\eqref{eq:bound_h_Vtil} we know that with probability $1-o\left(\frac{1}{n^2}\right)$ this last expression is $<5\theta$. Considering $\frac15\theta$ instead of $\theta$ gives the result.

Now we proceed to prove~\eqref{eq:bound_h_Vtil}. Note that
\begin{align}\label{eq:decompose_h_Vhat}
h(\widehat{V})\quad\le\quad\sum_{i}h(V_{z_i})\1_{\{z_i\in V\}}\1_{\{z_i\text{ good}\}}\quad+\quad\sum_{i}h(V_{z_i})\1_{\{z_i\text{ bad}\}}.
\end{align}

We know that for any $i$ using Corollary~\ref{cor:neverbacktrack} we have
\[
h(V_{z_i})\quad\le\quad\prstart{\tau_{z_i}<\infty}{x}\quad\le\quad\delta^{\frac{K}{2}}\quad\le\quad\frac{1}{(\log n)^2}
\]
since $\eps\le(\log n)^{-\beta}$ for some $\beta>0$, $\delta\lesssim\eps^{\frac13}$, and $K\gg1$.

Then by Lemma~\ref{lem:bound_num_bad} with probability $1-o\left(\frac{1}{n^2}\right)$ we can bound the second sum in~\eqref{eq:decompose_h_Vhat} as \[
\sum_{i}h(V_{z_i})\1_{\{z_i\text{ bad}\}}\quad\le\quad\frac{C}{(\log n)^2}\quad\ll\quad1.
\]

We know that
\begin{align*}
&\econd{h(V_{z_i})\1_{\{z_i\in V\}}\1_{\{z_i\text{ good}\}}}{\F_{i-1}}\\
&= h(V_{z_i})\1_{\{z_i\text{ good}\}}\prcond{\prcond{X_{\tau_K}\in V_{z_i},\text{ coupling of }X,\til{X}\text{ from }X_{\tau_K}\text{ fails}}{\F_{L_x+L_y}}{x}\ge\theta}{\F_{i-1}}{}.
\end{align*}
By Markov's inequality we have 
\begin{align*}
&\prcond{\prcond{X_{\tau_K}\in V_{z_i},\text{ coupling of }X,\til{X}\text{ from }X_{\tau_K}\text{ fails}}{\F_{L_x+L_y}}{x}\ge\theta}{\F_{i-1}}{}\\
&\le\quad\frac{\econd{\prcond{X_{\tau_K}\in V_{z_i},\text{ coupling of }X,\til{X}\text{ from }X_{\tau_K}\text{ fails}}{\F_{L_x+L_y}}{x}}{\F_{i-1}}}{\theta}\\
&\le\quad\max_{u\in V_{z_i}}\frac{\prcond{\text{coupling of }X,\til{X}\text{ fails}}{\F_{i-1}}{u}}{\theta}\quad\le\quad\frac{\theta^2\1_{\{z_i\text{ good}\}}+\1_{\{z_i\text{ bad}\}}}{\theta}
\end{align*}
for sufficiently large values of $B$ and $A$, by using Lemma~\ref{lem:coupling_succ_prob}.

This then gives
\[
\econd{h(V_{z_i})\1_{\{z_i\in V\}}\1_{\{z_i\text{ good}\}}}{\F_{i-1}}\quad\le\quad\theta h(V_{z_i}).
\]
Let
\[
R_i:=\quad h(V_{z_i})\1_{\{z_i\in V\}}\1_{\{z_i\text{ good}\}},\qquad M_k:=\quad\sum_{i=1}^{k}(R_i-\econd{R_i}{\F_{i-1}}).
\]
Then $(M_k)$ is a martingale and $|M_k-M_{k-1}|\le h(V_{z_k})$. Also \[
M_{L_x+L_y}\quad\ge\quad\sum_{i}R_i-\theta\sum_ih(V_{z_i})\quad=\quad\sum_{i}R_i-\theta.
\]

Then by the Azuma-Hoeffding inequality we get that
\begin{align*}\prcond{\sum_{i}h(V_{z_i})\1_{\{z_i\in V\}}\1_{\{z_i\text{ good}\}}>2\theta}{T_{x,0},T_{y,0}}{}\quad&\le\quad\prcond{M_{L_x+L_y}>\theta}{T_{x,0},T_{y,0}}{}\\
&\le\quad\exp\left(-\frac{c\theta^2}{\sum_{i}h(V_{z_i})^2}\right)\\
&\le\quad\exp\left(-c\theta^2(\log n)^2\right)\quad\ll\quad\frac{1}{n^2}.
\end{align*}
In the second line we used that $\sum_{i}h(V_{z_i})^2\le\left(\max_i h(V_{z_i})\right)\sum_{i}h(V_{z_i})\le\frac{1}{(\log n)^2}$.

This finishes the proof.
\end{proof}

\section{Proof of cutoff}\label{sec:proof_of_cutoff}

\subsection{Lower bounding $\sum_{z}\prstart{X_{\tau_{L}}=z,\Omega_0(L),\Omega_1(L)}{x}\prstart{X_{\tau_{L}}=\eta(z),\Omega_0(L),\Omega_1(L)}{\eta(y)}$}

To prove concentration of the probability that a walk from $x$ and a walk from $y$ hits level $L$ of the corresponding trees at $z$ and $w$ respectively that are long-range neighbours, we will use Lemma~\ref{lem:w_pair_sum}. This lemma is based on a result of Chatterjee in~\cite{Steins_method}, and was stated in this form and used for non-backtracking walks on the configuration model in~\cite{cutoff_NBRW_on_sparse_random_graphs}.

\begin{proposition}\label{pro:aux_bound_Ptau_x_y}
Let $x$, $y$ be two vertices of the graph $G$ at graph distance $>2K$ and let $T_{x,0}$ and $T_{0,y}$ be two possible realisations of the first $K$ levels of the quasi-tree of corresponding to $G^*$ centred at $x$ and $y$, respectively. Then for any $\theta>0$ for sufficiently large values of $B$ and $A$ we have
\[\mathbb{P} \bigg(\sum_{z}\prcond{X_{\tau_{L}}=z,\Omega_0(L),\Omega_1(L)}{G^*}{x}\prcond{X_{\tau_{L}}=\eta(z),\Omega_0(L),\Omega_1(L)}{G^*}{y}\ge\frac1n(1-\theta)\bigg|\hspace{5cm}\]
\[\hspace{3cm}\bigg|\B_K^*(x)=T_{x,0},\B_K^*(y)=T_{y,0},\B_K^*(x)\cap \B_K^*(y)=\emptyset\bigg)\quad\ge\quad1-o\left(\frac{1}{n^2}\right).\]
\end{proposition}

\begin{proof}
For $u\in\partial T_x$, $v\in\partial T_y$ let \[w_x(u):=\prcond{X_{\tau_{L}}=u,\Omega^X_2}{G^*}{x},\quad w_y(v):=\prcond{Y_{\tau_{L}}=w,\Omega^Y_2}{G^*}{y}.\]
where we recall that $\Omega_2^X$ is the event that the coupling of $X$ and $\til{X}$ is successful and $\Omega_2^Y$ is the analogous event for $Y$.

Then we have
\begin{align*}
&\sum_{z}\prcond{X_{\tau_{L}}=z,\Omega_0(L),\Omega_1(L)}{G^*}{x}\prcond{X_{\tau_{L}}=\eta(z),\Omega_0(L),\Omega_1(L)}{G^*}{y}\\
&\ge\quad\sum_{u\in\partial T_x}\sum_{v\in\partial T_y}w_x(u)w_y(v)\1_{v=\eta(u)}.
\end{align*}

Let $w_{u,v}=w_x(u)w_y(v)$ for $u\in\partial T_x$, $v\in\partial T_y$ and $0$ otherwise. By the definition of $\truncprime{\cdot}$ we have $w_x(u),w_y(v)\le\frac{1}{\sqrt{n}\log n}$ for $u\in\partial T_x$, $v\in\partial T_y$, hence $w_{u,v}\le\frac{1}{n(\log n)^2}$ for all $u$, $v$.

Using Lemma~\ref{lem:w_pair_sum} for these weights and $a=\frac{c}{|\I|-1}$ where $c$ is some sufficiently small constant and noting that $\max_{i\ne j}(w_{i,j}+w_{j,i})\le\frac{2}{n(\log n)^2}$, $m\le\frac{1}{|\I|-1}$ and $|\I|\le n$ we get that
\begin{align*}
&\prcond{\sum_{u\in\partial T_x}\sum_{v\in\partial T_y}w_x(u)w_y(v)\1_{v=\eta(u)}>\frac{1}{|\I|-1}\left(m-a\right)}{\F_{L_x+L_y}}{}\\
&\le\quad\exp\left(-\frac{c^2}{8}(\log n)^2\right)\quad\ll\quad\frac{1}{n^2}
\end{align*}
where
\[
m=\frac{1}{|\I|-1}\left(\sum_{u\in\partial T_x}w_x(u)\right) \left(\sum_{v\in\partial T_y}w_y(v)\right)=\frac{1}{|\I|-1}\prcond{\Omega_2^X}{\F_{L_x+L_y}}{x}\prcond{\Omega_2^Y}{\F_{L_x+L_y}}{y}.
\]

By Lemma~\ref{lem:coupling_succ_prob_from_x} we know that for sufficiently large values of $B$ and $A$, conditional on the event $\{\B_K^*(x)=T_{x,0},\B_K^*(y)=T_{y,0},\B_K^*(x)\cap \B_K^*(y)=\emptyset\}$, with probability $1-o\left(\frac{1}{n^2}\right)$ we have
\[
\prcond{\Omega_2^X}{\F_{L_x+L_y}}{x}\prcond{\Omega_2^Y}{\F_{L_x+L_y}}{y}\quad\ge\quad(1-\theta)^2.
\]
We also know that $\frac{1}{|\I|-1}\ge\frac1n$.

Putting these together gives the result.
\end{proof}

\subsection{Upper bounding the mixing time}

For a random walk $X$ on $G^*$ we define a stopping time $\tau$ as follows. Let $\tau_{K\text{-root}}$ be the first time the walk hits a $K$-root and let $\tau_{2L}$ be defined as before for the walk $(X_s)_{s\ge\tau_{K\text{-root}}}$. Let 
\begin{align}\label{def:tau}
\tau=\tau_{K\text{-root}}+\tau_{2L}.
\end{align}

To upper bound the mixing time we show that with probability close to 1 the value of $\tau$ is upper bounded by $t_0+O(t_w)$ and show that for any starting point, the distribution of $X_{\tau}$ is close to uniform. This gives an upper bound on the hitting time of large sets, then we use a result comparing mixing and hitting times.

We denote the uniform distribution on the vertices of $G$ by $\U$.

\begin{lemma}\label{lem:bound_dtv_tau}
For any $\theta>0$ for sufficiently large values of $B$ we have
$$\pr{\max_{x}\dtv{\prcond{X_{\tau}\in\cdot}{G^*}{x}}{\U(\cdot)}\le\theta}\quad=\quad1-o(1).$$
\end{lemma}

\begin{proof} We are conditioning on $G^*$ throughout the proof, but it will be omitted from the notation.

It is sufficient to show that
$$\pr{\max_{x:K\text{-root}}\dtv{\prstart{X_{\tau_{2L}}\in\cdot}{x}}{\U(\cdot)}\le\theta}\quad=\quad1-o(1).$$
By Lemma~\ref{lem:bound_tv_dist_at_tau_two_ell} for any $K$-root $x$ we have
\begin{align*}
&\dtv{\prstart{X_{\tau_{2L}}\in\cdot}{x}}{\U(\cdot)}\\
&\le\quad\sum_y\left(\U(y)-\sum_{z}\prstart{X_{\tau_{L}}=z,\Omega_0(L),\Omega_1(L)}{x}\prstart{X_{\tau_{L}}=\eta(z),\Omega_0(L),\Omega_1(L)}{\eta(y)}\right)^++o(1)\\
&\le\footnotemark\sum_{y:\text{non-}K\text{-root}}\U(y)+\sum_{y:\B_K^*(x)\cap \B_K^*(y)\ne\emptyset}\U(y)\\
&+\sum_{\substack{y:K\text{-root},\\ \B_K^*(x)\cap \B_K^*(y)=\emptyset}}\left(\U(y)-\sum_{z}\prstart{X_{\tau_{L}}=z,\Omega_0(L),\Omega_1(L)}{x}\prstart{X_{\tau_{L}}=\eta(z),\Omega_0(L),\Omega_1(L)}{y}\right)^++o(1).
\end{align*}
\footnotetext{Here we change the indexing from $y$ to $\eta(y)$ and use that $\U(y)=\U(\eta(y))$.}
\noindent By Lemma~\ref{lem:few_non_K_roots} we know that whp for all $x$ we have $\sum_{y:\text{non-}K\text{-root}}\U(y)=o(1)$ and by Lemma~\ref{lem:few_close_vertices} we know that for all $x$ and all realisations of $G^*$ we have $\sum_{y:\B_K^*(x)\cap \B_K^*(y)\ne\emptyset}\U(y)=o(1)$.
Also
$$\pr{\exists\text{ }K\text{-root }x:\sum_{\substack{y:K\text{-root},\\\B_K^*(x)\cap B_K^*(y)=\emptyset}}\left(\U(y)-\sum_{z}\prstart{X_{\tau_{L}}=z,\Omega_0(L),\Omega_1(L)}{x}\prstart{X_{\tau_{L}}=\eta(z),\Omega_0(L),\Omega_1(L)}{y}\right)^+>\theta}$$
$$\le\quad\mathbb{P}\bigg(\exists\text{ }K\text{-roots }x,y\text{ with }\B_K^*(x)\cap \B_K^*(y)=\emptyset:\hspace{8cm}$$
$$\hspace{2.8cm}\sum_{z}\prstart{X_{\tau_{L}}=z,\Omega_0(L),\Omega_1(L)}{x}\prstart{X_{\tau_{L}}=\eta(z),\Omega_0(L),\Omega_1(L)}{y}<\frac1n(1-\theta)\bigg)$$
$$\le\quad\sum_{x,y}\mathbb{P}\bigg(\sum_{z}\prstart{X_{\tau_{L}}=z,\Omega_0(L),\Omega_1(L)}{x}\prstart{X_{\tau_{L}}=\eta(z),\Omega_0(L),\Omega_1(L)}{y}<\frac1n(1-\theta),\hspace{3cm}$$
$$\hspace{9.3cm}{x,y\text{ }K\text{-roots},\: \B_K^*(x)\cap \B_K^*(y)=\emptyset}\bigg)$$
$$\le\footnotemark\sum_{\substack{x,y:\\d_G(x,y)>2K}}\mathbb{P}\bigg(\sum_{z}\prstart{X_{\tau_{L}}=z,\Omega_0(L),\Omega_1(L)}{x}\prstart{X_{\tau_{L}}=\eta(z),\Omega_0(L),\Omega_1(L)}{y}<\frac1n(1-\theta)\bigg|\hspace{3cm}$$
$$\hspace{9.3cm}\bigg|{x,y\text{ }K\text{-roots},\:\B_K^*(x)\cap \B_K^*(y)=\emptyset}\bigg){}.$$
\footnotetext{We use here that $2K<R$, so for $x,y$ with $d_G(x,y)\le2K$ we have $\pr{\B_K^*(x)\cap \B_K^*(y)=\emptyset}=0$.}

From Proposition~\ref{pro:aux_bound_Ptau_x_y} we know that each of the probabilities appearing in the last sum is $o\left(\frac{1}{n^2}\right)$, hence the whole sum is $o(1)$.
\end{proof}

\begin{lemma}\label{lem:bound_tau}
Let $\tau$ be defined as above. Then for any $\theta$ and $B$ the following holds. With high probability $G^*$ is such that for any $B'\to\infty$ for any vertex $x$ we have
\[
\prcond{\tau>t_0+B't_w}{G^*}{x}\quad\le\quad\theta.
\]
\end{lemma}

\begin{proof}
From Lemma~\ref{lem:hit_K_root_quickly} we know that with high probability $G^*$ is such that with high probability starting from any $x$ 
\[
\tau_{K\text{-root}}\quad\le\quad\frac{bK}{\eps}\quad\asymp\quad\frac1{\eps}\log g(n)\quad\ll\quad\frac{1}{\eps}\sqrt{g(n)}\quad\asymp\quad t_w.
\]
We will use Lemma~\ref{lem:speed} to bound $\tau_{2L}$, but we only know how to relate a walk on $G^*$ with a walk on $T$ up to $\tau_L$, and only if they start from a $K$-root, so we will need to break down $\tau_{2L}$ into smaller parts that we are able to bound.

Let $\tau'$ be the first time when the walk ${(X_k)_{k\ge\tau_L}}$ hits a $K$-root and let $\tau''$ be defined as $\tau_{L}$ for the walk ${(X_k)_{k\ge\tau_L+\tau'}}$.

Let $S$ be the quasi-tree corresponding to $G^*$ and $x$ and let $Y$ be the random walk on $S$ corresponding to $X$. Let $S'$ be the quasi-tree corresponding to $G^*$ and $X_{\tau_L}$ and let $Y'$ be the random walk on $S'$ corresponding to $(X_k)_{k\ge\tau_L}$. Let $S''$ be the quasi-tree corresponding to $G^*$ and $X_{\tau_L+\tau'}$ and let $Y''$ be the random walk on $S''$ corresponding to $(X_k)_{k\ge\tau_L+\tau'}$.

Let $A_1$ be the event that up to time $\tau'$ the walk $Y'$ does not cross the long-range edge from the root of $S'$. Note that on the event $A_1$ the walk $Y$ will be at level $\ge L$ in the quasi-tree $S$ at time $\tau_L+\tau'$.

Let $A_2$ be the event that up to time $\tau''$ the walk $Y''$ does not cross the long-range edge from the root of $S''$. Note that on the event $A_1\cap A_2$ the walk $Y$ will be at level $\ge 2L$ in the quasi-tree $S$ at time $\tau_L+\tau'+\tau''$.

By Lemma~\ref{lem:neverbacktrack} the events $A_1$ and $A_2$ have high probability for any $G^*$. On this event $A_1\cap A_2$ we have
\[
\tau_{2L}\quad\le\quad\tau_L+\tau'+\tau''.
\]
We know that with high probability
\[
\tau'\quad\le\quad\frac{bK}{\eps}\quad\ll\quad t_w.
\]
Now we will bound $\tau_L$ and $\tau''$.

Let $T$ be a random quasi-tree associated to $G$ and let $\til{X}$ be a random walk on it starting from the root, coupled to $(G^*,X)$ as in Definition~\ref{def:coupling}. Let $A_3$ be the event that the coupling is successful. From Lemmas~\ref{lem:bound_num_bad} and~\ref{lem:coupling_succ_prob_from_x} we know that with high probability $G^*$ is such that for sufficiently large values of $B$ we have $\prcond{A_3}{G^*}{}\ge1-\theta'$. On the event $A_3$ we have $\tau^X_L=\tau^{\til{X}}_L$.

Let $s=\frac12\left(t_0+\left(B+B''\right)t_w\right)$ with $B''$ to be chosen later. By Lemma~\ref{lem:speed} we know that for sufficiently large values of $C$ we have
\[
\pr{d_T(\rho,\til{X}_s)<\nu s-C\sqrt{\eps s}}\quad\le\quad\theta'.
\]
Here
\begin{align*}
\nu s-C\sqrt{\eps s}\quad&=\quad\frac12\nu t_0+\frac12\nu\left(B+B''\right)t_w-C\sqrt{\frac12\eps t_0+\frac12\eps\left(B+B''\right)t_w}\\
&\ge\quad L+\frac12B''\frac{1}{\h}\frac{\log n}{\sqrt{g(n)}}-C\sqrt{\frac{1}{2c}\left(\frac{\log n}{\h}+\left(B+B''\right)\frac{1}{\h}\frac{\log n}{\sqrt{g(n)}}\right)}\quad\ge\quad L
\end{align*}
for sufficiently large values of $B''$ since
\[
\frac{\log n}{\h}\quad\gg\quad\frac{1}{\h}\frac{\log n}{\sqrt{g(n)}}\qquad\text{and}
\]
\[
\sqrt{\frac{\log n}{\h}}\quad\asymp\quad\sqrt{g(n)}\quad\asymp\quad \frac{1}{\h}\frac{\log n}{\sqrt{g(n)}}.
\]

This then gives
\[
\pr{\tau^{\til{X}}_{L}>s}\quad\le\quad\pr{d_T(\rho,\til{X}_s)<L}\quad=\quad\pr{d_T(\rho,\til{X}_s)<\nu s-C\sqrt{\eps s}}\quad\le\quad\theta'.
\]
Therefore
\[
\pr{\tau^{\til{X}}_{L}>s}\quad\le\quad\pr{\tau^{\til{X}}_{L}>s}+\pr{A_3^c}\quad\le\quad2\theta'.
\]
Applying the same reasoning for $(X_k)_{k\ge\tau_L+\tau'}$ instead of $X$ we also get that
\[
\pr{\tau'>s}\quad\le\quad2\theta'.
\]
Overall this shows that
\begin{align*}
&\pr{\tau>t_0+(B+B''+1)t_w}\\
&\le\quad\pr{A_1^c}+\pr{A_2^c}+\pr{\tau_{K\text{-root}}>\frac{bK}{\eps}}+\pr{\tau'>\frac{bK}{\eps}}+\pr{\tau_L>s}+\pr{\tau''>s}\\
&\le\quad o(1)+o(1)+o(1)+o(1)+2\theta'+2\theta'\quad<\quad5\theta'.
\end{align*}
We have
\begin{align*}
&\pr{\max_{x:K\text{-root}}\prcond{\tau>t_0+(B+B''+1)t_w}{G^*}{x}>\theta}\\
&\le\quad\frac{\E{\max_{x:K\text{-root}}\prcond{\tau>t_0+(B+B''+1)t_w}{G^*}{x}}}{\theta}\\
&=\quad\frac{\max_{x:K\text{-root}}\prstart{\tau>t_0+(B+B''+1)t_w}{x}}{\theta}\quad<\quad\frac{5\theta'}{\theta}.
\end{align*}
By letting $\theta'$ go to 0 arbitrarily slowly, we will get that $\frac{5\theta'}{\theta}=o(1)$, while $(B+B''+1)$ goes to $\infty$ arbitrarily slowly. This finishes the proof.
\end{proof}

\begin{lemma}\label{lem:bound_trelabs}
With high probability $G^*$ is such that the absolute relaxation time of the random walk on $G^*$ satisfies $\trelabs\lesssim\frac1\eps$.
\end{lemma}

\begin{proof}[Sketch proof]
Let $K$ be the graph with the same vertices and edges as $G^*$, but all edges having weight 1. From \cite[Theorem 6.1]{random_matching} we know that whp $\trel^{\textrm{abs},K}\lesssim1$. We can get the required bound on $\trel^{\textrm{abs},G^*}$ by showing that the transition probabilities and the invariant distributions of the simple random walks on $G^*$ and $K$ satisfy $\eps P_{K}(x,y)\leq P_{G^*}(x,y)\lesssim P_{K}(x,y)$, $\pi_{G^*}(x)\asymp\pi_{K}(x)$, then using these to compare the Dirichlet forms~\cite[(2.1) and (2.2)]{comparison_thms_for_reversible_MCs} and using the characterisation of eigenvalues in terms of these Dirichlet forms.
\end{proof}

\begin{lemma}[Proposition 1.8 and Remark 1.9 in \cite{characterisation_of_cutoff}] \label{lem:hit_mix_comparison}
For any reversible irreducible finite chain and any $\theta\in\left(0,\frac14\right]$ we have
$$\hit{1-\theta/4}{\frac54\theta}\quad\le\quad\tmix{\theta}\quad\le\quad \hit{1-\theta/4}{\frac34\theta}+\Bigg\lceil\frac32\trelabs\left|\log\left(\frac\theta4\right)\right|\Bigg\rceil$$
where $\hit{\alpha}{\theta}=\min\left\{s:\forall x,\forall A\text{ with }\pi(A)\ge\alpha\text{ we have }\prstart{\tau_A>s}{x}\le\theta\right\}$.
\end{lemma}

\begin{lemma}[Corollary 3.4 in \cite{characterisation_of_cutoff}] \label{lem:hitting_times_comparison}
For any reversible irreducible finite chain and any $0<\theta_1<\theta_2<0$ and $0<\alpha_1\le\alpha_2<1$ we have
\[
\hit{\alpha_2}{\theta_2}\quad\le\quad\hit{\alpha_1}{\theta_2}\quad\le\quad\hit{\alpha_2}{\theta_2-\theta_1}+\Bigg\lceil\frac{1}{\alpha_1}\trel\log\left(\frac{1-\alpha_1}{(1-\alpha_2)\theta_1}\right)\Bigg\rceil.
\]
\end{lemma}

\begin{proposition}\label{pro:tmix_upper_bound}
For any $\theta\in(0,1)$ and any $t_w^+\gg t_w$ with high probability the mixing time of the random walk on $G^*$ satisfies\[
\tmix{\theta}\quad\le\quad t_0+t_w^+
\]
for all sufficiently large $n$.
\end{proposition}

\begin{proof}
From Lemma~\ref{lem:bound_dtv_tau} we know that for sufficiently large values of $B$ with high probability we have
\[
\max_{x}\dtv{\prcond{X_{\tau}\in\cdot}{G^*}{x}}{\U(\cdot)}\quad\le\quad\frac18\theta.
\]
From Lemma~\ref{lem:bound_tau} we know that for the above $B$ and for any $B'\to\infty$ with high probability we have
\[
\max_x\prcond{\tau>t_0+B't_w}{G^*}{x}\quad\le\quad\frac18\theta.
\]
Let $x$ be any vertex and let $A$ be any set with $\pi(A)\ge1-\frac{1}{4(\Delta+1)}\theta$. Then $\pi(A^c)\ge\frac{|A^c|}{(\Delta+1)n}=\frac{1}{\Delta+1}\U(A^c)$, hence $\U(A)\ge1-\frac{1}{4}\theta$.
Then on the above high probability event we have
\begin{align*}
&\prstart{\tau_A>t_0+B't_w}{x}\quad\le\quad\prstart{X_{\tau}\not\in A}{x}+\prstart{\tau>t_0+B't_w}{x}\\
&\le\quad\dtv{\prstart{X_{\tau}\in\cdot}{x}}{\U(\cdot)}+\U(A^c)+\prstart{\tau>t_0+B't_w}{x}\quad\le\quad\frac18\theta+\frac14\theta+\frac18\theta.
\end{align*}
Therefore $\hit{1-\frac{1}{4(\Delta+1)}\theta}{\frac12\theta}\le t_0+B't_w$. From Lemma~\ref{lem:bound_trelabs} we know that $\trel\le\trelabs\lesssim\frac1\eps\ll t_w$. These together with Lemmas~\ref{lem:hitting_times_comparison} and \ref{lem:hit_mix_comparison} give the result.
\end{proof}

\subsection{Lower bounding the mixing time}\label{sec:tmix_lower_bound}

The lower bound on the mixing time can be proved analogously to the proof in~\cite{random_matching}. The exploration and coupling used here will be very similar to the ones used for the upper bound, but we will only explore around one vertex $x$, and the threshold for the truncation criterion and the number of explored levels will be different.

Given a quasi-tree $T$ let us define a truncation event for a long-range edge $e$ and a constant $A>0$ as
\[
\trunctil{e}{A}:=\left\{\til{W}_T(e)>\log n-A\sqrt{\frac{\V\log n}{\h}}\right\}.
\]
Let
\[
t:=t_0-Bt_w,\qquad\til{L}:=\nu t+2D\sqrt{\eps t}
\]
for some constants $B$ and $D$ to be chosen later.

\begin{proposition}
Let $G$ be as before, assume that Assumption~\ref{assump:f} holds and let $R$, $K$ and $M$ be as in Definition~\ref{def:variable_values}. Let $T$ be the random quasi-tree associated to $G$. Let $T_0$ be any realisation of the first $K$ levels of $T$. Let $X$ be a simple random walk on $T$ started from its root. Then for all $\theta\in(0,1)$ there exist $B$ (depending on $\theta$) and $A$ (depending on $\theta$ and $B$) sufficiently large such that
\[
\prcond{\bigcup_{k\le\tau_{\til{L}}}\trunctil{(X_{k-1},X_k)}{A}}{B_K(\rho)=T_0}{}<\theta.
\]
\end{proposition}

\begin{proof}
Similarly to the proof of Proposition~\ref{pro:bound_trunc_prob} it is sufficient to show that
\[
\prcond{\til{W}_T(\xi_{\til{L}})\ge\log n-A\sqrt{\frac{\V\log n}{\h}}}{T_0}{}\quad\le\quad\theta.
\]
Similarly to the proof of Lemma~\ref{lem:if_W_large_then_Wtil_large} we can show that $\til{W}_T(e)\ge\log n-A\sqrt{\frac{\V\log n}{\h}}$ implies\\ ${W}_T(e)\ge\log n-A'\sqrt{\frac{\V\log n}{\h}}$ for all $A'>A$ for sufficiently large $n$.

From Proposition~\ref{pro:aux_entropy} we know that for a sufficiently large value of $C'$ we have
\[
\prcond{W_T(\xi_{\til{L}})\ge\til{L}\h+C'\sqrt{\til{L}\V}}{T_0}{}\quad\le\quad\theta.
\]
For any value of $A'$ for sufficiently large values of $B$ and $n$ we have
\[
\til{L}\h+C'\sqrt{\til{L}\V}\quad\le\quad\log n-A'\sqrt{\frac{\V\log n}{\h}},
\]
hence
\[
\pr{{W}_T(\xi_{\til{L}})\ge\log n-A'\sqrt{\frac{\V\log n}{\h}}}\quad\le\quad\theta.
\]
This finishes the proof.
\end{proof}

\begin{definition}\label{def:coupling_2}
Let $x$ be a vertex of $G$ and let us condition on it being a $K$-root with $K$-neighbourhood $T_0=\B^*_{K}(x)$.

Let us define $z_1,...,z_{L_x}$ and $V_{z_i}$ as in Definition~\ref{def:coupling}. For each $z_i$ we will define an exploration process of $G^*$ corresponding to the set $V_{z_i}$ by constructing a coupling between a subset of $G^*$ and a quasi-tree $T$ that is distributed like the random quasi-tree corresponding to the graph $G$, conditioned to be $T_{0}$ at the first $K$ levels around its root.

Let us define the exploration as in Definition~\ref{def:coupling} with the following modifications. We use the truncation criterion $\trunctil{e}{A}$ instead of $\trunc{e}{A}$. We explore up to level $\til{L}$. We do not consider the half-edges from the last explored level of $T$.

Let us define $\F_i$ and good and bad vertices as in Definition~\ref{def:coupling}.

Define the coupling of a random walk $X$ on $G^*$ from $u\in V_{z_i}$ with a random walk $\til{X}$ on $T$ from $u$ by moving them together until $\til{X}$ reaches level $\til{L}$ as long as none of the following happen.
\begin{enumerate}[(i)]
\item $\til{X}$ crosses a truncated edge;
\item $\til{X}$ visits level $K-1$ of $T$;
\item $\Omega_0(\til{L}-K)$ fails to hold, i.e.\ $\til{X}$ crosses more than $\frac{R}{2}$ short edges in a row.
\end{enumerate}
If none of these events happen, we say that the coupling is successful. Otherwise we say that the coupling fails.

Let us also define the coupling of a random walk $X$ on $G^*$ from $x$ and a random walk $\til{X}$ on $T$ from $x$ by moving them together until $\til{X}$ hits level $K$ and then using the above coupling from the level $K$ vertices. We say that the coupling is successful if $\Omega_0^X(K)$ holds (note that this only depends on $X$ up to the first time it hits level $K$) and the coupling from the level $K$ vertices is also successful. Otherwise let us say that the coupling fails. Let $\Omega_2^X$ denote the event that the coupling is successful.
\end{definition}

\begin{lemma}\label{lem:bound_num_bad_2}
Let $T_0$ be a possible realisation of the first $K$ levels of $T$. For each $i$ let $D_i$ be the set of vertices explored during the exploration process from the set $V_{z_i}$. Then for all sufficiently large values of $n$ we have
$$\left|\bigcup D_i\right|\quad\le\quad N:=\quad n\exp\left(-\frac13A\sqrt{\frac{\V\log n}{\h}}\right).$$
Also there exists a positive constant~$C$ (not depending on $T_0$) so that the number $\rm{Bad}$ of bad vertices $z$ satisfies
\[
\prcond{{\rm{Bad}}\ge C\sqrt{\frac{\log n}{\h}} }{\B_K^*(x_0)=T_0}{}\quad\le\quad o\left(\frac{1}{n^2}\right).
\]
\end{lemma}

\begin{proof}
The proof of the bound on $\left|\bigcup D_i\right|$ is analogous to the proof in Lemma~\ref{lem:bound_num_bad}. For the bound on the number of bad vertices we get that
\begin{align*}
&\prcond{{\rm{Bad}}\ge C\sqrt{\frac{\log n}{\h}} }{T_0}{}\quad\le\quad\binom{N}{C\sqrt{\frac{\log n}{\h}}} \left(\frac{b(R)^{K+1}}{n}\right)^{C\sqrt{\frac{\log n}{\h}}}\\
&\le\quad\exp\left(C\sqrt{\frac{\log n}{\h}}\left(\log n-\frac13A\sqrt{\frac{\V\log n}{\h}}+2K\log b(R)-\log n\right)\right)\quad\ll\quad\frac{1}{n^2}
\end{align*}
if $C$ is sufficiently large, since
\begin{align*}K\log b(R)\quad\lesssim\quad\log g(n)\frac{\log n}{g(n)}\log g(n)\quad&\ll\quad \frac{\log n}{\sqrt{g(n)}}\quad\asymp\quad\sqrt{\frac{\V\log n}{\h}},\qquad\text{and}\\
\sqrt{\frac{\log n}{\h}}\sqrt{\frac{\V\log n}{\h}}\quad&\asymp\quad\log n.
\end{align*}
\end{proof}

\begin{lemma}\label{lem:coupling_succ_prob_2}
Let us consider the setup of Definition~\ref{def:coupling_2} and consider the coupling of $X$ with $\til{X}$ from $u\in V_{z_i}$ . Then for all $\theta>0$ there exist $B$ and $A$ sufficiently large such that for all large enough $n$ we have
\[
\prcond{\text{the coupling of $X$ and $\til{X}$ succeeds}}{\F_{i-1}}{u}  \quad\ge\quad (1-\theta)\1_{\{z_i\text{ good}\}}.
\]
\end{lemma}

\begin{proof}
Analogous to the proof of Lemma~\ref{lem:coupling_succ_prob}.
\end{proof}

\begin{proposition}\label{pro:tmix_lower_bound}
For any $\theta\in(0,1)$ and any $t_w^+\gg t_w$ with high probability the mixing time of the random walk on $G^*$ satisfies
\[
\tmix{\theta}\quad\ge\quad t_0-t_w^+
\]
for all sufficiently large $n$.
\end{proposition}

\begin{proof} This is analogous to the proof of the lower bound in~\cite[Theorem 1.1]{random_matching}.
\end{proof}

\section{Specific graphs}\label{sec:specific_graphs}

\subsection{Graphs with polynomial growth of balls}

\begin{proposition}\label{pro:poly_ball_growth}
Let us assume that $G$ and $\eps$ satisfy the following. For any $t$ the size of all balls of radius $t$ are upper bounded by $ct^d$ where $c,d>0$ are constants; $\eps_n\lesssim(\log n)^{-\beta}$ for some constant $\beta>0$; and $\log(\eps)\gg-\log n$ (i.e.\ $\eps=n^{-o(1)}$). Then for any $\theta\in\left(0,1\right)$ and $t_w^+\gg t_w$ whp the mixing time of the random walk on $G^*$ satisfies
\[
\left|\tmixtext^{G^*}(\theta)-t_0\right|\quad\le\quad t_w^+
\]
for all sufficiently large $n$, where $t_0\asymp\frac1\eps\frac{\log n}{\log\left(\frac1\eps\right)}$ and $t_w\asymp\frac1\eps\sqrt{\frac{\log n}{\log\left(\frac1\eps\right)}}$.
This means that the chain exhibits cutoff around $t_0$ with window $t_w^+$.
\end{proposition}

\begin{proof}
Let $g(n)=\frac{\log n}{\log\left(\frac1\eps\right)}$ and let $f(t)=\log(1+t)$. Note that $f(t)\asymp\log t$ for $t\ge2$. Then by assumption $g(n)\to\infty$ as $n\to\infty$ and $f\left(\frac1\eps\right)\asymp\frac{\log n}{g(n)}$.

Also for any constant $\widehat{c}$ we have $\log\left(1+\frac{\widehat{c}}{\eps}\right)\asymp\log\left(\frac{1}{\eps}\right)+\log{\widehat{c}}\asymp\log\left(\frac1\eps\right)$, since $\frac1\eps\gg1$.\\
For any $x,y>e^2-1$ we have $\log\left(1+xy\right)\le\log(1+x)+\log(1+y)\le$\\ $2\max\{\log(1+x),\log(1+y)\}\le\log(1+x)\log(1+y)$.

We have already assumed that the sizes of balls of radius $t$ are upper bounded by $b(t)$ where $b(t)\asymp t^d$, hence $\log b(t)\asymp \log t\asymp f(t)$.

Let $X$ be a random walk on $G$. We will show that for all $t\lesssim\frac1\eps$ and all $b\in\Zpos$ we have  $H_b(X_t)\asymp(\log t)^b\asymp f(t)^b$. From Lemma~\ref{lem:H_b_properties}~\eqref{property:bound_by_set_size} we immediately get that
\[
H_b(X_t)\quad\le\quad\left(\log b(t)\right)^b\quad\lesssim\quad\left(\log t\right)^b\quad\asymp\quad f(t)^b.
\]
From Lemma~\ref{lem:Pt_upper_bound} we know that for all $t\lesssim n^2$ and any $u$ we have $\prstart{X_t=u}{v}\lesssim\frac{1}{\sqrt{t}}$.
Note that $\frac1\eps=n^{o(1)}\ll n^2$ for all sufficiently large $n$. This gives that for any $t\le\frac{C}{\eps}$ we have
\[
H_b(X_t)=\sum_{u}\prstart{X_t=u}{v}\left(-\log\prstart{X_t=u}{v}\right)^b\ge\sum_{u}\prstart{X_t=u}{v}\left(-\log\left(\frac{c'}{\sqrt{t}}\right)\right)^b\asymp\left(\log t\right)^b.
\]
This means that Assumption~\ref{assump:f} holds with $f(t)=\log(1+t)$ and $g(n)=\frac{\log n}{\log\left(\frac1\eps\right)}$. Propositions~\ref{pro:tmix_upper_bound} and \ref{pro:tmix_lower_bound} give the result.
\end{proof}

\subsection{Graphs with linear growth of entropy}

\begin{proposition}\label{pro:lin_entropy_growth}
Let us assume that $G$ and $\eps$ satisfy the following. For any vertex $x$ and any $t\le c\log n$ (where $c$ is some positive constant), the entropy of a simple random walk $X$ on $G$ starting from $x$ satisfies $H_1(X_t)\gtrsim t$; and $(\log n)^{-1}\ll\eps_n\lesssim(\log n)^{-\beta}$ for some constant $\beta>0$. Then for any $\theta\in\left(0,1\right)$ and $t_w^+\gg t_w$ whp the mixing time of the random walk on $G^*$ satisfies
\[
\left|\tmixtext^{G^*}(\theta)-t_0\right|\quad\le\quad t_w^+
\]
for all sufficiently large $n$, where $t_0\asymp\log n$ and $t_w\asymp\sqrt{\frac{\log n}{\eps}}$.

This means that the chain exhibits cutoff around $t_0$ with a window bounded by anything that is $\gg t_w$.
\end{proposition}

\begin{proof}
Let $g(n)=\eps\log n$ and let $f(t)=t$. Then by assumption $g(n)\to\infty$ as $n\to\infty$ and $f\left(\frac1\eps\right)=\frac{\log n}{g(n)}$.

Also $f\left(\frac{\widehat{c}}{\eps}\right)\asymp f\left(\frac1\eps\right)$ for any constant $\widehat{c}$, and $f(xy)=f(x)f(y)$ for any $x,y>0$.

Since each vertex has degree $\le\Delta$, the sizes of balls of radius $t$ are upper bounded by $b(t):=\Delta^{t+1}$, which gives $\log b(t)\asymp t\asymp f(t)$.

Let $X$ be a random walk on $G$. We will show that for all $t\lesssim\frac1\eps\ll\log n$ and all $b\in\{1,2,4\}$ we have $H_b(X_t)\asymp(\log t)^b\asymp f(t)^b$. From Lemma~\ref{lem:H_b_properties}~\eqref{property:bound_by_set_size} we immediately get that for any $b\in\Zpos$ we have
\[
H_b(X_t)\quad\le\quad\left(\log b(t)\right)^b\quad\lesssim\quad t^b\quad\asymp\quad f(t)^b.
\]
By assumption we also have $H_1(X_t)\gtrsim t$, which by Cauchy-Schwarz also gives $H_2(X_t)\gtrsim t^2$ and $H_4(X_t)\gtrsim t^4$.

This means that Assumption~\ref{assump:f} holds with $f(t)=t$ and $g(n)=\eps\log n$. Propositions~\ref{pro:tmix_upper_bound} and \ref{pro:tmix_lower_bound} give the result.
\end{proof}

Now we give some examples of notable families of graphs covered by Proposition~\ref{pro:lin_entropy_growth}.

\begin{proposition}\label{pro:local_expanders_lin_entropy}
For any positive constants $a$ and $b$ there exist positive constants $C$ and $c$ with the following property. If the graph $G$ is such that for any set $A\subseteq V$ of size $\le n^a$ we have $|\partial A|\ge b|A|$ (where $\partial A$ denotes the set of edges between $A$ and $V\setminus A$) then for any vertex $x$ of $G$ and any $t\le C\log n$, the entropy of a simple random walk $X$ on $G$ starting from $x$ satisfies $H_1(X_t)\ge ct$.
\end{proposition}

\begin{remark}\label{rmk:local_expanders_lin_entropy}
This means that families of locally expanding graphs (in the above sense) satisfy the condition of Proposition~\ref{pro:lin_entropy_growth}. These include in particular families of expanders.
\end{remark}

\begin{proof}
Using~\cite[Theorem 2]{evolving_sets} for the lazy random walk on $G$, the assumption on the size of boundaries of sets, and that $\pi(x)\asymp\frac1n$ for all $n$, we get that there exist positive constants  $c_1$, $c_2$ and $c_3$ such that
\[
P^t_{G,\lazytext}(x,y)\quad\le\quad c_2\exp\left(-c_3t\right)\qquad\forall t\le c_1\log n,\quad\forall x,y\in V.
\]

Then using that $P_{G}^k(x,y)\le P_{G}^{k-2}(x,y)$ and that $\max_{x,y}P_{G}^{k}(x,y)\asymp\max_{x,y}P_{G}^{k-1}(x,y)$ we get that
\[
\max_{x,y}P^t_{G,\lazytext}(x,y)\quad=\quad\max_{x,y}\sum_{k=0}^{t}\pr{\Bin{t}{\frac12}=k}P_{G}^k(x,y)\quad\gtrsim\quad\max_{x,y}P_{G}^{t/2}(x,y).
\]
Since for $t\le\frac12c_1\log n$ the transition probabilities of $P_{G}$ are exponentially small in $t$, we also get that the entropy $H_1(X_t)$ is growing linearly in $t$.
\end{proof}

\begin{definition}\label{def:lamplighter}
For a finite connected simple graph $K=(V(K),E(K))$, the corresponding lamplighter graph $G=(V(G),E(G))$ is defined as follows. Let $V(G)=\{0,1\}^{V(K)}\times V(K)$ and for $(\sigma,u),(\rho,v)\in V(G)$ let $\left\{(\sigma,u),(\rho,v)\right\}\in E(G)$ if and only if $\{u,v\}\in E(K)$ and $\sigma(w)=\rho(w)$ for all $w\in V(K)\setminus\{u,v\}$. \footnote{Imagine that a lamp is placed at each vertex of $K$ and a lamplighter performs a simple random walk on $V(K)$, at each step switching the lamps at the two corresponding vertices on or off randomly, independently of each other and everything else. The vertices of $G$ represent the possible configurations of the system and a random walk on $G$ corresponds to the dynamics of the system.}
\end{definition}

Note that the maximal degree of $G$ is 4 times the maximal degree of $K$.

For a path $y$ on $V(K)$ and a given time $t$, let $R_t(y):=\{y_0,...,y_t\}$ denote the range of $y$ up to time~$t$.

\begin{proposition}\label{pro:suitable_lamplighters_lin_entropy}
For any positive constants $c_1$ and $c_2$ there exist positive constants $c_3$ and $c_4$ with the following property. If a graph $K$ is such that for any $t\le c_1|V(K)|$ and any vertex $v$ in $K$, the expected size of the range of a random walk $Y$ from $v$ up to time $t$ satisfies $\mathbb{E}_{v}\left|R_t(Y)\right|\ge c_2t$ then the corresponding lamplighter graph $G$ is such that for any $t\le c_3\log|V(G)|$ and any $x\in V(G)$ the entropy of a simple random walk $X$ on $G$ starting from $x$ satisfies $H_1(X_t)\ge c_4t$.
\end{proposition}

\begin{proof}
Let $X$ be a simple random walk on $G$ and let $Y$ be the corresponding simple random walk on $K$.

Note that for any starting state $X_0=(\sigma,v)$, conditional on $(Y_0,...,Y_t)=(y_0,...,y_t)$ the state $X_t$ is uniform among all states $(\rho,u)$ where $u=y_t$ and $\rho(w)=\sigma(w)$ for all $w\in V(K)\setminus R_t(y)$. There are $2^{|R_t(y)|}$ such states in total, so we have
\[
H_1(X_t|Y_0=y_0,...,Y_t=y_t)\quad=\quad\log\left(2^{|R_t(y)|}\right)\quad\asymp\quad|R_t(y)|.
\]
Therefore
\[
H_1(X_t)\quad\ge\quad H_1(X_t|Y_0,...,Y_t)\quad\asymp\quad\mathbb{E}_{v}\left|R_t(Y)\right|.
\]
For $t\le c_1|V(K)|\asymp\log|V(G)|$ this is $\gtrsim t$ by assumption. This finishes the proof.
\end{proof}

\begin{remark}
This means that if $(K_n)$ is a sequence of graphs with bounded degrees and linearly growing range (in the above sense), then the sequence $(G_n)$ of the corresponding lamplighter graphs satisfies the condition of Proposition~\ref{pro:lin_entropy_growth}.
\end{remark}

\begin{remark}
One example of graphs $(K_n)$ satisfying the condition of Proposition~\ref{pro:suitable_lamplighters_lin_entropy} is $K_n=\Z_n^d$ where $d\ge3$ is fixed. Note that the corresponding family $(G_n)$ is not locally expanding, so it is not covered by Proposition~\ref{pro:local_expanders_lin_entropy}.
\end{remark}

\subsection{Remark on other functions $f$}

Note that for any graph $G$ with $n$ vertices and degrees bounded by $\Delta$ and any $b\in\{1,2,4\}$, a random walk $X$ on $G$ and $b(t):=\sup_{x}\left|\B_{G}(x,t)\right|$ satisfies
\[
(\log t)^b\quad\lesssim\quad H_b(X_t)\quad\lesssim\quad\left(\log b(t)\right)^b\quad\lesssim\quad t^b.
\]
It means that if we want to ensure
\begin{align}\label{eq:Hb_logb_f}
H_b(X_t)\quad\asymp\quad\left(\log b(t)\right)^b\quad\asymp\quad f(t)^b
\end{align}
with $f(t)=\log t$, it is sufficient to assume $\log b(t)\lesssim \log t$, while to ensure~\eqref{eq:Hb_logb_f} with $f(t)=t$, it is sufficient to assume $H_1(X_t)\gtrsim t$.

In case we want~\eqref{eq:Hb_logb_f} to hold with a function $\log t\ll f(t)\ll t$, we need to assume both a lower bound $H_1(X_t)\gtrsim f(t)$ and an upper bound $\log b(t)\lesssim f(t)$. It is less straightforward to see that graphs (and especially graphs of interest) satisfying these assumptions exist.

\section{Other values of $\eps$}\label{sec:other_eps}

\subsection{Larger values}

In Assumption~\ref{assump:f} we only considered weights $(\eps_n)$ satisfying $\eps_n\lesssim(\log n)^{-\beta}$ for some $\beta>0$. Now we cover the cases $(\log n)^{-\frac13}\ll\eps_n\ll1$ and $\eps_n\asymp1$.

\begin{assumption}\label{assump:f_large_eps}
$ $
\begin{itemize}
\item $f(t)$ is a continuous increasing function $\Rpos\to\Rpos$ that satisfies $f(0^+)=0^+$, $f(\infty)=\infty$ and there exist positive constants $c$ and $C$ such that $c\log t\leq f(t)\leq C t$.
\item For any positive constant $\widehat{c},$ there exist positive constants $c$ and $C$ depending on $\widehat{c}$ such that for any $t\le \frac{\widehat{c}}{\eps_n}$, $b\in\{1,2,4\}$ and for all $x_0$ and a random walk $X$ on $G$ with $X_0=x_0$ we have $c f(t)^b\leq H_b(X_t)\leq C f(t)^b$. 
\item There exist a function $b$ and positive constants $c$ and $C$ such that for any $t$ the size of any ball of radius $t$ in $G$ is upper bounded by $b(t)$, and we have $cf(t)\le \log b(t)\le C f(t)$ for all $t$.
\item $(\log n)^{-\frac13}\ll\eps_n\ll1$.
\item There exists a function $g(n)$ growing to infinity and there exist positive constants $c$ and $C$ such that $c\frac{\log n}{g(n)} \leq f\left(\frac{1}{\eps_n}\right)\leq C\frac{\log n}{g(n)}$.
\item For any positive constant $\widehat{c}$ there exist positive constants $c$ and $C$ such that $c f\left(\frac{1}{\eps_n}\right) \leq f\left(\frac{\widehat{c}}{\eps_n}\right)\leq C f\left(\frac{1}{\eps_n}\right)$ for all sufficiently large values of $n$.
\item There exists a positive constant $c$ such that for all $x,y\geq c$ we have $f(xy)\leq f(x)f(y)$.
\end{itemize}
\end{assumption}

\begin{proposition}\label{pro:larger_eps}
Let us assume that Assumption~\ref{assump:f_large_eps} holds. Then for any $\theta\in\left(0,1\right)$ whp the mixing time of the random walk on $G^*$ satisfies
\[
\left|\tmixtext^{G^*}(\theta)-t_0\right|\quad\ll\quad t_0
\]
where $t_0\asymp\frac1\eps g(n)$ as in Definition~\ref{def:variable_values}.
This means that the chain exhibits cutoff around $t_0$.
\end{proposition}

\begin{proof}
Let us recall $t_w$, $R$, $K$ and $M$ from Definition~\ref{def:variable_values}. Let $t=t_0-Bt_w$ and $\til{L}=\nu t+2D\sqrt{\eps t}$ as in Section~\ref{sec:tmix_lower_bound}.

Let us use the same truncation criterion and exploration process as in Section~\ref{sec:tmix_lower_bound}. It can be checked that in this case the results in Section~\ref{sec:tmix_lower_bound} still hold. This immediately gives the lower bound.

Analogously to \cite[Proposition 5.7]{random_matching} we also get the following.

\begin{proposition}
For all $\theta>0$, there exist ${B}$ (in the definition of ${t}$), ${A}$ (in the definition of the truncation criterion) depending on $\theta$ and ${B}$ and a positive constant $\Gamma$ sufficiently large such that for all $n$ sufficiently large, on the event $\{\B^*_K(x_0)=T_0\}$, for all $i$ and all $x\in \partial \B_K^*(x_0)$ descendants of $z_i\in \partial \B_{K/2}^*(x_0)$, on the event $\{z_i \text{ is good}\}$ we have for all $s \ge 0$ that
\[
\prcond{{d}_x(t+s)<e^{-\tfrac{s}{\trel(G^*)}}\cdot \frac{1}{1-\theta}  \exp\left(\Gamma \sqrt{\frac{\V\log n}{\h}}\right) + \theta}{\F_{i-1}}{}\geq 1-2\theta,
\]
where ${d}_x(r)=\dtv{\prcond{{X}_r\in \cdot}{G^*}{x}}{\pi}$ for every $r\in \Zpos$.
\end{proposition}

Since $\trel(G^*)\lesssim\frac1\eps$ and $g(n)\gg(\log n)^{\frac23}$, it is possible to choose $s$ with $\trel(G^*)\frac{\log n}{\sqrt{g(n)}}\ll s\ll\frac1\eps g(n)$. This then gives $e^{-\tfrac{s}{\trel(G^*)}}\cdot \frac{1}{1-\theta}  \exp\left(\Gamma \sqrt{\frac{\V\log n}{\h}}\right)\ll1$. We can finish the proof of the upper bound on the mixing time as before and we get an upper bound of the form $t_0+s$ where $s\ll t_0$.
\end{proof}

\begin{proposition}\label{pro:const_order_eps}
Let us assume that $\eps_n\asymp1$ and the graphs $G_n$ are connected, have bounded degree and $G_n$ has $n$ vertices. Then for any $\theta\in(0,1)$ whp the mixing time of the random walk on $G^*$ satisfies
\[
\left|\tmixtext^{G^*}(\theta)-t_0\right|\quad\ll\quad t_0
\]
where $t_0\asymp\log n$.
This means that the chain exhibits cutoff around $t_0$.
\end{proposition}

\begin{proof}
The proof is essentially the same as for the $\eps_n=1$ case in~\cite{random_matching}.
\end{proof}

\subsection{Smaller values}

\begin{proposition}\label{pro:smaller_eps}
Let us assume that $\eps\ll\frac{1}{\tmixtext^G(\theta)}$ for some $\theta\in(0,1)$. Then the random walk on $G^*$ exhibits cutoff if and only if the random walk on $G$ does.
\end{proposition}

The proof of this is not difficult and is deferred to Appendix~\ref{app:other_remainings_pfs}.

The following proposition will make use of a comparison between hitting and mixing times. We recall that
\[\hit{\alpha}{\theta}=\min\left\{s:\forall x,\forall A\text{ with }\pi(A)\ge\alpha\text{ we have }\prstart{\tau_A>s}{x}\le\theta\right\}\]
and we let $\hittext_{\alpha}=\hit{\alpha}{\frac14}$. We say that $\hittext_{\alpha}$-cutoff occurs if we have $\left|\hit{\alpha}{\theta}-\hittext_{\alpha}\right|\ll\hittext_{\alpha}$ for all $\theta\in(0,1)$.

\begin{proposition}\label{pro:eps_asymp_tmix_inv_no_cutoff_implies_no_cutoff}
Let us assume that the lazy random walk on the original graph $G$ does not exhibit cutoff and that $\eps\asymp\frac{1}{\tmixtext^{G,\lazytext}}$. Then the random walk on $G^*$ does not exhibit cutoff either.
\end{proposition}

\begin{proof}
Assume that the lazy random walk on $G$ does not exhibit cutoff, but the random walk on $G^*$ exhibits cutoff. Then by \cite[Remark 1.9]{characterisation_of_cutoff} we have $\trel^{\text{abs},G^*}\ll\tmixtext^{G^*}$, and by this and \cite[equation (1.4)]{characterisation_of_cutoff} (which holds with $\trel^{\text{abs},G^*}$ by Remark 1.9) the random walk on $G^*$ exhibits $\hittext_{\frac12}$-cutoff. Then by \cite[Corollary 3.4]{characterisation_of_cutoff} it also exhibits $\hittext_{\alpha}$-cutoff for all $\alpha\in\left(0,\frac12\right)$ and $\hittext^{G^*}_{\alpha_1}(u_1)=(1\pm o(1))\hittext^{G^*}_{\alpha_2}(u_2)$ for any $\alpha_1,\alpha_2\in\left(0,\frac12\right)$, $u_1,u_2\in(0,1)$.
By Lemma~\ref{lem:hit_hitlazy_comparison} we have $\hittext_{\alpha}^{\lazytext,G^*}(u)=(1\pm o(1))\frac12\hittext_{\alpha}^{G^*}((1\pm o(1))u)$ for all $\alpha$ and all $u$, hence
\begin{center}
the lazy walk on $G^*$ also exhibits $\hittext_{\alpha}$-cutoff for all $\alpha\in\left(0,\frac12\right)$ and satisfies
\end{center}
\begin{align}\label{eq:hit_same_order}
\hittext^{\lazytext,G^*}_{\alpha_1}(u_1)=(1\pm o(1))\hittext^{\lazytext,G^*}_{\alpha_2}(u_2)\qquad\forall\alpha_1,\alpha_2\in\left(0,\frac12\right),u_1,u_2\in(0,1).
\end{align}

Let $X$ be a lazy random walk on $G$ and $Y$ be a lazy random walk on $G^*$. Let $\tauLR$ be the first time that $Y$ crosses a long-range edge. Since $X$ does not exhibit cutoff, by \cite[Theorem 3]{characterisation_of_cutoff} it does not exhibit $\textrm{hit}_{\alpha}$-cutoff for any $\alpha\in\left(0,\frac12\right)$ either. Therefore for any $\alpha\in\left(0,\frac12\right)$ there exists $\theta\in(0,1)$ such that along a subsequence (of the graphs $G_n$) we have
\[
\textrm{hit}^X_\alpha(\theta)-\textrm{hit}^X_\alpha(1-\theta)\quad\asymp\quad\textrm{hit}^X_\alpha(\theta).
\]
If $\trel^{X}\ll\tmixtext^{X}$ then by \cite[Proposition 1.8]{characterisation_of_cutoff} and \cite[Corollary 3.4]{characterisation_of_cutoff} we have $\textrm{hit}^X_\alpha(\theta)\asymp\tmixtext^{X}$. If $\trel^{X}\asymp\tmixtext^{X}$ then by \cite[Corollary 3.1]{characterisation_of_cutoff} and the last display equation in the proof of \cite[Proposition 3.7]{characterisation_of_cutoff} we still have $\textrm{hit}^X_\alpha(\theta)\asymp\tmixtext^{X}$.

Let $t=\textrm{hit}^X_\alpha(\theta)-1$. Then there exist a vertex $x$ and a set $A$ such that $\pi^X(A)\ge\alpha$ and $\prstart{\tau^X_A>t}{x}>\theta$. For any vertex $z$ we have $\pi^Y(z)=\frac{\deg(z)}{2E}=\frac{\deg(z)+\eps}{2E+n\eps}(1\pm o(1))=\pi^X(z)(1\pm o(1))$, therefore $\pi^Y(A)\ge\frac12\alpha$ for all sufficiently large values of $n$.

Note that $t\asymp\tmixtext^X\asymp\frac1\eps$. For any vertices $y$ and $z$ that are adjacent in $G$ we have \[\prstart{Y_1=z,\tauLR>1}{y}\quad=\frac{1}{\deg(y)+\eps}\quad=(1-O(\eps))\frac{1}{\deg(y)}\quad=(1-O(\eps))\prstart{X_1=z}{y}.\]
Summing this over all paths $p$ of length $t$ in $G$ that start in $x$ and do not visit $A$, we get that
\[
\prstart{\tau^Y_A>t}{x}\quad\ge\prstart{\tau^Y_A>t,\tauLR>t}{x}\quad=(1-O(\eps))^t\prstart{\tau^X_A>t}{x}\quad\asymp\prstart{\tau^X_A>t}{x}\quad>\theta,
\]
therefore there exists $c_1>0$ such that \[
\textrm{hit}^Y_{\frac12\alpha}(c_1\theta)\quad\ge\quad\textrm{hit}^X_{\alpha}(\theta).
\]
Now let $s=\textrm{hit}^X_\alpha(1-\theta)$ and let $x$ be any vertex and $A$ be any set with $\pi^Y(A)\ge\frac12\alpha+\frac14$. Then we also have $\pi^X(A)\ge\alpha$, hence
\[
\prstart{\tau^Y_A\le s}{x}\quad\ge\quad\prstart{\tau^Y_A\le s,\tauLR>s}{x}\quad\asymp\quad\prstart{\tau^X_A\le s}{x}\quad\ge\quad\theta.
\]
This shows that there exists $c_2>0$ such that
\[
\textrm{hit}^Y_{\frac12\alpha+\frac14}(1-c_2\theta)\quad\le\quad\textrm{hit}^X_{\alpha}(1-\theta).
\]
Together these give
\[
\textrm{hit}^Y_{\frac12\alpha}(c_1\theta)-\textrm{hit}^Y_{\frac12\alpha+\frac14}(1-c_2\theta)\quad\asymp\quad\textrm{hit}^Y_{\frac12\alpha}(c_1\theta),
\]
contradicting \eqref{eq:hit_same_order}.
\end{proof}

\subsection{Completing the picture for expanders}

\begin{proposition}\label{pro:expander_example_no_cutoff}
There exists a sequence of expander graphs $G$ such that the simple random walk on $G$ exhibits cutoff, but for any $\eps\asymp\frac{1}{\log n}$ whp the simple random walk on $G^*$ does not exhibit cutoff.
\end{proposition}

\begin{proof}
Let $G$ be the 5-regular expander constructed in the proof of \cite[Theorem 1]{explicit_expanders}, with a sufficiently large value of $L$. Then the simple random walk on $G$ has cutoff.

Let $X$ and $Y$ be simple random walks on $G$ and $G^*$, respectively. Let $\tauLR$ and $\tauLR^{(2)}$ be the first and second time, respectively when $Y$ crosses a long-range edge. Also let $\tau^X_{\ell}$ be the first time that $X$ reaches level $3h+2$ and $\tau^Y_{\ell}$ be the first time that $Y$ reaches level $3h+2$.

From \cite{explicit_expanders} we know that the mixing time of $X$ is $\tmixtext^X(\theta)=(1\pm o(1))t_0$ for all $\theta\in(0,1)$, where $t_0=\frac53\left(5L^2-3L+1\right)h$. Let $L$ be sufficiently large so that $t_0\ge8L^2h$. We also know that for any vertex $v$ at level $>h$, the mixing time of $X$ started from $v$ satisfies $\tmixtext^X(x;\theta)<6L^2h$ for all $\theta\in(0,1)$.

Since $\eps\asymp\frac{1}{\log n}\asymp h$, we have $\prcond{\tauLR>2t_0}{G^*}{\rho}\ge\theta_1$ for some $\theta_1\in(0,1)$. By the regularity of $G$ we also have $(Y_k)_{k\le2t_0}\big|\left\{Y_0=\rho,\tauLR>2t_0\right\}\eqdist(X_k)_{k\le2t_0}\big|\left\{X_0=\rho\right\}$. From \cite[Claim 2.3]{explicit_expanders} we know that starting $X$ from the root $\rho$ we have $\tau^X_\ell=(1+o(1))t_0$ whp. Together these show that for any realisation of $G^*$ and for any constant $a\in(0,1)$ we have
\begin{align*}
\prcond{\tau^Y_{\ell}>(1-a)t_0}{G^*}{\rho}\quad&\ge\quad\prcond{\tau^Y_{\ell}>(1-a)t_0,\tauLR>2t_0}{G^*}{\rho}\\
&=\quad\prcond{\tauLR>2t_0}{G^*}{\rho}\prstart{\tau^X_{\ell}>(1-a)t_0}{\rho}\quad\ge\:\theta_1(1-o(1)).
\end{align*}
Level $3h+2$ contains at least a constant $b>0$ proportion of all vertices, so this implies
\begin{align}\label{eq:eg_expander_hit_bound}
\hittext^Y_{b}\left(\frac12\theta_1\right)\quad\ge\quad(1-o(1))t_0.
\end{align}
For any vertex $x$ at level $>h$ and any $t<2t_0$ we have
\begin{align*}
\dtv{\prcond{Y_t=\cdot}{G^*}{x}}{\pi(\cdot)}\quad&\le\quad \prcond{\tauLR\le2t_0}{G^*}{x}\dtv{\prcond{Y_t=\cdot}{G^*;\tauLR\le2t_0}{x}}{\pi(\cdot)}\\
&+\quad\prcond{\tauLR>2t_0}{G^*}{x}\dtv{\prcond{Y_t=\cdot}{G^*;\tauLR>2t_0}{x}}{\pi(\cdot)}\\
&\le\quad(1-\theta_1)\quad+\quad\theta_1\dtv{\prstart{X_t=\cdot}{x}}{\pi(\cdot)}.
\end{align*}
In particular this shows that
\begin{align}\label{eq:eg_expander_mix_bound}
\tmixtext^Y\left(x;1-\frac12\theta_1\right)\quad\le\quad\tmixtext^X\left(x;\frac12\right)\quad<\quad6L^2h\qquad\text{for all }x\text{ at level }>h.
\end{align}
Now we will also upper bound the mixing time from vertices at level $\le h$.

Note that levels $0$ to $\frac54h$ in $G$ contain $\asymp 4^{\frac54h}$ vertices and the total number of vertices is $\asymp4^{3h}$. Therefore the probability that $G^*$ has any long-range edges between two vertices at level $\le\frac54h$ is $\lesssim4^{-\frac12h}=o(1)$. So whp $G^*$ is such that all long-range edges starting from a level $\le\frac54h$ lead to a level $>\frac54h$.

Note that we have $\prcond{\tauLR<\frac14h,\tauLR^{(2)}>7L^2h}{G^*}{}\ge\theta_2$ for some $\theta_2\in(0,1)$. By the strong Markov property and by the regularity of $G$, for any vertex $v$ we have
\[
(Y_k)_{k=\tauLR}^{\tauLR+6L^2h}\bigg|\left\{\tauLR<\frac14h,\tauLR^{(2)}>7L^2h,Y_{\tauLR}=v\right\}\quad\eqdist\quad(X_k)_{k=0}^{6L^2h}\big|\left\{X_0=v\right\}.
\]
Let $t=6L^2h+\frac14h$ and let $U$ be the set of vertices at levels $>\frac54h$. Then on the above high probability event for any vertex $x$ at level $\le h$ we have
\begin{align*}&\dtv{\prcond{Y_t=\cdot}{G^*}{x}}{\pi(\cdot)}\\
&\le\quad\left(1-\prcond{\tauLR<\frac14h,\tauLR^{(2)}>7L^2h}{G^*}{x}\right)\\
&\qquad+\quad\sum_{u\in U,s\le\frac14h}\prcond{\tauLR=s,\tauLR^{(2)}>7L^2h,Y_{\tauLR}=u}{G^*}{x}\dtv{\prstart{X_{t-s}=\cdot}{u}}{\pi(\cdot)}\\
&\le\quad(1-\theta_2)\quad+\quad\theta_2\cdot\frac12.
\end{align*}
In particular this shows that with high probability $G^*$ is such that
\begin{align}\label{eq:eg_expander_mix_bound_2}
\tmixtext^Y\left(x;1-\frac12\theta_2\right)\quad<\quad6L^2h+\frac14h\qquad\text{for all }x\text{ at level }\le h.
\end{align}
Assume that $Y$ exhibits cutoff. Then we have $\trel^{\text{abs},Y}\ll\tmixtext^Y$ and we have $\tmixtext^Y(\theta)=(1\pm o(1))\tmixtext^Y$ for any $\theta\in(0,1)$. By \cite[Remark 1.9]{characterisation_of_cutoff} and \cite[Corollary 3.4]{characterisation_of_cutoff} we also have $\hittext^Y_{\alpha}(\theta)=(1\pm o(1))\tmixtext^Y$ for any $\alpha,\theta\in(0,1)$. Then \eqref{eq:eg_expander_hit_bound}, \eqref{eq:eg_expander_mix_bound} and \eqref{eq:eg_expander_mix_bound_2} would give a contradiction. So $Y$ does not exhibit cutoff.
\end{proof}

\begin{proposition}\label{pro:expander_example_cutoff}
There exists a sequence of expander graphs $G$ such that the simple random walk on $G$ exhibits cutoff and for any $\eps\asymp\frac{1}{\log n}$ whp the simple random walk on $G^*$ also exhibits cutoff.
\end{proposition}

\begin{proof}
First let us consider a random sequence $G$, the sequence of random $3$-regular graphs. From \cite[Theorem 1]{cutoff_random_regular_graphs} we know that whp the simple random walk on $G$ exhibits cutoff.

It can be proved using the methods of~\cite{RWs_on_random_graph} and \cite{random_matching} that the simple random walk on $G^*$ also exhibits cutoff whp.

In this case $T$ will be a weighted tree where each vertex has three edges of weight 1 and one edge of weight $\eps$.

Then there also exists a sequence $G$ of non-random 3-regular expanders such that whp the random walk on $G^*$ exhibits cutoff.
\end{proof}

\begin{lemma}\label{lem:SRW_cutoff_implies_lazy_cutoff}
Assume that the simple random walk on a sequence of graphs $G$ exhibits cutoff. Then the lazy random walk on $G$ also exhibits cutoff.
\end{lemma}

Propositions~\ref{pro:expander_example_no_cutoff} and~\ref{pro:expander_example_cutoff} and Lemma~\ref{lem:SRW_cutoff_implies_lazy_cutoff} immediately imply the following.

\begin{corollary}\label{cor:expander_examples}
There exists a sequence of expander graphs $G$ such that the lazy random walk on $G$ exhibits cutoff and for any $\eps\asymp\frac{1}{\log n}$ the simple random walk on $G^*$ whp does not exhibit cutoff.

Also there exists a sequence of expander graphs $G$ such that the lazy random walk on $G$ exhibits cutoff and for any $\eps\asymp\frac{1}{\log n}$ the simple random walk on $G^*$ whp also exhibits cutoff.
\end{corollary}

\begin{proof}[Proof of Lemma~\ref{lem:SRW_cutoff_implies_lazy_cutoff}]
Let $X$ be a simple random walk and let $Y$ be a lazy random walk on $G$.

By \cite[Remark 1.9]{characterisation_of_cutoff} we know that $\trelabs\ll\tmixtext^X$ and that $X$ exhibits $\hittext_{\alpha}$-cutoff for some $\alpha\in(0,1)$.

Then by Lemma~\ref{lem:hit_hitlazy_comparison}\footnote{$G$ has bounded degree, therefore $\tmixtext^X(\theta)\gg1$ for all $\theta\in(0,1)$ and so $\hittext_{\alpha}^X(\theta)\gg1$ for all $\theta\in(0,1)$.} the walk $Y$ also exhibits $\hittext_{\alpha}$-cutoff and $\trel\lesssim\trelabs\ll\hittext^{Y}_{\alpha}(\theta)$ for all $\theta\in(0,1)$.

Then by \cite[Corollary 3.1]{characterisation_of_cutoff} we have
\[
\hittext^Y_{\beta}(\theta)\quad=\quad(1\pm o(1))\hittext^Y_{\alpha}(\theta\pm o(1))
\]
for all $\beta\in(0,1)$. In particular $Y$ exhibits $\hittext_{\beta}$-cutoff for some $\beta\in\left(0,\frac12\right)$. Then by \cite[Theorem 3]{characterisation_of_cutoff} the chain $Y$ also exhibits cutoff.
\end{proof}

\begin{lemma}\label{lem:tmix_bound_by_tmixlazy} Let $G$ be connected and have degrees bounded by $\Delta$.
Then there exists a $\theta\in(0,1)$ (depending on $\Delta$) such that
\[
\tmixtext^G(1-\theta)\quad\lesssim\quad\tmixtext^{G,\lazytext}.
\]
\end{lemma}

\begin{proof}Let
$$\tavetext(\theta):=\min\left\{t:\:\max_{x}\dtv{\frac{P^t(x,\cdot)+P^{t+1}(x,\cdot)}{2}}{\pi(\cdot)}\le\theta\right\}.$$

Using that by \cite[Theorem 1.4]{mix_hit} we have $\tavetext\left(\frac14\right)\asymp\tmixtext^{\lazytext}\left(\frac14\right)$, the definition of $\tavetext\left(\frac14\right)$, that $\pi(y)\asymp\pi(z)$ for all $y$, $z$ and that the degrees are bounded, we get the result. See Appendix~\ref{app:other_remainings_pfs} for more details.
\end{proof}

\subsection{Completing the picture for vertex-transitive graphs with polynomial growth of balls}

In this section we will consider vertex-transitive graphs $G$, and denote the volume of a ball of radius $r$ in $G$ by $V(r)$.

\begin{assumption}\label{assump:for_eps_power_of_n}
${}$
\begin{itemize}
\item $G$ is vertex-transitive.
\item There exist constants $C_1,C_2,a_1,a_2>0$ such that $C_1n^{-a_2}\le\eps\le C_2n^{-a_1}$. Let $a:=\frac{\log\left(\frac1\eps\right)}{\log n}\asymp1$.
\item $\eps\gg\frac{1}{\tmixtext^G(\theta)}$ for all $\theta\in(0,1)$.
\item There exist constants $c_1,c_2,c_3,c_4>0$ with the following property. For any vertex $x$, for any positive integers $t$ and $r$ with $t\le\frac{c_1}{\eps}$, $c_2\sqrt{t}\le r\le c_3\sqrt{t}$, and for any $y\in B_{G}(x,r)$, we have $P^t_{G,\lazytext}(x,y)\ge \frac{c_4}{V(r)}$.
\item For any $c'>0$ there exists $c>0$ such that $V\left(\frac{c}{\sqrt{\eps}}\right)\le c'n$.
\end{itemize}
\end{assumption}

First we show the following.

\begin{proposition}\label{pro:Gstar_no_cutoff} Assume that $G$ and $\eps$ satisfy Assumption~\ref{assump:for_eps_power_of_n}. Then whp the random walk on $G^*$ does not exhibit cutoff, and $\tmixtext^{G^*}\left(\frac{1}{4}\right)\asymp \frac{1}{\eps}$.
\end{proposition}

After that we show that vertex-transitive graphs with polynomial growth of balls with $\eps\asymp n^{-a}\gg\frac{1}{\tmixtext^G(\theta)}$ satisfy Assumption~\ref{assump:for_eps_power_of_n}, hence the corresponding $G^*$ do not exhibit cutoff.

\begin{proposition}\label{pro:transitive_poly_satsifies_assump}
Assume that there exist positive constants $\Delta$, $c$ and $d$ such that each $G$ is a vertex-transitive graph of degree $\Delta$, satisfying $V(r)\le c r^d$ for all $r$. Assume further that $\eps\asymp n^{-a}\gg\frac{1}{\tmixtext^G(\theta)}$ for all $\theta\in(0,1)$ where $a\asymp 1$. Then $G$ and $\eps$ satisfy Assumption~\ref{assump:for_eps_power_of_n}.
\end{proposition}

\begin{corollary}\label{cor:eps_power_of_n}
Let $G$ and $\eps$ be as in Proposition~\ref{pro:transitive_poly_satsifies_assump}. Then whp the random walk on $G^*$ does not exhibit cutoff, and $\tmixtext^{G^*}\left(\frac{1}{4}\right)\asymp \frac{1}{\eps}$.

\end{corollary}

Finally we show that the simple random walk and lazy random walk on vertex-transitive graphs $G$ with polynomial growth of balls do not have cutoff and conclude that for any $\eps\lesssim n^{-\Theta(1)}$ whp $G^*$ does not exhibit cutoff.

\begin{proposition}\label{pro:Gstar_transitive_poly_no_cutoff}
Let $G$ be as in Proposition~\ref{pro:transitive_poly_satsifies_assump} and let $\eps\lesssim n^{-a}$ where $a\asymp1$. Then whp the random walk on $G^*$ does not exhibit cutoff, and when $\varepsilon\gg \frac{1}{\diam{G}^2}$, we have $\tmixtext^{G^*}\left(\frac{1}{4}\right)\asymp \frac{1}{\eps}$.
\end{proposition}

In what follows we will be working towards the proof of Proposition~\ref{pro:Gstar_no_cutoff}, under Assumption~\ref{assump:for_eps_power_of_n}.

\begin{lemma}\label{lem:vertex_transitive_lower_bound_tmix}
For any constant $C>0$ there exists $\theta\in(0,1)$ such that $\tmixtext^{G^*}(\theta)\ge \frac{C}{\eps}$.
\end{lemma}

\begin{proof}
Let $t=\frac{C}{\eps}$. Note that we have
\begin{align*}
&1-\dtv{P^t_{G^*}(x,\cdot)}{\pi_{G^*}(\cdot)}\quad=\quad\sum_{y}P^t_{G^*}(x,y)\wedge\pi_{G^*}(y)\\
&\le\quad\pr{\tauLR>t}\sum_{y}\prcond{X_t=y}{\tauLR>t}{x}\wedge\pi_{G^*}(y)\quad+\quad\pr{\tauLR\le t}\\
&=\quad1-\pr{\tauLR>t}\left(1-\sum_{y}\prcond{X_t=y}{\tauLR>t}{x}\wedge\pi_{G^*}(y)\right).
\end{align*}
Since each vertex in $G$ has the same degree, we have $\pi_{G^*}(y)=\frac1n=\pi_{G}(y)$ and\\ $\prcond{X_t=y}{\tauLR>t}{x}= P^t_{G}(x,y)$,
hence
\begin{align*}\sum_{y}\prcond{X_t=y}{\tauLR>t}{x}\wedge\pi_{G^*}(y)\quad&=\quad\sum_{y}P^t_{G}(x,y)\wedge\pi_{G}(y)\\
&=\quad1-\dtv{P^t_{G}(x,\cdot)}{\pi_{G}(\cdot)}\quad\ll\quad1,
\end{align*}
where for the last bound we used that $t\asymp\frac1\eps\ll\tmixtext^G(\theta)$ for all $\theta\in(0,1)$.

Since $t\asymp\frac1\eps$, we also know that $\pr{\tauLR>t}\gtrsim1$. Together these show that
\[
\dtv{P^t_{G^*}(x,\cdot)}{\pi_{G^*}(\cdot)}\:\ge\:\pr{\tauLR>t}\left(1-\sum_{y}\prcond{X_t=y}{\tauLR>t}{x}\wedge\pi_{G^*}(y)\right)\:\gtrsim\:1,
\]
hence $t\le\tmixtext^{G^*}(\theta)$ for some $\theta\in(0,1)$.
\end{proof}

\begin{definition}\label{def:ri_neighbourhood_root}
For a vertex $z$ of $G^*$ and positive integers $r_{0}$, ..., $r_{k}$ we define the $(r_{0},...,r_{k})$-neighbourhood around $z$ as follows. Let us consider a ball of radius $r_{0}$ in $G$ centred at $z$. This will be level 0. Given a level $i$ consisting of balls of radius $r_{i}$ in $G$, we consider the long-range edge of $G^*$ from each vertex of these balls, except for the centres, and attach a copy of the ball of radius $r_{i+1}$ in $G$ around the other endpoint. These new balls will form level $(i+1)$. We continue this up to level $k$. Note that the neighbourhood might contain multiple copies of a given vertex of $G$. (In particular copies of the same vertex might appear in multiple levels of the neighbourhood.) Let $\iota$ be the function mapping each vertex of the neighbourhood to the corresponding vertex of $G$. We say that $z$ is an $(r_{0},...,r_{k})$-root if the image of the level $k$ vertices of the neighbourhood under $\iota$ has size $\asymp\prod_{i=0}^{k}V(r_i)$.\footnote{More precisely it has size $\ge c'\prod_{i=0}^{k}V(r_i)$ where $c'>0$ is a sufficiently small fixed constant. Note that level $k$ contains $\prod_{i=0}^{k}(V(r_i)-1)\asymp\prod_{i=0}^{k}V(r_i)$ vertices; being a root means that the number of different vertices of $G$ they correspond to is also of the same order.}
\end{definition}

\begin{lemma} For given $(r_{0},...,r_{k})$ and given vertex $z$ we have
\[
\pr{z\text{ is an }(r_{0},...,r_{k})\text{-root}}\quad=\quad1-O\left(\frac1n\prod_{i=0}^{k}V(r_i)\right).
\]
\end{lemma}

\begin{proof}
Let us consider the vertices $v_1$, $v_2$, ..., $v_N$ at level $k$ of the neighbourhood of $z$ (where $N=\prod_{i=0}^{k}(V(r_i)-1)\asymp\prod_{i=0}^{k}V(r_i)$).

For each $i$ let $A_{i}$ be the event that $\iota(v_i)$ agrees with one of $\iota(v_1)$, ..., $\iota(v_{i-1})$. Note that for each $i\ne j$ we have $\pr{\iota(v_i)=\iota(v_j)}\lesssim\frac{1}{n}$, hence for each $i$ we have $\pr{A_i}\lesssim\frac{i-1}{n}\lesssim\frac{N}{n}$. Then
\begin{align*}\pr{z\text{ is not an }(r_{0},...,r_{k})\text{-root}}\quad&=\quad\pr{\left|\left\{\iota(v_1),...,\iota(v_N)\right\}\right|<c'N}\\
&=\quad\pr{\sum_{i}\1_{A_{i}^c}<c'N}\quad=\quad\pr{\sum_{i}\1_{A_i}>(1-c')N}\\
&\le\quad\frac{\E{\sum_{i}\1_{A_i}}}{(1-c')N}\quad\lesssim\quad\frac{N\cdot\frac{N}{n}}{N}\quad\asymp\quad\frac{N}{n}.
\end{align*}
This finishes the proof.
\end{proof}

\begin{corollary}\label{cor:many_r_roots}
If $\prod_{i=0}^{k}V(r_i)\ll n$ then for any $c\in(0,1)$ with high probability $G^*$ is such that at least $cn$ of its vertices are $(r_{0},...,r_{k})$-roots.
\end{corollary}

\begin{proof}
Let $N=\prod_{i=0}^{k}V(r_i)$. Then
\[
\pr{\sum_{z}\1_{\left\{z\text{ an }(r_{0},...,r_{k})\text{-root}\right\}}<cn} \quad\le\quad\frac{n\pr{z\text{ not an }(r_{0},...,r_{k})\text{-root}}}{(1-c)n}\quad\lesssim\quad\frac{N}{n}\quad\ll\quad1.\qedhere
\]
\end{proof}

\begin{remark}
Note that for any $r_0$ and any realisation of $G^*$, any vertex is an $(r_0)$-root.
\end{remark}

\begin{definition}\label{def:A_B} Let $c_0>0$ be a sufficiently small constant, to be chosen later.
Let $A$ be such that $V\left(\frac{c_0}{\sqrt{\eps}}\right)= n^{A}$. Note that $n\ge V\left(\frac{c_0}{\sqrt{\eps}}\right)\gtrsim \eps^{-\frac12}\asymp n^{\frac{a}{2}},$ hence $1\ge A\gtrsim 1$. Let $M\in\Zpos$ and let $B\in[0,A)$ be such that
\[
AM+B\ge1+\frac{2\log\log n}{\log n} \qquad\text{and}\qquad (M-1)A+B\le1.
\]
For sufficiently large values of $n$, such $M$ and $B$ exist. Since $A$ is bounded away from 0, the value of $M$ is bounded. Let $r$ be such that $V(r)\asymp n^B$. Note that for any $j$ we have $V(j)\le V(j+1)\le\Delta V(j)$, hence we can choose $r$ to make the multiplicative implicit constants in $V(r)\asymp n^B$ arbitrarily small.
\end{definition}

For two given vertices $x$ and $y$ of $G^*$ let $\Omega_{x,y}$ be the event that $y$ is an $(r_{y,0},...,r_{y,M-1})$-root and the $(r_{y,0},...,r_{y,M-1})$-neighbourhood of $y$ is disjoint from the $(r_x)$-neighbourhood of $x$, where $r_x=r_{y,1}=...=r_{y,M-1}=\frac{c_0}{\sqrt{\eps}}$ and $r_{y,0}=r$.

\begin{lemma}\label{lem:tau_transtion_probs}
There exists a constant $c>0$ such that for sufficiently small values of $c_0$ and $r$, for any vertices $x$ and $y$ having graph distance at least $r_x+r_{y,0}$ in $G$, we have
\[
\prcond{\prcond{X_{\tau}=y}{G^*}{x}\ge\frac{c}{n}}{\Omega_{x,y}}{}\quad=\quad1-o\left(\frac{1}{n^2}\right).
\]
\end{lemma}

\begin{proof} Let us work conditional on the event $\Omega_{x,y}$.
Let $\mathcal{X}$ be the set of vertices at level 0 of the $(r_x)$-neighbourhood of $x$ and let $\mathcal{Y}$ be the set of vertices at level $M-1$ of the $(r_{y,0},...,r_{y,M-1})$-neighbourhood of $y$. Then we have $|\mathcal{X}|=V(r_x)\asymp n^A$ and $|\mathcal{Y}|\asymp\prod_{i=0}^{M-1}V(r_{y,i})\asymp n^{B+(M-1)A}$ (since $y$ is a root).

Let $X$ be a lazy random walk on $G^*$ and let us consider a stopping time defined as $\tau=\tauLR^{(M)}+r^2$, where $\tauLR^{(k)}$ is the $k$th time that $X$ crosses a long-range edge.

Let $Y$ and $\widehat{Y}$ be two processes on the vertices of $G^*$ defined as follows. In the first $r^2$ steps $Y$ moves as a lazy simple random walk on $G$, while $\widehat{Y}$ moves as a lazy random walk on $G^*$. Then they cross a long-range edge. Afterwards they both move as a lazy random walk on $G^*$. Let $\tauLR^{(k,Y)}$ be the $k$th time that $(Y_{r^2+1+i})_{i\ge0}$ crosses a long-range edge and let $\tauLR^{(k,\widehat{Y})}$ be the $k$th time that $(\widehat{Y}_{r^2+1+i})_{i\ge0}$ crosses a long-range edge.

Note that for any $m\ge0$ and any path $(z_0,z_1,..z_m)$ we have
\[
\prstart{\left(\widehat{Y}_{i}\right)_{i=0}^{r^2+1+\tauLR^{(k,\widehat{Y})}}=(z_m,z_{m-1},...,z_0)}{z_m}=\prstart{\left(\eta\left(X_0\right),\left({X}_{i}\right)_{i=0}^{\tauLR^{(k,X)}+r^2}\right)=(z_0,z_1,...,z_m)}{z_0}.
\]

Also note that $r^2\lesssim\frac1\eps$ , hence for any path $y_0,...,y_{r^2}$ in $G$ and for any realisation of $G^*$ we have
\begin{align*}\prcond{\widehat{Y}_1=y_1,...,\widehat{Y}_{r^2}=y_{r^2}}{G^*}{y_0}\quad&=\quad(1-O(\eps))^{r^2}\prcond{Y_1=y_1,...,Y_{r^2}=y_{r^2}}{G^*}{y_0}\\
&\asymp\quad\prcond{Y_1=y_1,...,Y_{r^2}=y_{r^2}}{G^*}{y_0}.
\end{align*}
Hence for any vertices $x$ and $y$ we have
\[
\prcond{X_{\tauLR^{(k,X)}+r^2}=y}{G^*}{x}=\prcond{\widehat{Y}_{r^2+1+\tauLR^{(k,\widehat{Y})}}=\eta(x)}{G^*}{y}\asymp\prcond{Y_{r^2+1+\tauLR^{(k,Y)}}=\eta(x)}{G^*}{y}.
\]
Therefore
\begin{align}\label{eq:decomposition_with_reversal}
\prcond{X_{\tau}=y}{G^*}{x}\quad\asymp\sum_{z,w}\prcond{X_{\tauLR^{(1,X)}}=z}{G^*}{x}\prcond{Y_{r^2+1+\tauLR^{(M-1,Y)}}=w}{G^*}{y}\1_{\eta(z)=w}.
\end{align}
To simplify notation let $\tau^Y=r^2+1+\tauLR^{(M-1,Y)}$.

We will now bound the terms in the sum in~\eqref{eq:decomposition_with_reversal}.
For any $z\in\mathcal{X}$ we have
\begin{align*}w^\mathcal{X}_{z}&:=\quad\prcond{X_{\tauLR-1}=z}{G^*}{x}\quad=\quad\sum_{j\ge0}\prcond{\tauLR=j+1,X_{j}=z}{G^*}{x}\\
&\gtrsim\quad\sum_{j=C_0^2/\eps}^{C_1^2/\eps}\eps P^j_{G,\lazytext}(x,z)\quad\gtrsim\quad \frac{1}{V(r_x)}\quad\asymp\quad n^{-A}.
\end{align*}
The first $\gtrsim$ is because $\tauLR\sim\Geompos{\frac12\frac{\eps}{\Delta+\eps}}$, so $\prcond{\tauLR=j+1}{G^*}{x}\asymp\eps$ for $j\in\left[C_0^2/\eps,C_1^2/\eps\right]$ for any constants $C_0<C_1$, and $\prcond{X_j=z}{G^*,\tauLR=j+1}{x}\asymp P^j_{G,\lazytext}(x,z)$. The last $\gtrsim$ is true for appropriate choices of $C_0$ and $C_1$ and sufficiently small values of $c_0$ by Assumption~\ref{assump:for_eps_power_of_n}.

Similarly, for all $z$ in the $(r_{y,0})$-neighbourhood of $y$ we have
\[
\prcond{Y_{r^2}=z}{G^*}{y}\quad\gtrsim\quad\frac{1}{V(r_{y,0})}\quad\asymp\quad n^{-B}.
\]
Repeatedly using the above results we also get that for any $z\in\mathcal{Y}$ we have
\[
w^\mathcal{Y}_{z}:=\quad\prcond{Y_{\tau^Y}=z}{G^*}{y}\quad\gtrsim\quad n^{-B-(M-1)A}.
\]
Let us condition on the $(r_x)$-neighbourhood of $x$ and the $(r_{y,0},...,r_{y,M-1})$-neighbourhood of $y$ and let $\I$ be the set of yet unpaired long-range half edges. Consider a uniform random matching of $\I$ to complete the graph $G^*$. Let
\[
w_{i,j}:=\begin{cases}
w^\mathcal{X}_{i}w^\mathcal{Y}_{j}\wedge n^{-B-MA} & \text{if }i\in\mathcal{X},j\in\mathcal{Y},\\
0 & \text{otherwise}.
\end{cases}
\]
Then the right-hand side of~\eqref{eq:decomposition_with_reversal} can be written as $\sum_{i\in\I}w_{i,\eta(i)}$.

By Lemma~\ref{lem:w_pair_sum} we get that
$$\pr{\sum_{i\in\I}w_{i,\eta(i)}<\frac12m}\quad\le\quad\exp\left(-\frac{m}{16n^{-B-MA}}\right)$$
where $m=\frac{1}{|\I|-1}\sum_{i\in\I}\sum_{j\in\I\setminus\{i\}}w_{i,j}$.

Note that
\[
n-|\I|\quad\lesssim\quad n^A+n^{B+(M-1)A}\quad\lesssim\quad n.
\]
We used here that $A\le1$ and $(M-1)A+B\le1$. We can make the constant in the $\lesssim$ arbitrarily small by choosing the constant $c_0$ in $\frac{c_0}{\sqrt{\eps}}=r_{x}=r_{y,1}=...=r_{y,M-1}$ small and choosing $r=r_{y,0}$ to be small. \footnote{We used the last point of Assumption~\ref{assump:for_eps_power_of_n} here.}
Hence $|\I|\asymp n$. This gives
\[
m\quad\asymp\quad\frac1n|\mathcal{X}||\mathcal{Y}|n^{-B-MA}\quad\asymp\quad\frac1n.
\]
Then
\[
\frac{m}{16n^{-B-MA}}\quad\gtrsim\quad n^{B+MA-1}\quad\gg\footnotemark\quad\log n,
\]
\footnotetext{We used here that $B+MA=1+\frac{\omega(1)+\log\log n}{\log n}$.}
hence
\[
\exp\left(-\frac{m}{16n^{-B-MA}}\right)\quad\ll\quad\frac{1}{n^2}.
\]
This finishes the proof.
\end{proof}

\begin{lemma}\label{lem:neighhbourhoods_disjoint}
For any vertex $x$ and any realisation of $G^*$ at least a constant proportion of vertices $y$ are such that the image under $\iota$ of the $(r_x)$-neighbourhood of $x$ is disjoint from the image under $\iota$ of the $(r_{y,0},...,r_{y,M-1})$-neighbourhood of $y$.
\end{lemma}

\begin{proof} We will omit writing that we are considering the images under $\iota$. 

The $(r_x)$-neighbourhood of $x$ intersects level $j$ of the $(r_{y,0},...,r_{y,M-1})$-neighbourhood of $y$ if and only if $x$ is in the $(r_{y,0},...,r_{y,j-1},r_{y,j}+r_x)$-neighbourhood of $y$. This is equivalent to $y$ being in the $(r_{y,j}+r_x,r_{y,j-1},...,r_{y,0})$-neighbourhood of $x$.

For $j\in\{1,2,...,M-1\}$ the $(r_{y,j}+r_x,r_{y,i-j},...,r_{y,0})$-neighbourhood of $x$ contains $\asymp n^{jA+B}$ vertices. The $(r_{y,0}+r_x)$-neighbourhood of $x$ contains $\asymp n^{A}$ vertices. Taking union over $j=0,1,...,M-1$ we get that there are
\[
\lesssim\quad n^{A}+\sum_{j=1}^{M-1}n^{jA+B}\quad\lesssim\quad n^A+n^{(M-1)A+B}
\]
vertices $y$ with the neighbourhoods intersecting.

This is $\lesssim n$ and we can make the constant in the $\lesssim$ arbitrarily small by choosing $c_0$ and $r$ sufficiently small. (We are using the last point of Assumption~\ref{assump:for_eps_power_of_n} here.)
\end{proof}

\begin{lemma}\label{lem:bound_tau_2}
For any $c\in(0,1)$ there exists a $C>0$ such that for any vertex $x$ and any realisation of $G^*$ we have
\[
\prcond{\tau>\frac{C}{\eps}}{G^*}{x}\quad\le\quad c.
\]
\end{lemma}

\begin{proof}
Use that $\tauLR^{M}$ is dominated by the sum of $M$ iid $\Geompos{\frac12\frac{\eps}{1+\eps}}$ random variables.
\end{proof}

\begin{lemma}\label{lem:vertex_transitive_bound_hit}
There exist constants $C>0$ and $\alpha,\theta\in(0,1)$ such that whp $\hittext^{G^*}_{\alpha}(\theta)\le\frac{C}{\eps}$.
\end{lemma}

\begin{proof}
For each vertex $x$ let $V_x=\{y:\:\Omega_{x,y}\text{ holds}\}$. Let
\[
\Omega_0:=\quad\left\{|V_x|\ge C_1n\quad\forall x;\quad\prcond{X_{\tau}=y}{G^*}{x}\ge\frac{C_2}{n}\quad\forall x,\forall y\in V_x\right\}.
\]
First we show that for sufficiently small values of $C_1$ and $C_2$, the event $\Omega_0$ holds whp, and then we show that on this event we have $\hittext^{G^*}_{\alpha}(\theta)\le\frac{C}{\eps}$ for some $C$, $\alpha$ and $\theta$.

Let $c$ be the constant in Lemma~\ref{lem:neighhbourhoods_disjoint} let $\Omega_{\text{root}}$ be the event that at least $\left(1-\frac{c}{2}\right)n$ vertices $y$ are $(r_{y,0},...,r_{y,M-1})$-roots. By Corollary~\ref{cor:many_r_roots} this is a high probability event. Also on this event for any given vertex $x$, the event $\Omega_{x,y}$ holds for at least a constant proportion of vertices $y$. This shows that for $C_1$ sufficiently small we have
\[
\pr{|V_x|\ge C_1n\quad\forall x}\quad=\quad1-o(1).
\]
Also by Lemma~\ref{lem:tau_transtion_probs} we have
\begin{align*}&\pr{\exists\: x,y\in V_x:\quad\prcond{X_{\tau}=y}{G^*}{x}<\frac{C_2}{n}}\quad\le\quad\sum_{x,y}\pr{\Omega_{x,y},\:\prcond{X_{\tau}=y}{G^*}{x}<\frac{C_2}{n}}\\
&\le\quad\sum_{\substack{x,y:\\d_G(x,y)\ge r_{x}+r_{y,0}}}\prcond{\prcond{X_{\tau}=y}{G^*}{x}<\frac{C_2}{n}}{\Omega_{x,y}}{}\quad\le\quad n^2\cdot o\left(\frac{1}{n^2}\right)\quad=\quad o(1).
\end{align*}
Together these show that $\pr{\Omega_0}=1-o(1)$.

In what follows we work on the event $\Omega_0$. Let $x$ be any given vertex of $G^*$.

Note that
\[
1-\dtv{\prcond{X_{\tau}=\cdot}{G^*}{x}}{\pi_{G^*}(\cdot)}\quad=\quad\sum_{y}\prcond{X_{\tau}=\cdot}{G^*}{x}\wedge\pi_{G^*}(y)\quad\gtrsim\quad1,
\]
since $\pi_{G^*}(y)\asymp\frac1n$ for all $y$, $\prcond{X_{\tau}=\cdot}{G^*}{x}\gtrsim\frac1n$ for all $y\in V_x$ and we have $|V_x|\gtrsim n$. Let $\theta'\in(0,1)$ be a constant such that $\dtv{\prcond{X_{\tau}=\cdot}{G^*}{x}}{\pi_{G^*}(\cdot)}\le1-\theta'$ (for all $x$ and all $G^*$ where $\Omega_0$ holds).

Let $A$ be any subset of the vertices with $\pi_{G^*}(A)\ge1-\alpha$ where $\alpha\in\left(0,\frac{\theta'}{3}\right)$. Let $C$ be such that Lemma~\ref{lem:bound_tau_2} holds with $c=\frac{\theta'}{3}$. Then
\begin{align*}\prcond{\tau_A>\frac{C}{\eps}}{G^*}{x}\quad&\le\quad\dtv{\prcond{X_{\tau}=\cdot}{G^*}{x}}{\pi_{G^*}(\cdot)}+\pi_{G^*}(A^c)+\prcond{\tau>\frac{C}{\eps}}{G^*}{x}\\
&\le\quad(1-\theta')+\frac{\theta'}{3}+\frac{\theta'}{3}\quad=\quad1-\frac{\theta'}{3}.
\end{align*}
This shows that on the high probability event $\Omega_0$ we have $\hittext^{G^*,\lazytext}_{\alpha}({\theta})\le\frac{{C}}{\eps}$ for $\alpha$ and ${C}$ as above and ${\theta}=1-\frac{\theta'}{3}$.

As in the proof of Lemma~\ref{lem:SRW_cutoff_implies_lazy_cutoff} we get that $\hittext^{G^*,\lazytext}_{\alpha}({\theta})=2(1\pm o(1))\hittext^{G^*}_{\alpha}({\theta}\pm o(1))$. Hence we also have $\hittext^{G^*}_{\alpha}(\til{\theta})\le\frac{\til{C}}{\eps}$ for some $\alpha$, $\til{C}$ and $\til{\theta}$.
\end{proof}

Now we are ready to conclude that whp $G^*$ does not exhibit cutoff.

\begin{proof}[Proof of Proposition~\ref{pro:Gstar_no_cutoff}]
Let $C$, $\theta$ and $\alpha$ be as in Lemma~\ref{lem:vertex_transitive_bound_hit} and work on the high probability event that $\hittext^{G^*}_{\alpha}(\theta)\le\frac{C}{\eps}$.

By Lemma~\ref{lem:bound_trelabs} we have that $\trelabs(G^*)\lesssim\frac1\eps$, and hence using \cite[Proposition 3.3 and Corollary 3.4]{characterisation_of_cutoff} together with the above bound on $\hittext^{G^*}_{\alpha}(\theta)$ we get that $\hittext^{G^*}_{1/2}(\theta')\lesssim\frac1\eps$ for any $\theta'\in(0,1)$. Applying~\cite[Proposition 1.8 and Remark 1.9]{characterisation_of_cutoff} we conclude that $\tmixtext^{G^*}(\theta')\lesssim\frac1\eps$ for any $\theta'\in(0,1)$.

Let $c$ be such that $\tmixtext^{G^*}(1-\theta')\le\frac{c}{\eps}$.  By Lemma~\ref{lem:vertex_transitive_lower_bound_tmix} we have that $\tmixtext^{G^*}(\theta'')\ge\frac{2c}{\eps}$ for some~$\theta''$. Together these show that the walk on $G^*$ does not exhibit cutoff.

If $\trelabs(G^*)\gtrsim\frac1\eps$ then we immediately get $\tmixtext^{G^*}\left(\frac{1}{4}\right)\gtrsim\trelabs(G^*)\gtrsim\frac{1}{\eps}$. If $\trelabs(G^*)\ll\frac1\eps$ then using Lemma~\ref{lem:vertex_transitive_lower_bound_tmix}, the comparison between the mixing and hitting times from~\cite[Proposition~1.8 and Remark~1.9]{characterisation_of_cutoff}, and \cite[Proposition~3.3]{characterisation_of_cutoff}, we again get $\tmixtext^{G^*}\left(\frac{1}{4}\right)\gtrsim \frac{1}{\eps}$.
\end{proof}

Now we proceed to prove Proposition~\ref{pro:transitive_poly_satsifies_assump}. The hardest part is establishing the lower bound on the transition probabilities. We prove this via a series of lemmas. First we show the lower bound and a matching upper bound for the diagonal entries $P_{G,\lazytext}^{t}(o,o)$. Then we upper bound the difference $\left|P_{G,\lazytext}^{t}(o,o)-P_{G,\lazytext}^{t}(o,x)\right|$ in terms of the diagonal transition probabilities.

\begin{lemma}\label{lem:poly_growth_exponent}
In the setup of Proposition~\ref{pro:transitive_poly_satsifies_assump} there exist constants $C$ and $L$ such that for any $r,k>0$ we have $V(rk)\le Ck^LV(r)$.
\end{lemma}

\begin{proof}
This is a consequence of \cite[Corollary 1.5]{structure_thm}.
\end{proof}

\begin{lemma}\label{lem:diagonal_lower_bound}
In the setup of Proposition~\ref{pro:transitive_poly_satsifies_assump} for any constant $c>0$, any $t>0$ and any vertex $o$ we have
\[
P_{G,\lazytext}^{t}(o,o)\quad\gtrsim\quad\frac{1}{V(c\sqrt{t})}.
\]
\end{lemma}

The proof of this is similar to the proof of~\cite[Lemma 6.12]{universality_of_fluctuations} and is deferred to Appendix~\ref{app:estimates_for_transitive_poly_growth_graphs}.

\begin{lemma}\label{lem:diagonal_upper_bound}
Let $G$ be as in Proposition~\ref{pro:transitive_poly_satsifies_assump}. There exists a constant $C$ such that for any vertex $o$, any $t\le C\diam{G}^2$ and any constant $c>0$ we have
\[
P_{G,\lazytext}^t(o,o)\quad\lesssim\quad\frac{1}{V(c\sqrt{t})}.
\]
\end{lemma}

The proof of this relies on a bound of mixing times in terms of spectral profile and is quite technical. Full details are provided in Appendix~\ref{app:estimates_for_transitive_poly_growth_graphs}.

\begin{lemma}[Gradient inequality]\label{lem:off_diag_comparison}
Let $G$ be a vertex-transitive graph on $n$ vertices. Then for any vertices $o$, $x$, $y$ and any $t,s>0$ we have
\[
\left|P_{G,\lazytext}^{t+s}(o,x)-P_{G,\lazytext}^{t+s}(o,y)\right|\quad\lesssim\quad d(x,y)\frac{1}{\sqrt{s}}P_{G,\lazytext}^{t}(o,o).
\]
\end{lemma}

\begin{proof} The proof is analogous to the proof \cite[Proposition 6.1]{universality_of_fluctuations}, with the modification that instead of a continuous-time chain we directly consider a lazy chain, and instead of $q(u)=\sqrt{u}e^{-2su}$ we use $q(u)=\sqrt{u}(1-u)^s$.
\end{proof}

\begin{proof}[Proof of Proposition~\ref{pro:transitive_poly_satsifies_assump}]
Let $P=P_{G,\lazytext}$. From Lemma~\ref{lem:off_diag_comparison} we get that
\[
P^{2t}(o,x)\quad\ge\quad P^{2t}(o,o)-C'\frac{d(o,x)}{\sqrt{t}}P^{t}(o,o)
\]
where $C'$ is some positive constant.

From Lemmas~\ref{lem:diagonal_lower_bound} and \ref{lem:diagonal_upper_bound} we know that for any $t$ we have $P^{2t}(o,o)\ge\frac{a_1}{V(\sqrt{t})}$, and for any $t\le C\diam{G}^2$ we have $P^{t}(o,o)\le\frac{a_2}{V(\sqrt{t})}$ where $a_1,a_2>0$ are constants. Let $a_4>0$ be small enough so that $a_1-C'a_2a_4>0$ and let $0<a_3<a_4$.

Then for any $r$ with $a_3\sqrt{t}\le r\le a_4\sqrt{t}$ and $x\in B(o,r)$ we have
\[
P^{2t}(o,x)\quad\ge\quad\left(a_1-C'a_2a_4\right)\frac{1}{V(\sqrt{t})}\quad\asymp\quad\frac{1}{V(\sqrt{t})}\quad\asymp\quad\frac{1}{V(r)},
\]
where the last $\asymp$ follows from Lemma~\ref{lem:poly_growth_exponent}. We can get a bound on $P^{2t+1}(o,x)$ with $t\le C\diam{G}^2$ analogously. To show that the bound holds for any $t\le\frac{c_1}{\eps}$ it remains to check that $\frac1\eps\lesssim\diam{G}^2$. For this it is sufficient to prove that $\tmixtext^{G}(\theta)\lesssim\diam{G}^2$ for some $\theta\in(0,1)$.

Let $t=\diam{G}^2$. We know that for sufficiently small $c$ we have $P^t(o,x)\gtrsim\frac{1}{V(\sqrt{t})}=\frac1n$ for any $x\in B(o,c\sqrt{t})$. Then
\[
1-\dtv{P^t(o,\cdot)}{\pi(\cdot)}\quad=\quad\sum_{x}P^t(o,x)\wedge\pi(x)\quad\gtrsim\quad1,
\]
i.e.\ $\exists\:\theta'$ with $\tmixtext^{G,\lazytext}\left(\theta'\right)\le\diam{G}^2$. Using this, that $\trelabs\lesssim\frac1\eps$ by Lemma~\ref{lem:bound_trelabs}, that $\tmixtext^G(\theta)\gg\frac1\eps$ for all $\theta$ by assumption, Lemmas~\ref{lem:hit_mix_comparison} and~\ref{lem:hitting_times_comparison}, and equation~\eqref{eq:hit_hitlazy_comparison}, we get that $\tmixtext^{G}\left(\theta'\right)\le\diam{G}^2$ for some $\theta$.

For the last point of Assumption~\ref{assump:for_eps_power_of_n} it is sufficient to prove that for any $c'$ we have $V(c\diam{G})\le c'n$ for sufficiently small values of $c$.

Let $x$ and $y$ be two vertices of $G$ with $d(x,y)=\diam{G}$ and let $x_0,...,x_{\diam{G}}$ be a geodesic between $x$ and $y$. Then for any $k\in\Zpos$ the balls $B(x_{(3j-1)r},r)$ where $r=\big\lfloor\frac{\diam{G}}{3k}\big\rfloor\asymp\diam{G}$, $j=1,2,...,k$ are disjoint, hence $V(r)\le\frac{n}{k}$.

This finishes the proof.
\end{proof}

Now we proceed to prove Proposition~\ref{pro:Gstar_transitive_poly_no_cutoff}

\begin{lemma}\label{lem:G_lazy_nocutoff}
Let $G$ be as in Proposition~\ref{pro:transitive_poly_satsifies_assump}. Then the lazy random walk on $G$ does not exhibit cutoff.
\end{lemma}

\begin{proof} Let $X$ be a lazy random walk on $G$, and let $P$ denote its transition matrix.

First we show that $\exists\theta\in(0,1)$ such that we have $\tmixtext^{G,\lazytext}(\theta)\gtrsim\diam{G}^2$.

Consider any sequence $t\ll\diam{G}^2$. Then for a sufficiently small constant $c$ we have $P^t(o,x)\gtrsim\frac{1}{V(\sqrt{t})}$ for all $x\in B(o,c\sqrt{t})$. Therefore $\prstart{X_t\in B(o,c\sqrt{t})}{o}\gtrsim1$, while $\pi\left(B(o,c\sqrt{t})\right)=\frac{V(\sqrt{t})}{n}\ll1$, hence $\dtv{P^t(o,\cdot)}{\pi(\cdot)}\gtrsim1$, where the constant in $\gtrsim1$ only depends on the choice of $c$.

This shows that there exists $\theta\in(0,1)$ such that for any $t\ll\diam{G}^2$ for sufficiently large $n$ we have $\tmixtext^G(\theta)>t$. This shows $\tmixtext^G(\theta)\gtrsim\diam{G}^2$.

Now we show that for any $t\asymp\diam{G}^2$ there exists $\theta\in(0,1)$ with $\tmixtext^{G,\lazytext}(1-\theta)\le t$.

It is sufficient to consider $t=c\cdot\diam{G}^2$ where $c$ is a sufficiently small constant. Then by Lemma~\ref{lem:poly_growth_exponent} we have $V(\sqrt{t})\asymp n$, and we know that $P^t(o,x)\gtrsim\frac{1}{V(\sqrt{t})}$ for all $x\in B(o,c\sqrt{t})$. Therefore
$$1-\dtv{P^t(o,\cdot)}{\pi(\cdot)}\quad\ge\quad\sum_{x\in B(o,c\sqrt{t})}P^t(o,x)\wedge\pi(x)\quad\gtrsim\quad n\cdot\frac1n\quad\asymp\quad1.$$
This shows that $\tmixtext^{G,\lazytext}(1-\theta)\le t$ for some $\theta\in(0,1)$.
This finishes the proof.
\end{proof}

Lemma~\ref{lem:G_lazy_nocutoff} and Lemma~\ref{lem:SRW_cutoff_implies_lazy_cutoff} immediately gives following.

\begin{corollary}\label{cor:G_transitive_poly_no_cutoff}
Let $G$ be as in Proposition~\ref{pro:transitive_poly_satsifies_assump}. Then the simple random walk on $G$ does not exhibit cutoff.
\end{corollary}

\begin{lemma}\label{lem:tmix_order_transitive_poly}
Let $G$ be as in Proposition~\ref{pro:transitive_poly_satsifies_assump}. Then $\tmixtext^{G,\lazytext}\left(\frac14\right)\asymp\diam{G}^2$, $\tmixtext^{G}(\theta)\gtrsim\diam{G}^2$ for all $\theta\in(0,1)$, and $\tmixtext^{G}(1-\theta)\lesssim\diam{G}^2$ for some $\theta\in(0,1)$.
\end{lemma}

\begin{proof}
First we show that $\tmixtext^{G,\lazytext}\left(\frac14\right)\lesssim\diam{G}^2$.

Let $t=c\diam{G}^2$ where $c$ is a sufficiently small constant so that $P_{G,\lazytext}^{t}(o,o)\asymp\frac1n$ (we are using Lemma~\ref{lem:diagonal_lower_bound} and Lemma~\ref{lem:diagonal_upper_bound} here). Let $s=C\diam{G}^2$ where $C$ is a sufficiently large constant. From Lemma~\ref{lem:diagonal_lower_bound} we know that $P_{G,\lazytext}^{t+s}(o,o)\gtrsim\frac1n$ where the constant in $\gtrsim$ does not depend on $C$.

Let $a_1$, $a_2$ and $a_3$ be constants such that $P_{G,\lazytext}^{t}(o,o)\le a_1\frac1n$, $P_{G,\lazytext}^{t+s}(o,o)\ge a_2\frac1n$ for all $C$, and in Lemma~\ref{lem:off_diag_comparison} we have
\begin{align}\label{eq:off_diag_comparison_const}
\left|P_{G,\lazytext}^{t+s}(o,x)-P_{G,\lazytext}^{t+s}(o,y)\right|\quad\le\quad a_3 d(x,y)\frac{1}{\sqrt{s}}P_{G,\lazytext}^{t}(o,o).
\end{align}

Let $\theta\in(0,1)$ be a sufficiently small constant and let us choose $C$ such that $\frac{a_3}{\sqrt{C}}\frac{a_{1}}{a_{2}}\le\theta$.

Then by~\eqref{eq:off_diag_comparison_const} for any $o$ and $x$ we have
\[
\left|P_{G,\lazytext}^{t+s}(o,x)-P_{G,\lazytext}^{t+s}(o,o)\right|\quad\le\quad \frac{a_3}{\sqrt{C}}P_{G,\lazytext}^{t}(o,o)\quad\le\quad\theta P_{G,\lazytext}^{t+s}(o,o),
\]
i.e.\ for all $o$ and $x$ we have
\[
(1-\theta)P_{G,\lazytext}^{t+s}(o,o)\quad\le\quad P_{G,\lazytext}^{t+s}(o,x)\quad\le\quad(1+\theta)P_{G,\lazytext}^{t+s}(o,o).
\]

Choosing $\theta$ sufficiently small this shows that $\tmixtext^{G,\lazytext}\left(\frac14\right)\le t+s\asymp\diam{G}^2$.

Then by Lemma~\ref{lem:tmix_bound_by_tmixlazy} we also get that $\tmixtext^{G}(1-\theta)\lesssim\diam{G}^2$ for some $\theta\in(0,1)$.

Now we show that for any $\theta\in(0,1)$ we have $\tmixtext^{G,\lazytext}\left(\theta\right)\gtrsim\diam{G}^2$.

From~\cite[Corollary 2.8]{structure_thm} we know that for any sufficiently large $n$ the graph $G$ has $(c,a)$-moderate growth for some $a$ and $c$. Then by~\cite[Proposition 11.1]{exclusion_process} we get that $\trel\asymp\diam{G}^2$. Then using~\cite[Theorem 12.5]{MTMC_book} gives $\tmixtext^{G,\lazytext}\left(\theta\right)\gtrsim\trel\asymp\diam{G}^2$ for all $\theta\in(0,1)$.

By~\cite[Corollary 9.5]{mix_hit} we also get a lower bound of order $\diam{G}^2$ for the mixing time if instead of a lazy walk we consider a $p$-lazy walk on $G$ for some $p\in(0,1)$ bounded away from 0 and 1, i.e.\ a walk with transition matrix $pI+(1-p)P_G$.

Finally we show that $\tmixtext^{G}(\theta)\gtrsim\diam{G}^2$ for all $\theta\in(0,1)$.

Let $H$ be a graph obtained as follows. Let $H$ have the same vertex set as $G$. For each pair $\{x,y\}$ with $x\ne y$ let $H$ have twice as many edges between $x$ and $y$ as the number of paths of length 2 between $x$ and $y$ in $G$, and for each $x$ let it have $\deg_G(x)$ loops from $x$. Note that the simple random walk on this graph $H$ has transition matrix $P^2$.

Also note that each vertex in $H$ has degree $2\deg_G(o)^2$ and has exactly $\deg_G(o)$ loops. Let $\til{G}$ be a graph obtained by deleting all loops in $H$.

Then $\til{G}$ is a vertex-transitive graph with bounded degrees and polynomial growth of balls (and having multiple edges). Also either $\til{G}$ is connected (if $G$ is not bipartite) or it has two connected components of size $\frac{n}{2}$ (if $G$ is bipartite). Note that the transition matrix of the SRW on $\til{G}$ satisfies $\frac{1}{\deg_G(o)}I+\left(1-\frac{1}{\deg_G(o)}\right)P_{\til{G}}=P^2$. Also each connected component of $\til{G}$ has diameter $\asymp\diam{G}$.

Using the above result for the $\frac{1}{\deg_G(o)}$-lazy random walk on one connected component of $\til{G}$ we get that $\tmixtext^{G}(\theta)\gtrsim\diam{G}^2$ for all $\theta\in(0,1)$.
\end{proof}

\begin{proof}[Proof of Proposition~\ref{pro:Gstar_transitive_poly_no_cutoff}]
${}$

If $\eps\ll\frac{1}{\diam{G}^2}$ then by Proposition~\ref{pro:smaller_eps}, Corollary~\ref{cor:G_transitive_poly_no_cutoff} and Lemma~\ref{lem:tmix_order_transitive_poly} the random walk on $G^*$ does not exhibit cutoff.

If $\eps\asymp\frac{1}{\diam{G}^2}$ then by Proposition~\ref{pro:eps_asymp_tmix_inv_no_cutoff_implies_no_cutoff}, Corollary~\ref{cor:G_transitive_poly_no_cutoff} and Lemma~\ref{lem:tmix_order_transitive_poly} the random walk on $G^*$ does not exhibit cutoff.

If $\eps\asymp n^{-\Theta(1)}\gg\frac{1}{\diam{G}^2}$ then by Corollary~\ref{cor:eps_power_of_n} and Lemma~\ref{lem:tmix_order_transitive_poly} the random walk on $G^*$ does not exhibit cutoff.

Any sequence $\eps\lesssim n^{-\Theta(1)}$ has a subsequence which falls into one of the above categories.
\end{proof}

\subsection{Proof of Theorems~\ref{thm:results_lin_entropy}, \ref{thm:results_poly_balls}, \ref{thm:results_expanders} and \ref{thm:results_poly_growth_vertex_transitive}}

\begin{proof}[Proof of Theorem~\ref{thm:results_lin_entropy}]
This follows from Proposition~\ref{pro:lin_entropy_growth} (case $\frac{1}{n}\ll\eps\lesssim(\log n)^{-1/3}$), Proposition~\ref{pro:larger_eps} (case $(\log n)^{-1/3}\ll\eps\ll1$), Proposition~\ref{pro:const_order_eps} (case $\eps\asymp1$), Proposition~\ref{pro:smaller_eps} and Proposition~\ref{pro:eps_asymp_tmix_inv_no_cutoff_implies_no_cutoff}.
\end{proof}

\begin{proof}[Proof of Theorem~\ref{thm:results_poly_balls}]
As above, with Proposition~\ref{pro:poly_ball_growth} instead of Proposition~\ref{pro:lin_entropy_growth}.
\end{proof}

\begin{proof}[Proof of Theorem~\ref{thm:results_expanders}]
From Remark~\ref{rmk:local_expanders_lin_entropy} we know that expanders have linear growth of entropy. It is immediate to see that $\tmixtext^{\lazytext}\left(\frac14\right)\asymp\log n$ and $\tmixtext(\theta)\gtrsim\log n$ for all $\theta\in(0,1)$.
Then from Lemma~\ref{lem:tmix_bound_by_tmixlazy} we get that $\tmixtext(\theta)\asymp\log n$ for some $\theta\in(0,1)$, so the first three points of Theorem~\ref{thm:results_expanders} follow from Theorem~\ref{thm:results_lin_entropy}.

The last two points follow from Corollary~\ref{cor:expander_examples}
\end{proof}

\begin{proof}[Proof of Theorem~\ref{thm:results_poly_growth_vertex_transitive}]
This follows from Theorem~\ref{thm:results_poly_balls} and Proposition~\ref{pro:Gstar_transitive_poly_no_cutoff}.
\end{proof}

\section{General sequences of graphs}\label{sec:general_graphs}

In this section we discuss some conjectures and questions regarding more general families of graphs and present a sketch of the proof of Theorem~\ref{thm:results_general_graphs}.

\subsection{Conjectures and open problems}\label{subsec:conjectures}

In the two main regimes we considered in the above proofs (graphs with polynomial growth of balls, graphs with linear growth of entropy) we expressed $\h$, which captures the average entropy of a walk on the quasi-tree, in terms of $\eps$, and proved that cutoff holds in case $\h\ll\log n$, with mixing time of order $\frac{\log n}{\eps\h}$. In both cases the order of the entropy of the first long-range edge crossed by the walk did not depend on the starting point. Using this and the assumption on the growth of balls allowed us to prove entropic concentration~\eqref{eq:entropy_deviations}. We believe that this phenomenon holds more generally, in the following form.

For a given graph $G$, weight $\eps$ and vertex $\rho$ let
\[
h(\eps,\rho):=\quad H\left(X_{\tauLR}^{(\rho)}\right)
\]
be the entropy of the first long-range edge crossed by a walk on $G^*$, when started from $\rho$.

Also let
\[
h(\eps):=\quad \frac1n\sum_{\rho}h(\eps,\rho).
\]

\begin{conjecture}\label{conj:h_unif}
Let $(G_n)$ be a sequence of graphs with (uniformly) bounded degrees and diverging sizes, and let $(\eps_n)$ be a sequence of constants in $(0,1)$. Let $(G_n^*)$ be as in Definition~\ref{def:Gstar}. If $h_n(\eps_n,\rho_n)\asymp h_n(\eps_n)$ for all $n$ and all $\rho_n$ and $h_n(\eps_n)\ll\log|V_n|$, then the random walk on $(G_n^*)$ exhibits cutoff with high probability, with a mixing time of order $\frac{\log|V_n|}{\eps_n h_n(\eps_n)}$.
\end{conjecture}

Note in particular, that for vertex-transitive graphs the assumption $h(\eps,\rho)\asymp h(\eps)$ holds for any $\eps$, so Conjecture~\ref{conj:h_unif} gives a sufficient condition on $\eps$ to ensure cutoff.

We established the result under Assumption~\ref{assump:f}, which in addition to $h(\eps,\rho)\asymp h(\eps)$ also assumes a matching upper bound on the volume of balls and makes some technical assumptions on $H_2$ and $H_4$. One way we could relax the assumption on the growth of balls is by using better upper bounds on the maximum distance a walk on $G$ reaches up to a given time. (This would allow us to choose a smaller radius $R$ in the quasi-tree, while still ensuring that the walk will likely not hit the boundary of $R$-balls.)

A-priori it seems plausible that for vertex-transitive graphs the entropy of the walk at a given time $t$ is always of the same order as the log of the size of a ball with radius the expected distance travelled by time $t$. The former is always bounded by the latter up to a multiplicative constant, but there is a counterexample for the other direction.
For example consider a version of the lamplighter graph where the base graph is a copy of $\Z_k$ and each lamp takes values in $\Z_k$ (and takes steps as a simple random walk on $\Z_k$). We would like to thank Gady Kozma for this example.

The condition $h(\eps,\rho)\asymp h(\eps)$ is included in Conjecture~\ref{conj:h_unif} as it should make it easier to control the fluctuations of the entropy, but one might also ask the following question.

\begin{open_problem} Is Conjecture~\ref{conj:h_unif} true if we replace the condition $h_n(\eps_n,\rho_n)\asymp h_n(\eps_n)\ll\log|V_n|$ with $\max_{\rho_n}h_n(\eps_n,\rho_n)\ll\log|V_n|$?
\end{open_problem}

Since we have $h(\rho,\eps)\lesssim\frac{1}{\eps}$ for any $\rho$ and any $\eps$ with $\frac{1}{\eps}\le\diam{G}$, a special case of the above would be the following.

\begin{conjecture}
Let $(G_n)$ be a sequence of graphs with (uniformly) bounded degrees and diverging sizes, and let $(\eps_n)$ be a sequence of constants in $(0,1)$ satisfying $\eps_n\gg\frac{1}{\log |V_n|}$. Let $(G_n^*)$ be as in Definition~\ref{def:Gstar}. Then the random walk on $(G_n^*)$ exhibits cutoff with high probability, with a mixing time of order $\frac{\log|V_n|}{\eps_n h_n(\eps_n)}$.
\end{conjecture}

With a minor modification of the proofs presented above we are able to prove a slightly weaker version of the above statement, namely Theorem~\ref{thm:results_general_graphs}. We present the required modifications below.

\subsection{Sketch proof of Theorem~\ref{thm:results_general_graphs}}

The proofs of parts~\eqref{cond:general_graph_small_eps} and~\eqref{cond:general_graph_eps_one_over_tmix} are as before. In case $\eps\asymp1$ the proof is essentially the same as the proof of the $\eps\equiv1$ case in~\cite{random_matching}. In what follows we assume $1\gg\eps\gg\frac{\log\log n}{\log n}$. In this case the proof is similar to the proof of cutoff presented in Sections~\ref{sec:quasi-tree}, \ref{sec:relating_Gstar_and_T} and~\ref{sec:proof_of_cutoff}, and we explain the required modifications below.

We choose the following values of the parameters.
\[
R:=C_1\frac{1}{\eps}\log\log n,\qquad K:=C_2\frac{\log\log n}{\log\left(\frac1\eps\right)},\qquad M:=\frac12K,
\]
where $C_1$ and $C_2$ are sufficiently large constants.

We prove concentration of entropy on the event that the graph distance in $T$ between two consecutive regeneration edges is at most $R$ and then use that with high probability this event will hold up to the level that the walk on the quasi-tree reaches by the mixing time. More precisely, we define\[
A_k:=\quad\left\{d_{T,\textrm{gr}}\left(X_{\sigma_{k-1}},X_{\sigma_k}\right)\le C_1\frac1\eps\log\log n\right\},
\]
where $d_{T,\textrm{gr}}$ denotes the graph distance on $T$, and
\[\h:=\quad\frac{1}{\E{\varphi_2-\varphi_1}}\cdot\E{\left(-\log\prcond{X'_{\sigma'_1}\in\til{\xi}'}{X',T'}{\rho}\right)\1_{A_1}}.\]
We prove that for any $b\ge2$ we have
\[
\E{\left(-\log\prcond{X'_{\sigma'_1}\in\til{\xi}'}{X',T'}{\rho}\right)^b\1_{A_1}}\lesssim\h\left(\frac{\log\log n}{\eps}\right)^{b-1}+(\log\log n)^b\ll\h(\log n)^{b-1},
\]
and define
\[
\V:=\quad\h\left(\frac{\log\log n}{\eps}\right)+(\log\log n)^2.
\]
Then we can prove that for all $\theta>0$ and $\widehat{C}>0$, there exists a positive constant $C$ so that for all $\frac{(K\log b(R))^2}{\V}\vee1\le k\le\frac{\widehat{C}\log n}{\h}=:\ell$ we have
\begin{align}\label{eq:entropy_deviations_for_reg_edges}
\prcond{\left|-\log\prcond{\xi_{\varphi_k}\in \til{\xi}}{T,\xi}{} -\h\varphi_k\right|>C\sqrt{\varphi_k\V},\:\bigcap_{i=1}^{\ell}A_{i}}{\B_K(\rho)=T_0}{}\quad\le\quad\theta.
\end{align}
For upper bounding the mixing time we let
\[
t_0:=\frac{\log n}{\nu\h},\qquad t_w:=\frac{1}{\nu\h}\sqrt{\frac{\V\log n}{\h}},\qquad L:=\frac12\nu(t_0+Bt_w),
\]
we use the same truncation criteria as before, and we bound the probability of crossing a truncated edge by using~\eqref{eq:entropy_deviations_for_reg_edges} with $\ell\ge L$, and that $\pr{\left(\bigcap_{i=1}^{\ell}A_{i}\right)^c}\ll 1$.

The lower bound on the mixing time follows analogously, with the above values of the parameters. \hfill $\qed$

\begin{remark}
We worked on event $A_1$ to be able to bound the variance of\\ $-\log\prcond{X'_{\sigma'_1}\in\til{\xi}'}{X',T'}{\rho}$ in terms of the expectation $\h$. It is easy to check that the value of $\h$ would be asymptotically the same without the indicator $\1_{A_1}$.
\end{remark}

\appendix

\section{Some auxiliary results}\label{app:aux_results}

The proofs of the following three lemmas are given in Appendix~\ref{app:other_remainings_pfs}.

\begin{lemma}\label{lem:Pt_upper_bound_lazy}
Let $G$ be a connected graph on $n$ vertices with all vertices having degree $\le\Delta$. Let $x$ be a vertex of $G$ and let $X$ be a lazy simple random walk on $G$, starting from $x$. Let $A$ be a positive constant. Then there exists a positive constant $c$ depending only on $\Delta$ and $A$ such that for any vertex $y$ and any $t<An^2$ we have
\[
\prstart{X_t=y}{x}\quad\le\quad\frac{c}{\sqrt{t}}.
\]
\end{lemma}

\begin{lemma}\label{lem:Pt_upper_bound}
Let $G$ be a connected graph on $n$ vertices with all vertices having degree $\le\Delta$. Let $x$ be a vertex of $G$ and let $X$ be a simple random walk on $G$, starting from $x$. Let $A$ be a positive constant. Then there exists a positive constant $c$ depending only on $\Delta$ and $A$ such that for any vertex $y$ and any $t<An^2$ we have
\[
\prstart{X_t=y}{x}\quad\le\quad\frac{c}{\sqrt{t}}.
\]
\end{lemma}

\begin{lemma}\label{lem:hit_hitlazy_comparison}
Let $\alpha\in(0,1)$ be given, let $\left(X^{(n)}\right)$ be a sequence of Markov chains satisfying $\hittext^{X^{(n)}}_{\alpha}(\theta)\gg1$ for all $\theta\in(0,1)$, and let $Y^{(n)}$ be the lazy version of $X^{(n)}$. Then for any $\theta\in(0,1)$ the hitting times of the chains satisfy
\begin{align}\label{eq:hit_hitlazy_comparison}
\hittext^{Y^{(n)}}_{\alpha}(\theta)\quad=\quad(1\pm o(1))\cdot2\hittext^{X^{(n)}}_{\alpha}(\theta\pm o(1)).
\end{align}
\end{lemma}

The following properties of $H_b(\cdot)$ and $H_b(\cdot|\cdot)$ (as defined in Definitions~\ref{def:Hb} and~\ref{def:Hb_conditional}) are easy to check.

\begin{lemma}\label{lem:H_b_properties}
Let $b\in\Zpos$ and let $(p_i)$, $(q_i)$ and $(r_i)$ be sequences of real numbers taking values in $[0,\theta]$ where $\theta\in(0,1)$ is a sufficiently small constant (depending on $b$). Let $W$ be a random variable taking values in a countable set $\mathcal{W}$. Then the following properties are satisfied.
\begin{enumerate}[(i)]
\item\label{property:decreasing}
If $p_i\ge q_i$ for all $i$, then
\[
H_b(p_1,p_2,...)\quad\ge\quad H_b(q_1,q_2,...).
\]
\item\label{property:splitting}
If $p_i=q_i+r_i$ for all $i$ then
\[
H_b(p_1,p_2,...)\quad\le\quad H_b(q_1,q_2,...)+H_b(r_1,r_2,...).
\]
\item\label{property:factorising}
If $p_i=q_ir_i$ for all $i$ then
\[
H_b(p_1,p_2,...)\quad\lesssim\quad H_b(q_1,q_2,...)+H_b(r_1,r_2,...).
\]
\item\label{property:bound_by_set_size}
\[
H_b(W)\quad\le\quad\left(\log|\mathcal{W}|\right)^b.
\]
\item\label{property:bound_by_set_size_two}
If $p_1+...+p_N=p<1$ and $p_i=0$ for $i>N$ then
\[
H_b(p_1,p_2,...)\quad\lesssim\quad p\left(\log N\right)^b+p(-\log p)^b.
\]
\end{enumerate}
\end{lemma}

\begin{lemma}\label{lem:H_b_cond_properties} Let $b\in\Zpos$ and let $W$ and $Z$ be random variables taking values in countable sets. Then
\[
H_b(W|Z)\quad\le\quad H_b(W)\quad\lesssim\quad H_b(W|Z)+H_b(Z).
\]
\end{lemma}

\section{Concentration of speed and entropy}\label{app:proofs_speed_entropy}

\begin{proof}[Proof of Lemma~\ref{lem:speed}]
We will mimic the proof of \cite[Lemma 3.11]{random_matching}.

From Corollary~\ref{cor:varphi_sigma_E_Var_bounds} we know that $\E{\varphi_2-\varphi_1}\asymp1$ and $\E{\sigma_2-\sigma_1}\asymp\frac1\eps$, hence $\nu\asymp\eps$.

For the a.s.\ convergence:
\begin{itemize}
\item By the ergodic theorem we have $\frac{\varphi_k}{k}\to \E {\varphi_2-\varphi_1}$ and $\frac{\sigma_k}{k}\to \E {\sigma_2-\sigma_1}$ almost surely.
These give
\[
\frac{d_T(\rho,X_{\sigma_k})}{\sigma_k}=\frac{\varphi_k}{\sigma_k}\to\frac{\E {\varphi_2-\varphi_1}}{\E {\sigma_2-\sigma_1}}=\nu \quad\text{a.s.}.
\]
\item Let $M_k=\max\{i:\sigma_i\le k\}$ and let $b_j=\sup_{i:i\le\sigma_{j+1}}d_T(\rho,X_i)-\varphi_{j+1}$. The variables $(b_j)_{j\ge1}$ are identically distributed and have finite mean and variance, hence $\frac{b_j}{j}\to0$ a.s.\ as $j\to\infty$.

Then
\[
\frac{\varphi_{M_k}}{\sigma_{M_k}}\frac{\sigma_{M_k}}{k}\quad\le\quad\frac{d_T(\rho,X_k)}{k}\quad\le\quad\frac{\varphi_{M_k+1}+b_{M_k}}{\sigma_{M_k+1}}\frac{\sigma_{M_k+1}}{k}
\]
and both sides converge to $\nu$ almost surely.
\footnote{$\frac{\varphi_{M_k}}{\sigma_{M_k}}\to\nu$ a.s.\ by above. $\frac{b_{M_k}}{\sigma_{M_k+1}}\le\frac{b_{M_k}}{M_k}\to0$ a.s.\ .
$\frac{\sigma_{M_k}}{k}\le1\le\frac{\sigma_{M_k+1}}{k}$ and $\frac{\sigma_{M_k+1}-\sigma_{M_k}}{k}\le\frac{\sigma_{M_k+1}-\sigma_{M_k}}{\sigma_{M_k}}=\frac{\sigma_{M_k+1}-\sigma_{M_k}}{M_k}\frac{M_k}{\sigma_{M_k}}\to0$ a.s., hence $\frac{\sigma_{M_k}}{k},\frac{\sigma_{M_k+1}}{k}\to1$ a.s.\ .}
\end{itemize}
For the bound on deviations for $X_{\sigma_k}$:
\begin{itemize}
\item For any $\alpha>0$ we have
\begin{align}\label{eq:speed_deviation}
&\pr{\left|d_T(\rho,X_{\sigma_k})-\nu\sigma_k\right|>C'\sqrt{\eps\sigma_k}}\\
&\le\quad\pr{\left|\varphi_k-k\E{\varphi_2-\varphi_1}\right|>\frac12C'\sqrt{\alpha k}} \nonumber\\
&\quad+\quad\pr{\left|\nu\sigma_k-\nu k\E{\sigma_2-\sigma_1}\right|>\frac12C'\sqrt{\alpha k}}\quad+\quad \pr{\sigma_k\le\alpha\frac{k}{\eps}}. \nonumber
\end{align}
\item We know that $\pr{\left|\varphi_k-\varphi_1-(k-1)\E{\varphi_2-\varphi_1}\right|>C''\sqrt{k\vr{\varphi_2-\varphi_1}}}\to0$ as $C''\to\infty$, that $\pr{\varphi_1\ge C''}\to0$ as $C''\to\infty$ and that $\Var(\varphi_2-\varphi_1)\lesssim1$ and $\E{\varphi_2-\varphi_1}\asymp1$ \footnote{By Corollary~\ref{cor:varphi_sigma_E_Var_bounds}.}. This allows us to control the first term in \eqref{eq:speed_deviation}.
\item We know that $\pr{\left|\nu(\sigma_k-\sigma_1)-\nu (k-1)\E{\sigma_2-\sigma_1}\right|>C''\nu\sqrt{k\Var(\sigma_2-\sigma_1)}}\to0$ as $C''\to\infty$, that $\pr{\nu\sigma_1>C''}\to0$ as $C''\to\infty$ and that $\E{\sigma_1}\lesssim\frac1\eps$, $\Var(\sigma_2-\sigma_1)\lesssim\frac{1}{\eps^2}$ and $\nu\asymp\eps$.
\footnote{By Corollary~\ref{cor:varphi_sigma_E_Var_bounds}.} This lets us control the second term in \eqref{eq:speed_deviation}.
\item The last term in \eqref{eq:speed_deviation} goes $\to0$ as $\alpha\to0$.
\item Hence for any $\theta>0$ there exists $C'$ such that for all $k$ we have
\[
\pr{\left|d_T(\rho,X_{\sigma_k})-\nu\sigma_k\right|>C'\sqrt{\eps\sigma_k}}\le\theta.
\]
\end{itemize}

For the bound on deviations for the whole sequence:
\begin{itemize}
\item Fix any $k\ge\frac{J}{\eps}$ and let $s^-=\left(\frac{\sqrt{\ell^2+\frac{4k}{\E{\sigma_2-\sigma_1}}}-\ell}{2}\right)^2$ and $s^+=\left(\frac{\sqrt{\ell^2+\frac{4k}{\E{\sigma_2-\sigma_1}}}+\ell}{2}\right)^2$ where $\ell$ is a large integer to be chosen later.
\item Note that
\begin{align*}\pr{\left|d_T(\rho,X_k)-\nu k\right|\ge C\sqrt{\eps k}}\quad&\le\quad\pr{k\not\in\left[\sigma_{s^-},\sigma_{s^+}\right]}\\
&\qquad+\quad\pr{d_T(\rho,X_k)\ge\nu k+C\sqrt{\eps k},\text{ }k\in\left[\sigma_{s^-},\sigma_{s^+}\right]}\\
&\qquad+\quad \pr{d_T(\rho,X_k)\le\nu k-C\sqrt{\eps k},\text{ }k\in\left[\sigma_{s^-},\sigma_{s^+}\right]}.
\end{align*}
We will show that the three terms on the right $\to0$ as $\ell\to\infty$ and $C\to\infty$ (in terms of $\ell$).
\item For any $s$ we have
\begin{align*}
&\pr{\left|\sigma_s-s\E{\sigma_2-\sigma_1}\right|\ge \ell\sqrt{s}\E{\sigma_2-\sigma_1}}\\
&\le\quad\pr{\left|(\sigma_s-\sigma_1)-(s-1)\E{\sigma_2-\sigma_1}\right|\ge\frac12\ell\sqrt{s}\E{\sigma_2-\sigma_1}}\\
&\qquad+\quad\pr{\left|\sigma_1-\E{\sigma_2-\sigma_1}\right|\ge\frac12\ell\sqrt{s}\E{\sigma_2-\sigma_1}}.
\end{align*}
We know that $(\sigma_s-\sigma_1)$ is the sum of $(s-1)$ iid random variables that have mean $\E{\sigma_2-\sigma_1}\asymp\frac1\eps$ and variance $\lesssim\frac{1}{\eps^2}\asymp\E{\sigma_2-\sigma_1}^2$. We also know that $\sigma_1$ has mean $\asymp\frac1\eps\asymp\E{\sigma_2-\sigma_1}$ and variance $\lesssim\frac{1}{\eps^2}\lesssim\E{\sigma_2-\sigma_1}^2$.
Hence both probabilities on the RHS $\to0$ as $\ell\to\infty$, uniformly in $s$.
\[
s^-\E{\sigma_2-\sigma_1}+\ell\sqrt{s^-}\E{\sigma_2-\sigma_1}\quad=\quad s^+\E{\sigma_2-\sigma_1}-\ell\sqrt{s^+}\E{\sigma_2-\sigma_1}\quad=\quad k,
\]
hence we get that $\pr{k\not\in\left[\sigma_{s^-},\sigma_{s^+}\right]}\to0$ as $\ell\to\infty$.

\item We know that $\eps k\ge J$ and that $\frac{4k}{\E{\sigma_2-\sigma_1}}\asymp \eps k\gtrsim1$, hence for any fixed $\ell$ we have $s^+-s^-=\ell\sqrt{\ell+\frac{4k}{\E{\sigma_2-\sigma_1}}}\asymp\sqrt{\eps k}$ and $\sqrt{s^-}\asymp \sqrt{\ell^2+\frac{4k}{\E{\sigma_2-\sigma_1}}}-\ell\asymp\sqrt{\eps k}$.

\item
\begin{align*}&\pr{d_T(\rho,X_k)\ge\nu k+C\sqrt{\eps k},\text{ }k\in\left[\sigma_{s^-},\sigma_{s^+}\right]}\\
&\le\quad\pr{\varphi_{s^-}\ge\nu\sigma_{s^-}+\frac13C\sqrt{\eps\sigma_{s^-}}}\quad+\quad\pr{\varphi_{s^+}-\varphi_{s^-}\ge\frac13C\sqrt{\eps k}}\\
&\qquad+\quad\pr{\sup_{i\in\left[\sigma_{s^-},\sigma_{s^+}\right]}d_T(\rho,X_i)\ge\varphi_{s^-}+\frac23C\sqrt{\eps k};\text{ }\varphi_{s^+}\le\varphi_{s^-}+\frac13C\sqrt{\eps k}}.
\end{align*}
\item By the statement for the subsequence $(X_{\sigma_i})$ we know that
$\pr{\varphi_{s^-}\ge\nu\sigma_{s^-}+\frac13C\sqrt{\eps\sigma_{s^-}}}\to0$ as $C\to\infty$.

\item $\varphi_{s^+}-\varphi_{s^-}$ has mean $(s^+-s^-)\E{\varphi_2-\varphi_1}\asymp(s^+-s^-)\asymp\sqrt{\eps k}$ and variance $(s^+-s^-)\vr{\varphi_2-\varphi_1}\lesssim(s^+-s^-)\asymp\sqrt{\eps k}$ hence
$\pr{\varphi_{s^+}-\varphi_{s^-}\ge\frac13C\sqrt{\eps k}}\to0$ as $C\to\infty$.

\item By noting that the walk has to backtrack at least $\frac13C\sqrt{\eps k}$ levels after first hitting level $\varphi_{s^-}+\frac23C\sqrt{\eps k}$ we get that \[
\pr{\sup_{i\in\left[\sigma_{s^-},\sigma_{s^+}\right]}d_T(\rho,X_i)\ge\varphi_{s^-}+\frac23C\sqrt{\eps k};\text{ }\varphi_{s^+}\le\varphi_{s^-}+\frac13C\sqrt{\eps k}}\hspace{2cm}\]
\[\hspace{5cm}\quad\lesssim\quad\delta^{\frac13C\sqrt{\eps k}}\quad\lesssim\quad\left(\delta^{\frac13\sqrt{\eps}}\right)^C.
\]
This $\to0$ as $C\to\infty$.

\item
\begin{align*}
&\pr{d_T(\rho,X_k)\le\nu k-C\sqrt{\eps k},\text{ }k\in\left[\sigma_{s^-},\sigma_{s^+}\right]}\quad\le\quad\pr{\varphi_{s^-}\le\nu\sigma_{s^+}-C\sqrt{\eps\sigma_{s^-}}}\\
&\le\quad\pr{\varphi_{s^-}\le\nu\sigma_{s^-}-\frac12C\sqrt{\eps\sigma_{s^-}}}\quad+\quad\pr{\nu\sigma_{s^-}-\frac12C\sqrt{\eps\sigma_{s^-}}<\nu \sigma_{s^+}-C\sqrt{\eps\sigma_{s^-}}}.
\end{align*}
\item By the result for the subsequence $X_{\sigma_i}$ we know that $\pr{\varphi_{s^-}\le\nu\sigma_{s^-}-\frac12C\sqrt{\eps\sigma_{s^-}}}\to0$ as $C\to\infty$.

\item $\nu\left(\sigma_{s^+}-\sigma_{s^-}\right)$ has mean and variance $\asymp s^+-s^-\gtrsim1$, hence\\ $\pr{\nu\left(\sigma_{s^+}-\sigma_{s^-}\right)\ge u(s^+-s^-)}\to0$ as $u\to\infty$.

$\eps\sigma_{s^-}$ has mean and variance $\asymp s^-\gtrsim1$, hence $\pr{\sqrt{\eps\sigma_{s^-}}\le\beta\sqrt{s^-}}\to0$ as $\beta\to0$.

Also $s^+-s^-\asymp\sqrt{s^-}$. Together these give that $\pr{\nu\left(\sigma_{s^+}-\sigma_{s^-}\right)>\frac12C\sqrt{\eps\sigma_{s^-}}}\to0$ as $C\to\infty$.
\end{itemize}

For the bound on the $\sup$:
\begin{itemize}
\item
\begin{align*}
&\pr{\sup_{s:\,s\leq t}d_T(\rho,X_s)>\nu t+ 2C\sqrt{\eps t}}\\
&\le\pr{d_T(\rho,X_t)>\nu t+ C\sqrt{\eps t}}+ \sum_{s<t}\pr{d_T(\rho,X_s)>\nu t+ 2C\sqrt{\eps t},\text{ }d_T(\rho,X_t)<\nu t+ C\sqrt{\eps t}}
\end{align*}
\item We already know that the first term $\to0$ as $C\to\infty$.
\item In the sum each term is $\le\delta^{C\sqrt{\eps t}}$, so the whole sum is $\le t\delta^{C\sqrt{\eps t}}$. This $\to0$ as $C\to\infty$.\qedhere
\end{itemize}
\end{proof}

\begin{proof}[Proof of Prop~\ref{pro:aux_entropy}.]
This is similar to the proof of \cite[Proposition 3.15]{random_matching}.

For the a.s.\ convergence:
\begin{itemize}
\item Let $Y_i=-\log\frac{\prcond{\xi_{\varphi_i}\in\til\xi}{\xi,T}{\rho}}{\prcond{\xi_{\varphi_{i-1}}\in\til\xi}{\xi,T}{\rho}}$ for $i\ge 1$ and let $Y_0=-\log\prcond{\xi_{\varphi_0}\in\til\xi}{\xi,T}{\rho}$.
We know that variables $(Y_i)_{i\ge2}$ are distributed as $-\log\prcond{X'_{\sigma'_1}\in\xi'}{T',X'}{\rho}$ (whether or not we are conditioning on $T_0$). Hence by the ergodic theorem
\[
\frac{-\log\prcond{\xi_{\varphi_k}\in\til\xi}{T,\xi}{\rho}}{k}=\frac{\sum_{i=0}^kY_i}{k}\to\E{-\log\prcond{X'_{\sigma'_1}\in\xi'}{T',X'}{\rho}}=:\til{h}\quad\text{a.s.\ }.
\]
\item $\frac{\varphi_k}{k}\to\E{\varphi_2-\varphi_1}$ a.s., hence we get the a.s.\ convergence to $\h$ along the subsequence $(\xi_{\varphi_k})$.
\item $\{\xi_k\in\til\xi\}$ is non-increasing in $k$, hence we get a.s.\ convergence to $\h$ for the full sequence $(\xi_k)$.
\end{itemize}

For the bound on deviations for $\xi_{\varphi_k}$:
\begin{itemize}
\item By Lemma~\ref{lem:aux_entropy_var_bound} and by the assumption $k\ge\frac{(K\log b(R))^2}{\V}$ we have $$\vrc{-\log\prcond{\xi_{\varphi_k}\in\til\xi}{\xi,T}{\rho}}{T_0}\quad\lesssim\quad k\V+(K\log b(R))^2\quad\lesssim\quad k\V.$$
\item Applying Chebyshev we get
\begin{align*}
&\prcond{\left|-\log \prcond{\xi_{\varphi_k}\in \til{\xi}}{T,\xi}{}-\h\varphi_k\right|>C\sqrt{k\V}}{T_0}{}\\
&\le\quad\frac{2\vrc{-\log \prcond{\xi_{\varphi_k}\in \til{\xi}}{T,\xi}{}-\h\varphi_k}{T_0}}{\left(C\sqrt{k\V}-\left|\econd{-\log \prcond{\xi_{\varphi_k}\in \til{\xi}}{T,\xi}{}}{T_0}-\econd{\h\varphi_k}{T_0}\right|\right)^2}
\end{align*}
We know from above that $\vrc{-\log \prcond{\xi_{\varphi_k}\in \til{\xi}}{T,\xi}{}}{T_0}\lesssim k\V$.\\
Also $\vrc{\h\varphi_k}{T_0}=\vrc{\varphi_k}{T_0}{\h}^2\lesssim k\V$.
\footnote{By Corollary~\ref{cor:varphi_sigma_E_Var_bounds} we know that $\vrc{\varphi_i-\varphi_{i-1}}{T_0}\le\econd{(\varphi_i-\varphi_{i-1})^2}{T_0}\lesssim1$ for $i=2,...,k$ and that $\vrc{\varphi_1}{T_0}=\vrc{\varphi_1-\varphi_0}{T_0}\le\econd{(\varphi_1-\varphi_0)^2}{T_0}\lesssim1$. By Lemma~\ref{lem:indep_between_regenerations} we know that $(\varphi_i-\varphi_{i-1})_{i=2}^k$ and $\varphi_1$ are all independent, hence
$\vrc{\varphi_k}{T_0}= \vrc{\varphi_1}{T_0}+\sum_{i=2}^k\vrc{\varphi_i-\varphi_{i-1}}{T_0}\lesssim k.$
By definition ${\h}^2\lesssim\V$.}
Together these give
\[
\vrc{-\log \prcond{\xi_{\varphi_k}\in \til{\xi}}{T,\xi}{}-\h\varphi_k}{T_0}\quad\lesssim\quad k\V.
\]
To bound the denominator note that
\begin{align*}\econd{-\log \prcond{\xi_{\varphi_k}\in \til{\xi}}{T,\xi}{}}{T_0}\quad&=\quad(k-1){\h}{\E{\varphi_2-\varphi_1}}+\econd{Y_0+Y_1}{T_0}\quad\text{and}\\
\econd{\h\varphi_k}{T_0}\quad&=\quad\h\left((k-1)\E{\varphi_2-\varphi_1}+\econd{\varphi_1}{T_0}\right).
\end{align*}
The same way as we proved the bound on $\econd{(Y_0+Y_1)^2}{T_0}$ for Lemma~\ref{lem:aux_entropy_var_bound} we get $$\econd{Y_0+Y_1}{T_0}\quad\lesssim\quad K\log b(R)\quad\lesssim\quad\sqrt{k\V}.$$

By Corollary~\ref{cor:varphi_sigma_E_Var_bounds} we know that
\[
\econd{\varphi_1}{T_0}\lesssim K+1\lesssim\sqrt{k\V},
\]
hence for any sufficiently large $C$ we have
\[
\left(C\sqrt{k\V}-\left|\econd{-\log \prcond{\xi_{\varphi_k}\in \til{\xi}}{T,\xi}{}}{T_0}-\econd{\h\varphi_k}{T_0}\right|\right)^2\quad\ge\quad\frac{C^2}{2}k\V.
\]
This then gives
\[
\prcond{\left|-\log \prcond{\xi_{\varphi_k}\in \til{\xi}}{T,\xi}{}-\h\varphi_k\right|>C\sqrt{k\V}}{T_0}{}\quad\lesssim\quad\frac{2}{C^2}.
\]
This $\to0$ as $C\to\infty$. We established the deviation bound for the sequence $(\xi_{\varphi_k})_{k\ge1}$.
\end{itemize}

For the bound on deviations for the whole sequence:
\begin{itemize}
\item Fix any $k$ and let $s^-=\left(\frac{\sqrt{\ell^2+\frac{4(k-K)}{\E{\varphi_2-\varphi_1}}}-\ell}{2}\right)^2$ and $s^+=\left(\frac{\sqrt{\ell^2+\frac{4(k-K)}{\E{\varphi_2-\varphi_1}}}+\ell}{2}\right)^2$ where $\ell$ is a large integer to be chosen later.
Then we have
\begin{align*}
&\prcond{\left|-\log \prcond{\xi_k\in \til{\xi}}{T,\xi}{} -\h k\right|>C\sqrt{k\V}}{T_0}{}\\
&\le\quad\prcond{-\log \prcond{\xi_k\in \til{\xi}}{T,\xi}{} >\h k+C\sqrt{k\V};\text{ }k\in[\varphi_{s^-},\varphi_{s^+}]}{T_0}{}\\
&\quad+\quad\prcond{-\log \prcond{\xi_k\in \til{\xi}}{T,\xi}{}< \h k-C\sqrt{k\V};\text{ }k\in[\varphi_{s^-},\varphi_{s^+}]}{T_0}{}\\
&\quad+\quad\prcond{k\not\in[\varphi_{s^-},\varphi_{s^+}]}{T_0}{}.
\end{align*}
We wish to show that each term on the RHS $\to0$ as $\ell\to\infty$ and $C\to\infty$ (in terms of $\ell$).

\item For any $s\ge2$ we have
\begin{align*}
&\pr{\left|\varphi_s-K-s\E{\varphi_2-\varphi_1}\right|\ge\ell\sqrt{s}\E{\varphi_2-\varphi_1}}\\
&\le\quad\pr{\left|(\varphi_s-\varphi_1)-(s-1)\E{\varphi_2-\varphi_1}\right|\ge\frac12\ell\sqrt{s}\E{\varphi_2-\varphi_1}}\\
&\quad+\quad\pr{\left|\varphi_1-K-\E{\varphi_2-\varphi_1}\right|\ge\frac12\ell\sqrt{s}\E{\varphi_2-\varphi_1}}.
\end{align*}
We know that $(\varphi_s-\varphi_1)$ has mean $(s-1)\E{\varphi_2-\varphi_1}$ and variance $\lesssim(s-1)\asymp (s-1)\E{\varphi_2-\varphi_1}^2$, hence $\pr{\left|(\varphi_s-\varphi_1)-(s-1)\E{\varphi_2-\varphi_1}\right|\ge\frac12\ell\sqrt{s}\E{\varphi_2-\varphi_1}}\to0$ as $\ell\to\infty$ (uniformly).

We also know that $\varphi_1-K=\varphi_1-\varphi_0$ has mean $\asymp1\asymp\E{\varphi_2-\varphi_1}$ and variance $\lesssim1$, hence $\pr{\left|\varphi_1-K-\E{\varphi_2-\varphi_1}\right|\ge\frac12\ell\sqrt{s}\E{\varphi_2-\varphi_1}}\to0$ as $\ell\to\infty$ (uniformly).
\[
s^-\E{\varphi_2-\varphi_1}+\ell\sqrt{s^-}\E{\varphi_2-\varphi_1}\quad=\quad s^+\E{\varphi_2-\varphi_1}-\ell\sqrt{s^+}\E{\varphi_2-\varphi_1}\quad=\quad k-K,
\]
hence we get that $\pr{k\not\in\left[\varphi_{s^-},\varphi_{s^+}\right]}\to0$ as $\ell\to\infty$.

\item By the monotonicity of events $\{\xi_i\in\til\xi\}$ we get that
\begin{align*}
&\prcond{-\log \prcond{\xi_k\in \til{\xi}}{T,\xi}{} >\h k+C\sqrt{k\V};\text{ }k\in[\varphi_{s^-},\varphi_{s^+}]}{T_0}{}\\
&\le\quad\prcond{-\log \prcond{\xi_{\varphi_{s^+}}\in \til{\xi}}{T,\xi}{} >\h \varphi_{s^-}+C\sqrt{\varphi_{s^-}\V}}{T_0}{}\quad\text{and}\\
&\prcond{-\log \prcond{\xi_k\in \til{\xi}}{T,\xi}{}< \h k-C\sqrt{k\V};\text{ }k\in[\varphi_{s^-},\varphi_{s^+}]}{T_0}{}\\
&\le\quad\prcond{-\log \prcond{\xi_{\varphi_{s^-}}\in \til{\xi}}{T,\xi}{}< \h\varphi_{s^+}-C\sqrt{\varphi_{s^-}\V}}{T_0}{}.
\end{align*}
\item For any $\alpha\in(0,1)$ we have
\begin{align*}
&\prcond{-\log \prcond{\xi_{\varphi_{s^+}}\in \til{\xi}}{T,\xi}{} >\h \varphi_{s^-}+C\sqrt{\varphi_{s^-}\V}}{T_0}{}\\
&\le\quad\prcond{-\log \prcond{\xi_{\varphi_{s^+}}\in \til{\xi}}{T,\xi}{} >\h \varphi_{s^+}+\alpha C\sqrt{\varphi_{s^+}\V}}{T_0}{}\\
&\quad+\quad\prcond{\h\left(\varphi_{s^+}-\varphi_{s^-}\right)\ge C\left(\sqrt{\V\varphi_{s^-}}-\alpha\sqrt{\V\varphi_{s^+}}\right)}{T_0}{}.
\end{align*}
We already know that the first term on RHS $\to0$ as $\alpha C\to\infty$.

We also know that $\h\asymp\sqrt{\V}$, that $\pr{\varphi_{s^+}-\varphi_{s^-}\ge u\left(s^+-s^-\right)}\to0$ as $u\to\infty$, and that $\pr{\varphi_{s^-}\le\frac12s^-}=0$ and $\pr{\varphi_{s^+}\ge\frac{1}{2\alpha^2}s^+}\to0$ as $\frac{1}{2\alpha^2}\to\infty$ \footnote{$\h\asymp\sqrt{\V}$ holds by assumption. $\pr{\varphi_{s^+}-\varphi_{s^-}\ge u\left(s^+-s^-\right)}\le\frac{\econd{\varphi_{s^+}-\varphi_{s^-}}{T_0}}{u(s^+-s^-)}=\frac{\E{\varphi_2-\varphi_1}}{u}\to0$ as $u\to\infty$.\\
$\pr{\varphi_{s^+}\ge vs^+}\le\frac{E[\varphi_{s^+}]}{vs^+}\le\frac{K+\E{\varphi_1-\varphi_0}+(s^+-1)\E{\varphi_2-\varphi_1}}{vs^+}\lesssim\frac{K+s^+}{vs^+}\lesssim\frac1v\to0$ as $v\to\infty$. Here we used that $s^+\asymp(k-K)\gtrsim K$, which is true by the assumption that $\log b(R)\gtrsim\h\asymp\sqrt{\V}$.}. Putting these together and using that for any fixed $\ell$ we have $s^+-s^-\asymp\sqrt{s^+}\asymp\sqrt{s^-}$ we get that the second term on the RHS $\to0$ as $C\to\infty$, $\alpha\to0$.

\item Similarly
\begin{align*}&\prcond{-\log \prcond{\xi_{\varphi_{s^-}}\in \til{\xi}}{T,\xi}{}< \h\varphi_{s^+}-C\sqrt{\varphi_{s^-}\V}}{T_0}{}
\\&\le\quad\prcond{-\log \prcond{\xi_{\varphi_{s^-}}\in \til{\xi}}{T,\xi}{}< \h\varphi_{s^-}-\frac12C\sqrt{\varphi_{s^-}\V}}{T_0}{}
\\&\quad+\quad\prcond{\h\left(\varphi_{s^+}-\varphi_{s^-}\right)\ge\frac12C\sqrt{\V\varphi_{s^-}}}{T_0}{}.
\end{align*}

Again both terms on RHS $\to0$ as $C\to\infty$.\qedhere
\end{itemize}
\end{proof}

\begin{proof}[Proof of Lemma~\ref{lem:aux_entropy_var_bound}]
The conditioning on $T_0$ will be assumed throughout, but often dropped from notation.
As in the proof of Proposition~\ref{pro:aux_entropy} let $Y_i=-\log\frac{\prcond{\xi_{\varphi_i}\in\til\xi}{\xi,T}{\rho}}{\prcond{\xi_{\varphi_{i-1}}\in\til\xi}{\xi,T}{\rho}}$ for $i\ge 1$ and let $Y_0=-\log\prcond{\xi_{\varphi_0}\in\til\xi}{\xi,T}{\rho}$. Note that
\[
\frac{\prcond{\xi_{\varphi_i}\in\til\xi}{\xi,T}{\rho}}{\prcond{\xi_{\varphi_{i-1}}\in\til\xi}{\xi,T}{\rho}}=\prcond{\xi_{\varphi_{i}}\in\til\xi}{\xi,T,\xi_{\varphi_{i-1}}\in\til\xi}{\rho}=\prcond{X_{\sigma_i}\in\xi(i)}{(X_t)_{t\ge\sigma_{i-1}},T(X_{\sigma_{i-1}})}{}.
\]
where $\xi(i)$ is the loop-erasure of a random walk $X^i$ on $T(X_{\sigma_{i-1}})$ started from $X_{\sigma_{i-1}}$.

Then
\begin{align}\label{eq:entropy_var_decomposition}
&\vrc{-\log\prcond{\xi_{\varphi_k}\in\til\xi}{\xi,T}{\rho}}{T_0}\quad=\quad
\vrc{\sum_{i=0}^kY_i}{T_0}\\
&=\quad\sum_{i=2}^k\vrc{Y_i}{T_0}\quad+\quad\vrc{Y_0+Y_1}{T_0} \nonumber\\
&\qquad+\quad2\sum_{i=2}^k\sum_{j=i+1}^k\cvc{Y_i}{Y_j}{T_0}\quad+\quad2\sum_{i=2}^k\cvc{Y_i}{Y_0+Y_1}{T_0}.\nonumber
\end{align}
We will now bound each of the terms on the right-hand side.

\textit{Bounding $\sum_{i=2}^k\vrc{Y_i}{T_0}$:}

Note that for each $i\ge2$ we have $(T(X_{\sigma_{i-1}}),X_{\sigma_i},\xi(i))\big|T_0\eqdist(T',X'_{\sigma'_1},\xi')$ where $T'$, $X'$, $\sigma'_1$ are defined as in Definition~\ref{def:Tprime} and $\xi'$ is a loop-erased random walk on $T'$ independently of $X'$.

Therefore the random variables $(Y_i|T_0)_{i\ge 2}$ have the same distribution as 
$-\log\prcond{X'_{\sigma'_1}\in\xi'}{T',X'}{\rho}$.
Hence
\[
\sum_{i=2}^k\vrc{Y_i}{T_0}\quad\le\quad k\V.
\]

\textit{Bounding $\vrc{Y_0+Y_1}{T_0}$:}

Let $A_{K+r}$ denote the set of long-range edges between levels $K+r-1$ and $K+r$ of $T$.
From Lemma~\ref{lem:bound_prob_Xsigma1_equals_x} we know that for any $r\ge1$, $e\in A_{K+r}$ we have $\pr{(X_{\sigma_1-1},X_{\sigma_1})=e}\lesssim(2\delta^2)^{r-1}\pr{\xi_{K+r}=e}$. We also know that  for any $e\in A_K$ we have $\pr{(X_{\sigma_1-1},X_{\sigma_1})=e}\le\pr{\xi_{K}=e}$. Hence

\begin{align*}\vrc{Y_0+Y_1}{T_0}\quad&\le\quad\econd{(Y_0+Y_1)^2}{T_0}\quad=\quad\econd{\left(-\log\prcond{X_{\sigma_1}\in\til{\xi}}{X,T}{\rho}\right)^2}{T_0}
\\&\lesssim\quad\econd{\econd{\sum_{r\ge0}\sum_{e\in A_{K+r}}\prstart{X_{\sigma_1}=e^+}{\rho}\left(-\log\prstart{\xi_{K+r}=e}{\rho}\right)^2}{T}}{T_0}
\\
\lesssim\quad&\econd{\econd{\sum_{r\ge1}(2\delta^2)^{r-1}\sum_{e\in A_{K+r}}\prstart{\xi_{K+r}=e}{\rho}\left(-\log\prstart{\xi_{K+r}=e}{\rho}\right)^2}{T}}{T_0}
\\&\quad+\quad\econd{\econd{\sum_{e\in A_{K}}\prstart{\xi_{K}=e}{\rho}\left(-\log\prstart{\xi_{K}=e}{\rho}\right)^2}{T}}{T_0}.
\end{align*}

We also know that for any $r\ge0$ we have $\sum_{e\in A_{K+r}}\prstart{\xi_{K+r}=e}{\rho}=1$ and $|A_{K+r}|\le b(R)^{K+r+1}$, hence for any realisation of $T$ we have
\[\sum_{e\in A_{K+r}}\prstart{\xi_{K+r}=e}{\rho}\left(-\log\prstart{\xi_{K+r}=e}{\rho}\right)^2\quad\le\quad\log\left(b(R)^{K+r+1}\right)^2\hspace{2cm}\]
\[\hspace{5cm}=\quad\left((K+r+1)\log b(R)\right)^2.\]
This then gives
\[
\vrc{Y_0+Y_1}{T_0}\:\lesssim\:\left(K\log b(R)\right)^2+\sum_{r\ge1}(2\delta^2)^{r-1}\left((K+r)\log b(R)\right)^2\:\lesssim\: \left(K\log b(R)\right)^2.
\]

\textit{Bounding $\sum_{i=2}^k\sum_{j=i+1}^k\cvc{Y_i}{Y_j}{T_0}$:}

We will show that there exist positive constants $a$ and $u$ such that for any $i\ge2$, $j\ge i+u$ we have
\begin{align} \label{eq:cov_bound}
\cvc{Y_i}{Y_j}{T_0}\quad\lesssim\quad\V e^{-a(j-i)}.
\end{align}

Once we established that, we immediately get that $\sum_{i=2}^k\sum_{j=i+1}^k\cvc{Y_i}{Y_j}{T_0}\lesssim k\V$.

In order to prove \eqref{eq:cov_bound} we will show that for some positive constants $a$, $b$ and $u$, for all $i\ge2$, $j\ge i+u$ there exists a random variable $Y_{i,j}$ and an event $B_{i,j}$ satisfying the following properties.
\begin{enumerate}[(i)]
\item \label{item:indep} $B_{i,j}$ and $Y_{i,j}\1_{B_{i,j}}$ are independent of $Y_j$,
\item \label{item:bound_prob} $\prstart{B_{i,j}^c}{T_0}\lesssim e^{-2a(j-i)}$,
\item \label{item:bound_diff} $|Y_i-Y_{i,j}|\1_{B_{i,j}}\le e^{-b(j-i)}$.
\end{enumerate}
Once we have these variables, we can prove \eqref{eq:cov_bound} as follows.
\begin{align*}
\cvstart{Y_i}{Y_j}{T_0}\quad=\quad& \estart{(Y_i-\estart{Y_i}{T_0})(Y_j-\estart{Y_j}{T_0})\1_{B_{i,j}^c}}{T_0}\\ &+\quad\estart{(Y_i-\estart{Y_i}{T_0})(Y_j-\estart{Y_j}{T_0})\1_{B_{i,j}}}{T_0}
\end{align*}
Using Cauchy-Schwarz, \eqref{item:bound_prob} and that $\estart{(Y_j-\estart{Y_j}{T_0})^4}{T_0}\lesssim\V^2$ for all $j$ \footnote{This holds because $Y_j|T_0\eqdist Y'$ and we assumed that $\E{(Y'-\E{Y'})^4}\lesssim\E{Y'}^4$.} we get that
\begin{align*}
&\estart{(Y_i-\estart{Y_i}{T_0})(Y_j-\estart{Y_j}{T_0})\1_{B_{i,j}^c}}{T_0}
\\&\le\quad\sqrt[4]{\estart{(Y_i-\estart{Y_i}{T_0})^4}{T_0}\estart{(Y_j-\estart{Y_j}{T_0})^4}{T_0}\prstart{B_{i,j}^c}{T_0}^2}\quad\lesssim\quad\V e^{-a(j-i)}.
\end{align*}
Using \eqref{item:indep}, Cauchy-Schwarz, \eqref{item:bound_diff} and that $\estart{(Y_j-\estart{Y_j}{T_0})^2}{T_0}\lesssim\V$ for all $j$ we get that
\begin{align*}&\estart{(Y_i-\estart{Y_i}{T_0})(Y_j-\estart{Y_j}{T_0})\1_{B_{i,j}}}{T_0}\quad=\quad\estart{(Y_i-Y_{i,j})(Y_j-\estart{Y_j}{T_0})\1_{B_{i,j}}}{T_0}
\\&\le\quad\sqrt{\estart{(Y_i-Y_{i,j})^2\1_{B_{i,j}}}{T_0}\estart{(Y_j-\estart{Y_j}{T_0})^2}{T_0}}\quad\lesssim\quad \sqrt{\V}e^{-b(j-i)}.\end{align*}
This finishes the proof of \eqref{eq:cov_bound}.

Now let us define $Y_{i,j}$ and $B_{i,j}$ and prove that they have the required properties. As before, let $X^i$ be a random walk on $T(X_{\sigma_{i-1}})$ from $X_{\sigma_{i-1}}$ and let $\xi(i)$ be its loop-erasure. Let $\xi(i,j)$ be the loop-erasure of $X^i$ up to the first time that it hits the level $\varphi_{j-1}$ of $X_{\sigma_{j-1}}$. Let
\begin{align*} &Z_i:=\prcond{X_{\sigma_i}\in\xi(i)}{X,T(X_{\sigma_{i-1}})}{T_0}, \quad Z_{i,j}:=\prcond{X_{\sigma_i}\in\xi(i,j)}{X,T(X_{\sigma_{i-1}})}{T_0},\\& Y_{i,j}:=-\log Z_{i,j} \quad \text{and} \quad
B_{i,j}:=\left\{\varphi_i-\varphi_{i-1}\le j-i ,\text{ }d_g(X_{\sigma_{i-1}},X_{\sigma_i})\le \frac{j-i}{\eps^2}\right\}.\end{align*}
Also note that $Y_i=-\log Z_i$.

Then by Lemma~\ref{lem:tail_bounds} (used for tree $T(X_{\sigma_{i-1}})$), and using that $\sigma_i-\sigma_{i-1}\ge d_g(X_{\sigma_{i-1}},X_{\sigma_i})$ we have
\[
\prstart{B_{i,j}^c}{T_0}\quad\le\quad \prstart{\varphi_i-\varphi_{i-1}\ge j-i}{T_0}+\prstart{d_g(X_{\sigma_{i-1}},X_{\sigma_i})\ge \frac{j-i}{\eps^2}}{T_0}\hspace{2cm}\]
\[\hspace{8cm}\lesssim\quad e^{-c_1(j-i)}+e^{-c_2\frac{j-i}{\eps}}.
\]
This shows \eqref{item:bound_prob}.

Note that $B_{i,j}$ and $Y_{i,j}$ only depend on $(X_{s})_{s\le\sigma_{j-1}}$ and $T\setminus T(X_{\sigma_{j-1}})$, whereas $Y_j$ depends only on $T(X_{\sigma_{j-1}})$ and the walk $(X_s)_{s\ge\sigma_{j-1}}$ on this tree. From Lemma~\ref{lem:indep_between_regenerations} we know that these are independent of each other. This proves \eqref{item:indep}.

Now we turn to proving \eqref{item:bound_diff}. Note that $|Y_i-Y_{i,j}|=|-\log Z_i+\log Z_{i,j}|\le\frac{|Z_i-Z_{i,j}|}{Z_i\wedge Z_{i,j}}$.

Let $A_{i,j}=\{X^i\text{ revisits level }\varphi_{i}\text{ after hitting level }\varphi_{j-1}\}$. Note that
\begin{align*} |Z_i-Z_{i,j}|\quad &\le\quad \prcond{X_{\sigma_i}\in\xi(i),X_{\sigma_i}\not\in\xi(i,j)}{X,T(X_{\sigma_{i-1}})}{T_0}
\\&\quad +\quad \prcond{X_{\sigma_i}\not\in\xi(i),X_{\sigma_i}\in\xi(i,j)}{X,T(X_{\sigma_{i-1}})}{T_0}
\\&\le\quad\prcond{X_{\sigma_i}\in\xi(i),A_{i,j}}{X,T(X_{\sigma_{i-1}})}{T_0}\\&\quad+ \quad\prcond{X_{\sigma_i}\in\xi(i,j),A_{i,j}}{X,T(X_{\sigma_{i-1}})}{T_0}.
\end{align*}

We will show that on event $B_{i,j}$ the first and second terms on the right-hand side of the line above are $\lesssim \delta^{\frac{j-i}{3}}Z_{i}$ and $\lesssim\delta^{\frac{j-i}{3}}Z_{i,j}$ respectively, and that $Z_{i,j}\asymp Z_{i}$. These together immediately imply \eqref{item:bound_diff}.

First note that
\begin{align*}&\prcond{X_{\sigma_i}\in\xi(i,j),A_{i,j}}{X,T(X_{\sigma_{i-1}})}{T_0}\quad\le\quad \prcond{X_{\sigma_i}\in\xi(i,j)}{X,T(X_{\sigma_{i-1}})}{T_0}\delta^{\varphi_{j-1}-\varphi_i}
\\&\hspace{5cm}\le\quad\prcond{X_{\sigma_i}\in\xi(i,j)}{X,T(X_{\sigma_{i-1}})}{T_0}\delta^{j-i-1}\quad\lesssim\quad\delta^{\frac{j-i}{3}}Z_{i,j}.
\end{align*}
Let $e_1,...,e_{\varphi_i-\varphi_{i-1}}$ be the long-range edges leading from $X_{\sigma_{i-1}}$ to $X_{\sigma_i}$ and let $e_0^+=X_{\sigma_{i-1}}$. Then
\begin{align*}\prcond{X_{\sigma_i}\in\xi(i),A_{i,j}}{X,T(X_{\sigma_{i-1}})}{T_0}\quad\le\quad\pr{X^i\text{ hits level }\varphi_{j-1}\text{ and then hits }X_{\sigma_i}}
\end{align*}
\begin{align}\label{eq:bound_inxii_Aij}
\le\quad \sum_{m=0}^{\varphi_{i}-\varphi_{i-1}-1}\left(\prod_{\ell\ne m}\prstart{X^i\text{ hits }e_{\ell+1}^+}{e_\ell^+}\right) \prstart{X^i\text{ hits level }\varphi_{j-1}\text{ and then hits }e_{m+1}^+}{e_m^+}.
\end{align}

By Lemma~\ref{lem:visit_far_then_hit_e} and using that $\varphi_{j-1}-\ell(e_m^+)\ge j-i-1$ we get that
\begin{align*}&\prstart{X^i\text{ hits level }\varphi_{j-1}\text{ and then hits }e_{m+1}^+}{e_m^+}
\\&\lesssim\quad\ell_0\delta^{j-i-1}\prstart{\text{hit }e_{m+1}^+}{e_m^+}\quad+\quad\delta^{j-i-1}e^{-c_3\eps\ell_0}\left(\ell_0+\frac{1}{\eps}\right)
\end{align*}
for some positive constant $c_3$. On event $B_{i,j}$ this is then
\[\lesssim\quad\ell_0\delta^{j-i-1}\prstart{\text{hit }e_{m+1}^+}{e_m^+}\quad+\quad\delta^{j-i-1}e^{-c_3\eps\ell_0}\left(\ell_0+\frac{1}{\eps}\right)\eps^{-1}\Delta^{\frac{j-i}{\eps^2}}\prstart{\text{hit }e_{m+1}^+}{e_m^+}.
\]
Choosing $\ell_0=\frac{j-i}{c_3\eps^3}\log\Delta$ we get that this is
\[
\lesssim\quad(j-i)\frac{1}{\eps^3}\delta^{j-i-1}\prstart{\text{hit }e_{m+1}^+}{e_m^+}\quad\lesssim\quad\delta^{\frac{j-i}{2}}\prstart{\text{hit }e_{m+1}^+}{e_m^+}.\text{ }\footnotemark
\]
\footnotetext{We used here that $j-i$ is at least some sufficiently large constant.}
Plugging this back to \eqref{eq:bound_inxii_Aij} we get that on event $B_{i,j}$
\begin{align*}&\prcond{X_{\sigma_i}\in\xi(i),A_{i,j}}{X,T(X_{\sigma_{i-1}})}{T_0}\quad\lesssim\quad(\varphi_{i}-\varphi_{i-1})\delta^{\frac{j-i}{2}}\prod_{\ell}\prstart{X^i\text{ hits }e_{\ell+1}^+}{e_\ell^+}
\\&\lesssim\quad(j-i)\delta^{\frac{j-i}{2}}\prstart{X^i\text{ hits }X_{\sigma_i}}{}\quad\lesssim\quad(j-i)\delta^{\frac{j-i}{2}}\frac{1}{1-\delta}\prstart{X_{\sigma_i}\in\xi(i)}{}
\\&\lesssim\quad\delta^{\frac{j-i}{3}} \prcond{X_{\sigma_i}\in\xi(i)}{X,T(X_{\sigma_{i-1}})}{T_0}\quad\asymp\quad\delta^{\frac{j-i}{3}}Z_{i}.
\end{align*}
We also have (on event $B_{i,j}$) that
\begin{align*}\prcond{X_{\sigma_i}\in\xi(i)}{X,T(X_{\sigma_{i-1}})}{T_0}\quad&\ge\quad(1-\delta)\prcond{X_{\sigma_i}\in\xi(i,j)}{X,T(X_{\sigma_{i-1}})}{T_0},
\\\prcond{X_{\sigma_i}\in\xi(i,j)}{X,T(X_{\sigma_{i-1}})}{T_0}\quad&\ge\quad \prcond{X_{\sigma_i}\in\xi(i),X_{\sigma_i}\in\xi(i,j)}{X,T(X_{\sigma_{i-1}})}{T_0}
\\&\ge\quad\left(1-c_4\delta^{\frac{j-i}{3}}\right)\prcond{X_{\sigma_i}\in\xi(i)}{X,T(X_{\sigma_{i-1}})}{T_0},
\end{align*}
hence
\[
Z_{i}\:=\:\prcond{X_{\sigma_i}\in\xi(i)}{X,T(X_{\sigma_{i-1}})}{T_0}\:\asymp\: \prcond{X_{\sigma_i}\in\xi(i,j)}{X,T(X_{\sigma_{i-1}})}{T_0}\:=\: Z_{i,j}.
\]
This finishes the proof of \eqref{item:bound_diff}.

\textit{Bounding $\sum_{i=2}^k\cvc{Y_i}{Y_0+Y_1}{T_0}$:}

By Cauchy-Schwarz and the previous bounds we have
\begin{align*}\sum_{i=2}^k\cvc{Y_i}{Y_0+Y_1}{T_0}\quad&\le\quad\sqrt{\vrstart{\sum_{i=2}^kY_i}{T_0}\vrstart{Y_0+Y_1}{T_0}}
\\&\lesssim\quad\sqrt{k\V\cdot\left(K\log b(R)\right)^2}\quad\le\quad k\V+\left(K\log b(R)\right)^2.\qedhere
\end{align*}
\end{proof}

\section{Estimating the entropy (Lemma~\ref{lem:main_aux_entropy})}\label{app:proof_of4lemmas}

The proof of Lemma~\ref{lem:main_aux_entropy} follows immediately from the following four results.

\begin{lemma}\label{lem:aux_entropy_one}
Let $G$ be as before. Let $\rho$ be any vertex of $G$ and let $T$ be any realisation of the random quasi-tree corresponding to $G$, rooted at $\rho$. Let $v$ be any vertex of $G$ such that the ball of radius $\frac{R}{2}$ around $v$ is contained in the ball of radius $R$ around $\rho$ (the balls are in graph distance). Let $X$ and $\til{X}$ be independent random walks on $T$ from $v$ and let $\tau_1$ be the first time when $X$ hits level 1 and $\tilde{\tau}_1$ be the first time when $\til{X}$ hits level 1. Let $Y$ be a simple random walk on $G$ from $v$ and let $E$ be a random variable taking values on $\Znonneg$ such that $\prcond{E=k}{E\ge k,Y_k=u}{v}=\frac{\eps}{\deg(u)+\eps}$ for all $k$ and all $u$. Let $\left(\til{Y},\til{E}\right)$ be an independent copy of $\left(Y,E\right)$. Let $b\in\Zpos$.

Assume the following three properties hold.
\begin{enumerate}[(i)]
\item\label{assump:h_b_aux} There exists $h_b$ such that for every choice of $v$ we have \[
\econd{\left(-\log\prcond{Y_E=\til{Y}_{\til{E}}}{Y}{v}\right)^b}{v}\quad\asymp\quad h_b.
\]
\item\label{assump:ball_growth_aux} The size of all $R$-balls in $G$ is upper bounded by a $b(R)$ satisfying
\[
\left(\log b(R)\right)^b\left(1-\frac{\eps}{\Delta+\eps}\right)^{\frac{R}{2}}\ll h_b.
\]
\item\label{assump:R_aux} $R\gg\frac1\eps$.
\end{enumerate}
Then we have
\begin{align}\label{eq:aux_entropy_one_result}
\econd{\left(-\log\prcond{X_{\tau_1}=\til{X}_{\til{\tau}_1}}{X,T}{v}\right)^b}{T,v}\quad\asymp\quad h_b.
\end{align}
(Here $h_b$ and $b(R)$ can both depend on $n$.)
\end{lemma}

\begin{proof}
Let $X^a$ be a random walk on $T^a$. For a vertex $u$ let $Y^{(u)}$ denote a walk that starts from $u$ and has transition probabilities as $Y$. Let us define $X^{(u)}$ and $X^{a,(u)}$ analogously. First we will compare $Y_E$ to $X^a_{\tau^a_1}$ and then we compare $X^a_{\tau^a_1}$ to $X_{\tau_1}$.

Since the ball of radius $\frac{R}{2}$ around $v$ is contained in the $R$-ball of $\rho$ in $T$, we have
\[
\left((Y_{k\wedge E})_{k=1}^{R/2},E\wedge\frac{R}{2}\right)\eqdist \left((X^a_{k\wedge(\tau^a_1-1)})_{k=1}^{R/2},(\tau^a_1-1)\wedge\frac{R}{2}\right).
\]
In particular
\[\sum_{e}\prstart{Y_E=e^{-},E<\frac{R}{2}}{v}\left(-\log\prstart{Y_E=e^{-},E<\frac{R}{2}}{v}\right)^b\]
\begin{align}\label{eq:sum_small_E}
=\quad\sum_{e}\prstart{X^a_{\tau^a_1-1}=e^{-},\tau^a_1-1<\frac{R}{2}}{v}\left(-\log\prstart{X^a_{\tau^a_1-1}=e^{-},\tau^a_1-1<\frac{R}{2}}{v}\right)^b.
\end{align}

We will show that both $H_b(Y^{(v)}_E)$ and $H_b(X^{a,(v)}_{\tau_1})$ are at distance $\ll h_b$ from this sum. Once we established that, it is immediate to see that $H_b(X^{a,(v)}_{\tau_1})\asymp h_b$.

First let us consider $H_b(Y^{(v)}_E)$. By properties \eqref{property:decreasing} and \eqref{property:splitting} in Lemma~\ref{lem:H_b_properties} we get that
\[
H_b(Y^{(v)}_E)\quad\ge\quad\sum_{e}\prstart{Y_E=e^{-},E<\frac{R}{2}}{v}\left(-\log\prstart{Y_E=e^{-},E<\frac{R}{2}}{v}\right)^b
\]
and
\begin{align*}
H_b(Y^{(v)}_E)\quad\le&\quad\sum_{e}\prstart{Y_E=e^{-},E<\frac{R}{2}}{v}\left(-\log\prstart{Y_E=e^{-},E<\frac{R}{2}}{v}\right)^b
\\&+\sum_{e}\prstart{Y_E=e^{-},E\ge\frac{R}{2}}{v}\left(-\log\prstart{Y_E=e^{-},E\ge\frac{R}{2}}{v}\right)^b
\end{align*}
hence
\begin{align*}&\left|H_b(Y^{(v)}_E)-\sum_{e}\prstart{Y_E=e^{-},E<\frac{R}{2}}{v}\left(-\log\prstart{Y_E=e^{-},E<\frac{R}{2}}{v}\right)^b\right|
\\&\le\quad\sum_{e}\prstart{Y_E=e^{-},E\ge\frac{R}{2}}{v}\left(-\log\prstart{Y_E=e^{-},E\ge\frac{R}{2}}{v}\right)^b.
\end{align*}
Note that $\left(Y^{(v)}_E|E\ge k,Y_k=u\right)\eqdist Y^{(u)}_E$, hence\[
\prstart{Y_E=e^{-},E\ge\frac{R}{2}}{v}\quad=\quad\pr{E\ge\frac{R}{2}}\sum_{u}\prcond{Y_{\frac{R}{2}}=u}{E\ge\frac{R}{2}}{v}\prstart{Y_E=e^{-}}{u},
\]
therefore (using properties~\eqref{property:factorising} and \eqref{property:bound_by_set_size} in Lemma~\ref{lem:H_b_properties} and all three assumptions of Lemma~\ref{lem:aux_entropy_one}) we get that
\begin{align*}&\sum_{e}\prstart{Y_E=e^{-},E\ge\frac{R}{2}}{v}\left(-\log\prstart{Y_E=e^{-},E\ge\frac{R}{2}}{v}\right)^b
\\&\lesssim\quad \pr{E\ge\frac{R}{2}}\left(H_b\left(\left(Y^{(v)}_{\frac{R}{2}}\middle|E\ge\frac{R}{2}\right)\right)+\sum_{u}\prcond{Y_{\frac{R}{2}}=u}{E\ge\frac{R}{2}}{v}H_b\left(Y^{(u)}_E\right)\right)
\\&\lesssim\quad\left(1-\frac{\eps}{\Delta+\eps}\right)^{\frac{R}{2}}\left(\left(\log b(R)\right)^b+h_b\right)\quad\ll\quad h_b\quad\asymp\quad H_b\left(Y^{(v)}_E\right).
\end{align*}
In particular it shows that the sum in \eqref{eq:sum_small_E} is $\asymp h_b$.

Now let us proceed to approximating $H_b(X^{a,(v)}_{\tau^a_1})=H_b(X^{a,(v)}_{\tau^a_1-1})$. As before we get that
\begin{align*}&\left|H_b(X^{a,(v)}_{\tau^a_1-1})-\sum_{e}\prstart{Y_E=e^{-},E<\frac{R}{2}}{v}\left(-\log\prstart{Y_E=e^{-},E<\frac{R}{2}}{v}\right)^b\right|
\\&\le\quad\sum_{e}\prstart{X^a_{\tau^a_1-1}=e^{-},\tau^a_1-1\ge\frac{R}{2}}{v}\left(-\log\prstart{X^a_{\tau^a_1-1}=e^{-},\tau^a_1-1\ge\frac{R}{2}}{v}\right)^b
\\&\lesssim\quad \pr{\tau^a_1-1\ge\frac{R}{2}}H_b\left(\left(X^{a,(v)}_{\tau^a_1-1}\middle|\tau^a_1\ge\frac{R}{2}\right)\right)\:\lesssim\:\left(1-\frac{\eps}{\Delta+\eps}\right)^{\frac{R}{2}}\left(\log b(R)\right)^b\:\ll\: h_b.
\end{align*}
This finishes the proof of $H_b(X^{a,(v)}_{\tau_1})\asymp h_b$.

Let $X^{(v)}$ be the restriction of $X^{a,(v)}$ to $T$ (which is indeed distributed as a random walk on $T$). Then $X^{(v)}_{\tau_1}\eqdist\left(X^{a,(v)}_{\tau^a_1}\middle|X^{a,(v)}_{\tau^a_1}\ne\rho^a\right)$. From Lemma~\ref{lem:neverbacktrack} we know that $\prstart{X^a_{\tau^a_1}\ne\rho^a}{v}=1-o(1)$. Together these show that
\begin{align*}&H_b(X^{(v)}_{\tau_1})\quad=\quad\sum_{e}\prcond{X^a_{\tau^a_1}=e^+}{X^a_{\tau^a_1}\ne\rho^a}{v}\left(-\log\prcond{X^a_{\tau^a_1}=e^+}{X^a_{\tau^a_1}\ne\rho^a}{v}\right)^b
\\&=\quad\sum_{e}\frac{\prstart{X^a_{\tau^a_1}=e^+}{v}}{\prstart{X^a_{\tau^a_1}\ne\rho^a}{v}}\left(-\log\frac{\prstart{X^a_{\tau^a_1}=e^+}{v}}{\prstart{X^a_{\tau^a_1}\ne\rho^a}{v}}\right)^b\\&\asymp\quad
\sum_{e}\prstart{X^a_{\tau^a_1}=e^+}{v}\left(-\log\prstart{X^a_{\tau^a_1}=e^+}{v}\right)^b+1
\\&=\quad H_b(X^{a,(v)}_{\tau^a_1})+1-\prstart{X^a_{\tau^a_1}=\rho^a}{v}\left(-\log\prstart{X^a_{\tau^a_1}=\rho^a}{v}\right)^b\quad\asymp\quad H_b(X^{a,(v)}_{\tau^a_1})\quad\asymp\quad h_b
\end{align*}
where the sums are taken over all long-range edges $e$ of $T$ from the $R$-ball of $\rho$.
This finishes the proof.
\end{proof}

\begin{lemma}\label{lem:aux_entropy_two}
Let us consider the setup of Lemma~\ref{lem:aux_entropy_one} and assume that the assumptions also hold for $R=n$, with the same values of $h_b$. Let $\xi$ and $\til{\xi}$ be independent loop-erased random walks on $T$. Then we have
\begin{align}\label{eq:aux_entropy_two_result}
\econd{\left(-\log\prcond{\xi_1=\til{\xi}_1}{\xi,T}{\rho}\right)^b}{T}\quad\asymp\quad\econd{\left(-\log\prcond{X_{\tau_1}=\til{X}_{\til{\tau}_1}}{X,T}{\rho}\right)^b}{T}.
\end{align}
\end{lemma}

\begin{proof}
We first present the proof of the $\gtrsim$ direction, which is very quick. Then we will prove the $\lesssim$ direction in case $R=n$. Finally we use this result to conclude the $\lesssim$ direction for any $R$ satisfying assumptions~\eqref{assump:ball_growth} and \eqref{assump:R} in Lemma~\ref{lem:aux_entropy_one}. To simplify formulae we will sometimes drop the conditioning on $T$ from the notation, but it is assumed throughout.

For any long-range edge $e$ from the $R$-ball of $\rho$ in $T$ we have $\prcond{\xi_1=e}{T}{\rho}\ge(1-\delta)\prcond{X_{\tau_1}=e}{T}{\rho}$, hence
\[\sum_{e}\prcond{\xi_1=e}{T}{\rho}\left(-\log\prcond{\xi_1=e}{T}{\rho}\right)^b\hspace{5cm}\]
\[\hspace{2cm}\gtrsim\quad \sum_{e}\prcond{X_{\tau_1}=e}{T}{\rho}\left(-\log\prcond{X_{\tau_1}=e}{T}{\rho}\right)^b\quad\asymp\quad h_b.\]
This proves the $\gtrsim$ direction.

Now let us assume that $R=n$ and prove the $\lesssim$ direction. In this case the $R$-ball $B$ of $\rho$ in $T$ is the whole graph $G$, so \eqref{eq:aux_entropy_one_result} holds for all vertices $v$ in $B$.

For any $u$ in $B$ and any $k\ge1$ let $A_{k,u}$ be the event that $X$ leaves $B$ and returns to it at least $k$ times and the $k$th return is at vertex $u$. For $k=0$ we take $A_{0,u}$ such that $\pr{A_{0,\rho}}=\1_{u=\rho}$.

Then for any long-range edge $e$ from $B$ we have
\[\prstart{\xi_1=e}{\rho}\quad=\quad\sum_{k\ge0}\sum_{u}\prstart{A_{k,u}}{\rho}\prstart{X_{\tau_1}=e^+}{u}\prstart{\tau_{e^-}=\infty}{e^+}.
\]
Note that $\sum_{k\ge0}\prstart{A_{k,u}}{\rho}=\estart{a_u}{\rho}$ where $a_u=\#(\text{returns of }X\text{ to }B\text{ via }u)+\1_{u=\rho}$. We know that
\begin{itemize}
\item $\sum_{u}\estart{a_u}{\rho}=\E{\#(\text{returns of }X\text{ to }B)}\le\E{\Geompos{1-\delta}}\le\frac{1}{1-\delta},$
\item $\sum_e\prstart{X_{\tau_1}=e^+}{u}=1,$
\item $\sum_e\prstart{X_{\tau_1}=e^+}{u}\left(-\log\prstart{X_{\tau_1}=e^+}{u}\right)^b\asymp h_b,$
\item $1-\delta\le\prstart{\tau_{e^-}=\infty}{e^+}\le1.$
\end{itemize}

This then gives
\begin{align*}&\econd{\left(-\log\prcond{\xi_1=\til\xi_1}{\xi,T}{\rho}\right)^b}{T}\quad=\quad\sum_e\prstart{\xi_1=e}{\rho}\left(-\log\prstart{\xi_1=e}{\rho}\right)^b
\\&=\sum_e\left(\sum_u\estart{a_u}{\rho}\prstart{X_{\tau_1}=e^+}{u}\prstart{\tau_{e^-}=\infty}{e^+}\right)\left(-\log\left(\sum_u\estart{a_u}{\rho}\prstart{X_{\tau_1}=e^+}{u}\prstart{\tau_{e^-}=\infty}{e^+}\right)\right)^b
\\&\le\sum_e\sum_u\estart{a_u}{\rho}\prstart{X_{\tau_1}=e^+}{u}\prstart{\tau_{e^-}=\infty}{e^+}\left(-\log\left(\estart{a_u}{\rho}\prstart{X_{\tau_1}=e^+}{u}\prstart{\tau_{e^-}=\infty}{e^+}\right)\right)^b
\end{align*}
\begin{align}\label{eq:not_return}
&\asymp\quad\sum_u\estart{a_u}{\rho}\sum_e\prstart{X_{\tau_1}=e^+}{u}\prstart{\tau_{e^-}=\infty}{e^+}\left(-\log\prstart{\tau_{e^-}=\infty}{e^+}\right)^b\\
\label{eq:entropy_from_u}
&\quad+\quad\sum_u\estart{a_u}{\rho}\sum_e\left(\prstart{X_{\tau_1}=e^+}{u}\left(-\log\prstart{X_{\tau_1}=e^+}{u}\right)^b\right)\prstart{\tau_{e^-}=\infty}{e^+}\\
\label{eq:entropy_of_au}
&\quad+\quad\sum_u\left(\estart{a_u}{\rho}\left(-\log\estart{a_u}{\rho}\right)^b\right)\sum_e\prstart{X_{\tau_1}=e^+}{u}\prstart{\tau_{e^-}=\infty}{e^+}.
\end{align}
The first term~\eqref{eq:not_return} on the right-hand side is $\lesssim1$ and the second one~\eqref{eq:entropy_from_u} is $\lesssim1$. We will prove that
\begin{align}\label{eq:entropy_au_bound}
\sum_u\estart{a_u}{\rho}\left(-\log\estart{a_u}{\rho}\right)\quad\lesssim\quad h_b.
\end{align}
Once we establish that, it follows that the third term~\eqref{eq:entropy_of_au} is $\lesssim h_b$, hence\\ $\E{\left(-\log\prcond{\xi_1=\til\xi_1}{\xi}{\rho}\right)^b}\lesssim h_b$ as required. Let us proceed to prove~\eqref{eq:entropy_au_bound}.\\
We know that
\[
\estart{a_u}{\rho}-\1_{u=\rho}=\sum_{k\ge1}\prstart{A_{k,u}}{\rho},
\]
hence
\[
\sum_u\estart{a_u}{\rho}\left(-\log\estart{a_u}{\rho}\right)^b\quad\lesssim\quad\sum_{k\ge1}\sum_{u}\prstart{A_{k,u}}{\rho}\left(-\log\prstart{A_{k,u}}{\rho}\right)^b.
\]
We will prove by induction on $k$ that for each $k\ge1$ we have
\[
\sum_{u}\prstart{A_{k,u}}{\rho}\left(-\log\prstart{A_{k,u}}{\rho}\right)^b\quad\lesssim\quad k\delta^k h_b
\]
For any vertex $w$ of $B$ we have $\prstart{A_{1,u}}{w}=\sum_{u}\prstart{X_{\tau_1}=u}{w}\prstart{\tau_{u}<\infty}{u^+}$, hence
\begin{align*}&\sum_{u}\prstart{A_{1,u}}{w}\left(-\log\prstart{A_{1,u}}{w}\right)^b
\\&\asymp\quad\sum_{u}\prstart{X_{\tau_1}=u^+}{w}\left(-\log \prstart{X_{\tau_1}=u^+}{w}\right)^b\prstart{\tau_{u}<\infty}{u^+}
\\&\quad+\quad\sum_{u}\prstart{X_{\tau_1}=u}{w}\prstart{\tau_{u}<\infty}{u^+}\left(-\log\prstart{\tau_{u}<\infty}{u^+}\right)^b
\\&\le\quad\delta\sum_{u}\prstart{X_{\tau_1}=u^+}{w}\left(-\log \prstart{X_{\tau_1}=u^+}{w}\right)^b\:+\:\delta(-\log\delta)\sum_{u}\prstart{X_{\tau_1}=u^+}{w}\:\lesssim\:\delta h_b,
\end{align*}
which is the required bound for $k=1$.\\
Also for all $k\ge1$ we have $\sum_{u}\prstart{A_{k,u}}{w}=\prstart{X\text{ leaves }B\text{ and returns }\ge k\text{ times}}{w}\le\delta^k$.\\
Note that $\prstart{A_{k,u}}{\rho}=\sum_{w}\prstart{A_{k-1,w}}{\rho}\prstart{A_{1,u}}{w}$, hence (assuming the bound for $k-1$) we get
\begin{align*}&\sum_{u}\prstart{A_{k,u}}{\rho}\left(-\log\prstart{A_{k,u}}{\rho}\right)^b\\
&\le\quad\sum_{w}\sum_{u}\prstart{A_{k-1,w}}{\rho}\prstart{A_{1,u}}{w}\left(-\log\left(\prstart{A_{k-1,w}}{\rho}\prstart{A_{1,u}}{w}\right)\right)^b
\\&\asymp\quad\sum_{w}\prstart{A_{k-1,w}}{\rho}\left(-\log\prstart{A_{k-1,w}}{\rho}\right)^b\sum_{u}\prstart{A_{1,u}}{w}\\
&\qquad+\quad\sum_{w}\prstart{A_{k-1,w}}{\rho}\sum_{u}\prstart{A_{1,u}}{w}\left(-\log\prstart{A_{1,u}}{w}\right)^b
\\&\lesssim\quad(k-1)\delta^{k-1}\delta h_b\quad+\quad\delta^{k-1}\delta h_b\quad\asymp\quad k\delta^k h_b.
\end{align*}
Therefore
\[
\sum_u\estart{a_u}{\rho}\left(-\log\estart{a_u}{\rho}\right)\quad\le\quad\sum_{k\ge1}\sum_{u}\prstart{A_{k,u}}{\rho}\left(-\log\prstart{A_{k,u}}{\rho}\right)\quad\lesssim\quad\sum_{k\ge1}k\delta^k h_b\quad\lesssim\quad h_b.
\]
This finishes the proof of the $\lesssim$ direction in~\eqref{eq:aux_entropy_two_result} in case $R=n$.

Now let us consider any $R$ satisfying  assumptions~\eqref{assump:ball_growth} and \eqref{assump:R} in Lemma~\ref{lem:aux_entropy_one}. Let $B$ be the $R$-ball of $\rho$ in $T$. Let $\til{T}$ be a realisation of the random quasi-tree with root $\rho$, with radius $n$, containing $T$ as a subtree. Let $\til{B}$ be the ball of $\rho$ in $\til{T}$. Let $\hat{T}$ be a graph obtained by removing the long-range edges of $\til{T}$ that start from the vertices in $\til{B}\setminus B$ and removing any parts that got disconnected from $\rho$. Note that $T$ is a subgraph of $\hat{T}$. Let $\til{X}$ be a random walk on $\til{T}$ and let $\til{\xi}$ be its loop-erasure. Note that the restriction of $\til{X}$ to $\hat{T}$ is a random walk on $\hat{T}$, but it is possible that it only runs for a finite number of steps (if $\til{\xi}_1\in B\setminus B'$). Let $\hat{X}$ be a random walk on $\hat{T}$ from $\rho$ such that in case $\til{\xi}_1\in B\setminus B'$ it agrees with the restriction of $\til{X}$ to $\hat{T}$ and otherwise it is independent of $\til{X}$. Let $\hat{\xi}$ be the loop-erasure of $\hat{X}$. Note that for any long-range edge $e$ in $\til{T}$ from $B'$ we have
\begin{align*}&\pr{\hat{\xi}_1=e}\quad=\quad\pr{\hat{\xi}_1=e,\til{\xi}_1\not\in B'}+\pr{\hat{\xi}_1=e,\til{\xi}_1\in B'}
\\&=\:\pr{\til{\xi}_1\not\in B'}\prcond{\hat{\xi}_1=e}{\til{\xi}_1\not\in B'}{}+\pr{\til{\xi}_1=e}\:=\:\pr{\til{\xi}_1\not\in B'}\pr{\hat{\xi}_1=e}+\pr{\til{\xi}_1=e},
\end{align*}
hence\[
\pr{\hat{\xi}_1=e}\quad=\quad\frac{\pr{\til{\xi}_1=e}}{1-\pr{\til{\xi}_1\not\in B'}}\quad=\quad\prcond{\til{\xi}_1=e}{\til{\xi}_1\in B'}{}.
\]
This means that $\hat{\xi}_1\eqdist\left(\til{\xi}_1\middle|\til{\xi}_1\in B'\right)$.\\
Let $X$ be the restriction of $\hat{X}$ to $T$, which is a random walk on $T$. Let $\xi$ be the loop-erasures of $X$. Then $\xi=\hat{\xi}$.\\
We also know that $1\ge\pr{\til{\xi}_1\in B'}\ge\pr{\tau_1<R}(1-\delta)\gtrsim1$. Overall this shows that
\begin{align*}\sum_{e\in B'} \prstart{\xi_1=e}{\rho}\left(-\log\prstart{\xi_1=e}{\rho}\right)^b\quad &\asymp\quad\sum_{e\in B'} \prstart{\til{\xi}_1=e}{\rho}\left(-\log\prstart{\til{\xi}_1=e}{\rho}\right)^b
\\&\le\quad\sum_{e\in B} \prstart{\til{\xi}_1=e}{\rho}\left(-\log\prstart{\til{\xi}_1=e}{\rho}\right)^b\quad\asymp\quad h_b.
\end{align*}
This finishes the proof.
\end{proof}

\begin{lemma}\label{lem:aux_entropy_three}
Let $G$ be as before. Let $\rho$ be any vertex of $G$ and let $T$ be any realisation of the random quasi-tree corresponding to $G$, rooted at $\rho$. Let $\xi$ and $\til{\xi}$ be independent loop-erased random walks on $T$. Also let $T'$ be any realisation of the conditioned random quasi-tree corresponding to $G$ as in Definition~\ref{def:Tprime}, rooted at $\rho$. Let $\xi'$ and $\til{\xi}'$ be independent loop-erased random walks on $T'$. Let $b\in\Zpos$. Assume that there exist $h_b$ such that for any choice of $\rho$ and $T$ we have
\begin{align}\label{eq:aux_entropy_three_assumption}
\econd{\left(-\log\prcond{\xi_1=\til{\xi}_1}{\xi,T}{\rho}\right)^b}{T}\quad\asymp\quad h_b.
\end{align}
Then we have
\begin{align}
\econd{\left(-\log\prcond{\xi'_1=\til{\xi}'_1}{\xi',T'}{\rho}\right)^b}{T'}\quad\asymp\quad h_b.
\end{align}
\end{lemma}

\begin{proof}
For any quasi-tree $t$ we have $\left(\xi',\til{\xi}'\middle|T'=t\right)\eqdist\left(\xi,\til{\xi}\middle|T=t\right)$. The result follows from assumption \eqref{eq:aux_entropy_three_assumption}.
\end{proof}

\begin{lemma}\label{lem:aux_entropy_four}
Let $G$ be as before. Let $\rho$ be any vertex of $G$ and let $T$ be any realisation of the random quasi-tree corresponding to $G$, rooted at $\rho$. Let $\xi$ and $\til{\xi}$ be independent loop-erased random walks on $T$.

Also let $T'$, $X'$ and $\xi'$ be as in Definition~\ref{def:Tprime}, with $T'$ rooted at $\rho$. Let $\left(\til{X}',\til{\xi}'\right)$ be an independent copy of $\left(X',\xi'\right)$ given $T'$.

Let $b\in\Zpos$. Assume that there exist $h_b$ such that for any choice of $\rho$ and $T$ we have
\begin{align}\label{eq:aux_entropy_four_assumption}
\econd{\left(-\log\prcond{\xi_1=\til{\xi}_1}{\xi,T}{\rho}\right)^b}{T}\quad\asymp\quad h_b.
\end{align}
Then we have
\begin{align}\label{eq:aux_entropy_four}
\econd{\left(-\log\prcond{X'_{\sigma'_1}\in\til{\xi}'}{X',T'}{\rho}\right)^b}{T'}\quad\asymp\quad\econd{\left(-\log\prcond{\xi'_1=\til{\xi}'_1}{\xi',T'}{\rho}\right)^b}{T'}.
\end{align}
\end{lemma}

\begin{proof} For $r\ge1$ let $A_r(T')$ denote the set of long-range edges of $T'$ between levels $r-1$ and $r$. Note that
\[\E{\left(-\log\prcond{X'_{\sigma'_1}\in\xi'}{X',T'}{\rho}\right)^b}\hspace{7cm}\]
\[\hspace{2cm}=\quad\E{\sum_{r\ge1}\sum_{e\in A_r(T')}\left(-\log\prcond{\xi'_r=e}{T'}{\rho}\right)^b\prcond{X'_{\sigma'_1}=e^+}{T'}{\rho}}.
\]
We will show that the sum of the terms with $r=1$ is of order $\E{\left(-\log\prcond{\xi'_1=\til{\xi}'_1}{\xi',T'}{\rho}\right)^b}$, while the sum of the terms with $r\ge2$ is of strictly smaller order.

From Lemma~\ref{lem:bound_prob_Xprime_sigma1_equals_x} we know that for any $e\in A_1(T')$ we have $\prcond{X'_{\sigma'_1}=e^+}{T'}{\rho}\asymp\prcond{\xi'_1=e}{T'}{\rho}$ and for any $e\in A_r$ ($r\ge1$) we have $\prcond{X'_{\sigma'_1}=e^+}{T'}{\rho}\lesssim(2\delta^2)^{r-1}\prcond{\xi'_r=e}{T'}{\rho}$.
This gives that
\begin{align*}&\E{\sum_{e\in A_1(T')}\left(-\log\prcond{\xi'_1=e}{T'}{\rho}\right)^b\prcond{X'_{\sigma'_1}=e^+}{T'}{\rho}}
\\&\asymp\:\E{\sum_{e\in A_1(T')}\left(-\log\prcond{\xi'_1=e}{T'}{\rho}\right)^b\prcond{\xi'_1=e}{T'}{\rho}}\:=\:\E{\left(-\log\prcond{\xi'_1=\til{\xi}'_1}{\xi',T'}{\rho}\right)^b}.
\end{align*}
We also get that
\[\E{\sum_{r\ge2}\sum_{e\in A_r(T')}\left(-\log\prcond{\xi'_r=e}{T'}{\rho}\right)^b\prcond{X'_{\sigma'_1}=e^+}{T'}{\rho}}\]
\begin{align}\label{eq:bound_r_ge_two_terms}
\lesssim\quad \sum_{r\ge2}(2\delta^2)^{r-1}\E{\sum_{e\in A_r(T')}\left(-\log\prcond{\xi'_r=e}{T'}{\rho}\right)^b\prcond{\xi'_r=e}{T'}{\rho}}.
\end{align}
For a given $T'$ and given $e\in A_r(T')$ let $(e_i)_{i=1}^r$ be the long-range edges leading from $\rho$ to $e$ and let $e_0^+=\rho$. Then we have
\begin{align*}&\left(-\log\prcond{\xi'_r=e}{T'}{\rho}\right)^b\prcond{\xi'_r=e}{T'}{\rho}
\\&=\quad \left(\sum_{i=1}^r-\log\prcond{\xi'_i=e_i}{T',\xi'_{i-1}=e_{i-1}}{\rho}\right)^b\prod_{i=1}^r\prcond{\xi'_i=e_i}{T',\xi'_{i-1}=e_{i-1}}{\rho}
\end{align*}
\begin{align}\label{eq:expand_for_given_e}
\le\quad r^{b-1}\sum_{i=1}^r\left(-\log\prcond{\xi'_i=e_i}{T',\xi'_{i-1}=e_{i-1}}{\rho}\right)^b\prcond{\xi'_i=e_i}{T',\xi'_{i-1}=e_{i-1}}{\rho}.\quad\footnotemark
\end{align}

\footnotetext{We used here that for $r,b\in\Zpos$ and $a_1,...,a_r>0$ we have $\left(\sum_{i=1}^r a_i\right)^b\le r^{b-1}\sum_{i=1}^r a_i^b$. This can be proved by induction on $b$ and repeated use of the rearrangement inequality.}

Note that given $T'$, the process $\xi'$ has the distribution of a loop-erased random walk $\xi$ on $T'$. Hence \[
\prcond{\xi'_i=e_i}{T',\xi'_{i-1}=e_{i-1}}{\rho}\:=\:\prcond{\xi_i=e_i}{T',\xi_{i-1}=e_{i-1}}{\rho}\:=\:\prcond{\xi^{i}_1=e_i}{T'(e_{i-1}^+)}{e_{i-1}^+}
\]
where $\xi^{i}$ is a loop-erased random walk on tree $T'(e_{i-1}^+)$, started from $e_{i-1}^+$.

By assumption \eqref{eq:aux_entropy_four_assumption} and by using that $\xi'$ is a loop-erased random walk on $T'$ we know that for any realisation of $T'$ and for any fixed $e_{i-1}\in A_{i-1}(T')$ we have
\begin{align*}&\econd{\sum_{e_i\in A_{1}(T'(e_{i-1}^+))} \left(-\log\prcond{\xi^{i}_1=e_i}{T'(e_{i-1}^+)}{e_{i-1}^+}\right)^b\prcond{\xi^{i}_1=e_i}{T'(e_{i-1}^+)}{e_{i-1}^+}}{T'}
\\&\asymp\quad\econd{\sum_{e\in A_1(T')}\left(-\log\prcond{\xi'_1=e}{T'}{\rho}\right)^b\prcond{\xi'_1=e}{T'}{\rho}}{T'}.
\end{align*}
Plugging this back into \eqref{eq:bound_r_ge_two_terms} we get that
\begin{align*}&\E{\sum_{r\ge2}\sum_{e\in A_r(T')}\left(-\log\prcond{\xi'_r=e}{T'}{\rho}\right)^b\prcond{X'_{\sigma'_1}=e^+}{T'}{\rho}}
\\&\lesssim\: \sum_{r\ge2}(2\delta^2)^{r-1}r^{b-1}\E{\sum_{e\in A_r(T')}\sum_{i=1}^r\left(-\log\prcond{\xi^{i}_1=e_i}{T'(e_{i-1}^+)}{e_{i-1}^+}\right)^b\prcond{\xi^{i}_1=e_i}{T'(e_{i-1}^+)}{e_{i-1}^+}}
\end{align*}
\begin{align*}&\asymp\quad\sum_{r\ge2}(2\delta^2)^{r-1}r^{b}\E{\sum_{e\in A_1(T')}\left(-\log\prcond{\xi'_1=e}{T'}{\rho}\right)^b\prcond{\xi'_1=e}{T'}{\rho}}
\\&\ll\quad\E{\sum_{e\in A_1(T')}\left(-\log\prcond{\xi'_1=e}{T'}{\rho}\right)^b\prcond{\xi'_1=e}{T'}{\rho}}.
\end{align*}
This finishes the proof.
\end{proof}

\section{Some miscellaneous statements}\label{app:other_remainings_pfs}

\begin{proof}[Proof of Prop~\ref{pro:smaller_eps}]
Let $X$ be a random walk on $G$ and $Y$ be a random walk on $G^*$, starting from the same vertex. Let us couple them such that they move together until $\tauLR$, the first time that $Y$ crosses a long-range edge. Note that for any $t\asymp \tmixtext^G(\theta)\ll\frac1\eps$ we have
\[
\dtv{P^t_X(x,\cdot)}{P^t_Y(x,\cdot)}\quad\le\quad\prstart{\tauLR<t}{x}\quad\ll\quad1,
\]
where the $\ll$ follows from a union bound. Also
\begin{align*}
\dtv{\pi_X(\cdot)}{\pi_Y(\cdot)}\quad&=\quad\frac12\sum_{y}\left|\frac{\deg(y)}{2|E|}-\frac{\deg(y)+\eps}{2|E|+n\eps}\right|\quad=\quad\sum_{y}\frac{\left|2\eps |E|-\eps n\deg(y)\right|}{4|E|(2|E|+n\eps)}\\
&\lesssim\quad\eps\cdot n\cdot\frac{n}{n^2}\quad\asymp\quad\eps\quad\ll\quad1.
\end{align*}
(Here $\deg$ is the degree in $G$, and $E$ is the set of edges of $G$.) Therefore for any $t\asymp \tmixtext^G(\theta)$ we have
\[\left|\dtv{P^t_Y(x,\cdot)}{\pi_Y(\cdot)}-\dtv{P^t_X(x,\cdot)}{\pi_X(\cdot)}\right|\hspace{5cm}\]
\[\hspace{2cm}\le\quad\dtv{P^t_X(x,\cdot)}{P^t_Y(x,\cdot)}+\dtv{\pi_X(\cdot)}{\pi_Y(\cdot)}\quad\ll\quad1.\]
This shows that $Y$ has cutoff if and only if $X$ does.
\end{proof}

\begin{proof}[Proof of Lemma~\ref{lem:tmix_bound_by_tmixlazy}]
From \cite[Theorem 1.4]{mix_hit} we know that
$$\tavetext\left(\frac14\right):=\quad\min\left\{t:\:\max_{x}\dtv{\frac{P^t(x,\cdot)+P^{t+1}(x,\cdot)}{2}}{\pi(\cdot)}\le\frac14\right\}\quad\asymp\quad\tmixtext^{\lazytext}\left(\frac14\right).$$

Let $t=\tavetext\left(\frac14\right)$. Then for any $x$ we have
\begin{align*}\frac34\quad&\le\quad1-\dtv{\frac{P^t(x,\cdot)+P^{t+1}(x,\cdot)}{2}}{\pi(\cdot)}\\
&=\quad\sum_{y}\left(\frac{P^t(x,y)+P^{t+1}(x,y)}{2}\right)\wedge\pi(y)\\
&\le\footnotemark\quad\sum_{y}P^t(x,y)\wedge\pi(y)\quad+\quad\sum_{y}P^{t+1}(x,y)\wedge\pi(y).
\end{align*}

\footnotetext{Here we used that for any $A,B,C\in\R_{\ge0}$ we have $\left(\frac{A+B}{2}\right)\wedge C\le(A\wedge C)+(B\wedge C)$.
}

This shows that for any $x$ we have
$\sum_{y}P^t(x,y)\wedge\pi(y)\ge\frac38$ or $\sum_{y}P^{t+1}(x,y)\wedge\pi(y)\ge\frac38$.

Using that the degrees are bounded and $\pi(y)\asymp\pi(z)$ for all $y$, $z$ we get
\[
\sum_{y}P^{t+1}(x,y)\wedge\pi(y)\quad=\quad\sum_{y}\left(\sum_{z:z\sim y}\frac{1}{\deg(z)}P^{t}(x,z)\right)\wedge\pi(y)\quad\asymp\quad\sum_{z}P^{t}(x,z)\wedge\pi(z),
\]
therefore in either case there exists a $\theta\in(0,1)$ such that
\[
1-\dtv{P^{t}(x,\cdot)}{\pi(\cdot)}\quad=\quad\sum_{y}P^{t}(x,y)\wedge\pi(y)\quad\ge\quad\theta.
\]
This shows that $\tmix{1-\theta}\le t\asymp\tmixtext^{\lazytext}$.
\end{proof}

\begin{proof}[Proof of Lemma~\ref{lem:Pt_upper_bound_lazy}]
By \cite[Theorem 1.2]{RW_on_Eulerian_digraphs} we know that there exists a constant $c_2$ such that for all $a>0$ we have
\begin{align}\label{eq:t_unif_bound}
\tunif{a}\quad\le\quad\frac{c_2}{a}\left((mn)\wedge\frac{m^2}{a}\right)
\end{align}
where $m$ is the number of edges of $G$,
\[
\tunif{a}\quad=\quad\min\left\{t\ge0:\text{ }\max_{u,v}\left|\frac{P^t(u,v)}{\pi(v)}-1\right|\le a\right\},
\]
$P$ is the transition matrix of the walk $X$ and $\pi$ is the corresponding invariant distribution.

Let us apply~\eqref{eq:t_unif_bound} with $a=\frac{\sqrt{c_2}\Delta n}{\sqrt{t}}\ge\frac{\sqrt{c_2}\Delta}{\sqrt{A}}$. We know that $m\le\Delta n$ and $\pi(y)\le\frac{\Delta}{n}$, hence we get that
\[
\tunif{a}\quad\le\quad\frac{c_2}{a}\left(\left(\Delta n^2\right)\wedge\frac{\Delta^2n^2}{a}\right)\quad\le\quad\frac{c_2\Delta^2n^2}{a^2}\quad=\quad t,
\]
therefore
\[
\left|\frac{P^t(x,y)}{\pi(y)}-1\right|\quad\le\quad a,
\]
hence
\[
P^t(x,y)\quad\le\quad(a+1)\pi(y)\quad\le\quad a\left(1+\frac{\sqrt{A}}{\Delta\sqrt{c_2}}\right)\pi(y)\quad\le\quad\left(c_2\Delta^2+\Delta\sqrt{A}\right)\frac{1}{\sqrt{t}}.
\]
This finishes the proof.
\end{proof}

\begin{proof}[Proof of Lemma~\ref{lem:Pt_upper_bound}]
Let $E$ be a renewal process with $\Geompos{\frac12}$ renewal times, independently of $X$. Then $Y_t:=X_{E_t}$ is a lazy simple random walk on $G$ from $x$. Let $P$ be the transition matrix of the chain $X$ and $\pi$ the corresponding invariant distribution.

Using Lemma~\ref{lem:Pt_upper_bound_lazy} and that for any vertex $x$ the transition probability $\prstart{X_{2t}=x}{x}$ is non-increasing in $t$ (see e.g. \cite[Proposition 10.25 (i)]{MTMC_book}), we get that for any $t$ and for any $x$ we have
\begin{align*}\frac{1}{\sqrt{t}}\quad&\gtrsim\quad\prstart{Y_{6t}=x}{x}\quad=\quad\sum_{k=0}^{6t}\pr{E_{6t}=k}P^{k}(x,x)
\\&\ge\quad\sum_{k=t}^{3t}\pr{E_{6t}=2k}P^{2k}(x,x)\quad\ge\quad\pr{E_{6t}\ge2t,E_{6t}\text{ even}}P^{2t}(x,x).
\end{align*}
We know that $\pr{E_{6t}\ge2t}=1-o(1)$ and $\pr{E_{6t}\text{ even}}\asymp1$,
hence $\pr{E_{6t}\ge2t,E_{6t}\text{ even}}\gtrsim1$, and so
\[
\frac{1}{\sqrt{t}}\quad\gtrsim\quad P^{2t}(x,x).
\]
Then using reversibility, Cauchy-Schwarz and that $\pi(z)\asymp\frac1n$ for all $z$, we get that for any vertices $x$ and $y$ and for any $t$ we have
\[
P^{2t}(x,y)\quad\asymp\quad\frac1n\frac{P^{2t}(x,y)}{\pi(y)}\quad\le\quad\frac1n\sqrt{\frac{P^{2t}(x,x)}{\pi(x)}\frac{P^{2t}(y,y)}{\pi(y)}}\quad\lesssim\quad\frac{1}{\sqrt{t}}.
\]
Then we also get that
\[
P^{2t+1}(x,y)\quad=\quad\sum_{z:\:z\sim y}P^{2t}(x,z)P(z,y)\quad\lesssim\quad\frac{1}{\sqrt{t}}.
\]
This finishes the proof.
\end{proof}

\begin{proof}[Proof of Lemma~\ref{lem:hit_hitlazy_comparison}]
For any time $t$, any state $x$, and any subset $A$ of the state space we have
\begin{align*}
\prstart{\tau^{Y}_A>2t+t^{\frac23}}{x}\quad&\le\quad\prstart{\tau^{X}_A>t}{x}+\pr{\Bin{2t+t^{\frac23}}{\frac12}\le t},\qquad\text{and}\\
\prstart{\tau^{Y}_A<2t-t^{\frac23}}{x}\quad&\ge\quad\prstart{\tau^{X}_A<t}{x}-\pr{\Bin{2t-t^{\frac23}}{\frac12}\ge t}.
\end{align*}
Using this and that $\hittext^{X^{(n)}}_{\alpha}(\theta)\gg1$ we get the result.
\end{proof}

\section{Estimates for walks on vertex-transitive graphs of polynomial growth}~\label{app:estimates_for_transitive_poly_growth_graphs}


\begin{proof}[Proof of Lemma~\ref{lem:diagonal_lower_bound}]
The proof is similar to the proof of~\cite[Lemma 6.12]{universality_of_fluctuations}.

In what follows let $X$ be a lazy random walk on $G$; let $P=P_{G,\lazytext}$; for a set $A$ let $\tau_A$ be the hitting time of $A$ by $X$; let $B_\ell=B_{G}(o,\ell)$; and for a set $A$ let
\[
\alpha_{A}(x):=\quad\lim_{s\to\infty}\prcond{X_s=x}{\tau_{A}>s}{o}
\]
be the quasi-stationary distribution of $P$ corresponding to $A$.\\
By Cauchy-Schwarz, reversibility and vertex-transitivity we get that for any time $t$ we have
\[
\max_{x}P^{2t+1}(o,x)\le\max_{x}P^{2t}(o,x)\le\max_{x}\sqrt{P^{2t}(o,o)P^{2t}(x,x)}= P^{2t}(o,o)\le2P^{t+1}(o,o).
\]
Hence for any $t$ and $\ell$ we have
\[
P^{t}(o,o)\quad\gtrsim\quad\frac{\prstart{X_t\in B_\ell}{o}}{V(\ell)}\quad\ge\quad\frac{\prstart{\tau_{B_\ell^c}>t}{o}}{V(\ell)}.
\]
Using that for any $v\in B\left(o,\frac12\ell\right)$ we have $B\left(v,\frac12\ell\right)\subseteq B\left(o,\ell\right)$, and using transitivity we get that
\[
\prstart{\tau_{B_\ell^c}>t}{o}\quad\ge\quad\max_{v\in B\left(o,\frac12\ell\right)}\prstart{\tau_{B\left(v,\frac12\ell\right)^c}>t}{o}\quad=\quad\max_{v\in B\left(o,\frac12\ell\right)}\prstart{\tau_{B\left(o,\frac12\ell\right)^c}>t}{v}.
\]
This is then
\[
\ge\quad\prstart{\tau_{B_{\ell/2}^c}>t}{\alpha_{B_{\ell/2}^c}}\quad=\quad\prstart{\tau_{B_{\ell/2}^c}>1}{\alpha_{B_{\ell/2}^c}}^t.
\]
It is known (see e.g.\ \cite[Theorem 3.33]{Aldous_Fill_book}) that $p(r):=\prstart{\tau_{B_{r}^c}>1}{\alpha_{B_{r}^c}}$ satisfies
\[
\frac{1}{1-p(r)}\quad=\quad\estart{\tau_{B_r^c}}{\alpha_{B_r^c}}\quad=\quad\sup_{\substack{f:\:f|_{B_r^c}\equiv0, \\ f\ne\textrm{const}}}\frac{\langle f,f\rangle_{\pi_{B_r}}}{\langle(I-P_{B_r})f,f\rangle_{\pi_{B_r}}},
\]
where $\pi_{B_r}$ is $\pi$ conditioned on $B_r$, i.e. $\pi_{B_r}(x)=\frac{\pi(x)}{\pi(B_r)}$ for $x\in B_r$. (In the current setup it is the uniform distribution on $B_r$.)

Plugging in $f(x)=d_G\left(x,B_r^c\right)$; noting that $f$ is $L^1$-Lipschitz, hence $\langle(I-P_{B_r})f,f\rangle_{\pi_{B_r}}\le1$; and noting that $f(x)\ge\frac12r$ for all $x\in B_{r/2}$ we get that
\[
\frac{1}{1-p(r)}\quad\ge\quad \pi_{B_r}(B_{r/2})\frac14r^2\quad\asymp\quad\frac{V(r/2)}{V(r)}r^2.
\]

By Lemma~\ref{lem:poly_growth_exponent} we get that $V(r)\asymp V(r/2)$ for all $r$, hence $p(r)\ge 1-\frac{c'}{r^2}$, where $c'$ is a positive constant.

Using this for $\ell=c\sqrt{t}$, $r=\ell/2$ we get that
\[P^t(o,o)\quad\gtrsim\quad\frac{1}{V(c\sqrt{t})}\left(1-\frac{c'}{(c\sqrt{t}/2)^2}\right)^t\quad\asymp\quad\frac{1}{V(c\sqrt{t})}\]
where $c'$ is some positive constant. This gives the required result.
\end{proof}

For the proof of the upper bound on $P_{G,\lazytext}^{t}(o,o)$ we will also need the following lemma.

\begin{lemma}\label{lem:Lambda_lower_bound}
Let $G$ be as in Proposition~\ref{pro:transitive_poly_satsifies_assump} and let $P=P_{G,\lazytext}$ be the transition matrix of the lazy simple random walk on $G$. For $s\in(0,1)$ let
\[
\Lambda(s):=\quad\min_{A\ne\emptyset:\:\pi(A)\le s}\left(1-\lambda_1(P_A)\right)
\]
where $P_A$ is the substochastic matrix obtained by restricting $P$ to the set $A$ (i.e. deleting the rows and columns corresponding to $A^c$), and $\lambda_1(P_A)$ is its largest eigenvalue. Then for any $s\in\left(0,\frac12\right]$ we have
\[
\Lambda(s)\quad\ge\quad\frac{1}{2\Delta V^{-1}(2ns)^2}
\]
where $V^{-1}(k)$ denotes the minimal radius $r$ such that $V(r)\ge k$.
\end{lemma}

\begin{proof}
Let $A$ be any subset of the vertices with $m:=|A|\le\frac{n}{2}$ and let $r=V^{-1}(2m)$. (It exists since $2m\le n$.) Let
\[
K(x,y):=\quad\frac{1}{V(r)}\1_{y\in B(x,r)}
\]
Note that $K$ is reversible with invariant distribution $\pi$. Let $K_A$ be the substochastic matrix obtained by restricting $K$ to the set $A$ and let $\lambda_1(K_A)$ be the largest eigenvalue of $K_A$. We will compare this to $\lambda_1(P_A)$.

Note that by \cite[Theorem 3.33]{Aldous_Fill_book} we have
\[
1-\lambda_1(K_A)\quad=1-\prstart{\tau_{A^c}>1}{\alpha_{A^c}}\quad=\prstart{\tau_{A^c}=1}{\alpha_{A^c}}\quad=\sum_{x\in A}\alpha_{A^c}(x)\frac{\left|A^c\cap B(x,r)\right|}{V(r)}\quad\ge\frac12,
\]
where the probabilities concern a random walk with transition matrix $K$. The last $\ge$ is because $V(r)\ge2m=2|A|$ and so $\frac{\left|A^c\cap B(x,r)\right|}{V(r)}\ge\frac12$ for each $x\in A$.

Also note that
\begin{align}
\label{eq:gap_KA_PA}
1-\lambda_1(K_A)\quad=\quad \inf_{\substack{f:\:f|_{A^c}\equiv0,\\f\ne\textrm{const}}}\frac{\Dirform_K(f)}{\langle f,f\rangle_{\pi}},\qquad\text{and}\qquad
1-\lambda_1(P_A)\quad=\quad \inf_{\substack{f:\:f|_{A^c}\equiv0,\\f\ne\textrm{const}}}\frac{\Dirform_P(f)}{\langle f,f\rangle_{\pi}},
\end{align}
where $\Dirform_K$ and $\Dirform_P$ are the Dirichlet forms associated to $K$ and $P$ respectively.

For two vertices $x$ and $y$ of $G$ let $\mathcal{P}_{xy}$ be the set of paths form $x$ to $y$ in $G$ and let $\nu_{xy}$ be the probability measure on $\mathcal{P}_{xy}$ selecting a uniformly chosen geodesic from $x$ to $y$ in $G$.

Then by \cite[Theorem 13.20]{MTMC_book} for any function $f$ we have
\[
\Dirform_P(f)\quad\ge\quad\frac1B\Dirform_K(f)
\]
where $B$ is the congestion ratio defined as
\[
B:=\quad\max_{e:\:Q_P(e)\ne0}\left(\frac{1}{Q_P(e)}\sum_{x,y}Q_K(x,y)\sum_{\Gamma:\:e\in\Gamma\in\mathcal{P}_{xy}}\nu_{xy}(\Gamma)|\Gamma|\right).
\]
Here $Q_P(x,y)=\pi(x)P(x,y)$, $Q_K(x,y)=\pi(x)K(x,y)$.

For a fixed vertex $v$ let
\[
B(v):=\quad\sum_{\substack{e:\:e=(v,\cdot),\\Q_P(e)\ne0}}\left(\frac{1}{Q_P(e)}\sum_{x,y}Q_K(x,y)\sum_{\Gamma:\:e\in\Gamma\in\mathcal{P}_{xy}}\nu_{xy}(\Gamma)|\Gamma|\right).
\]
By vertex-transitivity we have $B(v)\ge B$. For a given $v$ let us consider the contributions to $B(v)$ from different values of $x$ as follows
\[
f(u,v):=\quad\sum_{\substack{e:\:e=(v,\cdot),\\Q_P(e)\ne0}}\left(\frac{1}{Q_P(e)}\sum_{y}Q_K(u,y)\sum_{\Gamma:\:e\in\Gamma\in\mathcal{P}_{uy}}\nu_{uy}(\Gamma)|\Gamma|\right).
\]
Since $f(\gamma u,\gamma v)=f(u,v)$ for any vertices $u$, $v$ and any graph automorphism $\gamma$, and since $G$ is finite, we can apply the mass transport principle \cite[equation (8.4)]{Lyons_Peres_book} to get
\begin{align*}
B(o)=\sum_{u}f(u,o)&=\sum_{u}f(o,u)=\sum_{u}\sum_{\substack{e:\:e=(u,\cdot),\\Q_P(e)\ne0}}\left(\frac{1}{Q_P(e)}\sum_{y}Q_K(o,y)\sum_{\Gamma:\:e\in\Gamma\in\mathcal{P}_{oy}}\nu_{oy}(\Gamma)|\Gamma|\right)\\
&=\sum_{e:\:Q_P(e)\ne0}\left(\frac{1}{Q_P(e)}\sum_{y}Q_K(o,y)\sum_{\Gamma:\:e\in\Gamma\in\mathcal{P}_{oy}}\nu_{oy}(\Gamma)|\Gamma|\right).
\end{align*}
Note that $Q_P(e)=\frac{1}{\Delta n}$ whenever $Q_P(e)\ne0$; $\sum_{y}Q_K(o,y)=\pi(o)=\frac1n$; and any $y$ with $Q_K(o,y)\ne0$ is within graph distance $r$ from $o$, hence we get
\[
B\quad\le\quad B(o)\quad=\quad\Delta n\sum_{y}Q_K(o,y)\sum_{\Gamma\in\mathcal{P}_{oy}}\nu_{oy}(\Gamma)|\Gamma|^2\quad\le\quad\Delta r^2.
\]
Using this bound, the characterisations in \eqref{eq:gap_KA_PA} and the comparison $\Dirform_P(f)\ge\frac1B\Dirform_K(f)$ for all $f$ with $f|_{A^c}\equiv0$, $f\ne\textrm{const}$ we get that
\[
1-\lambda_1(P_A)\quad\ge\quad\frac{1-\lambda_1(K_A)}{\Delta r^2}\quad\ge\quad\frac{1}{2\Delta r^2}.
\]

Applying this bound for all non-empty sets $A$ with $\pi(A)\le s$ gives the required result.
\end{proof}

\begin{proof}[Proof of Lemma~\ref{lem:diagonal_upper_bound}]
In what follows let $P=P_{G,\lazytext}$. Let $\Lambda(s)$ and $V^{-1}(r)$ be defined as in Lemma~\ref{lem:Lambda_lower_bound}.

From~\cite[Corollary 2.1]{MT_bounds_via_spectral_profile} we know that for any $M$ we have
\[
\dinfty{P^t(o,\cdot)}{\pi(\cdot)}\quad\le\quad M\qquad\text{for}\quad t\quad\ge\quad\int_{\frac{4}{n}}^{\frac{4}{M}}\frac{4}{s\Lambda(s)}ds\quad=:I_M.
\]
Let $M=\frac{n}{m}$ where $m\le\frac{1}{16}n$. Then for $t\ge I_M$ we have
\[
\frac{n}{m}\quad\ge\quad\dinfty{P^t(o,\cdot)}{\pi(\cdot)}\quad\ge\quad\left(\frac{P^t(o,o)}{\pi(o)}-1\right)\quad=\quad n\left(P^t(o,o)-\frac1n\right),
\]
hence
\[
P^t(o,o)\quad\le\quad\frac1n+\frac{1}{m}.
\]
Now we proceed to bound $I_M$. Using Lemma~\ref{lem:Lambda_lower_bound} we get that
\[
I_M\quad\le\quad\sum_{j=1}^{\lceil\log_2m\rceil}\int_{\frac{4}{n}2^{j-1}}^{\frac{4}{n}2^{j}}\frac{4}{s\Lambda(s)}ds\quad\le\quad\sum_{j=1}^{\lceil\log_2m\rceil}8\Delta V^{-1}\left(8\cdot2^j\right)^2.
\]
By Lemma~\ref{lem:poly_growth_exponent} we get that the above sum is dominated by the last term and this last term is $\asymp V^{-1}\left(8\cdot2^{\lceil\log_2m\rceil}\right)^2\le V^{-1}(16m)^2$.
\footnote{$r_j:=V^{-1}\left(8\cdot 2^j\right)$ and let $L$ be as in Lemma~\ref{lem:poly_growth_exponent}. Then $2^{J-j}=\frac{8\cdot2^J}{8\cdot2^j}\le\frac{V(r_J)}{V(r_j-1)}\lesssim \left(\lceil\frac{r_J}{r_j-1}\rceil\right)^L\asymp\left(\frac{r_J}{r_j}\right)^L$. This shows that $\sum_{j=1}^{J}r_j^2\lesssim r_J^2\sum_{j=1}^{J}2^{-2(J-j)/L}\asymp r_J^2$.}
Note that we needed the assumption $m\le\frac{1}{16}n$ to ensure that all terms in the sum are defined.

This shows that
\[
I_M\quad\lesssim\quad V^{-1}\left(16m\right)^2.
\]
Let $C$ be a constant such that $I_M\le CV^{-1}\left(16m\right)^2.$ Let us fix any $t\le C\diam{G}^2$ and let $m=\bigg\lfloor\frac{1}{16}V\left(\lfloor\sqrt{t/C}\rfloor\right)\bigg\rfloor$. Then $I_M\le CV^{-1}\left(16m\right)^2\le CV^{-1}\left(V\left(\lfloor\sqrt{t/C}\rfloor\right)\right)^2\le C\lfloor\sqrt{t/C}\rfloor^2\le t$.
\footnote{Here we used that for $r\le\diam{G}$ we have $V^{-1}(V(r))=r$.}
Hence
\[
P^t(o,o)\quad\lesssim\quad\frac1n+\frac{1}{m}\quad\asymp\quad\frac{1}{V(c\sqrt{t})}
\]
for any constant $c>0$.
\end{proof}

\textbf{Acknowledgements} Zsuzsanna Baran was supported by DPMMS EPSRC DTP. Jonathan Hermon's research was supported by an NSERC grant. An\dj ela \v{S}arkovi\'c was supported by DPMMS EPSRC International Doctoral Scholarship and by the Trinity Internal Graduate Studentship.

\bibliography{weighted_case_arXiv_v2.bib}
\bibliographystyle{plain}

\end{document}